\documentclass[11pt]{amsart}
\usepackage[centertags]{amsmath}
\usepackage{amsfonts}
\usepackage{amssymb}
\usepackage{amsthm}
\usepackage{newlfont}
\usepackage{amscd}
\usepackage{mathrsfs}
\usepackage[all,cmtip]{xy}
\usepackage{tikz}
\usepackage[pagebackref=false,colorlinks]{hyperref}

\definecolor{mycolor}{HTML}{F7F8E0}
\definecolor{myorange}{RGB}{245,156,74}
\definecolor{cadetgrey}{rgb}{0.57, 0.64, 0.69}
\definecolor{calpolypomonagreen}{rgb}{0.12, 0.3, 0.17}

\hypersetup{pdffitwindow=true,linkcolor=calpolypomonagreen,citecolor=calpolypomonagreen,urlcolor=calpolypomonagreen}
\usepackage{cite}
\usepackage{fancyhdr}
\usepackage{stmaryrd}

\usepackage[OT2,OT1]{fontenc}
\newcommand\cyr{%
\renewcommand\rmdefault{wncyr}%
\renewcommand\sfdefault{wncyss}%
\renewcommand\encodingdefault{OT2}%
\normalfont
\selectfont}
\DeclareTextFontCommand{\textcyr}{\cyr}

\newcommand{\Mod}[1]{\ (\text{mod}\ #1)}

\usepackage[toc, page] {appendix}

\numberwithin{equation}{section}

\newtheorem{thm}{Theorem}[section]

\newtheorem{cor}[thm]{Corollary}
\newtheorem{lem}[thm]{Lemma}
\newtheorem{prop}[thm]{Proposition}

\newtheorem{assu}[thm]{Assumption}
\newtheorem{conj}[thm]{Conjecture}

\theoremstyle{definition}
\newtheorem{defn}[thm]{Definition}
\newtheorem{choice}[thm]{Choice}
\newtheorem{rem}[thm]{Remark}

\newcommand{\KS}{\mathbf{KS}}

\newcommand{\ks}{\boldsymbol{\kappa}}
\newcommand{\kn}{\widetilde{\boldsymbol{\delta}}}

\newcommand{\sha}{\textrm{{\cyr SH}}}

\begin{document}
\title{Refined applications of Kato's Euler systems for modular forms}
\author{Chan-Ho Kim}
\address{Center for Mathematical Challenges, Korea Institute for Advanced Study, 85 Hoegiro, Dongdaemun-gu, Seoul 02455, Republic of Korea}
\email{chanho.math@gmail.com}
\date{\today}
\thanks{Competing interests: The author declares none.}
\begin{abstract}
We discuss refined applications of Kato's Euler systems for modular forms of \emph{higher} weight at  good primes (with more emphasis on the non-ordinary ones) beyond the one-sided divisibility of the main conjecture and the finiteness of Selmer groups.
These include a proof of the Mazur--Tate conjecture on Fitting ideals of Selmer groups over $p$-cyclotomic extensions
 and a new interpretation of the Iwasawa main conjecture via the non-triviality of Kato's Kolyvagin systems with structural applications.
Some applications to Birch and Swinnerton-Dyer conjecture are also discussed.
\end{abstract}
\maketitle

\setcounter{tocdepth}{1}
\tableofcontents

\section{Introduction} \label{sec:intro}
In this article, we study the more refined nature of Kato's Euler systems and its applications to the Iwasawa theory for modular forms of higher weight at good primes over the cyclotomic tower.
In particular, we focus on 
the Mazur--Tate conjecture on Fitting ideals of Selmer groups \cite{mazur-tate} (Theorem \ref{thm:main_thm_mazur_tate_conj_simple}) and the Iwasawa main conjecture \`{a} la Kato\footnote{We follow Kato's formulation for the \emph{main conjecture} in this article if there is no further specification.} \cite{kato-euler-systems} and Conjecture A \`{a} la Coates--Sujatha \cite{coates-sujatha-fine-selmer} (Theorem \ref{thm:main_thm_main_conj_simple}).
Before stating our main results, we need to introduce some notation.

Let $p$ be an odd  prime  and fix embeddings $\iota_p : \overline{\mathbb{Q}}  \hookrightarrow \overline{\mathbb{Q}}_p$ and $\iota_{\infty} : \overline{\mathbb{Q}}   \hookrightarrow \mathbb{C}$.
Let $f = \sum_{n \geq 1} a_n(f) q^n \in S_k( N, \psi )$ be a newform with character $\psi$ and assume that $(N,p)=1$.
Let $F_\pi = \mathbb{Q}_p( \iota_p(a_n(f)); n \geq 1)$, $\mathcal{O}_\pi$ the ring of integers of $F_\pi$, $\pi$ a uniformizer (which is compatible with $\iota_p$) lying above $p$, and $\mathbb{F}_\pi = \mathcal{O}_\pi/ \pi \mathcal{O}_\pi$ the residue field of $F_{\pi}$.
Write $\overline{f} = \sum_{n\geq 1} a_n(\overline{f}) q^n \in S_k( N, \overline{\psi})$ where $a_n(\overline{f})$ and $\overline{\psi}$ are the complex conjugates of $a_n(f)$  and $\psi$, respectively.
For any field $L$, denote by $G_L$ the absolute Galois group of $L$.

Let 
$\rho_f : G_{\mathbb{Q}} \to \mathrm{Aut}_{F_\pi}(V_f) \simeq \mathrm{GL}_2(F_\pi)$
 be the $\pi$-adic continuous cohomological Galois representation associated to $f$  \cite[$\S$14.10]{kato-euler-systems}.
We say that \textbf{$\rho_f$ has large image} if the image of $G_{\mathbb{Q}(\zeta_{p^\infty})}$ under $\rho_f$ contains a conjugate of $\mathrm{SL}_2(\mathbb{Z}_p)$,
We fix a Galois-stable $\mathcal{O}_\pi$-lattice $T_f \subseteq V_f$ and 
the residual representation is denoted by $\overline{\rho} = \overline{\rho}_f$.
Denote by $N(\overline{\rho})$ the prime-to-$p$ Artin conductor of $\overline{\rho}$ and by $W_f := V_f/T_f$ the discrete Galois module associated to $f$.
Let $\omega$ be the Teichm\"{u}ller character and write $T_{f, i} := T_f \otimes \omega^i$ and $W_{\overline{f},-i} = W_{\overline{f}} \otimes \omega^{-i}$.
Write $M(r) := M \otimes_{\mathbb{Z}_p} \mathbb{Z}_p(r)$ to be the $r$-th Tate twist of a Galois module $M$  for $r \in \mathbb{Z}$. From now on, we assume that $1 \leq r \leq k-1$.

\subsection{Mazur--Tate conjecture}
In \cite{mazur-tate}, Mazur and Tate studied ``Iwasawa theory for elliptic curves over \emph{finite} abelian extensions over $\mathbb{Q}$" and formulated the ``weak" main conjecture for elliptic curves (Conjecture \ref{conj:mazur-tate}).
The \emph{finite layer} Iwasawa theory due to Mazur and Tate is more complicated than the usual Iwasawa theory (over the \emph{infinite} extension) since neither we can directly apply the theory of Iwasawa modules to finite layer objects nor we can ignore ``finite errors".
In addition, the structure of the Iwasawa algebra is much simpler than that of the group rings for the finite layer objects.
This finite layer Iwasawa theory with the use of Fitting ideals is  a part of \emph{refined Iwasawa theory} \cite{kurihara-fitting, kurihara-plms, kurihara-iwasawa-2012}.

In order to consider the most interesting case, we put the following condition in this subsection.
\begin{assu} \label{assu:even-weight-central-critical}
$k$ is even and $r = k/2$, i.e. we focus on the central critical $L$-values.
\end{assu}
Since it is known that the non-central critical $L$-values of modular forms do not vanish, the central critical one is only interesting from the point of view of the ``weak" vanishing conjecture \cite[Conj. 1]{mazur-tate}, which is an equivariant analogue of the rank part of Birch and Swinnerton-Dyer conjecture.

The following conjecture is known as the ``weak" main conjecture \cite[Conj. 3]{mazur-tate}\footnote{The original conjecture concerns general abelian extensions of $\mathbb{Q}$.}, which basically claims Mazur--Tate elements are contained in the Fitting ideals of dual Selmer groups of modular forms over finite extensions.
\begin{conj}[Mazur--Tate] \label{conj:mazur-tate}
Assume $k$ is even and $\overline{\rho}$ is irreducible.
Then 
$$ \theta_{ \mathbb{Q}(\zeta_{p^n}), k/2 }(\overline{f}) \in \mathrm{Fitt}_{\mathcal{O}_\pi[\mathrm{Gal}(\mathbb{Q}(\zeta_{p^n})/\mathbb{Q})]} \left( \mathrm{Sel}( \mathbb{Q}(\zeta_{p^n}), W_{\overline{f}}(k/2) )^\vee \right) $$
where
 $\theta_{\mathbb{Q}(\zeta_{p^n}), k/2}(\overline{f})$ is the Mazur--Tate element of $\overline{f}$ ``at $s=k/2$" over $\mathbb{Q}(\zeta_{p^n})$ normalized by the optimal periods (reviewed in $\S$\ref{subsec:optimal_periods}), and $(-)^\vee$ is the Pontryagin dual.
\end{conj}
It is known that the ``weak" vanishing conjecture is an easy corollary of the ``weak" main conjecture (Corollary \ref{cor:mazur-tate-weak-vanishing}).
When an elliptic curve has good ordinary reduction at $p$, Conjecture \ref{conj:mazur-tate} follows from the standard ingredients of Iwasawa theory \cite[$\S$2]{kim-kurihara}. Notably, the control theorem applies well in the good ordinary setting.
In \cite{kim-kurihara}, the author and Kurihara gave a simple proof of the Mazur--Tate conjecture for elliptic curves with supersingular reduction with $a_p(E) = 0$ by using the $\pm$-Iwasawa theory of elliptic curves. See also \cite{kataoka-thesis}.
In fact, the first significant progress towards the Mazur--Tate conjecture for modular forms (at non-ordinary primes) was made in the work of Emerton--Pollack--Weston \cite{epw2}, which was first announced around 2010, and their argument relies heavily on the $p$-adic local Langlands correspondence \`{a} la Colmez.

Recall that the \textbf{fine Selmer group of $W_{\overline{f}}(k/2)$ over $\mathbb{Q}(\zeta_{p^n})$} 
$$\mathrm{Sel}_0( \mathbb{Q}(\zeta_{p^n}), W_{\overline{f}}(k/2) ) \subseteq \mathrm{Sel}( \mathbb{Q}(\zeta_{p^n}), W_{\overline{f}}(k/2) )$$
is defined by putting the strict local condition  at $p$.
The \emph{simplified} version of the first main theorem of this article is as follows.
\begin{thm}[Theorem \ref{thm:main_thm_mazur_tate_conj}] \label{thm:main_thm_mazur_tate_conj_simple}
Let $p$ be an odd prime and $f \in S_k( N, \psi )$ a newform such that $(N,p)=1$ and $\rho_f$ has large image. 
We assume that 
\begin{itemize}
\item[(FL)] $2 \leq k \leq p-1$ and $k$ is even,
\item[($\alpha \neq \beta$)] $X^2 - a_p(f) X + \psi(p)p^{k-1}$  has distinct roots,
\item[(non-ord)] $\overline{\rho} \vert_{G_{\mathbb{Q}_p}}$  is irreducible, and
\item[(fin)] $\mathrm{Sel}_0( \mathbb{Q}(\zeta_{p^\infty}), W_{\overline{f}}(k/2) )^\vee$ has no non-trivial finite Iwasawa submodule.
\end{itemize}
Then, for $n \geq 1$, Conjecture \ref{conj:mazur-tate} holds
\begin{equation} \label{eqn:mazur-tate-conjecture}
 \left( \theta_{\mathbb{Q}(\zeta_{p^{n}}), k/2}(\overline{f}) \right) 
  \subseteq \mathrm{Fitt}_{\mathcal{O}_\pi[\mathrm{Gal}(\mathbb{Q}(\zeta_{p^n})/\mathbb{Q})]} \left( \mathrm{Sel}(\mathbb{Q}(\zeta_{p^n}), W_{\overline{f}}(k/2))^\vee \right) .
\end{equation}
\end{thm}
Our approach is completely different from that of our former work.
No signed Iwasawa theory is applicable since the notion of finite layer signed objects are unclear  yet when $a_p(f) \neq 0$.
It is known that there are many examples to satisfy the condition (fin) on fine Selmer groups \cite[Ex. 1.21]{kim-kurihara}.

We first study the failure of control theorem of Selmer groups in the non-ordinary case.
We explicitly analyze the error term in the control theorem and relate the Fitting ideals of the error terms with Perrin-Riou's $\theta$-elements, the finite layer analogue of Perrin-Riou's algebraic $p$-adic $L$-functions. 
Perrin-Riou's $\theta$-elements and Mazur--Tate elements are also related by Kato's Euler system divisibility and the Euler system construction of Mazur--Tate elements.

The Euler system construction of Mazur--Tate elements comes from the construction of the \emph{integral} finite layer Coleman maps, which send \emph{integral} Kato's Euler systems to \emph{integral} Mazur--Tate elements.
As an expert can expect, our approach is based on an integral refinement of Perrin-Riou's local Iwasawa theory via integral $p$-adic Hodge theory and Wach modules.

However, a delicate integrality question naturally arises from the generalization and is not precisely detected before.
How to make \emph{three different} integral properties arising from finite layer Coleman maps, Kato's Euler systems, and Mazur--Tate elements compatible (Conjecture \ref{conj:comparison-zeta-elements})?
We partially resolve this problem (Theorem \ref{thm:comparison-zeta-elements}).
For example, the construction of the integral finite layer Coleman maps given in \cite{epw2} is not strong enough to deduce (\ref{eqn:mazur-tate-conjecture}) since a certain constant error term is inevitable in their work (Remark \ref{rem:main-theorem-coleman}).
Our construction removes their constant error term when $2 \leq k \leq p-1$ and precisely controls the error term even for modular forms of arbitrary weight (Theorem \ref{thm:main_thm_mazur_tate_conj} and Remark \ref{rem:main_thm_mazur_tate_conj}).
This shows the novelty of our local Iwasawa theory which is extensively discussed in $\S$\ref{subsec:mazur-tate-optimal-kato-zeta}.

\subsection{Kurihara numbers and the main conjecture}
Let $E/\mathbb{Q}$ be an elliptic curve.
In a series of his papers \cite{kurihara-fitting, kurihara-plms, kurihara-iwasawa-2012, kurihara-munster}, Kurihara introduced numerical invariants $\widetilde{\delta}_{m}(E) \in \mathbb{F}_p$ attached to $E$, which we call \emph{Kurihara numbers}, in the context of \emph{Kolyvagin systems of Gauss sum type} and proved a number of remarkable theorems relating these invariants to the size and structure of the Selmer group of $E$ over $\mathbb{Q}$.  This is another part of \emph{refined Iwasawa theory}.

In \cite{kks}, the author and his collaborators recognized the direct connection between Kato's Euler systems and Kurihara numbers.
We established a simple numerical criterion for the verification of the main conjecture for weight two modular forms at good primes via Kurihara numbers.
The argument of \cite{kks} relied heavily on working with weight two modular forms because the formal groups of modular abelian varieties plays a crucial role in the computation of  the integral image of the dual exponential map. 
In this present work, we extend the work of \cite{kks} to higher weight modular forms by instead using the integral finite layer Coleman maps mentioned earlier.
The construction of the integral finite layer Coleman maps completely replaces the computation of the dual exponential map.
By using the rigidity of $\Lambda$-adic Kolyvagin systems, we also show the converse to the main result of \cite{kks}, so the numerical criterion itself is equivalent to the Iwasawa main conjecture with no Tamagawa defect\footnote{``No Tamagawa defect" means the $\pi$-indivisibility of Fontaine--Perrin-Riou's local Tamagawa numbers \cite{fontaine-perrin-riou}.}. See also \cite{sakamoto-p-selmer} for a different approach.
This equivalence yields two applications.
First, we can use the numerical criterion to verify the Iwasawa main conjecture.
Conversely, by establishing the Iwasawa main conjecture, we obtain the non-triviality of Kato's Kolyvagin systems and the non-vanishing of Kurihara numbers, so we are able to describe the structures of $p$-strict Selmer groups and usual Selmer groups in terms of Kato's Kolyvagin systems and (a higher $p$-power version of) Kurihara numbers, respectively.
In particular, the latter one provides a completely new structural application of the Iwasawa main conjecture. See \cite{mazur-rubin-book, kim-structure-selmer, kim-refined-tamagawa} for details.

Unfortunately, the mod $p$ numerical criterion \emph{must fail} when the Tamagawa defect is non-trivial.
In order to handle this issue, we use congruences of modular forms.
In the ordinary case, we can spread knowledge of the main conjecture of Mazur--Greenberg from one form to all congruent forms if we know the vanishing of $\mu$-invariant of the $p$-adic $L$-function of one form and the one-sided divisibility of the main conjecture \cite{gv, epw}.
The same type of result for Kato's main conjecture is proved in \cite{kim-lee-ponsinet} under the vanishing of $\mu$-invariant of Kato's zeta elements.
We also provide another numerical criterion to verify the $\mu=0$ for Kato's zeta elements.
Since modular forms of minimal level (in the sense of \cite{serre-serre-conjecture}) have no Tamagawa defect \cite{kim-ota}, these two numerical criteria are enough to verify the main conjecture for families of modular forms.

We say that a prime $\ell$ is a \textbf{Kolyvagin prime} if $\ell \nmid Np$, $\ell \equiv 1 \pmod{p}$, and $a_\ell(f) \equiv \psi(\ell) +1 \pmod{\pi}$.
Denote by $\mathcal{N}$ the set of square-free products of Kolyvagin primes.
For each Kolyvagin prime $\ell$, we fix a primitive root $\eta_\ell$ and define the corresponding discrete logarithm
$\mathrm{log}_{\eta_\ell} (a) \in \mathbb{Z}/(\ell - 1)\mathbb{Z}$ by $\eta^{ \mathrm{log}_{\eta_\ell} (a) }_{\ell} \equiv a \Mod{\ell}$. 
Denote by $\overline{ \mathrm{log}_{\eta_\ell} (a) } \in \mathbb{Z}/p\mathbb{Z}$ the mod $p$ reduction of $\mathrm{log}_{\eta_\ell} (a)$.
For $1 \leq r \leq k-1$, $a \in \mathbb{Q}$, and $m \in \mathbb{Q}_{>0}$, $\lambda^{\pm, \mathrm{opt}} (\overline{f}, z^{r-1}; a, m)$ is the optimally normalized modular symbol defined in $\S$\ref{sec:mazur-tate-betti} and $\overline{\lambda^{\pm, \mathrm{opt}} (\overline{f}, z^{r-1}; a, m)}$ its mod $\pi$ reduction.
\begin{defn}
Let $m \in \mathcal{N}$.
\begin{enumerate}
\item
When $i =0$, we define the \textbf{(mod $\pi$) Kurihara numbers for $W_{\overline{f}}(r)$ at $m$} by
\begin{equation*}
\widetilde{\delta}_{\mathbb{Q}(\zeta_m)}(\overline{f},r) := 
{ \displaystyle \sum_{a \in (\mathbb{Z}/m\mathbb{Z})^\times} \left( \prod_{\ell \vert m}  \overline{ \mathrm{log}_{\eta_\ell} (a) } \right)  \cdot \overline{ \lambda^{\pm, \mathrm{opt}} (\overline{f}, z^{r-1}; a, m) }} \in \mathbb{F}_\pi .
\end{equation*}
where the sign of modular symbols coincides with that of $(-1)^{r-1}$.
\item
When $i \neq 0$, we also define the \textbf{(mod $\pi$) Kurihara numbers for $W_{\overline{f},-i}(r)$ at $m$} by
\begin{equation*}
\widetilde{\delta}_{\mathbb{Q}(\zeta_m)}(\overline{f},i,r) := 
{ \displaystyle \sum_{a \in (\mathbb{Z}/mp\mathbb{Z})^\times} \left( \prod_{\ell \vert m} \overline{ \mathrm{log}_{\eta_\ell} (a_{(m)}) } \right) \cdot \overline{\omega^{i}( a_{(p)} )} \cdot \overline{ \lambda^{\pm, \mathrm{opt}} (\overline{f}, z^{r-1}; a, mp) } } \in \mathbb{F}_\pi 
\end{equation*}
where
$a_{(m)} \in (\mathbb{Z}/m\mathbb{Z})^\times$ is the mod $m$ reduction of $a$, 
 $a_{(p)} \in (\mathbb{Z}/p\mathbb{Z})^\times$ is the mod $p$ reduction of $a$, and the sign of modular symbols coincides with that of $(-1)^{r-1} \cdot \omega^i(-1)$.
Here, $\overline{\omega^{i}( a_{(p)} )}$ means the mod $p$ reduction of $\omega^{i}( a_{(p)} )$.
\end{enumerate}
\end{defn}
It is known that the non-vanishing question of Kurihara numbers is well-defined although the values depend on various choices.

The \textbf{collection of Kurihara numbers for $W_{\overline{f}}(r)$ or $W_{\overline{f},-i}(r)$} is defined by 
$$\kn(\overline{f},r) = \left \lbrace \widetilde{\delta}_{\mathbb{Q}(\zeta_m)}(\overline{f},r) \in \mathbb{F}_\pi : m \in \mathcal{N} \right\rbrace \textrm{ or } \ \kn(\overline{f},i,r) = \left \lbrace \widetilde{\delta}_{\mathbb{Q}(\zeta_m)}(\overline{f},i,r) \in \mathbb{F}_\pi : m \in \mathcal{N} \right\rbrace,$$
respectively. We say that either one does not vanish if 
\begin{equation} \label{eqn:kurihara_number}
\left \lbrace
    \begin{array}{ll}
     \widetilde{\delta}_{\mathbb{Q}(\zeta_m)}(\overline{f},r) \neq 0  & \textrm{ when } i=0 , \\ 
     \widetilde{\delta}_{\mathbb{Q}(\zeta_m)}(\overline{f},i,r) \neq 0   & \textrm{ when } i \neq 0
    \end{array}
    \right.
\end{equation}
for some square-free product $m$ of Kolyvagin primes.

We are ready to state the \emph{simplified} form of the second main result of this article, which summarizes
Theorem \ref{thm:main_thm_main_conj}, Theorem \ref{thm:main_thm_mu_zero_conj}, and Corollary \ref{cor:main_thm_main_conj_families}.
\begin{thm} \label{thm:main_thm_main_conj_simple}
Let $p \geq 5$ be a prime  and $f \in S_k( N, \psi )$ a newform such that $(N,p)=1$ and $\rho_f$ has large image.
Let $r$ be an integer with $1 \leq r \leq k-1$ such that $r \neq \frac{k-1}{2}$ (e.g. $k$ is even).
We assume that 
\begin{itemize}
\item[(FL)] $2 \leq k \leq p-1$,
\item[($\alpha\neq\beta$)] $X^2 - a_p(f) X + \psi(p)p^{k-1}$  has distinct roots, and
\item[(``non-CM")] $\overline{\rho} \vert_{G_{\mathbb{Q}_p}}$ is not a direct sum of two characters. 
\end{itemize}
Then the following three statements are equivalent.
\begin{enumerate}
\item $\kn(\overline{f},r)$ or $\kn(\overline{f},i,r)$ does not vanish as in (\ref{eqn:kurihara_number}).
\item Kato's Kolyvagin system for $T_{f,i}(k-r)$ is primitive. 
\item The main conjecture holds for $T_{f,i}(k-r)$ (Conjecture \ref{conj:kato-main-conjecture}) and the local Tamagawa numbers are not divisible by $\pi$.
\end{enumerate}
In this case, the module structure of the $p$-strict Selmer group $\mathrm{Sel}_0(\mathbb{Q}, W_{\overline{f},-i}(r) )$ is determined by Kato's Kolyvagin system for $T_{f,i}(k-r)$.
\end{thm} 
\begin{cor} \label{cor:main_thm_main_conj_simple}
Keep all the assumptions in Theorem \ref{thm:main_thm_main_conj_simple}.
\begin{enumerate}
\item 
 If
\begin{equation} \label{eqn:mu=0}
\sum_{a \in (\mathbb{Z}/p\mathbb{Z})^\times} \overline{ \lambda^{\pm, \mathrm{opt}} (\overline{f}, z^{r-1}; a(1+bp), p^n) } \cdot \overline{ \omega^{i}(a)  } \neq 0
\end{equation}
where the sign of modular symbols coincide with that of $(-1)^{r-1} \cdot \omega^{i}(-1)$
 for some $b \in \mathbb{Z}/p^{n-1}\mathbb{Z}$ with $n \geq 1$, then the $\mu$-part of the main conjecture and Conjecture A of Coates--Sujatha (Conjecture \ref{conj:coates-sujatha}) for $T_{f,i}(k-r)$ hold.
\item 
 If both mod $p$ numerical criteria (\ref{eqn:kurihara_number}) and (\ref{eqn:mu=0}) hold for $T_{f, i}(k-r)$, 
then 
 the main conjecture  and
Conjecture A of Coates--Sujatha
for $T_{f_0, i}(k-r)$ hold
 where $f_0$ runs over all newforms of weight $k$ and level $N_{f_0}$ such that $p \nmid N_{f_0}$ and $\overline{\rho}_{f_0} \simeq \overline{\rho}$.
\end{enumerate}
\end{cor}
\begin{rem} 
When $f$ is $p$-ordinary, Corollary \ref{cor:main_thm_main_conj_simple}.(2) extends to all members in the Hida family of $\overline{\rho}$ \cite{epw}.
Indeed, the $r = \frac{k-1}{2}$ case can be also treated by using Tate twists \cite[$\S$6.5]{rubin-book}.
It is expected that ($\alpha\neq \beta$)  always holds  and the large image assumption implies (``non-CM"). See $\S$\ref{subsec:further_remarks} for the detail.
\end{rem}
Theorem \ref{thm:main_thm_main_conj_simple} establishes the equivalence among three different statements. 
Since the non-triviality of Kolyvagin systems has the application to the structures of $p$-strict Selmer groups \cite{mazur-rubin-book}, Theorem \ref{thm:main_thm_main_conj_simple} provides a new application of the Iwasawa main conjecture to the structures of Selmer groups.
The application to the structure of Bloch--Kato Selmer groups is studied in subsequent works \cite{kim-structure-selmer, kim-refined-tamagawa} in a more general setting.
Theorem \ref{thm:main_thm_main_conj_simple} can also be interpreted as the equivalence between the strong version of ``Kolyvagin's conjecture for Kato's Kolyvagin systems" \cite{kolyvagin-selmer, wei-zhang-mazur-tate} and the Iwasawa main conjecture with no Tamagawa defect.
Kurihara first formulated the numerical criterion (\ref{eqn:kurihara_number}) as a conjecture for elliptic curves with good ordinary reduction \cite[Conj. 1, page 320]{kurihara-iwasawa-2012}.

Corollary \ref{cor:main_thm_main_conj_simple} reduces the main conjecture and Conjecture A of Coates--Sujatha \emph{for families of modular forms} to
 two easily verifiable and effectively computable\footnote{An efficient algorithm for the numerical verification of the Iwasawa main conjecture for elliptic curves is implemented at \href{https://github.com/aghitza/kurihara_numbers}{https://github.com/aghitza/kurihara\_numbers} due to Alexandru Ghitza \cite[Appendix A]{kim-survey}. Extensive numerical examples for the main conjecture are computed in \cite[$\S$8]{kks}.
An algorithm for the Iwasawa invariants for elliptic curves with good reduction is implemented at \href{https://github.com/rpollack9974/Iwasawa-invariants}{https://github.com/rpollack9974/Iwasawa-invariants} due to Robert Pollack and is used in \href{http://www.lmfdb.org}{http://www.lmfdb.org}.}
mod $p$ numerical criteria (\ref{eqn:kurihara_number}) and (\ref{eqn:mu=0}) \emph{for one modular form} by upgrading the one-sided divisibility of the main conjecture.
Theorem \ref{thm:main_thm_main_conj_simple} covers ordinary and non-ordinary cases on equal footing.
In other words, all the applications of the main conjecture and Conjecture A of Coates--Sujatha become numerically verifiable.
These include Birch and Swinnerton-Dyer conjecture (\S\ref{subsec:appendix-applications} and \ref{subsec:appendix-heegner}), Bloch--Kato's Tamagawa number conjecture for modular forms (Corollary \ref{cor:bloch-kato}), and other main conjectures\footnote{e.g. the (equivalent) main conjecture of Mazur--Greenberg \cite[$\S$17]{kato-euler-systems}, the (equivalent) signed main conjectures \cite{lei-loeffler-zerbes_wach}, and the Heegner point main conjecture for elliptic curves  \cite{wei-zhang-mazur-tate, burungale-castella-kim}.}.

\subsection{Mazur--Tate elements and the ``optimal" choice of Kato's zeta elements} \label{subsec:mazur-tate-optimal-kato-zeta}
The common key feature of the proofs of 
Theorem \ref{thm:main_thm_mazur_tate_conj_simple}
and Theorem \ref{thm:main_thm_main_conj_simple}
 is the construction of the integral finite layer Coleman maps, which turn Kato's Euler systems into Mazur--Tate elements over \emph{general} cyclotomic extensions of $\mathbb{Q}$.
Let $\alpha$ and $\beta$ be the roots of $X^2 - a_p(f)X + \psi(p)p^{k-1}$ with $\mathrm{ord}_p(\alpha) \leq  \mathrm{ord}_p(\beta)$
where $\mathrm{ord}_p$ is normalized by $\mathrm{ord}_p(p) = 1$.
It is known than the complex absolute values of $\alpha$ and $\beta$ (under any complex embedding) are $p^{\frac{k-1}{2}}$ \cite[(14.10.5)]{kato-euler-systems}.
\begin{assu} \label{assu:standard} 
As we have seen, the following conditions are assumed throughout this article.
\begin{enumerate}
\item[(LI)] $\rho_f$ has large image.
\item[($\alpha\neq\beta$)] When $f$ is non-ordinary at $p$, we assume $\alpha \neq \beta$.
\item[(``non-CM")]  When $f$ is ordinary at $p$, we assume that $\overline{\rho} \vert_{G_{\mathbb{Q}_p}}$ is not a direct sum of two characters.
\item[(ENV)]  When $r = \frac{k-1}{2}$ is an integer, $\overline{\alpha} \neq \zeta \cdot p^{\frac{k-1}{2}}$ where $\zeta$ is a root of unity.  
\end{enumerate}
It is expected that ($\alpha\neq\beta$) always holds and (non-split) also holds under (LI) (Remark \ref{rem:standard}).
\end{assu}
Let $m$ be a positive integer with $(m, Np)=1$.
In $\S$\ref{sec:canonical-Kato-Euler-systems},
for $\gamma \in T_f$, we construct the \emph{canonical} Kato's Euler system 
$z_{\mathbb{Q}(\zeta_{mp^n}), \gamma}(f ,k-r)$
 and its variation
$\mathfrak{z}_{\mathbb{Q}(\zeta_{mp^n}), \gamma}(f ,k-r)$
 in $\mathrm{H}^1( \mathbb{Q}(\zeta_{mp^n}), T_{f}(k-r) )$
 under the large image assumption.
 If $\gamma \in V_f$, then these cohomology classes lie in 
 $\mathrm{H}^1( \mathbb{Q}(\zeta_{mp^n}), V_{f}(k-r) )$ in general.
 The choice of $\gamma$ corresponds to the choice of the periods in Kato's explicit reciprocity law (Theorem \ref{thm:kato_interpolation_original}).  
Denote by
$$\theta_{\mathbb{Q}(\zeta_{mp^n}), r}(\overline{f})  \in \mathcal{O}_\pi[\mathrm{Gal}(\mathbb{Q}(\zeta_{mp^n})/\mathbb{Q})]$$
 the Mazur--Tate element of $\overline{f}$ ``at $s=r$" over $\mathbb{Q}(\zeta_{mp^n})$  normalized by the optimal periods (reviewed in $\S$\ref{sec:mazur-tate-betti}). 


The technical heart of this article is the following Euler system construction of Mazur--Tate elements.
\begin{thm}[Mazur--Tate elements] \label{thm:main_thm_finite_layer_coleman}
Keep all the statements in Assumption \ref{assu:standard}.
Then there exist the finite layer Coleman map over $\mathbb{Q}(\zeta_{mp^n})$
$$\mathrm{Col}_{\mathbb{Q}(\zeta_{mp^n}), r} : \mathrm{H}^1( \mathbb{Q}(\zeta_{mp^n}) \otimes \mathbb{Q}_p, V_{f}(k-r) ) \to F_\pi [ \mathrm{Gal}(\mathbb{Q}(\zeta_{mp^n})/\mathbb{Q}) ]$$
and the crystalline normalization of the optimal periods via Kato's period integral
$$\delta_{f, \mathrm{cris}} \in V_f$$
such that 
$$
\theta_{\mathbb{Q}(\zeta_{mp^n}), r}(\overline{f}) = \mathrm{Col}_{\mathbb{Q}(\zeta_{mp^n}), r} \left( \mathrm{loc}_p \left( \mathfrak{z}_{\mathbb{Q}(\zeta_{mp^n}) , \delta_{f, \mathrm{cris}}}(f, k-r) \right) \right) \in \mathcal{O}_\pi[\mathrm{Gal}(\mathbb{Q}(\zeta_{mp^n})/\mathbb{Q})] 
$$
where $r$ is an integer with $1 \leq r \leq k-1$.\\
Furthermore, the restriction of $\mathrm{Col}_{\mathbb{Q}(\zeta_{mp^n}), r}$ to $T_f$ becomes integral, i.e. we have
\begin{equation} \label{eqn:full-integrality-coleman}
\mathrm{Col}_{\mathbb{Q}(\zeta_{mp^n}), r}  : \mathrm{H}^1( \mathbb{Q}(\zeta_{mp^n}) \otimes \mathbb{Q}_p , T_{f}(k-r) ) \to \mathcal{O}_\pi [ \mathrm{Gal}(\mathbb{Q}(\zeta_{mp^n})/\mathbb{Q}) ] .
\end{equation}
\end{thm}
\begin{proof}
See $\S$\ref{sec:mazur-tate-etale} for the proof.
Especially,
 the recipe of  $\delta_{f, \mathrm{cris}}$ is given in  $\S$\ref{subsec:crystalline-integral-normalization},
and
 the construction of $\mathrm{Col}_{\mathbb{Q}(\zeta_{mp^n}), r}$ is given in $\S$\ref{subsec:finite-layer-coleman}. 
See also $\S$\ref{sec:canonical-Kato-Euler-systems} for the construction of $\mathfrak{z}_{\mathbb{Q}(\zeta_{mp^n}) , \gamma}(f, k-r)$ for $\gamma \in V_f$.
\end{proof}
Theorem \ref{thm:main_thm_finite_layer_coleman} is the finite layer, integral, and equivariant refinement of the local Iwasawa theory \`{a} la Perrin-Riou \cite{perrin-riou-local-iwasawa, benois-crystalline-duke}.
A substantial amount of this article is devoted to the construction of the finite layer Coleman maps and the study of their integrality property.
Although we follow the well-known philosophy\footnote{``The $p$-adic $L$-functions of modular forms can be obtained by putting Kato's Euler systems into Perrin-Riou's local machinery."} of Perrin-Riou's local Iwasawa theory,
 the technical details become extremely delicate if we aim to obtain this level of generality and precision.
As mentioned before, the main difficulty arises from the comparison among three different integrality properties:
\begin{enumerate}
\item
The uniform integrality of Mazur--Tate elements over general cyclotomic extensions depending only on the \emph{optimal} choice of complex periods $\Omega^{\pm}_{\overline{f}}$. This is the content of $\S$\ref{sec:mazur-tate-betti}.
\item 
The uniform integrality of the canonical Kato's Euler systems over general cyclotomic extensions depending only on the choice of $\gamma \in V_f$. This is the content of $\S$\ref{sec:canonical-Kato-Euler-systems} (e.g. Definition \ref{defn:canonical-kato-euler-systems}).
\item 
The uniform integrality of the finite layer Coleman maps over general cyclotomic extensions depending only on the choice of $\eta \in \mathbf{D}_{\mathrm{cris}}(V_{\overline{f}})$. This is the content of $\S$\ref{sec:mazur-tate-etale} (e.g. Proposition \ref{prop:normalization_local_points_first}).
\end{enumerate}
The constructions of the (non-signed and infinite layer) Coleman maps for non-ordinary higher weight modular forms usually do not care the integrality and cover the $p$-cyclotomic extensions only \cite{perrin-riou-local-iwasawa, kato-euler-systems, lei-loeffler-zerbes_wach, lei-compositio}.
On the other hand, the constructions based on the formal groups of modular abelian varieties \cite{kurihara-invent, kobayashi-thesis, otsuki, ota-thesis, kks, kataoka-thesis} can control the integrality of the finite layer Coleman maps (e.g. \cite[$\S$3]{kurihara-invent})  but these methods do not generalize to higher weight modular forms.
This integrality comparison issue is also noticed in the work of Emerton--Pollack--Weston \cite{epw2}, but they just leave the difference as a constant error term. See Remark \ref{rem:main-theorem-coleman}.

We deal with this integrality issue as follows.
We first choose the complex periods $\Omega^{\pm}_{\overline{f}}$ by the optimal periods in modular symbols (``integral" complex Hodge theory)
 since the optimal integrality of Mazur--Tate elements is essential for the arithmetic applications.
All the complex periods are controlled \emph{up to $p$-adic units}.
Then we choose a suitable $\varphi$-stable lattice in the Dieudonn\'{e} module associated to $V_{\overline{f}}$ via Wach modules (integral $p$-adic Hodge theory) to have the \emph{integral} finite layer Coleman maps.
As a result, the notion of the crystalline normalization $\delta_{f, \mathrm{cris}} \in V_f$ of the optimal periods naturally appears in the Euler system side.
In other words, the choice of $\delta_{f, \mathrm{cris}} \in V_f$ makes the complex periods in the Mazur--Tate element \emph{optimal} and the finite layer Coleman map \emph{integral}  \emph{simultaneously}. 
The cost is that we do not know whether $\delta_{f, \mathrm{cris}} \in T_f$ in general; in other words, the integrality of the canonical Kato's Euler systems associated to the optimal periods is \emph{unclear}.
This obscurity leads us to formulate Conjecture \ref{conj:comparison-zeta-elements} below, and we prove it in many cases (Theorem \ref{thm:comparison-zeta-elements}).

We now need to introduce more notation for the precise description.
Let 
$\Gamma =  \mathrm{Gal}( \mathbb{Q}(\zeta_{p^\infty}) /\mathbb{Q} )$
and $\Lambda_{\mathcal{I}w} = \mathcal{O}_\pi \llbracket \Gamma \rrbracket$ be the Iwasawa algebra.
Decompose 
$\Lambda_{\mathcal{I}w} = \bigoplus_{i=0}^{p-2} \Lambda^{(i)}_{\mathrm{Iw}}
$
where $\Lambda^{(i)}_{\mathrm{Iw}} := \Lambda_{\mathcal{I}w} / (\sigma - \omega^i(\sigma) : \sigma \in \Gamma)$.
Let
$j : \mathrm{Spec}(\mathbb{Q}(\zeta_{p^n})) \to \mathrm{Spec}(\mathbb{Z}[\zeta_{p^n}, 1/p])$
be the natural map.
For $a \in \lbrace 1, 2\rbrace$, the $a$-th Iwasawa cohomology for $T_{f}(k-r)$ over $\Lambda_{\mathcal{I}w}$ is defined by
$$\mathrm{H}^a_{\mathcal{I}w}(   T_{f}(k-r)) := \varprojlim_{n} \mathrm{H}^a_{\mathrm{\acute{e}t}}( \mathrm{Spec}(\mathbb{Z}[\zeta_{p^n}, 1/p]), j_{*}T_{f}(k-r)) $$
and $\mathrm{H}^a_{\mathrm{Iw}}(   T_{f, i}(k-r))$ is defined by the $\omega^i$-isotypic component of $\mathrm{H}^a_{\mathcal{I}w}(  T_{f}(k-r))$
where  $i \in \lbrace 0, 1, \cdots, p-3, p-2 \rbrace$.
We use $\mathrm{H}^a_{\mathcal{I}w}$ for the extension $\mathbb{Q}(\zeta_{p^\infty})/\mathbb{Q}$ and $\mathrm{H}^a_{\mathrm{Iw}}$ for the cyclotomic $\mathbb{Z}_p$-extension $\mathbb{Q}_{\infty}$ of $\mathbb{Q}$.

Let 
$\mathcal{Z}^{\mathrm{prim}}(T_f(k-r)) := \left\langle \mathbf{z}_{\mathbb{Q}, \gamma}(f, k-r) : \gamma \in T_f \right\rangle \subseteq \mathrm{H}^1_{\mathcal{I}w}(T_f(k-r))$
be the module generated by primitive Kato's zeta elements recalled in $\S$\ref{sec:canonical-Kato-Euler-systems}.
Denote by $\mathcal{Z}^{\mathrm{prim}}(T_{f,i}(k-r))  \subseteq\mathrm{H}^1_{\mathrm{Iw}}(T_{f, i}(k-r))$ the $\omega^i$-isotypic component of $\mathcal{Z}^{\mathrm{prim}}(T_f(k-r))$.

Let $T^{\pm}_f$ be the $\mathcal{O}_\pi$-submodule of $T_f$ whose eigenvalue is $\pm 1$ with respect to the complex conjugation  via (\ref{eqn:identification-betti-etale}) below, and $\gamma^{\pm}_{0 }$ a $\mathcal{O}_\pi$-generator of $T^{\pm}_f$, respectively.
Write $\gamma_{0 } = \gamma^+_{0 } + \gamma^-_{0 }$.
Then it is well-known that
$\mathcal{Z}^{\mathrm{prim}}(T_f(k-r))  = \Lambda_{\mathcal{I}w}\mathbf{z}_{\mathbb{Q}, \gamma_0 }(f, k-r)$
and $\mathcal{Z}^{\mathrm{prim}}(T_{f, i}(k-r))  = \Lambda^{(i)}_{\mathrm{Iw}} \mathbf{z}_{\mathbb{Q}, \gamma_0 }(f,i, k-r)$ (Proposition \ref{prop:kato-generated-by-one}).

When $p$ does not divide $N$, we also have the zeta element $\mathbf{z}_{\mathbb{Q}, \delta_{f, \mathrm{cris}} }(f, k-r)$ associated to $\delta_{f, \mathrm{cris}} \in V_f$
described in Theorem \ref{thm:main_thm_finite_layer_coleman} and denote by 
$\mathbf{z}_{\mathbb{Q},  \delta_{f, \mathrm{cris}} }(f, i, k-r)$ the image of $\mathbf{z}_{\mathbb{Q},  \delta_{f, \mathrm{cris}}}(f, k-r)$ in $\mathrm{H}^1_{\mathrm{Iw}}(V_{f, i}(k-r))$.
Since we do not know whether $\delta_{f, \mathrm{cris}} \in T_f$, 
it is even unclear for us whether
 $\mathbf{z}_{\mathbb{Q}, \delta_{f, \mathrm{cris}} }(f, k-r) \in \mathcal{Z}^{\mathrm{prim}}(T_f(k-r))$.
However, since $\delta_{f, \mathrm{cris}}$ is determined by optimal periods, we \emph{believe} that $\mathbf{z}_{\mathbb{Q}, \delta_{f, \mathrm{cris}} }(f, k-r)$ is  the ``most correct" zeta element for the formulation and the verification of the Iwasawa main conjecture, at least when $p \nmid N$.
Since Kato's main conjecture is formulated with $\mathbf{z}_{\mathbb{Q}, \gamma_0 }(f, k-r)$, it is natural to ask the following question.
\begin{conj} \label{conj:comparison-zeta-elements}
Assume that $\rho_f$ has large image (Assumption \ref{assu:standard}).
Let $T_f$ be a Galois-stable $\mathcal{O}_\pi$-lattice of $V_f$ and $r \in \mathbb{Z}$.  
If $p$ does not divide $N$, then
$$
\Lambda^{(i)}_{\mathrm{Iw}}\mathbf{z}_{\mathbb{Q}, \delta_{f, \mathrm{cris}} }(f, i, k-r) = \Lambda^{(i)}_{\mathrm{Iw}}\mathbf{z}_{\mathbb{Q}, \gamma_0 }(f,i, k-r) 
$$
 in $\mathrm{H}^1_{\mathrm{Iw}}(T_{f,i}(k-r))$ for all $i = 0, \cdots , p-2$.
\end{conj}
Conjecture \ref{conj:comparison-zeta-elements} claims that two different integral normalizations of Kato's zeta elements coming from modular symbols and Galois representations should coincide up to units; especially, $\delta_{f, \mathrm{cris}} \in T_f$.
Conjecture \ref{conj:comparison-zeta-elements} allows us to make the precise connection between Theorem \ref{thm:main_thm_finite_layer_coleman} and the main applications (Theorem \ref{thm:main_thm_mazur_tate_conj_simple} and Theorem \ref{thm:main_thm_main_conj_simple}).

We say that
 $\overline{\rho}$ is \textbf{$p$-ordinary} if the restriction of $\overline{\rho}$ to $G_{\mathbb{Q}_p}$ is reducible, 
 and $\overline{\rho}$ is \textbf{$p$-distinguished} if the semi-simplification of the restriction of $\overline{\rho}$ to $G_{\mathbb{Q}_p}$ is not a scalar.
\begin{thm} \label{thm:comparison-zeta-elements}
Suppose that $\overline{\rho}$ is irreducible.
If 
\begin{enumerate}
\item $(N,p)=1$ and $2 \leq k \leq p-1$ or
\item  $\overline{\rho}$ is $p$-ordinary and $p$-distinguished,
\end{enumerate}
then $\delta^+_{f, \mathrm{cris}}$ and $\delta^-_{f, \mathrm{cris}}$ generate $T_f$ over $\mathcal{O}_\pi$. In particular, Conjecture \ref{conj:comparison-zeta-elements} follows under Condition (1) or (2). In this case, the canonical periods are optimal periods.
\end{thm}
\begin{proof}
See $\S$\ref{subsec:canonical_periods}.
\end{proof}
Theorem \ref{thm:comparison-zeta-elements} shows that all three integral properties mentioned above are compatible
 when $2 \leq k \leq p-1$ or the residual representation is $p$-ordinary and $p$-distinguished.
Indeed, the Fontaine--Laffaille range assumption in 
 Theorem \ref{thm:main_thm_mazur_tate_conj_simple} and Theorem \ref{thm:main_thm_main_conj_simple}
can be replaced by Conjecture \ref{conj:comparison-zeta-elements} (with some modification for the latter).
See Theorem \ref{thm:main_thm_mazur_tate_conj}, Theorem \ref{thm:main_thm_main_conj}, and Theorem \ref{thm:main_thm_mu_zero_conj} for their general versions.
See also (\ref{eqn:betti-etale-comparison-extended-integral-renormalized}) for the picture of the integrality comparison.
\subsection{Further remarks} \label{subsec:further_remarks}
We rephrase Corollary \ref{cor:main_thm_main_conj_simple}.(2) from the viewpoint of the residual representation and discuss its generalization.
\begin{cor} \label{cor:main_thm_main_conj_families}
Let $p \geq 5$ be a prime and $2 \leq k \leq p-1$ an integer.
Let 
$$\overline{\rho} : \mathrm{Gal}(\overline{\mathbb{Q}}/\mathbb{Q}) \to \mathrm{GL}_2(\overline{\mathbb{F}}_p)$$
be a continuous, odd, and irreducible Galois representation such that
\begin{itemize}
\item the image of $\overline{\rho}$ contains a conjugate of $\mathrm{SL}_2(\mathbb{F}_p)$;
\item the restriction of $\overline{\rho}$ to $\mathrm{Gal}(\overline{\mathbb{Q}}_p/\mathbb{Q}_p)$ does not admit the trivial subrepresentation.
\end{itemize}
Let $S_k(\overline{\rho})$ be the set of newforms of weight $k$ defined by the following criterion:
\begin{quote} a newform $f$ of level $N_f$ lies in $S_k(\overline{\rho})$ if and only if $p \nmid N_f$ and $\overline{\rho} \simeq \overline{\rho}_f$.
\end{quote}
If there exists a newform $f_0 \in S_k(\overline{\rho})$ such that
\begin{enumerate}
\item $N_{f_0} = N(\overline{\rho})$,
\item $\alpha_0 \neq \beta_0$ in the sense of Assumption \ref{assu:standard},
\item $f_0$ and $r$ satisfy Assumption \ref{assu:standard}.(ENV),
\item the numerical criterion in Theorem \ref{thm:main_thm_main_conj_simple}.(1) holds, and
\item the numerical criterion in Corollary \ref{cor:main_thm_main_conj_simple}.(1) holds,
\end{enumerate}
then the main conjecture and Conjecture A of Coates--Sujatha hold for $T_{f, i}(k-r)$ where $f$ runs over all forms in $S_k(\overline{\rho})$.
\end{cor}
\begin{proof}
It follows from the combination of Theorem \ref{thm:main_thm_main_conj_simple}, Corollary \ref{cor:main_thm_main_conj_simple}, and \cite{kim-lee-ponsinet}.
\end{proof}

\begin{rem} \label{rem:nakamura}
Combining Corollary \ref{cor:main_thm_main_conj_families} with the recent work\footnote{There are also some related work by Colmez--Wang \cite{colmez-wang}, Fouquet--Wan \cite{fouquet-wan}, and Yiwen Zhou \cite{yiwen-zhou-thesis}.} of Nakamura \cite{nakamura-kato-deformation}, which extends \cite{kim-lee-ponsinet} to the universal deformation ring, we obtain the following statement.
We keep the following assumptions.
\begin{itemize}
\item $p \geq 5$,
\item the image of $\overline{\rho}$ contains a conjugate of $\mathrm{SL}_2(\mathbb{F}_p)$,
\item $\mathrm{End}_{\mathbb{F}_\pi[G_{\mathbb{Q}_p}]}(\overline{\rho} \vert_{G_{\mathbb{Q}_p}}) = \mathbb{F}_\pi$, and
\item $\left( \overline{\rho} \vert_{G_{\mathbb{Q}_p}} \right)^{ss} \not\simeq \left( \mathbf{1} \oplus \omega^{\pm 1} \right) \otimes \chi$
for any character $\chi : G_{\mathbb{Q}_p} \to \mathbb{F}^\times_\pi$.
\end{itemize}
Let $S(\overline{\rho})$ be the set of newforms of arbitrary weight and arbitrary level whose residual representations are isomorphic to $\overline{\rho}$.
If there exists a newform $f_0 \in S(\overline{\rho})$ satisfying all the conditions in Corollary \ref{cor:main_thm_main_conj_families},
then 
the  main conjecture and Conjecture A of Coates--Sujatha hold for $T_{f, i} (k-r)$ where $f$ runs over all forms in $S(\overline{\rho})$.
It is remarkable that this result is insensitive to the range of weight and the type of primes.
In other words, 
\emph{the Iwasawa main conjecture and the $\mu=0$ conjecture for modular forms of arbitrary weight and level at an arbitrary prime could be verified from two numerical computations on one specific form (or possibly two forms)}.
Nakamura's work depends heavily on Emerton's local-global compatibility \cite{emerton-local-global}.
\end{rem}

\begin{rem} \label{rem:imc-progresses}
In order to prove various types of main conjectures, there are two approaches in general: the Eisenstein congruence method and the Euler system method \cite[$\S$2.4 and $\S$2.5]{kato-icm}.
Kato himself proved one-sided divisibility of his main conjecture via the Euler system method \cite{kato-euler-systems} (Theorem \ref{thm:kato-divisibility}).
Recent years, there has been significant progresses towards the opposite divisibility via the Eisenstein congruence method.
In \cite{skinner-urban}, Skinner--Urban prove the opposite divisibility of the main conjecture for $T_{f, i}(k-r)$ when
 $f$ is good ordinary at $p$, $\psi = \mathbf{1}$, $k \equiv 2 \pmod{p-1}$, $r = k-1$, $i=0$, and there exists a prime $\ell$ exactly dividing the tame level such that $\overline{\rho}$ is ramified at $\ell$.
Building on this work, the multiplicative case is also carried out in \cite{skinner-pacific}.
In \cite{wan_hilbert}, the condition on the level is replaced by the existence of a real quadratic field satisfying a certain integral property.
In all the literature on the Eisenstein congruence approaches, the conditions $\psi = \mathbf{1}$ and $i=0$ are imposed as well as the critical point $r$ is fixed. In addition, the ordinary and non-ordinary cases are studied separately and there are certain tame level assumptions therein.
On the other hand, Theorem \ref{thm:main_thm_main_conj_simple} does \emph{not} require $\psi = \mathbf{1}$ and does cover all critical points and the \emph{full} cyclotomic extension $\mathbb{Q}(\zeta_{p^\infty})$ by varying $r$ and $i$. We also do \emph{not} impose any condition on the tame level and deal with ordinary and non-ordinary cases on equal footing.
\end{rem}

\begin{rem} \label{rem:main-theorem-mc}
If the Tamagawa defect is non-trivial, then all Kurihara numbers must vanish \cite[Prop. 6.2.6]{mazur-rubin-book}.
The minimal level condition $N = N(\overline{\rho})$ ensures that there is no Tamagawa defect \cite[Lem. 4.1.2]{epw}.
See also \cite[(H.T), Page 3150]{kazim-Lambda-adic}.
A congruent form of minimal level can be always found in the set of congruent newforms via level lowering, at least in the Fontaine--Laffaille range when $p\geq 5$ \cite{diamond-taylor-non-optimal}. See \cite{kim-ota} for a more quantitative analysis of this aspect.
\end{rem}

\begin{rem} \label{rem:standard}
We discuss the motivation of Assumption \ref{assu:standard}.
\begin{enumerate}
\item[(LI)]
The large image assumption prevents CM forms.
It is satisfied for all Galois-stable lattices if it is satisfied for one lattice for non-CM forms \cite[Rem. 12.8]{kato-euler-systems}.
Thus, the choice of $T_f$ does not affect any of result of this article. This condition is important in the construction of integral canonical Kato's Euler systems.
\item[($\alpha\neq\beta$)]
This assumption is used to ``$p$-destabilize" the local points \`{a} la Perrin-Riou constructed in \cite{perrin-riou-local-iwasawa} in order to construct Mazur--Tate elements of modular forms.
This holds when $k=2$ and is expected to hold in general because it follows from the Tate conjecture on algebraic cycles \cite{coleman-edixhoven}.
\item[(``non-CM")]
This condition is equivalent to the non-existence of mod $p$ companion forms \cite{gross-tameness, coleman-voloch}, and it  is also equivalent that $\mathrm{H}^0(\mathbb{Q}_p , W_{\overline{f}, -i}(r)) = 0$ for $i = 0, \cdots, p-2$ \cite[Prop. 13.2]{gross-tameness}\footnote{The unchecked compatibility in \cite{gross-tameness} is confirmed in \cite{bryden-cais-thesis}.}.
Since this means that $f$ is not congruent to an anomalous ordinary CM form, this must always hold under (LI).
See \cite{pollack-stevens-critical,loeffler-zerbes-wach-critical,castella-wang-erickson} for this topic.
\item[(ENV)] This condition naturally appears in the computation of local Iwasawa theory (e.g. Proposition \ref{prop:normalization_local_points_first}) and the root of unity $\zeta$ comes from the equivariant refinement. This condition also implies the non-vanishing of the $p$-th Euler factor of $\overline{f}$ at $r$, and it is automatic when $k$ is even. 
\end{enumerate}
\end{rem}

\begin{rem} \label{rem:main-theorem-coleman}
We discuss some remarks on Theorem \ref{thm:main_thm_finite_layer_coleman}.
\begin{enumerate}
\item
In \cite{epw2}, a completely different construction of the finite layer Coleman maps over $\mathbb{Q}(\zeta_{p^n})$ is given by using the $p$-adic local Langlands correspondence for $\mathrm{GL}_2(\mathbb{Q}_p)$ (e.g.~\cite{colmez-phi-gamma-mirabolic, colmez-gl2qp-phi-gamma}) for de Rham representations when  $\rho_f \vert_{G_{\mathbb{Q}_p}}$ is irreducible. It does not require  the ($\alpha\neq \beta$) assumption. 
However, the integrality of the finite layer Coleman map is \emph{less} controlled therein; as a result, the constant error term $C_f$ in Theorem \ref{thm:main_thm_mazur_tate_conj} is not specified in the main theorem of \cite{epw2}.
\item
In \cite{colmez-local-iwasawa-de-rham} and \cite{nakamura-local-iwasawa}, Colmez and Nakamura generalized Perrin-Riou's local Iwasawa theory to de Rham representations and de Rham $(\varphi, \Gamma)$-modules, respectively, but it seems that the construction of the local points is missing in their work.
\item Theorem \ref{thm:main_thm_finite_layer_coleman} shows that the Euler system and modular symbol constructions of Mazur--Tate elements \emph{precisely} coincide. On the way, we also confirm that the Kato--Perrin-Riou $p$-adic $L$-function \emph{precisely} coincides with the Mazur--Tate--Teitelbaum $p$-adic $L$-function at good primes (Theorem \ref{thm:comparison-kato-perrin-riou-mtt}). See \cite{ota-rank-part} for a similar construction.
As far as we know, this level of precision was not known before at least for higher weight non-ordinary modular forms.
\end{enumerate}
\end{rem}

\begin{rem}
It seems reasonable to expect that the Euler system construction of Mazur--Tate elements over general cyclotomic extensions has a variety of arithmetic applications as that of $p$-adic $L$-functions has. We expect this idea would apply to other Euler systems (e.g. see \cite{loeffler-zerbes-arizona, loeffler-zerbes-iwasawa-2017} for Loeffler--Zerbes' program on the construction of new Euler systems for higher rank groups).
\end{rem}

\subsection{Structure of this article}

The first half (from $\S$\ref{sec:galois-cohomology} to $\S$\ref{sec:local-iwasawa-theory-perrin-riou}) of this article can be regarded as a preliminary part except the construction of canonical Kato's Euler systems in $\S$\ref{sec:canonical-Kato-Euler-systems}. In the second half (from $\S$\ref{sec:mazur-tate-etale} to $\S$\ref{sec:mu-conjecture}), we prove the main theorems in this article. 

In $\S$\ref{sec:galois-cohomology}, we recall the convention and the basic facts on Galois representations arising from modular forms, Galois cohomology and Selmer groups, and Iwasawa cohomology.
In $\S$\ref{sec:mazur-tate-betti}, the relevant comparison isomorphisms among the realizations of modular motives and Kato's period maps are reviewed. Then we discuss the modular symbol construction of Mazur--Tate elements following \cite{mtt} and pin down the optimal periods depending only on the values of modular symbols. Then we compare Kato's period integrals with the optimal periods.
In $\S$\ref{sec:canonical-Kato-Euler-systems}, we review two kinds of Kato's Euler systems and construct the canonical Kato's Euler systems.
We also prove that the canonical Kato's Euler system is integral under the large image assumption and recall the statements of the main conjecture and Conjecture A of Coates--Sujatha.
In $\S$\ref{sec:p-adic-hodge-theory}, we briefly review $p$-adic Hodge theory with emphasis on its integral aspect via Wach modules. Especially, we discuss an integral $\varphi$-stable lattice of the crystalline Dieudonn\'{e} module constructed by Benois.
In $\S$\ref{sec:local-iwasawa-theory-perrin-riou}, we study Perrin-Riou's local Iwasawa theory and its integral refinement following Benois' approach.

In $\S$\ref{sec:mazur-tate-etale}, we construct $p$-stabilized equivariant local points and equivariant $p$-adic $L$-functions of $p$-stabilized modular forms from Kato's Euler systems via the Perrin-Riou's local machinery. 
Especially, we confirm these Kato--Perrin-Riou $p$-adic $L$-functions and the $p$-adic $L$-functions of Mazur--Tate--Teitelbaum constructed from modular symbols are the same ones. 
Then we ``$p$-destabilize" $p$-stabilized equivariant local points and construct the finite layer Coleman maps (Theorem \ref{thm:main_thm_finite_layer_coleman}).
In $\S$\ref{sec:mazur-tate}, we prove the ``weak" main conjecture of Mazur--Tate (Theorem \ref{thm:main_thm_mazur_tate_conj}). We review how the control theorem fails when $f$ is non-ordinary at $p$ and recall some properties of Fitting ideals. After doing some local computation, we give a step-by-step proof of Theorem \ref{thm:main_thm_mazur_tate_conj}.
In $\S$\ref{sec:IMC}, we prove the mod $p$ numerical criterion for Kato's main conjecture (Theorem \ref{thm:main_thm_main_conj}) by the explicit computation of mod $p$ Kolyvagin derivatives of Mazur--Tate elements.
In $\S$\ref{sec:mu-conjecture}, we prove the mod $p$ numerical criterion for Conjecture A of Coates--Sujatha (Theorem \ref{thm:main_thm_mu_zero_conj}) and discuss its application to elliptic curves (Corollary \ref{cor:mu-elliptic-curves}) by the explicit computation of modular symbols and Mazur--Tate elements.

In $\S$\ref{sec:app-bsd}, we discuss arithmetic applications of our main results to Birch and Swinnerton-Dyer  (BSD)  conjecture.
There have been significant progresses in BSD conjecture in recent years.
The $p$-converse to a theorem of Gross--Zagier and Kolyvagin and the $p$-part of the BSD formula for rank one elliptic curves are enormously developed \cite{kobayashi-gross-zagier, kobayashi-iwasawa-2012, wei-zhang-mazur-tate, venerucci-converse, berti-bertolini-venerucci, jetchev-skinner-wan, castella-cambridge, wan-rankin-selberg, skinner-converse, burungale-tian-p-converse, wan-main-conj-ss-ec, sprung-main-conj-ss, kazim-pollack-sasaki, castella-ciperiani-skinner-sprung,  castella-wan-perrin-riou-ss, kim-p-converse}.
 We discuss the applications of these mod $p$ numerical criteria to the $p$-part of the BSD formula 
 for elliptic curves of analytic rank one at a good supersingular prime $p$ following Kobayashi's $p$-adic Gross--Zagier formula \cite{kobayashi-gross-zagier}.
We also study the applications of these mod $p$ numerical criteria to the Heegner point context based on the work of Wei Zhang on Kolyvagin's conjecture \cite{wei-zhang-mazur-tate}.
More precisely, we are able to numerically verify Kolyvagin's conjecture for any good prime $p>3$.
\emph{In other words, we offer the numerical verification method of all the results of Wei Zhang for modular abelian varieties at any (non-ordinary) good prime $p> 3$.}

\subsection*{Acknowledgement}
We would like to thank
Denis Benois,
K\^{a}z{\i}m B\"{u}y\"{u}kboduk,
Ralph Greenberg,
Takenori Kataoka,
Mark Kisin,
Masato Kurihara,
Antonio Lei,
Yong Suk Moon,
Kentaro Nakamura,
Amanda Nam,
Kazuto Ota,
Robert Pollack,
Gautier Ponsinet,
Takamichi Sano, 
Ryotaro Sakamoto, and
Sug Woo Shin for helpful comments, discussions, and encouragement.
We especially thank Robert Pollack for making a preliminary version of \cite{epw2} available and detailed comments; Takenori Kataoka for pointing out several errors in an earlier version and for extremely helpful discussions.
This research was partially supported 
by a KIAS Individual Grant (SP054102) via the Center for Mathematical Challenges at Korea Institute for Advanced Study and
by the National Research Foundation of Korea(NRF) grant funded by the Korea government(MSIT) (No. 2018R1C1B6007009).

\section{Galois and Iwasawa cohomologies} \label{sec:galois-cohomology}
\subsection{Construction of Galois representations} \label{subsec:galois-representations}
We follow \cite[(4.5.1) and $\S$8.3]{kato-euler-systems}.
Let $N \geq 4$ and $\varpi : E \to Y_1(N)$ the universal elliptic curve over the modular curve and
$\mathcal{H}^1_p := \mathrm{R}^1\varpi_*\mathbb{Z}_p$ the \'{e}tale $\mathbb{Z}_p$-sheaf on $Y_1(N)$.
We define
$ V_{k,\mathbb{Z}_p}(Y_1(N)) := \mathrm{H}^1_{\mathrm{\acute{e}t}}(Y_1(N)_{\overline{\mathbb{Q}}}, \mathrm{Sym}^{k-2}_{\mathbb{Z}_p}(\mathcal{H}^1_p)) $,
$V_{\mathcal{O}_\pi}(f) := V_{k, \mathbb{Z}_p} (Y_1(N)) \otimes_{\mathbb{T}} \mathbb{T}/\wp_f$, and
$V_{F_\pi}(f) := V_{\mathcal{O}_\pi}(f) \otimes_{\mathbb{Z}_p}\mathbb{Q}_p$
where $\mathbb{T}$ is the Hecke algebra over $\mathbb{Z}_p$ faithfully acting on $V_{k,\mathbb{Z}_p}(Y_1(N))$ and $\wp_f$ is the height one prime ideal of $\mathbb{T}$ generated by the Hecke eigensystem of $f$
following \cite[$\S$6.3]{kato-euler-systems}.
Here, $V_{F_\pi}(f)$ is our $V_f$.
More explicitly, $\rho_f$ satisfies the following properties:
\begin{enumerate}
\item $\mathrm{det} ( \rho_f ) = \chi^{1-k}_{\mathrm{cyc}} \cdot \psi^{-1}$
where $\chi_{\mathrm{cyc}}$ is the $p$-adic cyclotomic character;
\item for any prime $\ell$ not dividing $Np$, we have
\begin{equation*}
\mathrm{det} \left( 1- \rho_f ( \mathrm{Frob}^{-1}_\ell ) \cdot u  : \mathrm{H}^0(I_\ell,  V_f ) \right) = 1 - a_{\ell}(f) \cdot u +  \psi  (\ell) \cdot \ell^{k-1} \cdot u^2
\end{equation*}
where $\mathrm{Frob}_\ell$ is the arithmetic Frobenius at $\ell$ in $G_{\mathbb{Q}_\ell} / I_\ell$ and $I_\ell$ is the inertia subgroup of $G_{\mathbb{Q}_\ell}$;
\item for the prime number $p$ lying under $\pi$,  we have
$$\mathrm{det} \left( 1- \varphi \cdot u  : \mathbf{D}_{\mathrm{cris}}  ( V_f ) \right) = 1 - a_{p}(f) \cdot u +  \psi (p) \cdot p^{k-1} \cdot u^2$$
where $\varphi$ is the Frobenius operator acting on $\mathbf{D}_{\mathrm{cris}}  ( V_f ) := ( V_f \otimes_{\mathbb{Q}_p} \mathbf{B}_{\mathrm{cris}})^{G_{\mathbb{Q}_p}}$ \cite{fontaine-semi-stable}. 
Here, $\mathbf{B}_{\mathrm{cris}}$ is the crystalline period ring recalled in $\S$\ref{sec:p-adic-hodge-theory}. 
\end{enumerate}
When $N <4$, we are still able to construct Galois representations using the $N \geq 4$ case and trace maps. 


\subsection{Galois cohomology and Selmer groups}
Let $\ell$ be any rational prime and $K_v$ a finite extension of $\mathbb{Q}_\ell$.
Let $M = T_f$, $V_f$, or $W_f$ or their Tate twists.
We write
$\mathrm{H}^1(K_v, M) := \mathrm{H}^1(\mathrm{Gal}(\overline{K}_v/K_v), M) $.
Also, for a number field $K$ and a rational prime $\ell$, we define the semi-local cohomology groups by
$\mathrm{H}^1(K \otimes \mathbb{Q}_\ell, M) := \bigoplus_{v\vert \ell} \mathrm{H}^1(K_v, M)$
where $v$ runs over the primes of $K$ dividing $\ell$.
The localization map at $p$ is denoted by
$\mathrm{loc}_p : \mathrm{H}^1(K, M) \to \mathrm{H}^1(K \otimes \mathbb{Q}_p, M) $.
For a place $v$ of $K$ dividing $p$, we define 
\begin{align*}
\mathrm{H}^1_f(K_v, V_f(j)) & := \mathrm{ker} \left( \mathrm{H}^1(K_v, V_f(j)) \to \mathrm{H}^1(K_v, V_f(j) \otimes \mathbf{B}_{\mathrm{cris}}) \right) , \\
\mathrm{H}^1_f(K_v, W_f(j)) & := \mathrm{im} \left( \mathrm{H}^1_f(K_v, V_f(j)) \to \mathrm{H}^1(K_v, W_f(j)) \right) , \textrm{ and} \\
\mathrm{H}^1_f(K_v, T_f(j)) & := \mathrm{preim} \left( \mathrm{H}^1(K_v, T_f(j)) \to \mathrm{H}^1_f(K_v, V_f(j)) \right) .
\end{align*}
For a place $v$ of $K$ not dividing $p$, we define 
\begin{align*}
\mathrm{H}^1_f(K_v, V_f(j)) & := \mathrm{H}^1_\mathrm{ur}(K_v, V_f(j)), \\
\mathrm{H}^1_f(K_v, W_f(j)) & := \mathrm{im} \left( \mathrm{H}^1_f(K_v, V_f(j)) \to \mathrm{H}^1(K_v, W_f(j)) \right) ,  \textrm{ and} \\
\mathrm{H}^1_f(K_v, T_f(j)) & := \mathrm{preim} \left( \mathrm{H}^1(K_v, T_f(j)) \to \mathrm{H}^1_f(K_v, V_f(j)) \right) .
\end{align*}
Since we assume $p>2$, we ignore the infinite place. See \cite[Rem. 1.3.7]{rubin-book}.
For notational convenience, we write $\mathrm{H}^1_{/f}(-) = \dfrac{\mathrm{H}^1(-)}{\mathrm{H}^1_{f}(-)}$.
The natural finite-singular quotient of the localization map at $p$ is denoted by
$\mathrm{loc}^s_p : \mathrm{H}^1(K, M) \to \mathrm{H}^1(K \otimes \mathbb{Q}_p, M ) \to \mathrm{H}^1_{/f}(K \otimes \mathbb{Q}_p, M )$.

Let $\Sigma$ be a finite set of places of $\mathbb{Q}$ containing $p$, $\infty$, and the bad primes for $\overline{f}$.
For a number field $K$, $K_{\Sigma}$ is the maximal extension of $K$ unramified outside the places lying over the places in $\Sigma$. Let $r$ be an integer.
We define the (discrete) \textbf{Selmer group of $W_{\overline{f}}(r)$ over $K$} by
$$\mathrm{Sel}(K, W_{\overline{f}}(r)) := \mathrm{ker} \left( \mathrm{H}^1(K_\Sigma/K, W_{\overline{f}}(r)) \to \prod_{\ell \in \Sigma} \dfrac{\mathrm{H}^1(K \otimes \mathbb{Q}_\ell, W_{\overline{f}}(r) )}{\mathrm{H}^1_f(K \otimes \mathbb{Q}_\ell, W_{\overline{f}}(r))}  \right) $$
where $\mathrm{H}^i(F/K, M) = \mathrm{H}^i(\mathrm{Gal}(F/K), M)$.

Let $\mathbb{Q}_{\infty} \subseteq \mathbb{Q}(\zeta_{p^\infty})$ be the cyclotomic $\mathbb{Z}_p$-extension of $\mathbb{Q}$ and $\mathbb{Q}_n$ the subextension of $\mathbb{Q}$ in $\mathbb{Q}_{\infty}$ of degree $p^n$.
We define the Selmer group over the cyclotomic tower $\mathbb{Q}_{\infty}$ or $\mathbb{Q}(\zeta_{p^\infty})$ by taking the direct limit.

\subsection{Structure of global and local Iwasawa cohomology} \label{subsec:iwasawa-cohomology}
We recall  the comparison of Galois and \'{e}tale cohomologies over the cyclotomic tower with a slight equivariant generalization.

Let $m$ be an integer $\geq 1$ with $(m,p)=1$, and $F$ a finite extension of $\mathbb{Q}$ in $\mathbb{Q}(\zeta_{mp^\infty})$.
Let $j_F : \mathrm{Spec}(F) \to \mathrm{Spec}(\mathcal{O}_F[1/p])$ be the natural map and
write
$$\mathrm{H}^1_{\mathcal{I}w}(\mathbb{Q}(\zeta_{m}),   T_{f}(k-r)) := \varprojlim_{F} \mathrm{H}^1_{\mathrm{\acute{e}t}}( \mathrm{Spec}(\mathcal{O}_F[1/p]), j_{F, *}T_{f}(k-r)) $$
where $F$ runs over finite extension of $\mathbb{Q}$ in $\mathbb{Q}(\zeta_{mp^\infty})$.
\begin{prop}[Kurihara] \label{prop:kurihara-etale-galois}
Let $m \geq 1$ be an integer and $\Sigma_m = \Sigma \cup \lbrace q, \textrm{ a prime}: q \vert m \rbrace$.
Then we have an isomorphism
$$\mathrm{H}^1_{\mathcal{I}w}(\mathbb{Q}(\zeta_{m}),   T_{f}(k-r)) \simeq  \varprojlim_{F} \mathrm{H}^1( \mathbb{Q}_{\Sigma_m}/F ,T_{f}(k-r)) $$
where $F$ runs over finite extensions of $\mathbb{Q}$ in $\mathbb{Q}(\zeta_{mp^\infty})$.
\end{prop}
\begin{proof}
It follows from the localization exact sequence of \'{e}tale cohomology. 
See \cite[$\S$6]{kurihara-invent}, \cite[Prop. 7.1.i)]{kobayashi-thesis}, and \cite[Prop. 5.1]{kim-lee-ponsinet}.
Although the $m=1$ case is treated in the reference, the proof goes through verbatim due to \cite[Lem. 8.5]{kato-euler-systems}.
\end{proof}
As a consequence of Proposition \ref{prop:kurihara-etale-galois},  $\varprojlim_{F} \mathrm{H}^1( \mathbb{Q}_{\Sigma_m}/F ,T_{f}(k-r))$  is independent of $\Sigma$.
We also have the decomposition
$\mathrm{H}^1_{\mathcal{I}w}(\mathbb{Q}(\zeta_{m}),  T_{f}(k-r))
= 
\bigoplus_{i=0}^{p-2} \mathrm{H}^1_{\mathrm{Iw}}(\mathbb{Q}(\zeta_{m}), T_{f, i}(k-r)) $.
Here, 
$\mathrm{H}^1_{\mathrm{Iw}}(\mathbb{Q}(\zeta_{m}), T_{f, i}(k-r)) :=  \varprojlim_F \mathrm{H}^1(\mathbb{Q}_{\Sigma_m}/F, T_{f, i}(k-r)) $
where $F$ runs over finite extensions of $\mathbb{Q}$ in $\mathbb{Q}_\infty(\zeta_{m})$.

\begin{lem} \label{lem:freeness-iwasawa-algebra}
If $\rho_f$ has large image, then $\mathrm{H}^1_{\mathcal{I}w}(\mathbb{Q}(\zeta_{m}),  T_{f}(k-r) )$ is free over $\Lambda_{\mathcal{I}w}$.
\end{lem}
\begin{proof}
It suffices to prove the statement for any integer $r$. Put  $r=k$.
If the image of $\overline{\rho}$ contains a conjugate of $\mathrm{SL}_2(\mathbb{F}_p)$, we first claim
\begin{equation} \label{eqn:vanishing-H0}
\mathrm{H}^0(\mathbb{Q}(\zeta_m), \overline{\rho}) = 0 .
\end{equation}
\begin{enumerate}
\item Suppose that $\#\mathbb{F}_\pi > 4$.
Let $\overline{\rho} : \mathrm{Gal}(\overline{\mathbb{Q}}/\mathbb{Q}) \to \mathrm{GL}_2(\mathbb{F}_\pi)$ and $H :=  \mathrm{Gal}(\overline{\mathbb{Q}}/\mathbb{Q}) / \mathrm{ker} (\overline{\rho})$. Thus, we have $\mathrm{SL}_2(\mathbb{F}_p) \subseteq H \subseteq \mathrm{GL}_2(\mathbb{F}_\pi)$ by assumption. Since $\mathrm{PSL}_2(\mathbb{F}_\pi)$ is simple if $\#\vert \mathbb{F}_\pi \vert > 4$, the maximal abelian quotient of $\mathrm{GL}_2(\mathbb{F}_\pi)$ is $\mathbb{F}^\times_\pi$ via the determinant. Thus, for any abelian extension $F$ of $\mathbb{Q}$, the image of $ \mathrm{Gal}(\overline{\mathbb{Q}}/F)$ under $\overline{\rho}$ contains $\mathrm{SL}_2( \mathbb{F}_p)$. There is no fixed point of $(\mathbb{F}_\pi)^{\oplus 2}$ under $\mathrm{SL}_2(\mathbb{F}_p)$-action.
\item Suppose that $\mathbb{F}_\pi = \mathbb{F}_3$. Since $\overline{\rho}$ is odd, $\overline{\rho}$ is surjective.
The image of  $\mathrm{Gal}(\overline{\mathbb{Q}}/\mathbb{Q}(\zeta_m))$ under the map
\[
\xymatrix{
\mathrm{Gal}(\overline{\mathbb{Q}}/\mathbb{Q}) \ar@{->>}[r]^-{\overline{\rho}} & \mathrm{GL}_2(\mathbb{F}_3) \ar@{->>}[r]^-{\mathrm{pr}} & \mathrm{PGL}_2(\mathbb{F}_3) \simeq S_4
}
\]
is isomorphic to $S_4$, $A_4$, $V_4$ (the Klein four group), or the trivial group.
Because $$
(\mathrm{pr} \circ \overline{\rho} ) ( \mathrm{Gal}(\overline{\mathbb{Q}}/\mathbb{Q}) ) / ( \mathrm{pr} \circ \overline{\rho} ) ( \mathrm{Gal}(\overline{\mathbb{Q}}/\mathbb{Q}(\zeta_m)) )$$
is abelian, $S_4$ and $A_4$ are only possibilities and $\mathrm{PSL}_2(\mathbb{F}_3) \simeq A_4$. Since $A_4$ is a unique subgroup of $S_4$ of index 2, the image of $\mathrm{Gal}(\overline{\mathbb{Q}}/\mathbb{Q}(\zeta_m))$ contains $\mathrm{PSL}_2(\mathbb{F}_3)$ in $\mathrm{PGL}_2(\mathbb{F}_3)$. Due to \cite[Lem. 2, IV-23]{serre-abelian-l-adic}, the image of $\mathrm{Gal}(\overline{\mathbb{Q}}/\mathbb{Q}(\zeta_m))$ contains $\mathrm{SL}_2(\mathbb{F}_3)$ in $\mathrm{GL}_2(\mathbb{F}_3)$.
Thus, there is no fixed point of $(\mathbb{F}_3)^{\oplus 2}$ under the action of $\mathrm{Gal}(\overline{\mathbb{Q}}/\mathbb{Q}(\zeta_m))$ via $\overline{\rho}$.
\end{enumerate}
Due to (\ref{eqn:vanishing-H0}),
$\mathrm{H}^1_{\mathcal{I}w}(\mathbb{Q}(\zeta_{m}), T_f ) \otimes_{\Lambda_{\mathcal{I}w}} \Lambda_{\mathcal{I}w} / (\gamma -1)$
is a free $\mathcal{O}_\pi$-module 
 where $\gamma$ is a generator of $\Gamma = \mathrm{Gal}(\mathbb{Q}(\zeta_{p^\infty})/\mathbb{Q} )$.
 Then the map defined by multiplication by $\gamma - 1$
$$\times ( \gamma - 1) : \mathrm{H}^1_{\mathcal{I}w}(\mathbb{Q}(\zeta_{m}), T_f ) \to \mathrm{H}^1_{\mathcal{I}w}(\mathbb{Q}(\zeta_{m}), T_f )$$
is injective.
Thus, $\mathrm{H}^1_{\mathcal{I}w}(\mathbb{Q}(\zeta_{m}), T_f )$ is free over $\Lambda_{\mathcal{I}w}$. See also \cite[Rem. 6.5]{sakamoto-stark-systems}.
\end{proof}

\section{Mazur--Tate elements I: Betti construction} \label{sec:mazur-tate-betti}
We review the modular symbol construction of Mazur--Tate elements for modular forms.
This construction works for modular forms of arbitrary weight and level at any odd prime.
We closely follow \cite{mtt}.

\subsection{Realizations of modular motives} \label{subsec:realization_of_modular_motives}
We recall several realizations of motives from modular forms and their identifications in \cite{kato-euler-systems}. 
Let $N \geq 4$ for this moment.

We first recall the Betti cohomology of modular curves in \cite[$\S$4.5, $\S$5.4, and $\S$6.3]{kato-euler-systems}.
Let $\varpi : E^{\mathrm{univ}} \to Y(M,N)$ be the universal elliptic curve and we define a local system $\mathcal{H}^1$ on $Y(M,N)(\mathbb{C})$ by
$\mathcal{H}^1 = \mathrm{R}^1\varpi_*(\mathbb{Z})$,
which is locally isomorphic to $\mathbb{Z}^2$.
See \cite[$\S$2.1]{kato-euler-systems} for the definition of $Y(M,N)$.
For any commutative ring $A$, we define
\begin{align*}
V_{k,A}(Y(M,N)) & = \mathrm{H}^1(Y(M,N)(\mathbb{C}), \mathrm{Sym}^{k-2}_{\mathbb{Z}}(\mathcal{H}^1) \otimes_{\mathbb{Z}} A), \\
V_{k,A}(Y_1(N)) & = \mathrm{H}^1(Y_1(N)(\mathbb{C}), \mathrm{Sym}^{k-2}_{\mathbb{Z}}(\mathcal{H}^1) \otimes_{\mathbb{Z}} A)
\end{align*}
 by using the natural quotient map $Y(M,N)$ to $Y_1(N)$ for the latter.

We recall the \'{e}tale cohomology of modular curves in \cite[$\S$8.3]{kato-euler-systems}.
We define a smooth \'{e}tale sheaf $\mathcal{H}^1_p$ on $Y(M,N)_{\mathrm{\acute{e}t}}$ by $\mathrm{R}^1\varpi_*(\mathbb{Z}_p)$.
Then $\mathcal{H}^1_p \simeq \mathcal{H}^1 \otimes \mathbb{Z}_p$ as sheaves on $Y(M,N)(\mathbb{C})$.
For $A = \mathbb{Z}_p$, $\mathbb{Q}_p$, or $\mathbb{Z}/p^n\mathbb{Z}$ ($n \geq 1$), we consider
\[
\xymatrix{
\mathrm{H}^1_{\mathrm{\acute{e}t}}(Y(M,N)_{\overline{\mathbb{Q}}}, \mathrm{Sym}^{k-2}_{\mathbb{Z}_p}(\mathcal{H}^1_p) \otimes_{\mathbb{Z}_p} A), &
\mathrm{H}^1_{\mathrm{\acute{e}t}}(Y_1(N)_{\overline{\mathbb{Q}}}, \mathrm{Sym}^{k-2}_{\mathbb{Z}_p}(\mathcal{H}^1_p) \otimes_{\mathbb{Z}_p} A)
}
\]
where the latter is defined by using the natural quotient map as before.

We identify Betti cohomology with \'{e}tale cohomology
\begin{equation} \label{eqn:identification-betti-etale}
V_{k,A}(Y_1(N)) \simeq \mathrm{H}^1_{\mathrm{\acute{e}t}}(Y_1(N)_{\overline{\mathbb{Q}}}, \mathrm{Sym}^{k-2}_{\mathbb{Z}_p}(\mathcal{H}^1_p) \otimes_{\mathbb{Z}_p} A)
\end{equation}
for $A = \mathbb{Z}_p$, $\mathbb{Q}_p$, or $\mathbb{Z}/p^n\mathbb{Z}$ ($n \geq 1$).
Via (\ref{eqn:identification-betti-etale}), $V_{k,A}(Y_1(N))$ admits a continuous action of $\mathrm{Gal}(\overline{\mathbb{Q}}/\mathbb{Q})$ for such $A$'s. 
One can construct the above objects even for $N <4$ by using the $N \geq 4$ case and the trace map as in $\S$\ref{subsec:galois-representations}.

Now we focus on one cuspidal eigenform via the projection \cite[$\S$6.3]{kato-euler-systems}.
Let 
$F = \mathbb{Q}(a_n(f); n \geq 1) \subseteq \mathbb{C}$,
which is a finite extension of $\mathbb{Q}$ and is stable under the complex conjugation.

Let $M_k(\Gamma_1(N))$ be the space of modular forms of weight $k$ and level $\Gamma_1(N)$ as a $\mathbb{Q}$-vector space.
We define the $F$-submodule $\wp^{\mathrm{dR}}_f$ of $M_k(\Gamma_1(N)) \otimes_{\mathbb{Q}} F$ generated by the images of $T_n \otimes 1 - 1 \otimes a_n(f)$ for all $n \geq 1$ where $T_n$ is the $n$-th Hecke operator.
Let $S(f) := \left( M_k(\Gamma_1(N)) \otimes_{\mathbb{Q}} F \right) / \wp^{\mathrm{dR}}_f $ which is a one-dimensional $F$-vector space.

We define the $F$-submodule $\wp^{\mathrm{B}}_f$ of $V_{k,F}(Y_1(N))$ generated by the images of $T_n \otimes 1 - 1 \otimes a_n(f)$ for all $n \geq 1$.
Let $V_F(f) = V_{k,F}(Y_1(N)) / \wp^{\mathrm{B}}_f $ which is a two dimensional $F$-vector space.

For any $F$-algebra $A$, we define
$V_A(f) := V_F(f) \otimes_F A$
and
$V_{F_{\pi}}(f)$ admits a continuous Galois action as in $\S$\ref{subsec:galois-representations} via (\ref{eqn:identification-betti-etale}), i.e. the notation is compatible with that in $\S$\ref{subsec:galois-representations}.

For any $x \in V_{F}(f)$, $x^{\pm} := \frac{1}{2} (1\pm \iota) x$ where $\iota$ is induced from the complex conjugation.
Let  $V_{F}(f)^\pm :=  V_{F}(f)^{\iota = \pm 1}$ and then $\mathrm{dim}_F \ V_{F}(f)^\pm = 1$, respectively.

Following $\S$\ref{subsec:realization_of_modular_motives}, the realizations of modular motives have the following comparisons
\begin{equation} \label{eqn:betti-etale-comparison}
\xymatrix{
V_{F_\pi}(f) & V_F(f) \otimes F_\pi \ar[l]_-{\simeq} & V_F(f) \ar[l] \ar[r] & V_F(f) \otimes \mathbb{C} & S(f) \ar[l]_-{\mathrm{per}_f} 
}
\end{equation}
and $\mathrm{per}_f$ is Kato's period map reviewed in $\S$\ref{subsec:canonical-katos-euler-systems}.

\subsection{Modular symbols, Mazur--Tate elements, and optimal periods} \label{subsec:optimal_periods}
We follow the convention of \cite[$\S$1, $\S$3, and $\S$4]{mtt}.

Let $S_k$ be the space of all cuspforms of weight $k$ and level $\Gamma_1(N) $ for some $N$, and $\mathcal{P}_k(\mathbb{C}) = \bigoplus_{r=1}^{k-1}\mathbb{C} \cdot z^{r-1}$ be the space of polynomials in the variable $z$ of degree $\leq k-2$ with $\mathbb{C}$-coefficients. More generally, $\mathcal{P}_k(R)$ is defined for any commutative ring $R$ in the same way.

As in \cite[$\S$3]{mtt}, for 
$f \in S_k$, $P\in \mathcal{P}_k(\mathbb{C})$, and $a, m \in \mathbb{Q}$ with $m>0$, we define the \textbf{modular symbol} by
$$
\lambda (f, z^{r-1}, -a, m)  = 2\pi \sqrt{-1} \cdot \int^{a/m}_{a/m+ i \infty} f(z) \cdot (mz-a)^{r-1} dz .
$$
We set
$\lambda^{\pm} (f, z^{r-1}, a, m) =  \lambda (f, z^{r-1}; a, m) \pm \lambda (f, z^{r-1}, -a, m) $, and
 we regard it as a function
$\lambda^{\pm} (f, -, -, -) : P_k(\mathbb{C}) \times \mathbb{Q} \times \mathbb{Q}_{>0} \to \mathbb{C}$.
By the work of Manin and Shimura \cite{manin-parabolic, shimura-periods}, there exist periods $\Omega^{\pm}_f$ such that
$$\lambda^{\pm} (f, -, -, -) : P_k(\mathcal{O}_{(\pi)})  \times \mathbb{Q} \times \mathbb{Q}_{>0} \twoheadrightarrow \mathcal{O}_{(\pi)} \cdot \Omega^{\pm}_f $$
where $\mathcal{O}_{(\pi)} = \mathcal{O}_{\pi} \cap \overline{\mathbb{Q}}$.
We call these periods the \textbf{optimal periods} of $f$.
\begin{rem} 
\begin{enumerate}
\item In \cite{pw-mt}, Pollack--Weston defined the cohomological periods of $f$, which are similar to our optimal periods.
Their normalization is equivalent to ours when $\lambda^{\pm} (f, 1; a, m)$ is considered.
\item The optimal periods are determined by \emph{several} values of modular symbols in general, so it seems difficult to deform the optimal periods integrally in $p$-adic families of modular forms. The notion of optimal periods does not require any mod $p$ multiplicity one type result.
\end{enumerate}
\end{rem}
We define the \textbf{optimally normalized modular symbol}
by
$$\lambda^{\pm, \mathrm{opt}}(f, z^{r-1}, a, m) = \dfrac{1}{\Omega^{\pm}_f} \cdot \left( \lambda(f, z^{r-1}, a, m) \pm \lambda(f, z^{r-1}, -a, m) \right) .$$
Following \cite[$\S$8]{mtt}, the \textbf{Mazur--Tate element of $f$ over $\mathbb{Q}(\zeta_m)$ at $s=r$} is defined by
$$
\theta_{\mathbb{Q}(\zeta_m), r}(f)  := \sum_{a \in (\mathbb{Z}/m\mathbb{Z})^\times} \lambda^{\pm, \mathrm{opt}}(f, z^{r-1}, a, m)  \cdot \sigma_a  \in  \mathcal{O}_{\pi}[\mathrm{Gal}( \mathbb{Q}(\zeta_m)/\mathbb{Q} )] .
$$
where the sign of modular symbols is that of $(-1)^{r-1}$. See also \cite{otsuki-rims}.

For $n \vert m$ and a commutative ring $R$, 
the $R$-linear map 
$$\pi_{m,n} : R[\mathrm{Gal}( \mathbb{Q}(\zeta_m)/\mathbb{Q} )] \to R[\mathrm{Gal}( \mathbb{Q}(\zeta_n)/\mathbb{Q} )]$$
is induced by the natural projection $\mathrm{Gal}( \mathbb{Q}(\zeta_m)/\mathbb{Q} ) \to \mathrm{Gal}( \mathbb{Q}(\zeta_n)/\mathbb{Q} )$
and
the $R$-linear map 
$$\nu_{n,m}: R[\mathrm{Gal}( \mathbb{Q}(\zeta_n)/\mathbb{Q} )] \to R[\mathrm{Gal}( \mathbb{Q}(\zeta_m)/\mathbb{Q} )]$$
is defined by
$\sigma \mapsto \sum_{\tau \in \mathrm{Gal}( \mathbb{Q}(\zeta_m)/\mathbb{Q} )} \tau$
where $\tau$ runs over the elements mapping to $\sigma$ under $\pi_{m,n}$.
We identify 
$\mathrm{Gal}( \mathbb{Q}(\zeta_m)/\mathbb{Q} ) \simeq (\mathbb{Z}/m\mathbb{Z})^\times$
 by $\sigma_a \mapsto a$ and regard a Galois character  as a Dirichlet character under this identification. 
\begin{prop} \label{prop:mazur-tate-elements-properties}
\begin{enumerate}
\item Let $m$ be a positive integer and $\ell$ a prime.
Then
$$\pi_{m\ell,m} \left(  \theta_{\mathbb{Q}(\zeta_{m\ell}), r}(f) \right) =
\left\lbrace
\begin{array}{ll}
 \left( a_\ell(f) - \ell^{r-1}  \cdot \mathrm{Frob}_\ell -  \psi(\ell) \cdot \ell^{k-1-r} \cdot \mathrm{Frob}^{-1}_\ell \right) \cdot \theta_{\mathbb{Q}(\zeta_m), r}(f) & (\ell \nmid m)\\
a_\ell (f) \cdot \theta_{\mathbb{Q}(\zeta_m), r}(f)
- \psi(\ell) \cdot \ell^{k-2} \cdot \nu_{m/\ell, m} \left( \theta_{\mathbb{Q}(\zeta_{m/\ell}), r}(f) \right) & (\ell \mid m)
\end{array}
\right.$$
where $\mathrm{Frob}_\ell \in \mathrm{Gal}( \mathbb{Q}(\zeta_m)/\mathbb{Q} )$ is the arithmetic Frobenius at $\ell$.
\item For a primitive Dirichlet character $\chi$ of conductor $m>1$, we have
$$\chi \left( \theta_{\mathbb{Q}(\zeta_{m}), r}(f)  \right) = m^{r} \cdot (r-1)! \cdot \dfrac{1}{\tau(\chi^{-1})} \cdot \dfrac{L(f, \chi^{-1}, r)}{(-2 \pi \sqrt{-1})^{r-1} \cdot \Omega^{\pm}_f}$$
where the sign $\pm$ is that of $(-1)^{r-1} \cdot \chi(-1)$ and $\tau(\chi) = {\displaystyle \sum_{\sigma_a \in \mathrm{Gal}(\mathbb{Q}(\zeta_{m})/\mathbb{Q})} } \chi(a) \cdot \zeta^{a}_m$ is the Gauss sum.
\end{enumerate}
\end{prop}
\begin{proof}
See \cite[$\S$4, $\S$8, and (9.1)]{mtt}.
\end{proof}

\section{Construction of canonical Kato's Euler systems} \label{sec:canonical-Kato-Euler-systems}
The goal of this section is to construct the  \emph{canonical and integral} Kato's Euler systems for modular forms of \emph{arbitrary} weight at \emph{any odd} prime under the large image assumption.
In \cite{kato-euler-systems}, two \emph{slightly different} cohomology classes are mainly used for arithmetic applications:
\begin{enumerate}
\item ${}_{c,d}z^{(p)}_{\mathbb{Q}(\zeta_m)}(f,r,r', \xi, S) := {}_{c,d}z^{(p)}_{m}(f,r,r', \xi, S) \in \mathrm{H}^1_{\mathrm{\acute{e}t}}(\mathrm{Spec} ( \mathbb{Z}[1/p, \zeta_m] ), j_* V_{\mathcal{O}_\pi}(f)(k-r))$ in \cite[(8.1.3) and Ex. 13.3]{kato-euler-systems}
\item $\mathbf{z}^{(p)}_\gamma \in \mathbf{H}^1(V_{\mathcal{O}_\pi}(f))$ where $\gamma \in V_{\mathcal{O}_\pi}(f)$  in \cite[Thm. 12.5]{kato-euler-systems}
\end{enumerate}
For the former one, the tame level Euler system relation is explicitly studied in \cite[Prop. 8.12]{kato-euler-systems}; thus, it can be used to bound the Selmer groups via the Euler system argument \cite[Theorem 13.4]{kato-euler-systems}. Since  this element depends on various choices, it would not give us the exact bound. Especially, 
the choices affect the corresponding periods in Kato's explicit reciprocity law \cite[Thm. 6.6]{kato-euler-systems}. Also, the Euler factors at bad primes are ignored.

The latter one is entirely canonical and suitable for the formulation of the Iwasawa main conjecture.
It gives the correct period and remembers the Euler factors at bad primes.
However, it is rather unclear that it satisfies the tame level Euler system relation integrally (c.f. \cite[Appendix A]{delbourgo-book}).
It also seems to be related to \cite[Wishful Thinking 8.2.2 and Rem. 8.2.5]{rubin-book}.
See also \cite{colmez-p-adic-BSD, wang_kato} for a more distribution-theoretic approach.

Our version of Kato's Euler system works for both purposes. This system of elements \emph{simultaneously} remembers all the Euler factors at bad primes and enjoys the integrality, the tame level Euler system relation, and the correct period.
This choice removes a certain restriction in \cite{kks}, and we give an affirmative answer to \cite[Ques. A.15]{kim-nakamura} under the large image assumption. This section is insensitive to the good reduction assumption.

Some part of this section is benefited from the discussion with Takenori Kataoka. See also \cite{kataoka-thesis} for the case of elliptic curves.

\subsection{The cohomology classes and Kato's explicit reciprocity law} 
For an integer $A \geq 1$, denote by $\mathrm{prime}(A)$ the set of primes dividing $A$.
Let $m \geq 1$ be an integer.
\begin{defn} \label{defn:xi_S}
Following \cite[$\S$5.1]{kato-euler-systems}, we define $\xi$ and $S$ as follows:
\begin{enumerate}
\item $\xi$ is a symbol $a(A)$ where $a, A \in \mathbb{Z}$, $A \geq 1$ and $S$ is a non-empty finite set of primes containing $\mathrm{prime}(mA)$, or
\item $\xi$ is an element of $\mathrm{SL}_2(\mathbb{Z})$ and $S$ is a non-empty finite set of primes containing $\mathrm{prime}(mN)$.
\end{enumerate}
\end{defn}
\begin{defn} \label{defn:r_r'_c_d}
Following \cite[(5.2.1) and (5.2.2)]{kato-euler-systems}, we define integers $r, r'$ and positive integers $c, d$ as follows:
\begin{enumerate}
\item $1 \leq r \leq k-1$, $1 \leq r' \leq k-1$, and at least one of $r$, $r'$ is $k-1$.
\item $c$ and $d$ are positive integers with $\mathrm{prime}(cd) \cap S = \emptyset$, and $(d,N) = 1$.
\end{enumerate}
\end{defn}
Write 
$\mathrm{H}^1(\mathbb{Z}[1/p, \zeta_m], V_{\mathcal{O}_\pi}(f)(k-r))
=
\mathrm{H}^1_{\mathrm{\acute{e}t}}( \mathrm{Spec}(\mathbb{Z}[\zeta_{m}, 1/p]), j_{*}V_{\mathcal{O}_\pi}(f)(k-r)) $
where $j : \mathrm{Spec}(\mathbb{Q}(\zeta_{m})) \to \mathrm{Spec}(\mathbb{Z}[\zeta_{m}, 1/p])$ is the natural map.
Denote by
$${}_{c,d}z^{(p)}_{\mathbb{Q}(\zeta_m)}(f,r,r', \xi, S) = {}_{c,d}z^{(p)}_m(f,r,r', \xi, S) \in \mathrm{H}^1(\mathbb{Z}[1/p, \zeta_m], V_{\mathcal{O}_\pi}(f)(k-r))$$
the (imprimitive) zeta element in \cite[(8.1.3)]{kato-euler-systems}
where $p \in S$ as well as all the ``variables" satisfy Definition \ref{defn:xi_S} and Definition \ref{defn:r_r'_c_d}.
By \cite[Prop. 8.12]{kato-euler-systems}, the cohomology classes 
$${}_{c,d}z^{(p)}_{\mathbb{Q}(\zeta_m)}(f,r,r', \xi, S) \in \mathrm{H}^1(\mathbb{Z}[1/p, \zeta_m], V_{\mathcal{O}_\pi}(f)(k-r))$$ satisfy the tame level Euler system relation except at primes dividing $cdN$. See also \cite[Ex. 13.3]{kato-euler-systems}.
Denote by 
$${}_{c,d}z_{\mathbb{Q}(\zeta_m)}(f,r,r', \xi, S)  = {}_{c,d}z_m(f,r,r', \xi, S)  \in S(f) \otimes \mathbb{Q}(\zeta_m)$$
the zeta modular form in \cite[$\S$6.3]{kato-euler-systems}
where $c$, $d$, $r$, $r'$, $\xi$, and $S$ satisfy Definition \ref{defn:xi_S} and Definition \ref{defn:r_r'_c_d}.
By \cite[Thm. 9.7]{kato-euler-systems}, we have the assignment
${}_{c,d}z^{(p)}_m(f,r,r', \xi, S) \mapsto {}_{c,d}z_m(f,r,r', \xi, S)$
under the dual exponential map.

\begin{thm}[Kato's explicit reciprocity law] \label{thm:kato_interpolation_original}
Let $\chi$ be a character on $(\mathbb{Z}/m\mathbb{Z})^\times$.
Let $\pm = (-1)^{k-r-1} \cdot \chi(-1)$ and 
 $$(u,v) := \left\lbrace \begin{array}{ll} (r+2-k, r) & \textrm{if } r'=k-1 \\ (k-r', r') & \textrm{if } r=k-1  \end{array} \right.$$
 following \cite[(4.2.4)]{kato-euler-systems}.
We add the following condition.
\begin{enumerate}
\item[($*$)] If $\xi \in \mathrm{SL}_2(\mathbb{Z})$, then assume $c \equiv d \equiv 1 \Mod{N}$.
\end{enumerate}
Then we have
$$\sum_{b \in (\mathbb{Z}/m\mathbb{Z})^\times} \chi(b) \cdot \mathrm{per}_f \left( \sigma_b \left( {}_{c,d}z_m(f, r, r', \xi, S) \right) \right)^{\pm} = L^{(S)} (f, \chi, r) \cdot (2 \pi i )^{k-r-1} \cdot \gamma^{\pm}$$
where
$\mathrm{per}_f$ is Kato's period map in \cite[$\S$6.3]{kato-euler-systems},
\begin{align*}
\gamma = c^2 d^2 \delta(f, r', a(A)) & - c^u d^2 \overline{\chi}(c) \delta(f, r', ac(A)) \\
& - c^2 d^v \psi(d) \delta(f, r', ``a/d"(A)) + c^u d^v \overline{\chi}(cd) \psi(d)\delta(f, r', ``ac/d"(A))
\end{align*}
if $\xi = a(A)$, and
$$\gamma = (c^2 - c^u \overline{\chi}(c)) (d^2 - d^v \overline{\chi}(d)) \delta(f, r', \xi)$$
if $\xi \in \mathrm{SL}_2(\mathbb{Z})$.
\end{thm}
\begin{proof}
See \cite[Thm. 6.6]{kato-euler-systems}.
\end{proof}

\subsection{The module of primitive zeta elements} 
Following $\S$\ref{subsec:iwasawa-cohomology},
write
$$
\mathrm{H}^1_{\mathcal{I}w}(\mathbb{Q}(\zeta_{m}), V_{\mathcal{O}_\pi}(f)(k-r))
= \varprojlim_{F} \mathrm{H}^1_{\mathrm{\acute{e}t}}( \mathrm{Spec}(\mathcal{O}_F[1/p]), j_{*}V_{\mathcal{O}_\pi}(f)(k-r))
$$
where $F$ runs over finite extensions of $\mathbb{Q}$ in $\mathbb{Q}(\zeta_{mp^\infty})$.
The following definition comes from \cite[Thm. 12.6]{kato-euler-systems}, \cite[Appendix A]{delbourgo-book}, and \cite[Thm. 6.1]{kataoka-thesis}.
\begin{defn}
We define be the \textbf{module of imprimitive zeta elements} 
$$\mathcal{Z}^{\mathrm{imp}}_{\mathbb{Q}(\zeta_m)}(f, k-r) \subseteq \mathrm{H}^1_{\mathcal{I}w}(\mathbb{Q}(\zeta_{m}), V_{\mathcal{O}_\pi}(f)(k-r))$$
 by the submodule generated by
$${}_{c,d} \mathbf{z}^{(p)}_{\mathbb{Q}(\zeta_{m})}(f,k-r, j, a(A), \mathrm{prime}(pmA))   := 
\varprojlim_n \left( {}_{c,d} z^{(p)}_{\mathbb{Q}(\zeta_{mp^n})}(f,k-r, j, a(A), \mathrm{prime}(pmA) ) \otimes ( \zeta_{p^n})^{\otimes k-r} \right) $$
where $1 \leq j \leq k-1$, $a, A \in \mathbb{Z}$ with $A \geq 1$, $c, d \in \mathbb{Z}$ such that $(c, 6pmA)= (d, 6pmN) =1$, and
$$
{}_{c,d} \mathbf{z}^{(p)}_{\mathbb{Q}(\zeta_{m})}(f,k-r, j, \alpha, \mathrm{prime}(pmN))
 := \varprojlim_n \left( {}_{c,d} z^{(p)}_{\mathbb{Q}(\zeta_{mp^n})}(f,k-r, j, \alpha, \mathrm{prime}(pmN)) \otimes ( \zeta_{p^n})^{\otimes k-r} \right) 
$$
where $1 \leq j \leq k-1$, $\alpha \in \mathrm{SL}_2(\mathbb{Z})$, $c, d \in \mathbb{Z}$ such that $(cd, 6pmN) =1$ and $c \equiv d \equiv 1 \pmod{N}$.
\end{defn}

\begin{choice} \label{choice:alpha_j}
Fix elements $\alpha_1, \alpha_2 \in \mathrm{SL}_2(\mathbb{Z})$ and integers $j_1, j_2$ such that
\begin{itemize}
\item  $1 \leq j_i \leq k-1$ $(i=1, 2)$, and
\item $\delta(f , j_1, \alpha_1)^+ \neq 0$ and $\delta(f , j_2, \alpha_2)^- \neq 0$ (\hspace{1sp}\cite[(13.6)]{kato-euler-systems}).
\end{itemize}
\end{choice}
For a non-zero $\gamma \in V_{F_{\pi}}(f)$, we write
$$
\gamma = b_1 \cdot \delta(f , j_1, \alpha_1)^+ + b_2 \cdot \delta(f , j_2, \alpha_2)^- \neq 0
$$
for some $b_1, b_2 \in F_\pi$.
\begin{choice} \label{choice:c_and_d}
Let $m \geq 1$ be an integer.
Fix $c, d \in \mathbb{Z}$ such that
\begin{itemize}
\item $(cd, 6pm) = 1$,
\item $c \equiv d \equiv 1 \Mod{N}$ (needed for \cite[Thm. 6.6]{kato-euler-systems}), and
\item $c^2 \neq 1$ and $d^2 \neq 1$.
\end{itemize}
\end{choice}
For a commutative ring $R$, let $Q(R)$ be the total quotient ring of $R$.
Let $G_{mp^\infty} = \mathrm{Gal}(\mathbb{Q}(\zeta_{mp^\infty})/\mathbb{Q})$ and $\mathcal{O}_\pi\llbracket G_{mp^\infty} \rrbracket = \Lambda_{\mathcal{I}w}[\mathrm{Gal}(\mathbb{Q}(\zeta_{m})/\mathbb{Q})]$.
\begin{defn}
For each $\gamma \in V_{\mathcal{O}_\pi}(f)$, we define the \textbf{primitive zeta element} by
\begin{align*}
\mathbf{z}^{(p)}_{\mathbb{Q}(\zeta_{m}), \gamma} := &
\left\lbrace \mu(c,d,j_1)^{-1} \cdot b_1 \cdot  {}_{c,d} \mathbf{z}^{(p)}_{\mathbb{Q}(\zeta_{m})}(f,k, j_1, \alpha_1, \mathrm{prime}(pmN))  \right\rbrace^{-} \\
& +
\left\lbrace \mu(c,d,j_2)^{-1} \cdot b_2 \cdot  {}_{c,d} \mathbf{z}^{(p)}_{\mathbb{Q}(\mu_{m})}(f,k, j_2, \alpha_2, \mathrm{prime}(pmN))  \right\rbrace^{+} 
\end{align*}
in $\mathrm{H}^1_{\mathcal{I}w}(\mathbb{Q}(\zeta_{m}), V_{\mathcal{O}_\pi}(f)) \otimes_{\mathcal{O}_\pi \llbracket G_{mp^\infty} \rrbracket} Q(\mathcal{O}_\pi \llbracket G_{mp^\infty} \rrbracket)$
where $$\mu(c,d,j) := (c^2 - c^{k+1-j} \cdot \sigma_c) \cdot (d^2 - d^{j+1} \cdot \sigma_d) \cdot \prod_{\ell \mid N} (1 - a_\ell(f) \ell^{-k} \sigma^{-1}_{\ell}) \in \mathcal{O}_\pi\llbracket G_{mp^\infty} \rrbracket $$
for $j \in \mathbb{Z}$ and all the choices in the RHS satisfy Choice \ref{choice:alpha_j} and Choice \ref{choice:c_and_d}.
However, $\mathbf{z}^{(p)}_{\mathbb{Q}(\zeta_{m}), \gamma}$ itself is independent of Choice \ref{choice:alpha_j} and Choice \ref{choice:c_and_d}.
Note that $\mu(c,d,j)$ is not a zero divisor in $\mathcal{O}_\pi\llbracket G_{mp^\infty} \rrbracket$ for all $j \in \mathbb{Z}$.
\end{defn}
Here, the primitivity means that the Euler factors at bad reduction primes are built in Kato's explicit reciprocity law.
It is not related to the primitivity of Kolyvagin systems.
\begin{defn}
For any Galois-stable lattice $T_f \subseteq V_{F_{\pi}}(f)$,
we define the \textbf{module of primitive zeta elements} by the submodule
$$\mathcal{Z}^{\mathrm{prim}}_{\mathbb{Q}(\zeta_m)} (T_f(k-r)) \subseteq \mathrm{H}^1_{\mathcal{I}w}(\mathbb{Q}(\zeta_{m}), T_f(k-r) ) \otimes Q(\mathcal{O}_\pi \llbracket G_{mp^\infty} \rrbracket)$$
 generated by
$\mathbf{z}^{(p)}_{\mathbb{Q}(\zeta_{m}),\gamma} \otimes \varprojlim_n \left( ( \zeta_{p^n})^{\otimes k-r} \right)$
where $\gamma$ runs over $T_f(k-r)$.
\end{defn}
A slight generalization of \cite[Lem. 13.10]{kato-euler-systems} from $\mathbb{Q}$ to $\mathbb{Q}(\zeta_m)$ (e.g.\cite[Fact $\#$1 and Fact $\#$2 in Appendix A]{delbourgo-book}) yields that
\begin{equation} \label{eqn:imp-prim-containment}
\mathcal{Z}^{\mathrm{imp}}_{\mathbb{Q}(\zeta_m)} (f, k-r) \subseteq \mathcal{Z}^{\mathrm{prim}}_{\mathbb{Q}(\zeta_m)} (V_{\mathcal{O}_{\pi}}(f)(k-r)) .
\end{equation}
By \cite[Key Claim in Appendix A]{delbourgo-book}, which generalizes \cite[$\S$13.12]{kato-euler-systems}, 
the index of  (\ref{eqn:imp-prim-containment}) is finite; thus, we have
$$\mathcal{Z}^{\mathrm{prim}}_{\mathbb{Q}(\zeta_m)} (V_{\mathcal{O}_{\pi}}(f)(k-r)) \subseteq \mathrm{H}^1_{\mathcal{I}w}(\mathbb{Q}(\zeta_{m}), V_{\mathcal{O}_\pi}(f)(k-r)) \otimes \mathbb{Q}_p.$$
We have the following picture up to now
\[
\xymatrix{
\mathcal{Z}^{\mathrm{imp}}_{\mathbb{Q}(\zeta_m)} (f, k-r) \ar@{^{(}->}[r]_-{\textrm{fin. index}} \ar@{^{(}->}[d] & \mathcal{Z}^{\mathrm{prim}}_{\mathbb{Q}(\zeta_m)} (V_{\mathcal{O}_{\pi}}(f)(k-r))  \ar@{^{(}->}[d] \ar@{^{(}-->}[dl]^-{?}  \\
\mathrm{H}^1_{\mathcal{I}w}(\mathbb{Q}(\zeta_{m}), V_{\mathcal{O}_\pi}(f)(k-r))  & \mathrm{H}^1_{\mathcal{I}w}(\mathbb{Q}(\zeta_{m}), V_{\mathcal{O}_\pi}(f)(k-r)) \otimes_{\mathbb{Z}_p} \mathbb{Q}_p .
}
\]
\begin{prop}
If $\rho_f$ has large image, then we have 
$$\mathcal{Z}^{\mathrm{prim}}_{\mathbb{Q}(\zeta_m)} (T_{f}(k-r))  \subseteq \mathrm{H}^1_{\mathcal{I}w}(\mathbb{Q}(\zeta_{m}),  T_{f}(k-r) ) $$ where $T_f$ is any Galois-stable lattice of $V_{F_{\pi}}(f)$.
\end{prop}
\begin{proof}
It follows from Lemma \ref{lem:freeness-iwasawa-algebra} and \cite[$\S$13.14]{kato-euler-systems}.
\end{proof}
\subsection{Canonical Kato's Euler systems} \label{subsec:canonical-katos-euler-systems}
We keep the large image assumption on $\rho_f$ in this subsection.
\begin{defn} \label{defn:canonical-kato-euler-systems}
\begin{enumerate}
\item 
For each $\gamma \in T_f$, we define the \textbf{canonical Kato's Euler system over the cyclotomic tower} by
\[
 \mathbf{z}_{\mathbb{Q}(\zeta_{m}), \gamma}(f, k-r)  := \mathbf{z}^{(p)}_{\mathbb{Q}(\zeta_{m}), \gamma} \otimes \varprojlim_n \left( ( \zeta_{p^n})^{\otimes k-r} \right)  \in  \mathrm{H}^1_{\mathcal{I}w}(\mathbb{Q}(\zeta_{m}),  T_{f}(k-r) ) .
\]
\item 
For each $\gamma \in T_f$, we define the \textbf{canonical Kato's Euler system}
$$z_{\mathbb{Q}(\zeta_{mp^n}), \gamma}(f, k-r) \in  \mathrm{H}^1(\mathbb{Z}[\zeta_{mp^n}, 1/p],  T_f(k-r) )$$
by the image of $\mathbf{z}_{\mathbb{Q}(\zeta_{m}), \gamma} (f, k-r)$
under the corestriction map
$$ \mathrm{H}^1_{\mathcal{I}w}(\mathbb{Q}(\zeta_{m}),  T_{f}(k-r) )  \to \mathrm{H}^1(\mathbb{Z}[\zeta_{mp^n}, 1/p],  T_f(k-r) ) .$$
\end{enumerate}
\end{defn}
Let $\gamma^+_{0} \in T^+_f$ and $\gamma^-_{0}  \in T^-_f$ be $\mathcal{O}_\pi$-generators of $T^+_f$ and $T^-_f$, respectively, and write $\gamma_{0} = \gamma^+_{0} + \gamma^-_{0} \in T_f$.
Write $\mathrm{H}^1_{\mathcal{I}w}(T_f(k-r)) =  \mathrm{H}^1_{\mathcal{I}w}(\mathbb{Q},  T_{f}(k-r) )$ for convenience.
\begin{prop} \label{prop:kato-generated-by-one}
If $\rho_f$ has large image, then 
the submodule $$\mathcal{Z}^{\mathrm{prim}}(T_f(k-r))  := \mathcal{Z}^{\mathrm{prim}}_{\mathbb{Q}}(T_f(k-r))  \subseteq \mathrm{H}^1_{\mathcal{I}w}(T_f(k-r))  $$
 is generated by $\mathbf{z}_{\mathbb{Q}, \gamma_{0}}(f, k-r )$ over $\Lambda_{\mathcal{I}w}$.
\end{prop}
\begin{proof}
See \cite[Prop. A.12]{kim-nakamura}.
\end{proof}
Definition \ref{defn:canonical-kato-euler-systems} yields a generalization of Kato's zeta morphism $T_f \to \mathrm{H}^1_{\mathcal{I}w}(T_f)$ defined by $\gamma \mapsto \mathbf{z}^{(p)}_\gamma$ in \cite[Thm. 12.5.(1)]{kato-euler-systems} by adding tamely ramified abelian extensions.
\begin{prop} \label{thm:kato_euler_system_relation}
Let $\ell$ be a prime not equal to $Np$ and $\gamma \in T_f$.
The corestriction map induced from the norm map from $\mathbb{Q}(\zeta_{m \ell p^n})$ to $\mathbb{Q}(\zeta_{m p^n})$
$$\mathrm{H}^1(\mathbb{Z}[\zeta_{m\ell p^n}, 1/p],  T_f(k-r) ) \to \mathrm{H}^1(\mathbb{Z}[\zeta_{mp^n}, 1/p],  T_f(k-r) )$$
sends $z_{\mathbb{Q}(\zeta_{m \ell p^n}), \gamma}(f, k-r)$ to
$$\left( 1 - a_\ell(\overline{f}) \cdot \mathrm{Frob}^{-1}_\ell \cdot \ell^{-r} + \overline{\psi}(\ell) \cdot \mathrm{Frob}^{-2}_\ell \cdot \ell^{k-1-2r} \right) \cdot  z_{\mathbb{Q}(\zeta_{m p^n}),  \gamma}(f, k-r)$$
where $\mathrm{Frob}_\ell$  is the arithmetic Frobenius at $\ell$ in $\mathrm{Gal}(\mathbb{Q}(\zeta_{m p^n})/\mathbb{Q})$.
\end{prop}
\begin{proof}
It immediately follows from \cite[Prop. 8.12]{kato-euler-systems}.
\end{proof}
Let $\gamma = \gamma^+ + \gamma^- \in V_f$ and $z_{\mathbb{Q}(\zeta_{mp^n}), \gamma}(f, k-r) \in \mathrm{H}^1(\mathbb{Z}[\zeta_{mp^n}, 1/p],  V_f(k-r) )$ be the canonical Kato's Euler system with respect to $\gamma$.
Write
$$  \mathrm{exp}^{*} \left(  \mathrm{loc}_p \left( z_{\mathbb{Q}(\zeta_{mp^n}), \gamma}(f, k-r) \right) \right) = c_{k-r, mp^n, \gamma} \cdot \omega_f$$
where
 $\mathrm{exp}^{*} =  \mathrm{exp}^{*}_{V_{f}(k-r), \mathbb{Q}(\zeta_{mp^n})} = \sum_{v} \mathrm{exp}^{*}_{V_{f}(k-r), \mathbb{Q}(\zeta_{mp^n})_v}$
is the sum of the Bloch--Kato dual exponential maps at $v$ a prime of $\mathbb{Q}(\zeta_{mp^n})$ dividing $p$,
  $c_{k-r, mp^n, \gamma} \in \mathbb{Q}(\zeta_{mp^n}) \otimes F_\pi$, and
 $\omega_f$ is the basis of $S(f)$ induced from the \emph{normalized} eigenform $f$.

Following the convention of Kato's period map $\mathrm{per}_f$ \cite[$\S$7.6, (7.13.6) and Theorem 16.2]{kato-euler-systems}, 
Kato's periods are defined by
$$\mathrm{per}_f (\omega_f ) = 
\Omega^{+}_{\gamma, \omega_f} \cdot \gamma^+ +
\Omega^{-}_{\gamma, \omega_f} \cdot \gamma^- .$$
 
\begin{thm} \label{thm:kato_interpolation}
For a primitive Dirichlet character $\chi$ of conductor $mp^n$ with $(m,Np)=1$, we have the interpolation formula
$$
 \sum_{b \in (\mathbb{Z}/mp^n\mathbb{Z})^\times}  \chi(b) \cdot 
\sigma_b \left( c_{k-r, mp^n, \gamma} \right) \cdot \omega_f 
 = 
(2 \pi \sqrt{-1}  )^{k-r-1} \cdot \dfrac{L^{(p)}(\overline{f}, \chi, r)}{  \Omega^{\pm}_{\gamma, \omega_f} } \cdot \omega_{f} 
$$
where
 $L^{(p)}(\overline{f}, \chi, r)$ is the $p$-imprimitive $L$-value of $\overline{f}$ at $s = r$ twisted by $\chi$, 
$\Omega^{\pm}_{\gamma, \omega_f}$ is Kato's period, and the sign of Kato's period coincides with the sign of $(-1)^{k-r-1} \cdot \chi(-1)$.
\end{thm}
\begin{proof}
It immediately follows from Theorem \ref{thm:kato_interpolation_original} and the definition of canonical Kato's Euler systems.
\end{proof}
Comparing the conventions between Kato's periods \cite[$\S$7.6]{kato-euler-systems} and Mazur--Tate--Teitelbaum's ones \cite[$\S$1 and $\S$7]{mtt}, we obtain
%
%
$$C^\pm_{\mathrm{per}} \cdot \Omega^{\pm}_{\gamma, \omega_f}  =  - (2\pi \sqrt{-1})^{k-2} \cdot \Omega^{\pm}_{\overline{f}} $$
where 
\begin{itemize}
\item $C^\pm_{\mathrm{per}} \in \mathcal{O}_{(\pi)}$ is uniquely determined up to $\mathcal{O}^{\times}_{(\pi)}$, respectively, 
\item (the sign of $C^\pm_{\mathrm{per}}$) = (the sign of $\Omega^{\pm}_{\gamma, \omega_f}$) = (the sign of $\Omega^{\pm}_{\overline{f}}$) if $k$ is even, and
\item (the sign of $C^\pm_{\mathrm{per}}$) = (the sign of $\Omega^{\mp}_{\gamma, \omega_f}$) = (the sign of $\Omega^{\pm}_{\overline{f}}$) if $k$ is odd.
\end{itemize}
Let  $\delta^{\pm}_f \in V^{\pm}_F(f)$ be non-zero elements determined by the relation
\begin{align} \label{eqn:delta_f}
\begin{split}
\mathrm{per}_f (\omega_f ) & = \gamma^+ \otimes \Omega^+_{\gamma, \omega_f} + \gamma^- \otimes \Omega^-_{\gamma, \omega_f}  \\
& = \delta^+_f \otimes C^\pm_{\mathrm{per}} \cdot \Omega^+_{\gamma, \omega_f} + \delta^-_f \otimes C^\mp_{\mathrm{per}} \cdot \Omega^-_{\gamma, \omega_f} .
\end{split}
\end{align}
where the sign convention of $C^{\pm}_{\mathrm{per}}$ follows the above.

For our purpose, we modify the Euler system $\left\lbrace z_{\mathbb{Q}(\zeta_{mp^n}), \gamma }(f, k-r) \right\rbrace_{\mathbb{Q}(\zeta_{mp^n})}$ in the sense of ``Varying the Euler factors" \cite[$\S$9.6]{rubin-book}.
\begin{cor} \label{cor:kato-euler-systems-modification}
For each $\gamma \in T_f$,
there exists a system of cohomology classes
$$\left\lbrace \mathfrak{z}_{\mathbb{Q}(\zeta_{mp^n}), \gamma}(f, k-r) \in  \mathrm{H}^1(\mathbb{Z}[\zeta_{mp^n}, 1/p],  T_f(k-r) ) \right\rbrace_{\mathbb{Q}(\zeta_{mp^n})}$$
such that
\begin{enumerate}
\item $\mathfrak{z}_{\mathbb{Q}(\zeta_{m \ell p^n}), \gamma}(f, k-r)$ maps to
$$\left( 1 - a_\ell(\overline{f}) \cdot \mathrm{Frob}^{-1}_\ell \cdot \ell^{1-r} + \overline{\psi}(\ell) \cdot \mathrm{Frob}^{-2}_\ell \cdot \ell^{k-2-2r} \right) \cdot  \mathfrak{z}_{\mathbb{Q}(\zeta_{m p^n}), \gamma}(f, k-r)$$
under the corestriction map in Theorem \ref{thm:kato_euler_system_relation},
\item
$\mathfrak{z}_{\mathbb{Q}(\zeta_{p^n}), \gamma}(f, k-r)
=z_{\mathbb{Q}(\zeta_{p^n}), \gamma}(f, k-r)$, and
\item if we choose the optimal period $\gamma = \delta_f \in V_f$, then we have
$$\sum_{b \in (\mathbb{Z}/mp^n\mathbb{Z})^\times}  \chi(b) \cdot   \mathrm{exp}^{*} \left(  \mathrm{loc}_p \left( \mathfrak{z}_{\mathbb{Q}(\zeta_{mp^n}), \delta_f }(f, k-r)^{\sigma_b} \right) \right)  = 
 - \dfrac{  L^{(p)}(\overline{f}, \chi, r)}{(2 \pi \sqrt{-1})^{r-1} \cdot \Omega^{\pm}_{\overline{f}}} \cdot \omega_{f} $$
under the same setting in Theorem \ref{thm:kato_interpolation}
where $\pm$ is the sign of $(-1)^{r-1} \cdot \chi(-1)$.
\end{enumerate}
\end{cor}
\begin{proof}
It is a simple application of \cite[Lem. 9.6.1]{rubin-book}.
\end{proof}
The \textbf{modified Kato's Euler system over the cyclotomic tower} is defined by
$$\mathfrak{Z}_{\mathbb{Q}(\zeta_{m}), \gamma}(f, k-r) :=
\varprojlim_n \mathfrak{z}_{\mathbb{Q}(\zeta_{mp^n}), \gamma}(f, k-r) $$
where $\gamma \in V_f$, and we have $\mathfrak{Z}_{\mathbb{Q}, \gamma}(f, k-r) = \mathbf{z}_{\mathbb{Q}, \gamma}(f, k-r)$.

The following Iwasawa main conjecture without $p$-adic $L$-functions is due to Kato \cite[Conj. 12.10]{kato-euler-systems}.
\begin{conj}[Kato's main conjecture] \label{conj:kato-main-conjecture}
Let $i \in \lbrace 0, \cdots , p-2 \rbrace$.
Let $\mathbf{z}_{\mathbb{Q}, \gamma_0 }(f, i, k-r) \in \mathrm{H}^1_{\mathrm{Iw}}(T_{f,i}(k-r))$ be Kato's zeta element for $T_{f,i}(k-r)$ over the cyclotomic $\mathbb{Z}_p$-extension of $\mathbb{Q}$  where $r \in \mathbb{Z}$.
If  $\rho_f$ has large image,  then 
\begin{equation} \label{eqn:kato-main-conjecture}
\mathrm{char}_{\Lambda^{(i)}_{\mathrm{Iw}}} \dfrac{  \mathrm{H}^1_{\mathrm{Iw}}(T_{f,i}(k-r)) }{ \Lambda^{(i)}_{\mathrm{Iw}}\mathbf{z}_{\mathbb{Q}, \gamma_0 }(f, i, k-r) }
= \mathrm{char}_{\Lambda^{(i)}_{\mathrm{Iw}}} \mathrm{H}^2_{\mathrm{Iw}}(T_{f,i}(k-r)) 
\end{equation}
 as ideals of $\Lambda^{(i)}_{\mathrm{Iw}}$.
\end{conj}
We say that the \textbf{main conjecture for $T_{f,i}(k-r)$ holds} if (\ref{eqn:kato-main-conjecture}) holds for the fixed $i$. It is easy to see that Conjecture \ref{conj:kato-main-conjecture} for all $i$ is independent of $r$\footnote{However, Tate twist changes $i$; for example, the statements for
$\mathrm{H}^a_{\mathrm{Iw}}(T_{f,i}(k-r))$ are equivalent to those for $\mathrm{H}^a_{\mathrm{Iw}}(T_{f,i+r'-r}(k-r'))$. See \cite[$\S$6.5]{rubin-book} for detail.}.
\begin{thm}[Kato] \label{thm:kato-divisibility}
One inclusion $\subseteq$ in (\ref{eqn:kato-main-conjecture}) holds.
\end{thm}
\begin{proof}
See \cite[Thm. 12.5.(4)]{kato-euler-systems} and \cite[Thm. 5.3.10.(i)]{mazur-rubin-book}.
\end{proof}
We also concern a slightly generalized form of Conjecture A of Coates--Sujatha \cite[Conj. A]{coates-sujatha-fine-selmer}.
\begin{conj}[Coates--Sujatha] \label{conj:coates-sujatha}
For each $i = 0, \cdots, p-2$, the finitely generated torsion $\Lambda^{(i)}_{\mathrm{Iw}}$-module $\mathrm{H}^2_{\mathrm{Iw}}(T_{f, i}(k-r))$ is indeed a finitely generated $\mathcal{O}_\pi$-module. In other words,
\begin{equation*}
 \mu \left( \mathrm{H}^2_{\mathrm{Iw}}(T_{f, i}(k-r)) \right)  = 0.
\end{equation*}
\end{conj}

\section{Review of $p$-adic Hodge theory} \label{sec:p-adic-hodge-theory}
We quickly review the basic notion of $p$-adic Hodge theory emphasizing its \emph{integral} aspect via the theory of Wach modules.
Especially, we focus on Benois' $\varphi$-stable integral lattice in the crystalline Dieudonn\'{e}  module, which gives us the relevant integrality.
The existence of such a lattice is essential for the construction of the compatible families of the \emph{integral} local points in $\S$\ref{sec:mazur-tate-etale}. 
Although no new result is presented in this section, the content and the convention of this section play important roles in proving the integrality of the finite layer Coleman maps.
We closely follow \cite[$\S$1]{benois-crystalline-duke}.
\subsection{Rings of $p$-adic periods}
We review the rings of $p$-adic periods \`{a} la Fontaine. See \cite{berger-intro, caruso-intro} for details.
For convenience, we record both notations due to Colmez and Fontaine.

Let $K/\mathbb{Q}_p$ be a finite \emph{unramified} extension and $W  = \mathcal{O}_K$ with finite residue field $\mathbb{F}$ and $C = \widehat{\overline{K}}$ the $p$-adic completion of an algebraic closure of $K$. 
Let $v : C \to \mathbb{R} \cup \lbrace \infty \rbrace$ be the valuation of $C$ normalized by $v(p) = 1$.
Let 
$\mathcal{R} = \widetilde{\mathbf{E}}^+ := { \displaystyle \varprojlim_{x \mapsto x^p} \mathcal{O}_C /p\mathcal{O}_C }$ and $ \mathbf{A}_{\mathrm{inf}} = \widetilde{\mathbf{A}}^+ := W(\widetilde{\mathbf{E}}^+)$
where $W(R)$ is the ring of Witt vector of (perfect) ring $R$.

An element $x \in \mathcal{R}$ can be viewed as a sequence $(x_n)_{n \geq 0}$ such that $x^p_{n+1} = x_n$ for all $n \geq 0$.
Let $\widehat{x}_n$ be any lifting of $x_n$ in $\mathcal{O}_C$, and then the sequence $\widehat{x}^{p^n}_{m+n}$ converges to some $x^{(m)}$ in $\mathcal{O}_C$ and it is independent of the choice of liftings. 
The ring $\mathcal{R}$ equipped with the valuation $v_{\mathcal{R}}(x) = v(x^{(0)})$ is a complete local ring of characteristic $p$ with residue field $\overline{\mathbb{F}}_p$. Furthermore, $\mathcal{R}$ is integrally closed in $\mathrm{Frac}(\mathcal{R})$.
\begin{choice} \label{choice:p-adic-orientation}
We fix a generator $\varepsilon = \left( \zeta_{p^n} \right)_{n \geq 0}$ of $\mathbb{Z}_p(1)$ and denote by the same letter the unique element $\varepsilon \in \mathcal{R}$ such that $\varepsilon^{(n)} = \zeta_{p^n}$.
\end{choice}
Then $W(\mathcal{R})$ and $W(\mathrm{Frac}(\mathcal{R}))$ are complete rings of characteristic zero containing $W(\overline{\mathbb{F}}_p)$. Moreover, they are endowed with the action of the Galois group $G_K$ and the Frobenius $\varphi$. Let
 $[-] : \mathcal{R} \to W(\mathcal{R} )$
 be the Teichm\"{u}ller lifting. 
 Then any $u =  (u_n)_{n \geq 0} \in W(\mathcal{R})$ can be written as
$u = \sum_{m \geq 0} [u^{p^{-m}}_m] \cdot p^m $.

Put $\pi = [\varepsilon] - 1 \in W(\mathcal{R})$ and $\pi_n = \varphi^{-n}(\pi) = [\varepsilon]^{1/p^n} -1 \in W(\mathcal{R})$ for $n \geq 0$; thus, $\pi = \pi_0$.
Let 
$S_n = W \llbracket \pi_n \rrbracket \subseteq W(\mathcal{R})$
be the $W$-subring of $W(\mathcal{R})$ generated by $\pi_n$. Write $S =S_0$. Then $S_n$ is stable under the action of $G_K$ and $\varphi$; more explicitly, we have
\[
\xymatrix{
g(\pi_n) = (1+\pi_n)^{\chi_{\mathrm{cyc}}(g)} - 1 , &
\varphi(\pi_n) = (1 + \pi_n)^p -1
}
\]
where $g \in G_K$ and $\chi_{\mathrm{cyc}} : G_K \to \mathrm{Gal}(K(\zeta_{p^\infty})/K) \simeq \mathbb{Z}^\times_p$ is the $p$-adic cyclotomic character.

Since $\pi_n$ is invertible in $W(\mathrm{Frac}(\mathcal{R}))$, we regard $S_n[\pi^{-1}_n]$ as a subring of $W(\mathrm{Frac}(\mathcal{R}))$.
We write
$$
\mathcal{O}_{\mathcal{E}, \pi_n}   = \widehat{S_n[\pi^{-1}_n]}   = \left\lbrace \sum_{k \in \mathbb{Z}} a_k \pi^k_n : a_k \in W, a_k \to 0 \textrm{ as }k \to -\infty \right\rbrace 
$$
where $\widehat{(-)}$ means the $p$-adic completion (the strong topology).
When $n=0$, we also write $\mathcal{O}_{\mathcal{E}, \pi}  = \mathcal{O}_{\mathcal{E}}$.
Denote by $\mathcal{O}^{\mathrm{ur}}_{\mathcal{E}, \pi_n}$ the maximal unramified extension of $\mathcal{O}_{\mathcal{E}, \pi_n}$ and by $\widehat{\mathcal{O}^{\mathrm{ur}}_{\mathcal{E}, \pi_n}}$ the $p$-adic completion of $\mathcal{O}^{\mathrm{ur}}_{\mathcal{E}, \pi_n}$.
Denote by
$S^{\mathrm{PD}} = S \left[ \dfrac{\pi^2}{2!}, \dfrac{\pi^3}{3!}, \cdots \right]$
the PD-envelope of $S$
and $I^{\langle i \rangle}$ the ideal of $S^{\mathrm{PD}}$ generated by $\dfrac{\pi^m}{m!}$, $m \geq i$.
Write
$\widehat{S^{\mathrm{PD}}} = \varprojlim_{i} S^{\mathrm{PD}} / I^{\langle i \rangle} $.

Let 
$\theta : \mathbf{A}_{\mathrm{inf}} = \widetilde{\mathbf{A}}^+ = W(\mathcal{R}) \to \mathcal{O}_C$
 be the map given by the formula
$\theta \left( \sum_{m \geq 0} [u_m] \cdot p^m \right) = \sum_{m \geq 0} u^{(0)}_m \cdot p^m $.
Then $\theta$ is a surjective ring homomorphism and $\mathrm{ker} \ \theta$ is generated by
 $\omega = \sum_{i=0}^{p-1} [\varepsilon]^{i/p} = \dfrac{[\varepsilon]-1}{[\varepsilon]^{1/p}-1}$
 where $[\varepsilon]^{1/p} = [(\varepsilon^{(1)}, \cdots)]$.
By inverting $p$, we extend $\theta$ to
$\theta : \mathbf{B}^+_{\mathrm{inf}} = \widetilde{\mathbf{B}}^+ = W(\mathcal{R})[1/p] \to C$
and define
$\mathbf{B}^+_{\mathrm{dR}} := \varprojlim_n \mathbf{B}^+_{\mathrm{inf}} / \left( \mathrm{ker} \ \theta \right)^n $.
It is known that $\mathbf{B}^+_{\mathrm{dR}}$ is a complete discrete valuation ring with residue field $C$ and admits $G_K$-action.
Since $\theta(\pi) = \theta([\varepsilon] - 1) = 0$, $[\varepsilon] - 1$ is ``small" with respect to the topology of  $\mathbf{B}^+_{\mathrm{dR}}$.
Thus, the series
$\mathrm{log}([\varepsilon]) =  \sum_{m \geq 1}^{\infty} (-1)^{m+1} \dfrac{\pi^m}{m} = - \sum_{m \geq 1}^{\infty} \dfrac{(1-[\varepsilon])^m}{m}
$
converges to the ``$p$-adic $2\pi i$" $t$ in $\mathbf{B}^+_{\mathrm{dR}}$ with respect to the $(\mathrm{ker} \ \theta)$-adic topology.
The element $t$ also generates the maximal ideal of $\mathbf{B}^+_{\mathrm{dR}}$ and $G_K$ acts on $t$ by $g \cdot t = \chi_\mathrm{cyc}(g) \cdot t$.
The field of fractions of $\mathbf{B}^+_{\mathrm{dR}}$ is denoted by $\mathbf{B}_{\mathrm{dR}}$, and $\mathbf{B}_{\mathrm{dR}}$ is the complete discrete valuation field equipped with $G_K$-action and the filtration given by $t$.

Let 
$\mathbf{A}_{\mathrm{cris}} := \widehat{ \mathbf{A}_{\mathrm{inf}}\left[ \dfrac{\omega^2}{2!}, \dfrac{\omega^3}{3!}, \cdots \right] } $ be
the $p$-adic completion of the PD-envelope of $\mathbf{A}_{\mathrm{inf}}$ with respect to $\mathrm{ker} \ \theta$.
It is easy to see $t \in \mathbf{A}_{\mathrm{cris}}$.
We define
$\mathbf{B}^+_{\mathrm{cris}} := \mathbf{A}_{\mathrm{cris}}[1/p]$ and
$\mathbf{B}_{\mathrm{cris}} := \mathbf{B}^+_{\mathrm{cris}}[1/t] =  \mathbf{A}_{\mathrm{cris}}[1/p, 1/t]$.
It is known that $t$ lies in $\mathbf{B}^+_{\mathrm{dR}}$ and $\mathbf{A}_{\mathrm{cris}}$.
Due to \cite[$\S$1.1.2]{benois-crystalline-duke}, $\widehat{S^{\mathrm{PD}}}$ embeds in $\mathbf{A}_{\mathrm{cris}}$.
\subsection{The classification of Galois representations via $(\varphi, \Gamma)$-modules}
Let $K_n = K(\zeta_{p^n})$ and $K_\infty = \cup_{n \geq 1} K_n$.
We recall Fontaine's theorem on the classification of continuous representations of $G_{K_n}$ over $\mathbb{Z}_p$ in terms of $(\varphi, \Gamma)$-modules \cite{fontaine-grothendieck}.

Let $\mathrm{Rep}_{\mathbb{Z}_p}(G_{K_n})$ be the category of finite $\mathbb{Z}_p$-modules equipped with a continuous linear action of $G_{K_n}$ and $M^{\mathrm{\acute{e}t}}_{\varphi, \Gamma}( \mathcal{O}_{\mathcal{E}, \pi_n} )$ the category of \'{e}tale $(\varphi, \Gamma_n)$-modules over $\mathcal{O}_{\mathcal{E}, \pi_n}$ where $\Gamma_n = \mathrm{Gal}(K_\infty/K_n)$. 
\begin{thm}[Fontaine]
The functor
\[
\xymatrix{
\mathbf{D}_{\mathcal{O}_{\mathcal{E}, \pi_n}} : \mathrm{Rep}_{\mathbb{Z}_p}(G_{K_n}) \to M^{\mathrm{\acute{e}t}}_{\varphi, \Gamma}( \mathcal{O}_{\mathcal{E}, \pi_n} ) , & T \mapsto \mathbf{D}_{\mathcal{O}_{\mathcal{E}, \pi_n}}(T) = \left(  T \otimes_{\mathbb{Z}_p} \widehat{\mathcal{O}^{\mathrm{ur}}_{\mathcal{E}, \pi_n}} \right)^{G_{K_\infty}}
}
\]
is an equivalence of categories.
The functor
\[
\xymatrix{
\mathbf{V}_{\mathcal{O}_{\mathcal{E}, \pi_n}} : M^{\mathrm{\acute{e}t}}_{\varphi, \Gamma}( \mathcal{O}_{\mathcal{E}, \pi_n} ) \to
 \mathrm{Rep}_{\mathbb{Z}_p}(G_{K_n}) , & M \mapsto \mathbf{V}_{\mathcal{O}_{\mathcal{E}, \pi_n}} (M) = \left( M \otimes_{ \mathcal{O}_{\mathcal{E}, \pi_n} } \widehat{\mathcal{O}^{\mathrm{ur}}_{\mathcal{E}, \pi_n}} \right)^{\varphi = 1}
}
\]
is a quasi-inverse to $\mathbf{D}_{\mathcal{O}_{\mathcal{E}, \pi_n}}$.
\end{thm}
\begin{proof}
See \cite[Thm. 3.4.3 and Rem. 3.4.4.(b)]{fontaine-grothendieck}.
More specifically, we take $L=K_n$ and $K= K_0$ since $K$ is unramified in \cite[Rem. 3.4.4.(b)]{fontaine-grothendieck}.
\end{proof}

\subsection{The classification of crystalline representations via Wach modules} 
We construct an integral lattice of $\mathbf{D}_{\mathrm{cris}}(V)$ in terms of Wach modules.
Let 
$\widetilde{S}_n := \widehat{\mathcal{O}^{\mathrm{ur}}_{\mathcal{E}, \pi_n}} \cap \mathbf{A}_{\mathrm{inf}} = \widehat{\left( \widehat{S_n[\pi^{-1}_n]}^{\mathrm{ur}} \right)} \cap \mathbf{A}_{\mathrm{inf}}$
 for $n \geq 0$ and say $\widetilde{S} = \widetilde{S}_0$.
Let $T$ be a free $\mathbb{Z}_p$-module of finite rank endowed with continuous linear action of $G_{K_n}$.
We define 
$\mathbf{D}_{S_n}(T) := \left( T \otimes_{\mathbb{Z}_p} \widetilde{S}_n \right)^{G_{K_\infty}} $
and
$\mathbf{D}_{S}(T) := \left( T \otimes_{\mathbb{Z}_p} \widetilde{S} \right)^{G_{K_\infty}} $.
Then we have
$\mathbf{D}_{S}(T) = \varphi^n \mathbf{D}_{S_n}(T) $
and
$\mathbf{D}_{S_n}(T)$ is a free $S_n$-module with natural actions of $\varphi$ and $\Gamma_n$.
One has
$\mathrm{rk}_{S_n} \mathbf{D}_{S_n}(T)
\leq \mathrm{rk}_{\mathcal{O}_{\mathcal{E}, \pi_n}} \mathbf{D}_{\mathcal{O}_{\mathcal{E}, \pi_n}}(T)
=  \mathrm{rk}_{\mathbb{Z}_p} (T) $.

We say that, as a representation of $G_{K_n}$,  $T$ is a \textbf{representation of finite height} if $\mathrm{rk}_{S_n} \mathbf{D}_{S_n}(T) = \mathrm{rk}_{\mathbb{Z}_p} (T)$.
\begin{thm}[Colmez] \label{thm:colmez-crystalline-finite-height}
Any crystalline representation of $G_K$ is of finite height.
 \end{thm}
 \begin{proof}
It follows from \cite{colmez-finite-height} since $K$ is unramified over $\mathbb{Q}_p$.
 \end{proof}
\begin{thm}[Wach] \label{thm:wach}
Let $T$ be a representation of finite height.
Then the following statements are equivalent:
\begin{enumerate}
\item $V = T \otimes \mathbb{Q}_p$ is crystalline as a representation of $G_{K_n}$.
\item There exists a free $S_n$-submodule $N \subseteq \mathbf{D}_{S_n}(T)$ of rank $d$ such that $\Gamma_n$ acts trivially on $(N/\pi N)(-h)$ for some $h \in \mathbb{Z}$.
\end{enumerate}
\end{thm}
The integer $h$ in Theorem \ref{thm:wach} satisfies $\mathbf{D}^{-h}_{\mathrm{cris}}(V) = \mathbf{D}_{\mathrm{cris}}(V)$ as in \cite[$\S$4.1.2]{benois-crystalline-duke}. See also Assumption \ref{assu:filtration-1}.
The lattice we want to construct naturally arises in the proof of $(2) \Rightarrow (1)$ in Theorem \ref{thm:wach}.
See \cite[Page 220--221]{benois-crystalline-duke}, \cite{wach-pot-crys} for details.
Let $N_S$ be the $N$ in Theorem \ref{thm:wach} when $n=0$.
%
%
%
Since 
$\left( N_S \widehat{\otimes}_S \widehat{S^{\mathrm{PD}}} \right)^{\Gamma_n} \subseteq \left( \mathbf{D}_S(T) \widehat{\otimes}_S \widehat{S^{\mathrm{PD}}} \right)^{\Gamma_n} \subseteq (T \otimes \mathbf{A}_{\mathrm{cris}})^{G_{K_n}}$,
the $W$-module
\begin{equation} \label{eqn:wach-lattice}
M := \left( \mathbf{D}_S(T) \widehat{\otimes}_S \widehat{S^{\mathrm{PD}}} \right)^{\Gamma_n} \subseteq \mathbf{D}_{\mathrm{cris}}(V) 
\end{equation}
is a $\varphi$-stable lattice of $\mathbf{D}_{\mathrm{cris}}(V)$. This is the lattice we need for the construction of equivariant interal local points.
Another description of the lattice $M$ can also be found in \cite[$\S$3.2]{benois-berger}.
\section{Local Iwasawa theory and explicit reciprocity law \`{a} la Perrin-Riou} \label{sec:local-iwasawa-theory-perrin-riou}
We review local Iwasawa theory, the big exponential map, and the explicit reciprocity law developed by Perrin-Riou \cite{perrin-riou-local-iwasawa} and its refinement due to Benois and Berger \cite{benois-crystalline-duke, berger-three-explicit-formulas, benois-berger, benois-near-central}.
Our description keeps the \emph{integrality} information as much as possible; as a consequence, we are able to construct \emph{equivariant} $p$-adic $L$-functions of modular forms with the precise denominator information.
The integrality in \cite{perrin-riou-local-iwasawa} is based on Fontaine--Laffaille modules as described in \cite[$\S$2.1.5 and 2.2.1.Prop. (iv)]{perrin-riou-local-iwasawa}, but the integrality here comes from Wach modules following \cite{benois-crystalline-duke}.
\subsection{Rings of power series with ($\varphi, \Gamma$)-action} \label{subsec:rings-of-power-series}
For $h \geq 1$ and a subfield $L \subseteq \mathbb{C}_p$, we put
$$\mathscr{H}_{h,L}  = \left\lbrace \sum_{n \geq 0} a_n X^n \in L \llbracket X \rrbracket : \lim_{n \to \infty} \dfrac{\vert a_n \vert_p}{n^h} = 0 \right\rbrace $$
and
$\mathscr{H}_{L}$ is the ring of convergent power series over $L$ on the open unit ball.

Let $K/\mathbb{Q}_p$ be a finite unramified extension and consider $\mathscr{H}_K \subseteq K\llbracket X \rrbracket$.
Let $\sigma : K \to K$ be the absolute Frobenius map on $K$ and extend it to $\mathscr{H}_K$ by
$\sigma \left( \sum_{n \geq 0} a_n X^n \right) =  \sum_{n \geq 0} a^{\sigma}_n X^n $.
We define the Frobenius map on $\varphi : \mathscr{H}_K \to \mathscr{H}_K$
by
$\varphi \left(  \sum_{n \geq 0} a_n X^n \right) = \sum_{n \geq 0} a^\sigma_n \left( \left( 1+X \right)^p-1 \right)^n $.
By abuse of notation, denote by $\varphi$ the Frobenius map  $\varphi \otimes \varphi$ on 
$\mathscr{H}_K \otimes_K \mathbf{D}_{\mathrm{cris}}(V)$
where $V$ is a crystalline representation of $G_K$.
Also, $\mathrm{Gal}(K_\infty/K)$ acts on $\mathscr{H}_K$ by
$\gamma \left(  \sum_{n \geq 0} a_n X^n \right) = \sum_{n \geq 0} a_n \left( \left( 1+X \right)^{\chi_{\mathrm{cyc}}(\gamma)}-1 \right)^n $
where $\gamma \in \mathrm{Gal}(K_\infty/K)$ and $\chi_{\mathrm{cyc}}$ is the $p$-adic cyclotomic character.

Let  $W \subseteq K$ be the ring of integers and the action of $\varphi$ and $\mathrm{Gal}(K_\infty/K)$ on $\mathscr{H}_K$ can be restricted to $W\llbracket X \rrbracket$.
Also, the Frobenius map $\varphi$ (abusing notation again) on $W\llbracket X \rrbracket \otimes M$ is defined by  $\varphi \otimes \varphi$ where $M$ is the $\varphi$-stable lattice constructed in (\ref{eqn:wach-lattice}).

As in \cite[$\S$1.1.3]{perrin-riou-local-iwasawa}, let $\psi$ be the $\sigma^{-1}$-semi-linear operator on $K\llbracket X \rrbracket$ characterized by\footnote{This $\psi$ is not the character of the modular form. They are clearly distinguished by the context.}
$$\varphi \circ \psi (f(X)) = p^{-1} \cdot \sum_{\zeta \in \mu_p} f(\zeta \cdot (1+X) - 1)$$
where $f(X) \in K\llbracket X \rrbracket$.
Then $\psi \circ \varphi = 1$. 
We define the $\psi=0$ subring of $W\llbracket X \rrbracket$ by
$$W\llbracket X \rrbracket^{\psi=0} = \left\lbrace f(X) \in W\llbracket X \rrbracket : \sum_{\zeta \in \mu_p} f(\zeta \cdot (1+X) - 1) =0 \right\rbrace $$
as in \cite[$\S$1.1.6]{perrin-riou-local-iwasawa} and \cite[$\S$4.1.2]{benois-crystalline-duke}.
It is known that 
$W\llbracket X \rrbracket^{\psi=0}$ is a $W$-module stable under the $\mathrm{Gal}(K_{\infty}/K)$-action (via the cyclotomic character) and
is a free $W\llbracket \mathrm{Gal}(K_{\infty}/K) \rrbracket$-module of rank one generated by $1+X$.
Furthermore, it admits the unique product structure $*$ compatible with the $\mathrm{Gal}(K_{\infty}/K)$-module structure such that
$(1+X) * (1+X) = (1+X)$. If one interpret $W\llbracket X \rrbracket^{\psi=0}$ as the space of $W$-valued measures on $\mathbb{Z}^\times_p$, then this product structure corresponds to the convolution product. See \cite[$\S$3.6.3]{perrin-riou-local-iwasawa} for details.

\subsection{The twist convention} \label{subsec:twist-convention}
We follow the convention given in \cite[2.1.6 Ex.]{perrin-riou-local-iwasawa}.
Let $\mathbb{Q}_p(1) $ be the $p$-adic representation of $G_K$ associated with the cyclotomic character.
We put 
$$\mathbb{Q}_p(i)
  = \left\lbrace \begin{array}{ll}
  \mathbb{Q}_p(1)^{\otimes i} , & i \geq 0 , \\
  \mathrm{Hom}_{\mathbb{Q}_p}( \mathbb{Q}_p(1), \mathbb{Q}_p)^{ \otimes - i} , & i < 0 .
  \end{array} \right.
$$
Let $K[-i] = \mathbf{D}_{\mathrm{cris}}(\mathbb{Q}_p(i))$ be the filtered $\varphi$-module associated to $\mathbb{Q}_p(i)$.
Let $e_{-i} = t^{-i} \otimes \varepsilon^{\otimes i}$ (as in \cite[Thm. 3.1.(2)]{benois-berger}) be the canonical basis of $K$-vector space $K[-i]$. Then we have
$$\varphi \cdot e_{-i} = p^{-i} \cdot e_{-i} .$$ The filtration of $K[-i]$ is given by
\begin{align*}
\mathrm{Fil}^{-i} K[-i] & = \mathbf{D}^{-i}_{\mathrm{cris}}(\mathbb{Q}_p(i)) = \mathbf{D}_{\mathrm{cris}}(\mathbb{Q}_p(i)) , \\
\mathrm{Fil}^{-i+1} K[-i] & = \mathbf{D}^{-i+1}_{\mathrm{cris}}(\mathbb{Q}_p(i)) = 0 .
\end{align*}
Let $V$ be a crystalline representation of $G_K$ and write  $V(i) = V \otimes_{\mathbb{Q}_p} \mathbb{Q}_p(i)$.
Let $D$ be a filtered $\varphi$-module associated to a crystalline representation of $G_K$
and write $D[i] = D \otimes_K K[i]$.
Then we have
$\mathbf{D}_{\mathrm{cris}}(V(i)) = \mathbf{D}_{\mathrm{cris}}(V)[-i]$.
 In other words, we have
$\mathbf{D}_{\mathrm{cris}}(V(i))  = \mathbf{D}_{\mathrm{cris}}(V) \otimes e_{-i} $ and
$ \mathbf{D}^{-j}_{\mathrm{cris}}(V(i)) =  \mathbf{D}^{i-j}_{\mathrm{cris}}(V) \otimes e_{-i} $
in terms of the filtration.
\begin{rem}
The convention on $e_{-i}$ is compatible with \cite[$\S$0.3]{benois-crystalline-duke}, \cite[Thm. 3.1.(2)]{benois-berger}, and \cite{benois-near-central}, but not with \cite[II.5]{berger-three-explicit-formulas} and \cite[Notation]{nakamura-local-iwasawa}.
\end{rem}
\subsection{Differential operators} \label{subsec:differential-operators}
We follow \cite[$\S$1.1.5]{perrin-riou-local-iwasawa}.
Let
$D = (1+ X) \cdot \dfrac{d}{dX}$
be the differential operator on $\mathscr{H}_K$.
Then we have $D(f)(e^z - 1) = \dfrac{dh}{dz}(z)$ where $h(z) = f(e^z -1)$.
For any $\alpha \in \mathbb{Z}_p$, we have
$D( f ( (1+X)^\alpha - 1 ) ) = \alpha \cdot D(f) ((1+X)^\alpha - 1)$.
Note that $D$ is a $W$-linear automorphism of  $W\llbracket X \rrbracket^{\psi=0}$; thus, $D^{i}$ is defined for all $i \in \mathbb{Z}$.
If $f \in K\llbracket X\rrbracket$ and $j \in \mathbb{Z}$, we define
$\Delta_j (f) = D^j (f)(0) $
following Kummer, Coates--Wiles, and Perrin-Riou. 
\subsection{Equation $(1-\varphi)G =g$} \label{subsec:solving-eqn-frobenius}
The main references are \cite[$\S$2.2]{perrin-riou-local-iwasawa} and \cite[$\S$4.1]{benois-crystalline-duke}.
We would like to consider equation
$(1-\varphi)G =g$
in $\mathscr{H}_K \otimes_{K} \mathbf{D}_{\mathrm{cris}}(V)$
where $V$ is a finite dimensional crystalline representation of $G_K$.
The maps
$\Delta_i : \mathscr{H}_K \to K$ and $D : \mathscr{H}_K \to \mathscr{H}_K $
naturally induce the maps (with the same notation)
\[
\xymatrix@R=0em{
\Delta_i : \mathscr{H}_K \otimes_{K} \mathbf{D}_{\mathrm{cris}}(V) \to  \mathbf{D}_{\mathrm{cris}}(V) , &
D : \mathscr{H}_K \otimes_{K} \mathbf{D}_{\mathrm{cris}}(V) \to \mathscr{H}_K \otimes_{K} \mathbf{D}_{\mathrm{cris}}(V)  .
}
\]
The Frobenius $\varphi$ acts on $\mathscr{H}_K \otimes_{K} \mathbf{D}_{\mathrm{cris}}(V)$ by $\varphi \otimes \varphi$.
Let $\mathscr{D}(V) := W\llbracket X \rrbracket^{\psi=0} \otimes_W \mathbf{D}_{\mathrm{cris}}(V) $,
which is a $W\llbracket  \mathrm{Gal}(K_{\infty}/K) \rrbracket$-module, and
$\mathscr{H}(V) := \left\lbrace G \in  \mathscr{H}_K \otimes_K \mathbf{D}_{\mathrm{cris}}(V) : (1 - \varphi) G \in \mathscr{D}(V) \right\rbrace $.
Since $D : W\llbracket X \rrbracket^{\psi=0} \to W\llbracket X \rrbracket^{\psi=0}$ is an automorphism, $D$ also naturally defines an automorphism of $\mathscr{D}(V)$ (with the same notation).
\begin{assu} \label{assu:filtration-1}
Let $h$ be any positive integer such that $\mathbf{D}^{-h}_{\mathrm{cris}}(V) = \mathbf{D}_{\mathrm{cris}}(V)$.
\end{assu}
Under Assumption \ref{assu:filtration-1}, we have an exact sequence
\begin{equation} \label{eqn:1-varphi-sequence-0-h}
\xymatrix{
0 \ar[r] &\mathrm{ker} (1-\varphi) \ar[r] & \mathscr{H}(V) \ar[r]^-{1-\varphi} & \mathscr{D}(V) \ar[r]^-{\widetilde{\Delta}_{[0,h]}} & { \displaystyle \bigoplus_{0 \leq j \leq h} \dfrac{ \mathbf{D}_{\mathrm{cris}}(V) }{ (1-p^j \varphi )\mathbf{D}_{\mathrm{cris}}(V) } } \ar[r] & 0
}
\end{equation}
with 
$\mathrm{ker} (1-\varphi) = 
 \sum_{j \geq 0} \left( \mathrm{log}(1+X) \right)^j \otimes \mathbf{D}_{\mathrm{cris}}(V)^{\varphi=p^{-j}}$
where $\widetilde{\Delta}_{[0,h]} = \bigoplus_{0 \leq j \leq h} \widetilde{\Delta}_{j}$ and 
\[
\xymatrix{
\widetilde{\Delta}_{j} :\mathscr{D}(V) \ar[r]^-{\Delta_j} & \mathbf{D}_{\mathrm{cris}}(V) \ar[r]^-{\textrm{natural}}_-{\textrm{quotient}} & \dfrac{ \mathbf{D}_{\mathrm{cris}}(V) }{ (1-p^j \varphi) \mathbf{D}_{\mathrm{cris}}(V) } .
}
\]
See \cite[$\S$2.2.7]{perrin-riou-local-iwasawa} for detail.

We now discuss an integral refinement of the above argument following \cite[$\S$4.1.2]{benois-crystalline-duke}.
\begin{assu} \label{assu:filtration-2}
We further assume that $\mathbf{D}^0_{\mathrm{cris}}(V) = \mathbf{D}_{\mathrm{cris}}(V)$, i.e., $h$ can be taken as 0.
\end{assu}
Let $T$ be a $G_K$-stable $\mathbb{Z}_p$-lattice of $V$ and $M$ the $\varphi$-stable $W$-lattice of $\mathbf{D}_{\mathrm{cris}}(V)$ associated to $T$ as in (\ref{eqn:wach-lattice}).
We define
$\mathscr{D}(T) := W\llbracket X \rrbracket^{\psi=0} \otimes_W M$
and 
$$\mathscr{H}(T) := \left\lbrace G \in W\llbracket X \rrbracket \otimes_W M : \sum_{\zeta \in \mu_p} G(\zeta \cdot (1+X) - 1) = p \cdot \varphi G(X)
 \right\rbrace  .$$
If $G(X) = \alpha(X) \otimes \beta$ with $\alpha(X) \in W\llbracket X \rrbracket$ and $\beta \in M$, then
$\varphi G(X) = \sigma (\alpha ) ((X+1)^p - 1) \otimes \varphi(\beta)$.
One has $\mathscr{D}(T) \otimes \mathbb{Q}_p = \mathscr{D}(V)$ and $\mathscr{H}(T) \otimes \mathbb{Q}_p = \mathscr{H}(V)$ due to \cite[Lem. 4.1.3.(i)]{benois-crystalline-duke}.
Also, if $\mathbf{D}^0_{\mathrm{cris}}(V) = \mathbf{D}_{\mathrm{cris}}(V)$ (Assumption \ref{assu:filtration-2}), then $\mathbf{D}^0_{\mathrm{cris}}(V) = \mathbf{D}_{\mathrm{cris}}(V)^{\varphi = p^{-j}} = 0$ for $j >0$ due to \cite[$\S$4.1.2]{benois-crystalline-duke}.
Under Assumption \ref{assu:filtration-2},
(\ref{eqn:1-varphi-sequence-0-h}) becomes the exact sequence
\[
\xymatrix{
0 \ar[r] & \mathbf{D}_{\mathrm{cris}}(V)^{\varphi = 1} \ar[r] & \mathscr{H}(V)  \ar[r]^-{1 - \varphi} & \mathscr{D}(V)^{\widetilde{\Delta}_0} \ar[r] & 0  .
}
\]
Following \cite[Proof of Lem. 4.1.3]{benois-crystalline-duke}, the integral refinement of the above sequence also holds 
\begin{equation*} \label{eqn:1-varphi-sequence-0-integral}
\xymatrix{
0 \ar[r] & M^{\varphi = 1} \ar[r] & \mathscr{H}(T)  \ar[r]^-{1 - \varphi} & \mathscr{D}(T)^{\widetilde{\Delta}_0} \ar[r] & 0  .
}
\end{equation*}
Solving the equation $(1-\varphi)G=g$ is compatible with twists in the following sense.
\begin{lem} \label{lem:twist-commutative}
For $j \geq 0$, the following diagram commutes
\[
\xymatrix{
\mathscr{H}(V(j)) \ar[r]^-{D^j \otimes e_j} \ar[d]_-{1-\varphi} & \mathscr{H}(V) \ar[d]^-{1-\varphi} \\
\mathscr{D}(V(j)) \ar[r]^-{D^j \otimes e_j} & \mathscr{D}(V)  .
}
\]
\end{lem}
\begin{proof}
See \cite[Lem. 4.1.3.(ii)]{benois-crystalline-duke}.
\end{proof}
\subsection{Finite layer integral exponential maps with twists}
Note that $\mathrm{H}^1_e = \mathrm{H}^1_f$ under our running assumption (Assumption \ref{assu:standard}.(ENV)) and \cite{bloch-kato}.
We recall \cite[$\S$4.2 and Thm. 4.3]{benois-crystalline-duke}.
Let 
\begin{equation*} \label{eqn:bloch-kato-exponential}
\mathrm{exp}_{V(j), K_n} : K_n \otimes  \mathbf{D}_{\mathrm{cris}}(V(j))  \to \mathrm{H}^1_f (K_n, V(j)) \subseteq \mathrm{H}^1 (K_n, V(j))
\end{equation*}
be the Bloch--Kato exponential map following \cite{bloch-kato} with kernel $K_n \otimes  \mathbf{D}^0_{\mathrm{cris}}(V(j)) $.
\begin{thm}[Benois] \label{thm:integral-exponential}
For $n \geq 0$, there exist a family of homomorphisms
\[
\xymatrix{
\Sigma^{\varepsilon}_{T,j, K_n} : \mathscr{H}(T) \to \mathrm{H}^1_f(K_n, T(j)) , & \Sigma^{\varepsilon}_{V,j, K_n} = \Sigma^{\varepsilon}_{T,j, K_n} \otimes \mathbb{Q}_p : \mathscr{H}(V) \to \mathrm{H}^1(K_n, V(j))
}
\]
for all $j \in \mathbb{Z}$
satisfying the following relations
\begin{enumerate}
\item for $G \in \mathscr{H}(T)$ and $n \geq 1$, 
$$\mathrm{cor}_{K_{n+1}/K_n} \Sigma^{\varepsilon}_{T,j, K_{n+1}}(G) = \Sigma^{\varepsilon}_{T, j, K_{n}}( \varphi ( G ) ) .$$
\item for $j \geq 1$ and $G_j \in \mathscr{H}(V(j))$,
$$\Sigma^{\varepsilon}_{V,j, K_n}(D^j G_j \otimes e_j) = (-1)^j  \cdot (j-1)! \cdot p^{(j-1)n} \cdot \mathrm{exp}_{V(j), K_n}( G_j(\zeta_{p^n} - 1) ) .$$
\end{enumerate}
\end{thm}
\begin{proof}
See \cite[Thm. 4.3]{benois-crystalline-duke}. The proof involves the lattice $M$ in (\ref{eqn:wach-lattice}) as described in \cite[$\S$4.4.3]{benois-crystalline-duke}.
\end{proof}
\begin{rem}
Our $\Sigma^{\varepsilon}_{T, j, K_n}$ is $\Sigma^{\varepsilon}_{T, j, n}$ in \cite[$\S$4.2]{benois-crystalline-duke}.
In \cite[Thm. 4.3]{benois-crystalline-duke}, there is a lower bound of the index $s$ such that Theorem \ref{thm:integral-exponential} holds for $n \geq s$.
Since $s$ is the index $n$ in Theorem \ref{thm:wach} following \cite[$\S$4.1.2]{benois-crystalline-duke}, we can take $s = 0$ for crystalline representations of $G_K$.
\end{rem}
\begin{lem} \label{lem:exponential-quotients}
Let $h$ be any integer such that $\mathbf{D}^{-h}_{\mathrm{cris}}(V) = \mathbf{D}_{\mathrm{cris}}(V)$.
\begin{enumerate}
\item If $G \in M^{\varphi=1}$, then $\Sigma^{\varepsilon}_{T,j, K_n} (G)$ maps to zero in $\dfrac{\mathrm{H}^1(K_n, T(j))}{\mathrm{H}^1(K_\infty/K_n, T(j)^{G_{K_\infty}})}$.
\item If $j \not\in [0,h]$, then $V(j)^{G_{K_n}}=0$; thus, $\mathrm{H}^1(K_\infty/K_n, T(j)^{G_{K_\infty}})=0$.
\end{enumerate}
\end{lem}
\begin{proof}
See \cite[$\S$5.1.6]{benois-crystalline-duke}. For the second statement, it is essential to observe that zero is not a Hodge--Tate weight of $V(j)$ if $j \not\in [0,h]$.
\end{proof}
For $j \not\in [0,h]$, the map $\Sigma^{\varepsilon}_{T,j,K_n}$ induces the homomorphism
$\Omega^{\varepsilon}_{V,j, K_n} : \mathscr{D}(V)^{\widetilde{\Delta}=0} \to \mathrm{H}^1(K_n, V(j))$
defined by
\begin{equation} \label{eqn:Omega-Sigma}
\Omega^{\varepsilon}_{V, j, K_n}(g) = \Sigma^{\varepsilon}_{V, j ,K_n}  ( \varphi^{-n} G )
\end{equation}
where $(1 - \varphi) G =g$ as in $\S$\ref{subsec:solving-eqn-frobenius}.
By Theorem \ref{thm:integral-exponential}.(2), we have 
\begin{equation} \label{eqn:Omega-exp}
\Omega^{\varepsilon}_{V,j, K_n}(g) = (-1)^j \cdot (j-1)! \cdot \mathrm{exp}_{V(j), K_n} (\Xi^{\varepsilon}_{V,j,K_n}(g))
\end{equation}
for $j \geq 1$
where 
$\Xi^{\varepsilon}_{V, j, K_n}(g) = p^{-n} \cdot \varphi^{-n} G_j(\zeta_{p^n}-1)$
and $G_j \in\mathscr{H}(V(j))$ is a solution to the equation $(1 - \varphi) G_j = (D^{-j} \otimes e_{-j})(g)$.
By Lemma \ref{lem:twist-commutative}, we have
\begin{equation} \label{eqn:solution-twists}
( D^{-j} \otimes e_{-j} )G = G_j .
\end{equation}
 See \cite[$\S$5.2.3]{benois-crystalline-duke} for detail.
Thus, one has
$\Xi^{\varepsilon}_{V, j, K_n} = \Xi^{\varepsilon}_{V(j), 0, K_n} \circ (D^{-j} \otimes e_{-j})$
as in \cite[$\S$3.1]{benois-berger}.
We also have
$\mathrm{cor}_{K_{n+1}/K_n} \Omega^{\varepsilon}_{V,j, K_{n+1}}(g) = \Omega^{\varepsilon}_{V,j, K_n}(g) $
by Theorem \ref{thm:integral-exponential}.(1).

\subsection{Perrin-Riou's big exponential map} \label{subsec:big-exponential-map}
We carefully state the combination of the theorems of Benois and Berger on a generalization and refinement of Perrin-Riou's big exponential map for crystalline representations over unramified base fields.

Let
$\Xi^{\varepsilon}_{V(j),0, K_n} : \mathscr{D}(V(j))^{\widetilde{\Delta} = 0} \to K_n \otimes  \dfrac{ \mathbf{D}_{\mathrm{cris}}(V(j)) }{ \mathbf{D}_{\mathrm{cris}}(V(j))^{\varphi=1} }$
be the map defined by
$g \mapsto p^{-n} \cdot \varphi^{-n} G_j(\zeta_{p^n}-1)$
where $(1-\varphi)G_j(X) = (D^{-j} \otimes e_{-j})(g)(X)$ and $G_j(X) \in \mathscr{H}(V(j))$.
Write $\Gamma = \mathrm{Gal}(K_\infty/K) \simeq \Gamma^1 \times \Delta$ with $\Delta \simeq \mathbb{F}^\times_p$, and it justifies the notation $\mathscr{H}_K(\Gamma) = \mathscr{H}_K[\Delta]$ by sending a topological generator of $\Gamma^1$ to $1+X$.
Write
 ${ \displaystyle \mathrm{H}^1_{\mathcal{I}w}(K, T(j)) = \varprojlim_n \mathrm{H}^1(K_n, T(j)) }$
where the transition maps are the corestriction maps.
\begin{thm}[Benois--Berger--Perrin-Riou] \label{thm:big-exponential-map}
Let $K$ be a finite unramified extension of $\mathbb{Q}_p$ and $V$ a crystalline representation of $G_K$.
For any integers $h$ and $j$ such that
 $ \mathbf{D}^{-h}_{\mathrm{dR}} (V) =  \mathbf{D}_{\mathrm{dR}} (V)$ and
 $h+j \geq 1$,
there exists a unique $\mathscr{H}_K(\Gamma)$-homomorphism
$$\Omega^{\varepsilon}_{V(j), h+j} : \mathscr{D}(V(j))^{\widetilde{\Delta}=0} \to  \mathscr{H}_K(\Gamma) \otimes_{\Lambda} \dfrac{ \mathrm{H}^1_{\mathcal{I}w}(K, T(j)) }{ T(j)^{G_{K_\infty}} }$$
satisfying the following properties:
\begin{enumerate}
\item For any $n \geq 0$, the diagram
\begin{equation} \label{eqn:exponential_diagram}
\begin{split}
\xymatrix{
\mathscr{D}(V(j))^{\widetilde{\Delta}=0}  \ar[rrr]^-{\Omega^{\varepsilon}_{V(j), h+j}} \ar[d]_-{\Xi^{\varepsilon}_{V(j), 0, K_n}} &  & & \mathscr{H}(\Gamma) \otimes_{\Lambda} \dfrac{ \mathrm{H}^1_{\mathcal{I}w}(K, T(j)) }{ T(j)^{G_{K_\infty}} } \ar[d]^-{\mathrm{pr}_{T(j),K_n}}  \\
K_n \otimes \dfrac{ \mathbf{D}_{\mathrm{cris}}(V(j)) }{ \mathbf{D}_{\mathrm{cris}}(V(j))^{\varphi=1} } \ar[rrr]^-{(h+j-1)! \cdot \mathrm{exp}_{V(j), K_n}} &  & & \dfrac{ \mathrm{H}^1(K_n, V(j)) }{ \mathrm{H}^1(K_\infty/K_n,V(j)^{G_{K_\infty}}) }
}
\end{split}
\end{equation}
commutes.
\item 
For any integer $j \in \mathbb{Z}$, $\Omega^{\varepsilon}_{V(j), h+j}$ can be defined by the formula
\begin{equation} \label{eqn:exponential_twist}
\Omega^{\varepsilon}_{V(j+1), h+j+1} = - \mathrm{Tw}^{\varepsilon}_{V(j), 1} \circ \Omega^{\varepsilon}_{V(j), h+j} \circ (D \otimes e_1) 
\end{equation}
where $\mathrm{Tw}^{\varepsilon}_{V(j), 1}$ sends $x \in \mathrm{H}^1_{\mathcal{I}w}(K, T(j))$ to $x \otimes \varepsilon \in \mathrm{H}^1_{\mathcal{I}w}(K, T(j+1))$. 
\item One has
\begin{equation} \label{eqn:exponential_shift}
\Omega^{\varepsilon}_{V(j), h+j+1} = \ell_{h+j}  \cdot \Omega^{\varepsilon}_{V(j), h+j} .
\end{equation}
where $\ell_{h+j} := (h+j) - \dfrac{\mathrm{log}(\gamma_1)}{\mathrm{log}(\chi_{\mathrm{cyc}}(\gamma_1))}$ where $\gamma_1$ is any topological generator of $\Gamma^1$.
Especially, by using (\ref{eqn:exponential_shift}), $\Omega^{\varepsilon}_{V, h}$ can be defined for any integer $h \in \mathbb{Z}$ if we invert $\ell_{i}$ in $\mathscr{H}(\Gamma)$ for all $i \in \mathbb{Z}$.
\end{enumerate}
\end{thm}
\begin{proof}
See
\cite[$\S$3.2.3 Thm.]{perrin-riou-local-iwasawa},
\cite[Thm. 3.1]{benois-berger}, and \cite[Thm. 1.3.2]{benois-near-central}. 
In (\ref{eqn:exponential_diagram}), $\mathrm{exp}_{V(j), K_n}$ is well-defined due to Lemma \ref{lem:exponential-quotients}.(1).
Note that
$\Omega^{\varepsilon}_{V(j), h+j}$ is often denoted by $\mathrm{Exp}^{\varepsilon}_{V(j),h+j}$ in the references.
\end{proof}
\begin{rem} \label{rem:exponential-map-h-zero}
Suppose that $\mathbf{D}^{0}_{\mathrm{cris}}(V) = \mathbf{D}_{\mathrm{cris}}(V) $.
Then we have $\mathbf{D}^{-1}_{\mathrm{cris}}(V(1)) = \mathbf{D}_{\mathrm{cris}}(V(1))$.
Applying Theorem \ref{thm:big-exponential-map} to $V(1)$ with $h=1$ and $j= 0$ and using (\ref{eqn:exponential_twist}), we obtain the map
$$\Omega^{\varepsilon}_{V, 0} : \mathscr{D}(V)^{\widetilde{\Delta}=0} \to \mathscr{H}(\Gamma) \otimes_{\Lambda} \dfrac{ \mathrm{H}^1_{\mathcal{I}w}(K, T(j)) }{ T(j)^{G_{K_\infty}} }$$
 defined by
$ \Omega^{\varepsilon}_{V, 0} = - \mathrm{Tw}^{\varepsilon}_{V(1), -1} \circ \Omega^{\varepsilon}_{V(1), 1} \circ (D^{-1} \otimes e_{-1}) $
without inverting any element in $\mathscr{H}(\Gamma)$ since $D$ is an automorphism as in $\S$\ref{subsec:differential-operators}.
\end{rem}
We keep assumption $\mathbf{D}^0_{\mathrm{cris}}(V) = \mathbf{D}_{\mathrm{cris}}(V)$ in Remark \ref{rem:exponential-map-h-zero}.
Following \cite[Thm. 5.2.4]{benois-crystalline-duke} and \cite[Page 634]{benois-berger}, we twist diagram (\ref{eqn:exponential_diagram}) as follows.
For $j > 0$ (with consideration of Lemma \ref{lem:exponential-quotients}.(2)), we have the commutative diagram
\[
\xymatrix{
\mathscr{D}(V)^{\widetilde{\Delta}=0}  \ar[rrr]^-{\mathrm{Tw}^{\varepsilon}_{V, j} \circ \Omega^{\varepsilon}_{V, 0}} \ar[d]_-{\Xi^{\varepsilon}_{V, j, K_n}} &  & & \mathscr{H}(\Gamma) \otimes_{\Lambda} \dfrac{ \mathrm{H}^1_{\mathcal{I}w}(K, T(j)) }{ T(j)^{G_{K_\infty}} } \ar[d]^-{\mathrm{pr}_{T(j),K_n}}  \\
K_n \otimes \dfrac{ \mathbf{D}_{\mathrm{cris}}(V(j)) }{ \mathbf{D}_{\mathrm{cris}}(V(j))^{\varphi=1} } \ar[rrr]^-{(-1)^j \cdot (j-1)! \cdot \mathrm{exp}_{V(j), K_n}} &  & &  \mathrm{H}^1(K_n, V(j)) 
}
\]
and we have
\begin{equation} \label{eqn:projection-Omega}
\Omega^{\varepsilon}_{V,j, K_n} = \mathrm{pr}_{T(j),K_n} \circ \mathrm{Tw}^{\varepsilon}_{V, j} \circ \Omega^{\varepsilon}_{V, 0}
\end{equation}
as in \cite[Cor. 5.2.5]{benois-crystalline-duke}.

\subsection{Perrin-Riou pairing and explicit reciprocity law} \label{subsec:explicit-reciprocity-law}
Let $V$ be a crystalline representation of $G_K$, and write $V^*(1) = \mathrm{Hom}_{\mathbb{Q}_p} (V, \mathbb{Q}_p)(1)$.
Then we have the following commutative diagram induced from the cup product parings
\begin{equation} \label{eqn:pairing-diagram}
\begin{split}
\xymatrix@C=0.75em@R=1.2em{
\mathrm{H}^1(K_n, V ) \ar@{->>}[d] & \times & \mathrm{H}^1(K_n, V^*(1) ) \ar[rrr]^-{[-,-]_{V,K_n}} &  & & \mathrm{H}^2(K_n, \mathbb{Q}_p(1))  \ar@{=}[d] \\
\mathrm{H}^1_{/f}(K_n, V) \ar[d]^-{\mathrm{exp}^*} & \times & \mathrm{H}^1_{f}(K_n, V^*(1)) \ar@{^{(}->}[u] \ar[rrr] & & & \mathrm{H}^2(K_n, \mathbb{Q}_p(1))  \\
K_n \otimes_K \mathbf{D}^0_{\mathrm{cris}}(V) \ar@{^{(}->}[d] & \times & K_n \otimes_K \dfrac{ \mathbf{D}_{\mathrm{cris}}(V^*(1)) }{ \mathbf{D}^0_{\mathrm{cris}}(V^*(1)) } \ar[u]^-{\mathrm{exp}} \ar[rrr] & & & K_n \otimes  \mathbf{D}_{\mathrm{cris}}(\mathbb{Q}_p(1))  \ar@{=}[d] \ar[u]_-{\mathrm{Tr}_{K_n/\mathbb{Q}_p} \otimes \mathrm{id}} \\
K_n \otimes_K \mathbf{D}_{\mathrm{cris}}(V)  & \times & K_n \otimes_K \mathbf{D}_{\mathrm{cris}}(V^*(1)) \ar@{->>}[u] \ar[rrr]^-{[-,-]_{\mathbf{D}(V), K_n}} & & & K_n \otimes  \mathbf{D}_{\mathrm{cris}}(\mathbb{Q}_p(1)) 
}
\end{split}
\end{equation}
where the maps in the rightmost are given by
\[
\xymatrix{
K_n \otimes  \mathbf{D}_{\mathrm{cris}}(\mathbb{Q}_p(1)) \ar[r]^-{\simeq}_-{ e_{-1} \mapsto 1 } & K_n \otimes \mathbb{Q}_p \ar[r]^-{\mathrm{Tr}_{K_n/\mathbb{Q}_p} \otimes \mathrm{id}} & \mathbb{Q}_p \ar[r]^-{\simeq} & \mathrm{H}^2(K_n, \mathbb{Q}_p(1)) .
}
\]
See \cite[Prop. 3.8]{bloch-kato} and \cite[Chap. II, Thm. 1.4.1.(4)]{kato-lecture-1} for details.
Following \cite[$\S$3.6.1]{perrin-riou-local-iwasawa}, we consider the Iwasawa-theoretic variation of (\ref{eqn:pairing-diagram}).
We first restrict $[-,-]_{V,K_n}$ to 
\[
[-,-]_{T,K_n} : \mathrm{H}^1(K_n, T )  \times \mathrm{H}^1(K_n, T^*(1) )  \to \mathbb{Z}_p 
\]
where $T^*(1) = \mathrm{Hom}_{\mathbb{Z}_p}(T, \mathbb{Z}_p)(1)$.
For $ x = (x_n) \in \mathrm{H}^1_{\mathcal{I}w}(K, T )$ and $ y = (y_n) \in \mathrm{H}^1_{\mathcal{I}w}(K, T^*(1) )$,
we define the \textbf{Perrin-Riou pairing}
$$\langle -, - \rangle_{T,K} : \mathrm{H}^1_{\mathcal{I}w}(K, T ) \times \mathrm{H}^1_{\mathcal{I}w}(K, T^*(1) ) \to \Lambda_{\mathcal{I}w} = \mathbb{Z}_p \llbracket \mathrm{Gal}(K_\infty/K) \rrbracket$$
 by
$\langle x, y \rangle_{T,K} := \varprojlim_n \sum_{\sigma \in \mathrm{Gal}(K_n/K)} [ ( x_n )^{\sigma^{-1}}, y_n ]_{T,K_n} \cdot \sigma\in \Lambda_{\mathcal{I}w} $.
It satisfies
$\langle \lambda \cdot x, y \rangle_{T,K} = \lambda \cdot \langle x, y \rangle_{T,K}
 = \langle x, \iota(\lambda) \cdot y \rangle_{T,K}$
where $\lambda \in \Lambda_{\mathcal{I}w}$ and $\iota : \Lambda_{\mathcal{I}w} \to \Lambda_{\mathcal{I}w}$ sends $\gamma \mapsto \gamma^{-1}$ for $\gamma \in \Gamma$.
We extend the pairing  linearly
$$  \mathscr{H}_K(\Gamma) \otimes_{\Lambda_{\mathcal{I}w}} \mathrm{H}^1_{\mathcal{I}w}(K, T ) \times  \mathscr{H}_K(\Gamma) \otimes_{\Lambda_{\mathcal{I}w}}  \mathrm{H}^1_{\mathcal{I}w}(K, T^*(1) ) \to  \mathscr{H}_K(\Gamma) .$$
Since $p$ is inverted in $\mathscr{H}_K(\Gamma)$, the torsion in $\mathrm{H}^1_{\mathcal{I}w}(K, T )$ is ignored.
Using the product structure $*$ on $W\llbracket X  \rrbracket^{\psi=0}$ described in $\S$\ref{subsec:rings-of-power-series}, we 
extend $[-,-]_{\mathbf{D}(V)} = [-,-]_{\mathbf{D}(V), K} $ to  a bilinear form of $\Lambda_{\mathcal{I}w}$-modules
$$[-,-]_{\mathscr{D}(V)}  : \mathscr{D}(V) \times \mathscr{D}(V^*(1))   \to K \otimes_{\mathbb{Z}_p} \Lambda_{\mathcal{I}w} $$
such that
$[g_1 \otimes d_1 , g_2 \otimes d_2 ]_{\mathscr{D}(V)} = [d_1, d_2]_{\mathbf{D}(V)} \cdot g_1 * g_2 \in K \otimes_{\mathbb{Z}_p} \Lambda_{\mathcal{I}w} $
for 
$g_1, g_2 \in W\llbracket X  \rrbracket^{\psi=0}$,  $d_1 \in \mathbf{D}_{\mathrm{cris}}(V)$, and  $d_2 \in \mathbf{D}_{\mathrm{cris}}(V^*(1))$.

The following conjecture (e.g.~\cite[$\S$3.6.4]{perrin-riou-local-iwasawa}, \cite[Thm. 4.2.3]{perrin-riou-documenta}) is known as the \emph{explicit reciprocity law} conjecture, which is now a theorem.
\begin{conj}[$\mathrm{R\acute{e}c}(V)$] \label{conj:reciprocity-law}
Let $V$ be a crystalline representation of $G_K$. For all $x \in \mathscr{D}(V)^{\widetilde{\Delta} = 0}$, $y \in \mathscr{D}(V^*(1))^{\widetilde{\Delta} = 0}$,
we have
$$\left\langle \Omega^{\varepsilon}_{V,h}(x), \Omega^{\varepsilon^{-1}}_{V^*(1), 1-h}(y) \right\rangle_V \cdot (1+X) = (-1)^{h-1} \cdot \mathrm{Tr}_{K/\mathbb{Q}_p} \left( \left[ x, y \right]_{\mathscr{D}(V)}  \right) $$
for any integer $h$.
In other words, the inverse of $\Omega^{\varepsilon}_{V,h}$ is the adjoint of $\Omega^{\varepsilon^{-1}}_{V^*(1),1-h}$ up to sign.
\end{conj}
\begin{thm}[Benois, Berger, Colmez, Kato--Kurihara--Tsuji]
Conjecture \ref{conj:reciprocity-law} holds.
\end{thm}
\begin{proof}[Remarks on proof]
Conjecture \ref{conj:reciprocity-law} was proved by Benois \cite{benois-crystalline-duke} with \cite{colmez-finite-height} (Theorem \ref{thm:colmez-crystalline-finite-height}) and Colmez \cite{colmez-local-iwasawa-de-rham} independently. There are also different proofs by Berger \cite{berger-three-explicit-formulas} and Kato--Kurihara--Tsuji \cite{kato-kurihara-tsuji}. 
\end{proof}
\begin{rem}
\begin{enumerate}
\item Although the image of $\Omega^{\varepsilon}_{V,h}$ and $\Omega^{\varepsilon^{-1}}_{V^*(1), 1-h}$
lie in the quotients of Iwasawa cohomologies by their torsions, the Perrin-Riou pairing is still well-defined. See \cite[3.6.2.Lem.]{perrin-riou-local-iwasawa}.
\item If one constructs $\Omega^{\varepsilon}_{V,h}$ or $\Omega^{\varepsilon^{-1}}_{V^*(1), 1-h}$ from Bloch--Kato exponentials, then one can also (equivalently) construct the other from Bloch--Kato dual exponentials.
In our case, we construct $\Omega^{\varepsilon^{-1}}_{V_f(k)^*(1), 0} = \Omega^{\varepsilon^{-1}}_{V_{\overline{f}}, 0}$
from Bloch--Kato exponential maps.
In \cite[$\S$3.2 and $\S$3.3.2]{lei-compositio} and \cite[$\S$3.5.1]{lei-loeffler-zerbes_wach}, Lei--Loeffler--Zerbes construct $\Omega^{\varepsilon}_{V_f(1),1}$ from Bloch--Kato exponential maps.
This discrepancy affects the interpolation formula of $p$-adic $L$-functions very  slightly. See $\S$\ref{subsec:interpolation} for details.
\end{enumerate}
\end{rem}

\begin{prop} \label{prop:choice-varepsilon}
\begin{enumerate}
\item For $\tau \in \mathrm{Gal}(K_\infty/K)$, 
$\Omega^{\tau \varepsilon}_{V,h}
=
\tau \Omega^{ \varepsilon}_{V,h}
=
 \Omega^{ \varepsilon}_{V,h} \circ \tau$
\item 
$\mathrm{Tw}^{\varepsilon}_{-j, V^*(1)}
= (-1)^j \cdot \mathrm{Tw}^{\varepsilon^{-1}}_{-j, V^*(1)}$
\end{enumerate}
\end{prop}
\begin{proof}
See \cite[3.2.4.Pre. rem. (v) and 3.6.5.Lem.]{perrin-riou-local-iwasawa}
\end{proof}
\section{Mazur--Tate elements II: \'{e}tale construction} \label{sec:mazur-tate-etale}
Refining Perrin-Riou's local Iwasawa theory reviewed in $\S$\ref{sec:local-iwasawa-theory-perrin-riou}, we develop an \emph{equivariant} local Iwasawa theory and apply it to modular forms
 in an \emph{explicit} manner.
As a by-product, we construct the \emph{equivariant} $p$-adic $L$-functions of modular forms directly from Kato's Euler systems and the Mazur--Tate elements of modular forms over general cyclotomic extensions of $\mathbb{Q}$.
In other words, we expand and refine \cite[Line 1, Page 270]{kato-euler-systems}.
More precisely, we give explicit proofs of the corrected versions of \cite[Thm. 16.4 and Thm. 16.6]{kato-euler-systems} on the construction of $p$-adic $L$-functions of modular forms at good primes (i.e. the crystalline case) from Kato's Euler systems and generalize them equivariantly. The key input is the explicit construction of the families of equivariant local points  for modular forms following the idea of Perrin-Riou \cite[$\S$2.3]{perrin-riou-local-iwasawa} and Ota \cite{ota-rank-part}.
It can be regarded as a finite layer refinement of the explicit reciprocity law for modular forms and generalizes the constructions given in \cite{kurihara-invent, otsuki, ota-thesis, ota-rank-part, kataoka-thesis}.

%
%

\subsection{Non-triviality of equivariant local points} \label{subsec:compatible-families-local-points}
We regard $V_{\overline{f}}$ as a crystalline representation of $\mathrm{Gal}(\overline{\mathbb{Q}}_p/\mathbb{Q}_p)$ via restriction.
Then the corresponding filtered $\varphi$-module $\mathbf{D}_{\mathrm{cris}}(V_{\overline{f}})$ has filtration given by
$$ \mathbf{D}^{i}_{\mathrm{cris}}(V_{\overline{f}}) = 
\left\lbrace
\begin{array}{ll}
\mathbf{D}_{\mathrm{cris}}(V_{\overline{f}}) & (i \leq 0) \\
S(\overline{f}) \otimes_F F_\pi  &  (1 \leq i \leq k-1) \\
0 &  (k \leq i)
\end{array}
\right.$$
as in \cite[(11.3.4)]{kato-euler-systems}.
\begin{defn}
Following (\ref{eqn:wach-lattice}), we define the $\varphi$-stable $\mathcal{O}_\pi$-lattice of $\mathbf{D}_{\mathrm{cris}}(V_{\overline{f}})$ by
$$M_{\overline{f}} := \left( \mathbf{D}_S(T_{\overline{f}}) \widehat{\otimes}_S \widehat{S}^{\mathrm{PD}} \right)^{\Gamma} \subseteq \mathbf{D}_{\mathrm{cris}}(V_{\overline{f}})$$
where $\Gamma   = \mathrm{Gal}(\mathbb{Q}_p(\zeta_{p^\infty})/\mathbb{Q}_p)$.
\end{defn}
\begin{rem}
When $2 \leq k \leq p-1$, one may use the lattice arising from Fontaine--Laffaille modules instead of $M_{\overline{f}}$.
See  
\cite[Thm. 4.1.(iii)]{bloch-kato},
\cite[$\S$2.1.5]{perrin-riou-local-iwasawa}, and
\cite[$\S$14.17]{kato-euler-systems} for details.
\end{rem}
Let $\overline{\alpha} \in \mathbb{C}_p$ be a root of $X^2 - a_p(\overline{f}) X + \overline{\psi}(p) p^{k-1}$ 
such that $\mathrm{ord}_p( \overline{\alpha} ) < k-1$ and $\overline{\beta}$ the other root.
Both $\overline{\alpha}$ and $\overline{\beta}$ cannot be a root of unity due to their complex absolute values, and all $\alpha$, $\beta$, $\overline{\alpha}$, and $\overline{\beta}$ lie in $\mathcal{O}_\pi[\alpha]$.
\begin{prop} \label{prop:surjectivity-of-1-varphi}
\begin{enumerate}
\item Suppose that $\overline{f}$ is non-ordinary at $p$.
Then the map 
$$\left(  1 - \varphi \right) : (\mathbb{Z}[\zeta_m] \otimes \mathbb{Z}_p) \otimes_{\mathbb{Z}_p} M_{\overline{f}} \otimes \mathcal{O}_{\pi}[\alpha] \to (\mathbb{Z}[\zeta_m] \otimes \mathbb{Z}_p) \otimes_{\mathbb{Z}_p} M_{\overline{f}} \otimes \mathcal{O}_{\pi}[\alpha]$$
 is surjective.
\item Suppose that $\overline{f}$ is ordinary at $p$ with unit Hecke eigenvalue $\overline{\alpha}$ at $p$.
Then the map 
$$\left(  1 - \varphi \right) : \dfrac{1}{1-\overline{\alpha}^{\phi(m)}} \cdot (\mathbb{Z}[\zeta_m] \otimes \mathbb{Z}_p) \otimes_{\mathbb{Z}_p} M_{\overline{f}} \otimes \mathcal{O}_{\pi}[\alpha] \to  (\mathbb{Z}[\zeta_m] \otimes \mathbb{Z}_p) \otimes_{\mathbb{Z}_p} M_{\overline{f}} \otimes \mathcal{O}_{\pi}[\alpha]$$
 is surjective.
\end{enumerate}
\end{prop}
\begin{proof}
We follow \cite{ota-rank-part}.
\begin{enumerate}
\item 
Under the non-ordinary assumption,
$\overline{\alpha}^{n} \not\equiv 1 \pmod{\pi}$ and $\overline{\beta}^{n} \not\equiv 1 \pmod{\pi}$ for any integer $n \geq 1$.
Let $g \in (\mathbb{Z}[\zeta_m] \otimes \mathbb{Z}_p) \otimes M_{\overline{f}}  \otimes \mathcal{O}_{\pi}[\alpha]$.
Since the $\varphi$-eigenvalues of $M_{\overline{f}}$ are non-units, $G = \sum_{n \geq 0} \varphi^n(x)$ converges and it satisfies $(1-\varphi)G = g$.
\item 
Let $g \in (\mathbb{Z}[\zeta_m] \otimes \mathbb{Z}_p) \otimes_{\mathbb{Z}_p} M_{\overline{f}} \otimes \mathcal{O}_{\pi}[\alpha]$.
Write $g = a \cdot g_{\overline{\alpha}} + b \cdot g_{\overline{\beta}}$
where $a, b \in \mathbb{Z}[\zeta_m] \otimes \mathbb{Z}_p$ and 
$g_{\overline{\alpha}} , 
g_{\overline{\beta}} \in M_{\overline{f}} \otimes \mathcal{O}_{\pi}[\alpha]$
such that
$ \varphi g_{\overline{\alpha}} = \overline{\alpha} \cdot  g_{\overline{\alpha}}$ and $\varphi g_{\overline{\beta}}  = \overline{\beta} \cdot g_{\overline{\beta}}$.
We write
\[
\xymatrix{
(1 - \overline{\alpha}^{\phi(m)}) \cdot G_{\overline{\alpha}} =  \sum_{0 \leq i \leq  \phi(m)-1} \varphi^i( a \cdot g_{\overline{\alpha}} ), &
 G_{\overline{\beta}} =  \sum_{n \geq 0} \varphi^n( b \cdot g_{\overline{\beta}} ) 
}
\]
and then $G_{\overline{\alpha}}$ lies only in $\dfrac{1}{1 - \overline{\alpha}^{\phi(m)}} \cdot (\mathbb{Z}[\zeta_m] \otimes \mathbb{Z}_p) \otimes_{\mathbb{Z}_p} \otimes M_{\overline{f}}  \otimes \mathcal{O}_{\pi}[\alpha]$. 
Since $\sigma^{\phi(m)} = 1$ on $\mathbb{Z}[\zeta_m] \otimes \mathbb{Z}_p$, we have $(1 - \varphi) (G_{\overline{\alpha}} + G_{\overline{\beta}}) = g$.
Thus, we have the surjectivity of
$$\left(  1 - \varphi \right) : \dfrac{1}{1- \overline{\alpha}^{\phi(m)}} \cdot (\mathbb{Z}[\zeta_m] \otimes \mathbb{Z}_p) \otimes_{\mathbb{Z}_p} M_{\overline{f}}  \otimes \mathcal{O}_{\pi}[\alpha] \to (\mathbb{Z}[\zeta_m] \otimes \mathbb{Z}_p) \otimes_{\mathbb{Z}_p} M_{\overline{f}}  \otimes \mathcal{O}_{\pi}[\alpha] .$$
\end{enumerate}
\end{proof}

Let
\begin{align*}
\mathscr{H}_{\mathbb{Q}(\zeta_m)}(T_{\overline{f}}) & = \left\lbrace  G \in \left( \mathbb{Z}[\zeta_m] \otimes \mathbb{Z}_p \right) \llbracket  X \rrbracket \otimes_{\mathbb{Z}_p} M_{\overline{f}}   : \sum_{\zeta \in \mu_p} G(\zeta(1+X)-1) = p \cdot \varphi G(X)  \right\rbrace  \\
& = \left( \left( \mathbb{Z}[\zeta_m] \otimes \mathbb{Z}_p \right) \llbracket X \rrbracket \otimes_{\mathbb{Z}_p} M_{\overline{f}} \right)^{\psi=1},
\end{align*}
and
$\mathscr{D}_{\mathbb{Q}(\zeta_m)}(T_{\overline{f}}) = \left( \left( \mathbb{Z}[\zeta_m] \otimes \mathbb{Z}_p \right) \llbracket X \rrbracket \right)^{\psi=0} \otimes_{\mathbb{Z}_p} M_{\overline{f}} $.

In order to make everything $\mathrm{Gal}(\mathbb{Q}(\zeta_m)/\mathbb{Q})$-equivariant, we fix a primitive $m$-th root of unity $\zeta_m \in \mu_m$.
\begin{choice} \label{choice:p-adic-orientation-enhanced}
Let $\varepsilon_m = \left( \zeta_{mp^n} \right)_{n \geq 0}$ be the choice of the compatible system of $mp^n$-th roots of unity compatible with  $\varepsilon$ chosen in Choice \ref{choice:p-adic-orientation}.
\end{choice}
Let
$$\mathrm{exp}_{ V_{\overline{f}}(r), \mathbb{Q}(\zeta_{mp^n})} :
\left( \mathbb{Q}(\zeta_{mp^n}) \otimes \mathbb{Q}_p \right) \otimes_{\mathbb{Q}_p}
\left( 
\dfrac{ \mathbf{D}_{\mathrm{cris}}(V_{\overline{f}}(r)) }{ \mathrm{Fil}^0\mathbf{D}_{\mathrm{cris}}(V_{\overline{f}}(r)) }
\right)
\to \mathrm{H}^1_f ( \mathbb{Q}(\zeta_{mp^n}) \otimes \mathbb{Q}_p, V_{\overline{f}}(r))$$
be the direct sum of the local Bloch--Kato exponential maps and then it is $\mathrm{Gal}(\mathbb{Q}(\zeta_m)/\mathbb{Q})$-equivariant.
The equivariant property allows us to define a family of homomorphism  
$$\Sigma^{\varepsilon_m}_{T_{\overline{f}}, r , \mathbb{Q}(\zeta_{mp^n}) } : \mathscr{H}_{\mathbb{Q}(\zeta_m)}(T_{\overline{f}}) \to \mathrm{H}^1(\mathbb{Q}(\zeta_{mp^n}) \otimes \mathbb{Q}_p, T_{\overline{f}}(r))$$
for $r \in \mathbb{Z}$ by
$\Sigma^{\varepsilon_m}_{T_{\overline{f}}, r , \mathbb{Q}(\zeta_{mp^n}) }  = \bigoplus_{v \mid p} \Sigma^{\varepsilon_m}_{T_{\overline{f}}, r , \mathbb{Q}(\zeta_{mp^n})_v } $ where $v$ runs over all primes of $\mathbb{Q}(\zeta_{mp^n})$ lying above $p$.
Note that
$\bigoplus_{v \mid p} \Sigma^{\varepsilon_m}_{T_{\overline{f}}, r , \mathbb{Q}(\zeta_{mp^n})_v } = \bigoplus_{v \mid p} \Sigma^{\varepsilon}_{T_{\overline{f}}, r , \mathbb{Q}(\zeta_{mp^n})_v }$
as $\mathrm{Gal}(\mathbb{Q}(\zeta_{mp^n})/\mathbb{Q}(\zeta_{m}))$-homomorphisms, but the former one also remembers the action of $\mathrm{Gal}(\mathbb{Q}(\zeta_m)/\mathbb{Q})$ under Choice \ref{choice:p-adic-orientation-enhanced}. This extension is natural due to the equivariant property of $\mathrm{exp}_{V_{\overline{f}}(r), \mathbb{Q}(\zeta_m)}$.
Thus, Theorem \ref{thm:integral-exponential} extends equivariantly in the natural manner. 
\begin{lem} \label{lem:1-varphi-differential-equations}
Let $\gamma_m \in \Gamma$ such that $\chi_{\mathrm{cyc}}(\gamma_m) = m \in \mathbb{Z}^\times_p$.
\begin{enumerate}
\item 
Suppose that $\overline{f}$ is non-ordinary at $p$.
For each $\eta \in M_{\overline{f}} \otimes_{\mathcal{O}_\pi} \mathcal{O}_\pi[\alpha]$, there exists
a unique 
$$G_{\mathbb{Q}(\zeta_m), \eta} \in \mathscr{H}_{\mathbb{Q}(\zeta_m)}(T_{\overline{f}}) \otimes_{\mathcal{O}_\pi} \mathcal{O}_\pi[\alpha]$$ such that
$ (1 - \varphi) G_{\mathbb{Q}(\zeta_m), \eta}  = \dfrac{1}{m} \cdot  \left( \zeta_m (1+X)^{1/m} \otimes \eta \right) $.
\item Suppose that $\overline{f}$ is ordinary at $p$.
For each $\eta \in M_{\overline{f}} \otimes_{\mathcal{O}_\pi} \mathcal{O}_\pi[\alpha]$, there exists
a unique 
$$G_{\mathbb{Q}(\zeta_m), \eta} \in \dfrac{1}{  1- \overline{\alpha}^{\phi(m)} } \cdot \mathscr{H}_{\mathbb{Q}(\zeta_m)}(T_{\overline{f}}) \otimes_{\mathcal{O}_\pi} \mathcal{O}_\pi[\alpha]$$ such that
$(1 - \varphi) G_{\mathbb{Q}(\zeta_m), \eta}  =   \dfrac{1}{m} \cdot  \left( \zeta_m (1+X)^{1/m} \otimes \eta \right)$.
\end{enumerate}
\end{lem}
\begin{proof}
Straightforward from Proposition \ref{prop:surjectivity-of-1-varphi}.
\end{proof}
\begin{rem} 
\begin{enumerate}
\item
In the non-ordinary case, we have $( 1- \overline{\alpha}^{\phi(m)} ) \in \mathcal{O}_\pi[\alpha]^\times$.
\item
The choice of $\zeta_m$ (Choice \ref{choice:p-adic-orientation-enhanced}) is explicitly used in Lemma \ref{lem:1-varphi-differential-equations}.
We give an explicit solution of $(1-\varphi)G =g$ when $g = m^{r-1} \cdot   \zeta_m \cdot   (1+X)^{1/m}   \otimes \eta_{\overline{\alpha}} \otimes e_{-r}$
in order to obtain the precise interpolation formula in $\S$\ref{subsec:computation-local-points}.
\end{enumerate}
\end{rem}
By twisting the bottom row in (\ref{eqn:pairing-diagram}) with the duality $\mathrm{Hom}_{F_\pi} (V_{f}(k-r), F_{\pi})(1) \simeq V_{\overline{f}}(r)$, we have the natural paring
\begin{equation*} 
[-,-] : \mathbf{D}_{\mathrm{cris}}(V_f) \times \mathbf{D}_{\mathrm{cris}}(V_{\overline{f}})   \to \mathbf{D}_{\mathrm{cris}}(F_{\pi}(1-k)) \simeq F_\pi
\end{equation*}
where the last isomorphism sends $e_{k-1}  \mapsto 1$ as in (\ref{eqn:pairing-diagram}).
\begin{prop} \label{prop:normalization_local_points_nonzero}
We recall the following assumption:
\begin{quote}
Assumption \ref{assu:standard}.(``non-CM"). If $\overline{f}$ is ordinary at $p$, the restriction of $\overline{\rho}_{\overline{f}}$ to $G_{\mathbb{Q}_p}$ is not a direct sum of two characters.
\end{quote}
Under this assumption, there exist elements
$\eta'_{\overline{\alpha}} , \eta'_{\overline{\beta}} \in \mathbf{D}_{\mathrm{cris}}(V_{\overline{f}}) \otimes_{\mathcal{O}_\pi} \mathcal{O}_\pi[\alpha]
$
such that
\[
\xymatrix@R=0em{
\varphi \eta'_{\overline{\alpha}} = \overline{\alpha} \cdot \eta'_{\overline{\alpha}} , & [\omega_{f}, \eta'_{\overline{\alpha}} ] \neq 0 , \\
\varphi \eta'_{\overline{\beta}} = \overline{\beta} \cdot \eta'_{\overline{\beta}} , & [\omega_{f}, \eta'_{\overline{\beta}} ] \neq 0 \\
}
\]
where $\omega_{f} \in S(f) \subseteq \mathbf{D}_{\mathrm{cris}}(V_{f})$ arises from the normalized eigenform $f$.
\end{prop}
\begin{proof}
In the non-ordinary case, the existence of such $\eta'_{\overline{\alpha}}$ and $\eta'_{\overline{\beta}}$ follows from \cite[Thm. 16.6.(1)]{kato-euler-systems}.
In the ordinary case, the existence of $\eta'_{\overline{\alpha}}$ also follows from  \cite[Thm. 16.6.(1)]{kato-euler-systems}; however, the existence of $\eta'_{\overline{\beta}}$ in the ordinary case follows from \cite[Rem. 4.12]{colmez-p-adic-BSD} under Assumption \ref{assu:standard}.(``non-CM") since it is of critical slope.
\end{proof}

\subsection{Uniform integrality of equivariant local points}
For an $\mathcal{O}_\pi$-module $B$, we put
$B^* := \mathrm{Hom}_{\mathcal{O}_\pi}(B, \mathcal{O}_\pi) $.
We recall the cup product pairing
\begin{equation*}
[-,-]_{V_{f}(k-r), \mathbb{Q}(\zeta_{mp^n})} : \mathrm{H}^1 (\mathbb{Q}(\zeta_{mp^n})  \otimes \mathbb{Q}_p , V_f(k-r)) \times \mathrm{H}^1 ( \mathbb{Q}(\zeta_{mp^n})  \otimes \mathbb{Q}_p , V_{\overline{f}}(r)) \to F_\pi ,
\end{equation*}
and consider its \emph{optimally integral} restriction
\begin{equation*}
[-,-]_{V_{f}(k-r),\mathbb{Q}(\zeta_{mp^n})}  : \mathrm{H}^1_{/f} (\mathbb{Q}(\zeta_{mp^n})  \otimes \mathbb{Q}_p , T_f(k-r)) \times \mathrm{H}^1_{/f} (\mathbb{Q}(\zeta_{mp^n})  \otimes \mathbb{Q}_p , T_f(k-r))^* \twoheadrightarrow \mathcal{O}_\pi 
\end{equation*}
where $\mathrm{H}^1_{/f}(\mathbb{Q}(\zeta_{mp^n}) \otimes \mathbb{Q}_p, T_{f}(k-r) )^{*}  \subseteq \mathrm{H}^1_f(\mathbb{Q}(\zeta_{mp^n})  \otimes \mathbb{Q}_p , V_{\overline{f}}(r) ) $ is characterized the surjectivity of the pairing.

\begin{lem} \label{lem:normalization_local_points_first}
Under Assumption \ref{assu:standard}.(``non-CM"),
$$\mathrm{H}^0(\mathbb{Q}(\zeta_{mp}) \otimes \mathbb{Q}_p , W_{\overline{f}}(r)) = 0$$
 for any integer $m$ with $(m,p)=1$ and any $1 \leq r \leq k-1$.
\end{lem}
\begin{proof}
It suffices to show $\mathrm{H}^0(\mathbb{Q}(\zeta_{mp}) \otimes \mathbb{Q}_p , W_{\overline{f}, -i}(r)[\pi]) = 0$.
By \cite{gross-tameness, coleman-voloch}, Assumption \ref{assu:standard}.(``non-CM") implies that $W_{\overline{f}}(r)[\pi]$ is wildly ramified at $p$. Thus, the restriction of $W_{\overline{f}}(r)[\pi]$ to $G_{\mathbb{Q}(\zeta_m)}$ is still wildly ramified at all primes above $p$. Thus, the local Galois representation does not split, so the conclusion follows.
\end{proof}

\begin{prop} \label{prop:normalization_local_points_first}
We keep Assumption \ref{assu:standard}.(``non-CM"), and we also assume that $\overline{\alpha} \neq \zeta \cdot p^{\frac{k-1}{2}}$ when $r = \frac{k-1}{2}$ is an integer where $\zeta$ is a root of unity (Assumption \ref{assu:standard}.(ENV)).  
By choosing $\eta'_{\overline{\alpha}}$ and $\eta'_{\overline{\beta}}$ as in Proposition \ref{prop:normalization_local_points_nonzero}, we are able to make the elements
$$\Sigma^{\varepsilon_m}_{T_{\overline{f}}, r, \mathbb{Q}(\zeta_{mp^n})} (G^{\sigma^{-n}}_{\mathbb{Q}(\zeta_m), \eta'_{\overline{\alpha}}}) , \Sigma^{\varepsilon_m}_{T_{\overline{f}}, r, \mathbb{Q}(\zeta_{mp^n})} (G^{\sigma^{-n}}_{\mathbb{Q}(\zeta_m), \eta'_{\overline{\beta}}}) \in \mathrm{H}^1_f( \mathbb{Q}(\zeta_{mp^n}) \otimes \mathbb{Q}_p, V_{\overline{f}}(r) ) \otimes \mathcal{O}_\pi[\alpha]$$
to lie in 
$\mathrm{H}^1_{/f}( \mathbb{Q}(\zeta_{mp^n}) \otimes \mathbb{Q}_p, T_{f}(k-r) )^* \otimes \mathcal{O}_\pi[\alpha]$
via the integral cup product pairing
where 
$1 \leq r \leq k-1$, and $\mathbb{Q}(\zeta_{mp^n})$ runs over all $n \geq 0$ and $m \geq 1$ with $(m,p) = 1$.
\end{prop}
\begin{proof}
Due to Proposition \ref{prop:normalization_local_points_nonzero}, we can take 
$\eta'_{\overline{\alpha}}$ and 
$\eta'_{\overline{\beta}}$ satisfying
$ \eta'_{\overline{\alpha}} , \eta'_{\overline{\beta}} \in M_{\overline{f}} \otimes \mathcal{O}_\pi[\alpha]$ by scaling. 
By Theorem \ref{thm:integral-exponential}, Lemma  \ref{lem:1-varphi-differential-equations}
and the construction of $\Sigma^{\varepsilon_m}_{T_{\overline{f}}, r, \mathbb{Q}(\zeta_{mp^n})}$ in \cite[$\S$4.2]{benois-crystalline-duke}, it is immediate to have
$$\Sigma^{\varepsilon_m}_{T_{\overline{f}}, r, \mathbb{Q}(\zeta_{mp^n})} (G^{\sigma^{-n}}_{\mathbb{Q}(\zeta_m), \bullet })  \in  \mathrm{H}^1_f(\mathbb{Q}(\zeta_{mp^n}) \otimes \mathbb{Q}_p, T_{\overline{f}}(r)) \otimes \mathcal{O}_\pi[\alpha] 
$$
for all $\bullet \in \lbrace  \eta'_{\overline{\alpha}}, \eta'_{\overline{\beta}} \rbrace$ unless $\bullet =  \eta'_{\overline{\alpha}}$ and $\overline{f}$ is ordinary.
Even when $\bullet =  \eta'_{\overline{\alpha}}$ and $\overline{f}$ is ordinary, 
$\Sigma^{\varepsilon_m}_{T_{\overline{f}}, r, \mathbb{Q}(\zeta_{mp^n})} (G^{\sigma^{-n}}_{\mathbb{Q}(\zeta_m), \eta'_{\overline{\alpha}} })$ still lies in $\mathrm{H}^1_f(\mathbb{Q}(\zeta_{mp^n}) \otimes \mathbb{Q}_p, T_{\overline{f}}(r)) \otimes \mathcal{O}_\pi[\alpha] $
due to the explicit computation in $\S$\ref{subsec:computation-local-points} below.
More specifically, it essentially follows  from the fact that the eigenvalue of $(1-\varphi)^{-1}$ on $\eta'_{\overline{\alpha}} \otimes e_{-r}$ is  $\dfrac{1}{1 - \overline{\alpha}^{\phi(m)} \cdot p^{-r}} = \dfrac{p^r}{p^r - \overline{\alpha}^{\phi(m)} }$ and it is integral since $1\leq r \leq k-1$.
Note that $p^r - \overline{\alpha}^{\phi(m)} \neq 0$ under our assumption with Weil number considerations.
The local duality argument in proof of \cite[Lem. 14.18]{kato-euler-systems} with Lemma \ref{lem:normalization_local_points_first} implies
$$\Sigma^{\varepsilon_m}_{T_{\overline{f}}, r, \mathbb{Q}(\zeta_{mp^n})} (G^{\sigma^{-n}}_{\mathbb{Q}(\zeta_m), \bullet })  \in \mathrm{H}^1_{/f}( \mathbb{Q}(\zeta_{mp^n}) \otimes \mathbb{Q}_p, T_{f}(k-r) )^* \otimes \mathcal{O}_\pi[\alpha] $$
as $\left[-, \Sigma^{\varepsilon_m}_{T_{\overline{f}}, r, \mathbb{Q}(\zeta_{mp^n})} (G^{\sigma^{-n}}_{\mathbb{Q}(\zeta_m), \bullet }) \right]_{V_{f}(k-r), \mathbb{Q}(\zeta_{mp^n})  }$.
\end{proof}
\begin{rem} 
The upshot of Proposition \ref{prop:normalization_local_points_first} is that the integrality of $\Sigma^{\varepsilon_m}_{T_{\overline{f}}, r, \mathbb{Q}(\zeta_{mp^n})} (G^{\sigma^{-n}}_{\mathbb{Q}(\zeta_m), \eta })$ 
in $\mathrm{H}^1_{/f}( \mathbb{Q}(\zeta_{mp^n}) \otimes \mathbb{Q}_p, T_{f}(k-r) )^* \otimes \mathcal{O}_\pi[\alpha]$
depends only on the choice of $\eta \in \mathbf{D}_{\mathrm{cris}}(V_{\overline{f}}) \otimes_{\mathcal{O}_\pi} \mathcal{O}_\pi[\alpha]$. 
Proposition \ref{prop:normalization_local_points_first} uses the theory of Wach modules in a strong way.

\end{rem}

\subsection{Crystalline normalization of equivariant local points} \label{subsec:crystalline-integral-normalization} 
Let
$\mathrm{per}^{\pm}_f :  \mathbb{Q}(\zeta_{mp^n}) \otimes  S(f)
\to  \mathbb{C} \otimes_F V_F(f)
 \to  \mathbb{C} \otimes_F V^{\pm}_F(f) $ be the projection of Kato's period integrals to the $\pm$-component, respectively.

We extend the comparison (\ref{eqn:betti-etale-comparison}) further including the crystalline comparison isomorphism due to Tsuji, Faltings, Nizio\l{} and others.
\begin{equation} \label{eqn:betti-etale-comparison-extended}
\begin{gathered}
{ \scriptsize
\xymatrix{
\mathbf{D}_{\mathrm{cris}} ( V_{F_\pi}(f) ) \otimes \mathbf{B}_{\mathrm{cris}} \ar[d]^-{\simeq} & & \mathbf{D}_{\mathrm{cris}} ( V_{F_\pi}(f) ) \ar[ll] & \mathbf{D}^{k-r}_{\mathrm{cris}} ( V_{F_\pi}(f) ) \ar@{_{(}->}[l] & S(f) \ar[d]_-{\mathrm{per}^+_f + \mathrm{per}^-_f} \ar@{_{(}->}[l]\\
V_{F_\pi}(f) \otimes \mathbf{B}_{\mathrm{cris}}  & V_{F_\pi}(f) \ar[l] & V_F(f) \otimes F_\pi \ar[l]_-{\simeq} & V_F(f) \ar[l] \ar[r] & V_F(f) \otimes \mathbb{C}  .
}
}
\end{gathered}
\end{equation}
The na\"{i}ve integral refinement of (\ref{eqn:betti-etale-comparison-extended}) (with certain unclear parts) can be written as follows
\begin{equation} \label{eqn:betti-etale-comparison-extended-integral}
\begin{gathered}
{ \scriptsize
\xymatrix{
 M_f \otimes_S \widehat{S}^{\mathrm{PD}} \otimes_{\mathbb{Z}_p} \widetilde{S} \ar@{^{(}->}[d] & & M_f \ar@{_{(}->}[ll] \\
\left( T_f \otimes \mathbf{A}_{\mathrm{cris}} \right)^{G_{\mathbb{Q}_p}} \otimes \mathbf{A}_{\mathrm{cris}} \ar[d] & &  M^{k-r}_f \ar@{_{(}->}[u] \\
T_f \otimes \mathbf{A}_{\mathrm{cris}} & & \mathcal{O}_\pi \omega_f \ar[d]_-{\mathrm{per}^+_f + \mathrm{per}^-_f} \ar@{--}[u]_-{\textrm{unclear}} \\
T_f  \ar[u] & &  \mathcal{O}_\pi \delta^+_f \otimes C^{\pm}_{\mathrm{per}} \cdot \Omega^+_{\gamma, \omega_f} + \mathcal{O}_\pi \delta^-_f \otimes C^{\mp}_{\mathrm{per}} \cdot \Omega^-_{\gamma, \omega_f}  \\
 \mathcal{O}_\pi \delta^+_f  + \mathcal{O}_\pi \delta^-_f \ar@{--}[u]_-{\textrm{unclear}} & & \mathcal{O}_\pi \delta^+_f  + \mathcal{O}_\pi \delta^-_f  \ar[ll] \ar[u]  
}
}
\end{gathered}
\end{equation}
where $\omega_f \in S(f) \otimes F_\pi = \mathbf{D}^{k-r}_{\mathrm{cris}}(V_{F_\pi}(f))$ is the class coming from the \emph{normalized} eigenform $f$ and the sign convention follows that of (\ref{eqn:delta_f}).
In (\ref{eqn:betti-etale-comparison-extended-integral}), the upper part is normalized by $T_f$ and the lower part is normalized by $\omega_f$.

Theorem \ref{thm:integral-exponential} (when $n=0$) and the local duality argument in \cite[Lem. 14.18]{kato-euler-systems} yield an integral refinement of (\ref{eqn:pairing-diagram})
\begin{equation} \label{eqn:pairing-diagram-integral}
{ \scriptsize
\begin{split}
\xymatrix@C=0.5em{
\mathrm{H}^1_{/f}(\mathbb{Q}_p, T_f(k-r)) \ar[d]^-{\mathrm{exp}^*} & \times & \mathrm{H}^1_{/f}(\mathbb{Q}_p, T_{f}(k-r))^*   \ar@{->>}[rrr] & & & \mathcal{O}_\pi \ar@{=}[d] \\
\left( M_{\overline{f}}(r) / M^0_{\overline{f}}(r) \right)^*  & \times &  M_{\overline{f}}(r) / M^0_{\overline{f}}(r)  \ar[u] \ar@{->>}[rrr]  & & & \mathcal{O}_\pi .
}
\end{split}
}
\end{equation}
\begin{rem} \label{rem:fontaine-laffaille-functor}
If the assignment $T_{\overline{f}} \mapsto M_{\overline{f}}$ is compatible with tensor product and dual, then
we have $\left( M_{\overline{f}}(r) / M^0_{\overline{f}}(r) \right)^* = M^0_f(k-r)$.
We do not know whether these properties hold or not in general.
In the Fontaine--Laffaille case, such compatible properties hold if we replace $M_{\overline{f}}$ by the corresponding Fontaine--Laffaille module \cite[(14.17.2)]{kato-euler-systems}.
\end{rem}

Since $\left( M_{\overline{f}} / M^r_{\overline{f}} \right)^*$ is a free $\mathcal{O}_\pi$-module of rank one,
we choose an $\mathcal{O}_\pi$-basis $\omega_{f, \mathrm{cris}}$ of $\left( M_{\overline{f}} / M^r_{\overline{f}} \right)^*$ instead of $\omega_{f}$. 
Then we define the \textbf{crystalline normalization of the optimal periods via Kato's period integral} by
\begin{equation}  \label{eqn:delta_f_cris}
\xymatrix{
\delta^+_{f, \mathrm{cris}} := \dfrac{ \mathrm{per}^+_f(\omega_{f, \mathrm{cris}}) }{ C^{\pm}_{\mathrm{per}} \cdot \Omega^+_{\gamma, \omega_f} }  \in V_F(f)^+ , & \delta^-_{f, \mathrm{cris}} := \dfrac{ \mathrm{per}^-_f(\omega_{f, \mathrm{cris}}) }{ C^{\mp}_{\mathrm{per}} \cdot \Omega^-_{\gamma, \omega_f} }  \in V_F(f)^- ,
}
\end{equation}
and $\delta_{f, \mathrm{cris}}  := \delta^+_{f, \mathrm{cris}} + \delta^-_{f, \mathrm{cris}}$ where the sign convention of $C^{\pm}_{\mathrm{per}}$ follows that of (\ref{eqn:delta_f}). 
This is the crystalline normalization of (\ref{eqn:delta_f}), but the periods $C^{\pm}_{\mathrm{per}} \cdot \Omega^+_{\gamma, \omega_f}$ and $C^{\mp}_{\mathrm{per}} \cdot \Omega^-_{\gamma, \omega_f}$ remain the same.

Due to (\ref{eqn:pairing-diagram-integral}), we may choose $\eta'_{\overline{\alpha}}$ and  $\eta'_{\overline{\beta}}$ in Proposition \ref{prop:normalization_local_points_first} such that 
the images of $\eta'_{\overline{\alpha}} $
and $\eta'_{\overline{\beta}} $ in $\mathbf{D}_{\mathrm{cris}}(V_{F_\pi}(\overline{f})) / \mathbf{D}^{r}_{\mathrm{cris}}(V_{F_\pi}(\overline{f}))$ lie in $M_{\overline{f}} / M^{r}_{\overline{f}}$  but not in $\pi M_{\overline{f}} / M^{r}_{\overline{f}}$ for all $1 \leq r \leq k-1$.
The independence of $r$ follows from Proposition \ref{prop:normalization_local_points_first}.
Since $\omega_{f, \mathrm{cris}}$ is an $\mathcal{O}_\pi$-basis of $\left( M_{\overline{f}} / M^r_{\overline{f}} \right)^*$, we obtain
$[\omega_{f, \mathrm{cris}}, \eta'_{\overline{\alpha}}  ] \in \mathcal{O}^\times_\pi$ and $[\omega_{f, \mathrm{cris}},  \eta'_{\overline{\beta}}  ]\in \mathcal{O}^\times_\pi $.
Let
\begin{equation} \label{eqn:correct-periods}
{ \small
\xymatrix{
\eta_{\overline{\alpha}} = \dfrac{1}{[\omega_{f, \mathrm{cris}}, \eta'_{\overline{\alpha}}  ]} \cdot \eta'_{\overline{\alpha}} , &
 \eta_{\overline{\beta}} = \dfrac{1}{[\omega_{f, \mathrm{cris}},  \eta'_{\overline{\beta}}  ]} \cdot \eta'_{\overline{\beta}} , &
 \eta_{\overline{f}} = \dfrac{\overline{\alpha}}{\overline{\alpha} - \overline{\beta}} \cdot \eta_{\overline{\alpha}}
+ \dfrac{\overline{\beta}}{\overline{\beta}- \overline{\alpha}} \cdot  \eta_{\overline{\beta}} .
}
}
\end{equation}
Then $\eta_{\overline{\alpha}}$ and $\eta_{\overline{\beta}}$ give us the \emph{correct} optimal periods for $p$-adic $L$-functions since $[\omega_{f, \mathrm{cris}},  \eta_{\overline{\alpha}} ] = [\omega_{f, \mathrm{cris}},  \eta_{\overline{\beta}} ] = 1$.
Also, $\eta_{\overline{f}}$ plays both roles of the \emph{correct} optimal period for the Mazur--Tate elements and the ``$p$-destabilization" since it also satisfies $[\omega_{f, \mathrm{cris}},  \eta_{\overline{f}} ] = 1$ and 
$$\overline{f}(z) = \dfrac{\overline{\alpha}}{\overline{\alpha} - \overline{\beta}} \cdot \overline{f}_{\overline{\alpha}}(z)
 + \dfrac{\overline{\beta}}{\overline{\beta}-\overline{\alpha}} \cdot \overline{f}_{\overline{\beta}}(z)$$
 where $\overline{f}_{\overline{\alpha}}(z) = \overline{f}(z) - \overline{\beta} \cdot \overline{f}(pz)$ and
 $\overline{f}_{\overline{\beta}}(z) = \overline{f}(z) - \overline{\alpha} \cdot \overline{f}(pz)$ are the $p$-stabilizations of $\overline{f}$.
\begin{rem} \label{rem:equivariant-local-points-integral}
 Since  $[\omega_{f, \mathrm{cris}},   \eta_{\overline{\alpha}}  ] = [\omega_{f, \mathrm{cris}},  \eta_{\overline{\beta}}  ] = [\omega_{f, \mathrm{cris}},  \eta_{\overline{f}}  ]= 1$, the images of $ \eta_{\overline{\alpha}} $, $ \eta_{\overline{\beta}} $, and $ \eta_{\overline{f}} $
  in $\mathbf{D}_{\mathrm{cris}}(V_{F_\pi}(\overline{f})) / \mathbf{D}^{r}_{\mathrm{cris}}(V_{F_\pi}(\overline{f}))$  also lie in $M_{\overline{f}} / M^{r}_{\overline{f}}$ but not in $\pi M_{\overline{f}} / M^{r}_{\overline{f}}$ for all $1 \leq r \leq k-1$.
\end{rem}



\begin{defn} \label{defn:equivariant-local-points}
For $1 \leq r \leq k-1$, we put
\begin{align*}
c^{\varepsilon_m}_{\mathbb{Q}(\zeta_{mp^n}), r, \bullet} := \Sigma^{\varepsilon_m}_{V_{\overline{f}}, r, \mathbb{Q}(\zeta_{mp^n})} ( G^{\sigma^{-n}}_{\mathbb{Q}(\zeta_{m}), \bullet} )  \in \mathrm{H}^1_f(\mathbb{Q}(\zeta_{mp^n}) \otimes \mathbb{Q}_p, V_{\overline{f}}(r)) \otimes \mathcal{O}_\pi[\alpha] 
\end{align*}
where $\bullet \in \lbrace \eta_{\overline{\alpha}}, \eta_{\overline{\beta}}, \eta_{\overline{f}} \rbrace$ is defined in (\ref{eqn:correct-periods}).
Due to Proposition \ref{prop:normalization_local_points_first} and Remark \ref{rem:equivariant-local-points-integral}, we further have
$$c^{\varepsilon_m}_{\mathbb{Q}(\zeta_{mp^n}), r, \eta_{\overline{\alpha}}},
c^{\varepsilon_m}_{\mathbb{Q}(\zeta_{mp^n}), r, \eta_{\overline{\beta}}},
c^{\varepsilon_m}_{\mathbb{Q}(\zeta_{mp^n}), r, \eta_{\overline{f}}} 
\in \mathrm{H}^1_{/f}(\mathbb{Q}(\zeta_{mp^n}) \otimes \mathbb{Q}_p, T_{f}(k-r))^* \otimes \mathcal{O}_\pi[\alpha] .$$
See \cite[2.3.6.Prop.]{perrin-riou-local-iwasawa} for the Fontaine--Laffaille case (with $m=1$).
\end{defn}
Thanks to Theorem \ref{thm:integral-exponential}.(2), we also have an explicit formula with the classical Bloch--Kato exponential map
\begin{equation} \label{eqn:explicit-formula-local-points}
c^{\varepsilon_m}_{\mathbb{Q}(\zeta_{mp^n}),r, \eta_{\overline{\alpha}}} = (-1)^{r} \cdot (r-1)! \cdot m^{r-1} \cdot p^{(r-1)n} \cdot \mathrm{exp}_{ V_{\overline{f}}(r), \mathbb{Q}(\zeta_{mp^n})} \left( G^{\sigma^{-n}}_{\mathbb{Q}(\zeta_{m}),\eta_{\overline{\alpha}}, r}(\zeta_{p^n} - 1) \right)
\end{equation}
where
$G^{\sigma^{-n}}_{\mathbb{Q}(\zeta_{m}),\eta_{\overline{\alpha}}, r}$
is the $\sigma^{-n}$-twist of $G_{\mathbb{Q}(\zeta_{m}),\eta_{\overline{\alpha}}, r}$
given by solution of equation
\begin{equation} \label{eqn:explicit-equation}
(1-\varphi) G_{\mathbb{Q}(\zeta_{m}),\eta_{\overline{\alpha}}, r}  = m^{r-1} \cdot  \zeta_m \cdot   (1+X)^{1/m}   \otimes \eta_{\overline{\alpha}} \otimes e_{-r} 
\end{equation}
due to (\ref{eqn:solution-twists}) and the shape of $G_{\mathbb{Q}(\zeta_{m}),\eta}$ in Lemma \ref{lem:1-varphi-differential-equations}.

\subsection{Explicit computation of equivariant local points} \label{subsec:computation-local-points}
We write down the explicit solution for (\ref{eqn:explicit-equation}) and compute $G^{\sigma^{-n}}_{\mathbb{Q}(\zeta_{m}),\eta_{\overline{\alpha}}, r}(\zeta_{p^n} - 1)$ 
by refining  \cite[Lem. 2.2.1]{lei-thesis}, \cite[Lem. 3.4]{lei-compositio}, and \cite[Prop. 3.19]{lei-loeffler-zerbes_wach} equivariantly.
Let 
\begin{align*}
g_{\mathbb{Q}(\zeta_m), \eta_{\overline{\alpha}},r} & = m^{r-1} \cdot  \gamma^{-1}_m \cdot \left(  \zeta_m \cdot   (1+X) \right) \otimes \eta_{\overline{\alpha}} \otimes e_{-r} \\
& = m^{r-1} \cdot  \gamma^{-1}_m \cdot \left(  \zeta_m \cdot  X  \right) \otimes \eta_{\overline{\alpha}}  \otimes e_{-r}
+  m^{r-1} \cdot  \left( \zeta_m \cdot   1 \right) \otimes \eta_{\overline{\alpha}}  \otimes e_{-r} 
\end{align*}
in $(\mathbb{Z}_p \otimes \mathbb{Z}[\zeta_m])\llbracket X \rrbracket \otimes M_{\overline{f}} \otimes \mathcal{O}_\pi[\alpha] \otimes \mathbb{Z}_p\cdot e_{-r} $
where $\gamma^{-1}_m \in \Gamma$ with $\chi_{\mathrm{cyc}}(\gamma^{-1}_m) = 1/m$.
Since the actions of $\varphi$ and $\Gamma$ commute, we have
$$\varphi^i \left(  m^{r-1} \cdot \gamma^{-1}_m \cdot \left( \zeta_m \cdot  X  \right) \otimes \eta_{\overline{\alpha}} \otimes e_{-r} \right) =  m^{r-1} \cdot \gamma^{-1}_m \left( \left(  \zeta^{p^i}_m \cdot \left( ( X +1)^{p^i} - 1 \right) \right) \otimes \varphi^i ( \eta_{\overline{\alpha}}  \otimes e_{-r} )  \right) .$$
Since $( X +1)^{p^i} - 1$ converges to zero in $(\mathbb{Z}_p \otimes \mathbb{Z}[\zeta_m]) \llbracket X \rrbracket$ as $i$ goes to $\infty$, 
$\varphi^i \left( \gamma^{-1}_m \cdot \left( 1/m \cdot \zeta_m \cdot  X  \right) \otimes \eta_{\overline{\alpha}} \otimes e_{-r}  \right)$
also converges to zero  as $i$ goes to $\infty$.
Thus,
$$G_{\mathbb{Q}(\zeta_m), \eta_{\overline{\alpha}},r} = \sum_{i \geq 0} \varphi^i \left(  m^{r-1} \cdot \gamma^{-1}_m \cdot \left(  \zeta_m \cdot  X  \right) \otimes \eta_{\overline{\alpha}} \otimes e_{-r} \right) + (1-\varphi)^{-1} \left( m^{r-1} \cdot \left(  \zeta_m \cdot   1 \right) \otimes \eta_{\overline{\alpha}} \otimes e_{-r} \right) $$
exists and satisfies the equation
$(1 - \varphi ) G_{\mathbb{Q}(\zeta_m), \eta_{\overline{\alpha}},r}  = g_{\mathbb{Q}(\zeta_m), \eta_{\overline{\alpha}},r} $.
Since $\frac{1}{1 - \overline{\alpha}^{\phi(m)} \cdot p^{-r}} = \frac{ p^{r} }{p^{r} - \overline{\alpha}^{\phi(m)} } \in \mathcal{O}_\pi[\alpha]$ and is non-zero, the integrality of $G_{\mathbb{Q}(\zeta_m), \eta_{\overline{\alpha}},r}$ does not break under Assumption \ref{assu:standard}.(ENV).
We compute the specialization of $G_{\mathbb{Q}(\zeta_m), \eta_{\overline{\alpha}},r}$  at $X = \zeta_{p^n} - 1$.
Suppose that $n \geq 1$. Then we have
\begin{align*}
& G_{\mathbb{Q}(\zeta_m), \eta_{\overline{\alpha}},r} (\zeta_{p^n} - 1) \\
 = & \sum_{i =0}^{n-1} \left( \left( m^{r-1} \cdot \zeta^{p^i}_m \cdot \zeta^{1/m}_{p^{n-i}} \right) \otimes \varphi^i ( \eta_{\overline{\alpha}} \otimes e_{-r} )  \right)  + (1-\varphi)^{-1} \left( \left( m^{r-1} \cdot \zeta^{p^n}_m \cdot   1 \right) \otimes  \varphi^n( \eta_{\overline{\alpha}} \otimes e_{-r} ) \right) .
\end{align*}
Thus, we have
\begin{align} \label{eqn:local-points-higher-class}
\begin{split}
& \left( p^{-n} \cdot \varphi^{-n} \right)  G_{\mathbb{Q}(\zeta_m), \eta_{\overline{\alpha}},r} (\zeta_{p^n} - 1) \\
= \ &  \dfrac{1}{p^n} \cdot  \sum_{i =0}^{n-1} \left( \left( m^{r-1} \cdot \zeta^{p^{i-n}}_m \cdot  \zeta^{1/m}_{p^{n-i}}  \right) \otimes \varphi^{i-n} ( \eta_{\overline{\alpha}} \otimes e_{-r} )  \right)  \\
& + \dfrac{1}{p^n} \cdot  (1-\varphi)^{-1} \left( \left( m^{r-1} \cdot \zeta_m \cdot   1 \right) \otimes   \eta_{\overline{\alpha}} \otimes e_{-r} \right)  .
\end{split}
\end{align}
The $n=0$ case is as follows
\begin{align} \label{eqn:local-points-bottom-class}
\begin{split}
& \mathrm{Tr}_{\mathbb{Q}(\zeta_{mp})/\mathbb{Q}(\zeta_m)} \left( 
\left( p^{-1} \cdot \varphi^{-1} \right)  G_{\mathbb{Q}(\zeta_m), \eta_{\overline{\alpha}},r} (\zeta_{p} - 1)
\right) \\
& = \left(1 - \dfrac{\varphi^{-1}}{p} \right)  \left(1 - \varphi \right)^{-1} \left( m^{r-1} \cdot \zeta_m \otimes \eta_{\overline{\alpha}} \otimes e_{-r} \right) .
\end{split}
\end{align}

Note that $``G_{\mathbb{Q}(\zeta_m), \eta_{\overline{\alpha}}, r} (0) = \left(1 - \varphi \right)^{-1} \left( m^{r-1} \cdot \zeta_m \otimes \eta_{\overline{\alpha}} \otimes e_{-r} \right)"$.
As consequences, we have the norm compatibilities and the interpolation formula of equivariant local points.
\begin{prop} \label{prop:local-points-properties}
Let
$ \mathrm{cor}_{\mathbb{Q}(\zeta_{mp^n})/\mathbb{Q}(\zeta_{mp^{n-1}})} = \bigoplus_{v \vert p}
\mathrm{cor}_{\mathbb{Q}(\zeta_{mp^n})_v/\mathbb{Q}(\zeta_{mp^{n-1}})_v}$ be the direct sum of local corestriction maps
where $v$ runs over the primes of $\mathbb{Q}(\zeta_m)$ lying over $p$.
\begin{enumerate}
\item For $n \geq 2$, we have
$$\mathrm{cor}_{\mathbb{Q}(\zeta_{mp^n})/\mathbb{Q}(\zeta_{mp^{n-1}})} \left(c^{\varepsilon_m}_{\mathbb{Q}(\zeta_{mp^n}),r, \eta_{\overline{\alpha}}} \right)  = \overline{\alpha} \cdot c^{\varepsilon_m}_{\mathbb{Q}(\zeta_{mp^{n-1}}),r, \eta_{\overline{\alpha}}}$$
 in
$\mathrm{H}^1_f(\mathbb{Q}(\zeta_{mp^{n-1}}) \otimes \mathbb{Q}_p, V_{\overline{f}}(r)) \otimes \mathcal{O}_{\pi}[\alpha]$.
\item For $n = 1$, we have
$$\mathrm{cor}_{\mathbb{Q}(\zeta_{mp})/\mathbb{Q}(\zeta_{m})} \left( c^{\varepsilon_m}_{\mathbb{Q}(\zeta_{mp}), r,\eta_{\overline{\alpha}}} \right)  = \left( \overline{\alpha} - p^{r-1} \cdot \mathrm{Frob}_p \right) \cdot c^{\varepsilon_m}_{\mathbb{Q}(\zeta_{m}), r, \eta_{\overline{\alpha}}}$$
 in
$\mathrm{H}^1_f(\mathbb{Q}(\zeta_{m}) \otimes \mathbb{Q}_p, V_{\overline{f}}(r)) \otimes \mathcal{O}_{\pi}[\alpha]$.
\item  If $\ell$ is a prime with $(\ell, pmN)=1$, then 
$$\mathrm{cor}_{\mathbb{Q}(\zeta_{m\ell p^n})/\mathbb{Q}(\zeta_{mp^{n}})} \left(  c^{\varepsilon_{m\ell}}_{\mathbb{Q}(\zeta_{m\ell p^n}), r, \eta_{\overline{\alpha}}} \right)  = - \ell^{r-1}  \cdot \left( c^{\varepsilon_m}_{\mathbb{Q}(\zeta_{mp^n}), r, \eta_{\overline{\alpha}}} \right)^{\mathrm{Frob}^{-1}_\ell}$$
in $\mathrm{H}^1_f(\mathbb{Q}(\zeta_{mp^{n}}) \otimes \mathbb{Q}_p, V_{\overline{f}}(r)) \otimes \mathcal{O}_{\pi}[\alpha]$
where $\mathrm{Frob}^{-1}_\ell \in \mathrm{Gal}(\mathbb{Q}(\zeta_{mp^{n}})/\mathbb{Q})$.
\item Let $\chi$ be a primitive Dirichlet character of conductor dividing $mp^n$ with $n \geq 1$.
Then we have
$$\chi \left( \sum_{\sigma \in \mathrm{Gal}(\mathbb{Q}(\zeta_{mp^n})/\mathbb{Q})} \left( c^{\varepsilon_m}_{\mathbb{Q}(\zeta_{mp^n}),r, \eta_{\overline{\alpha}}} \right)^{\sigma} \cdot \sigma \right) 
 = (-1)^r \cdot (r-1)! \cdot (mp^n)^{r-1}  \cdot \tau(\chi) \cdot \dfrac{ 1 }{ \overline{\alpha}^n } \cdot \eta_{\overline{\alpha}} \otimes e_{-r} .
$$
\item Let $\chi$ be a primitive Dirichlet character of conductor $m$ with $(m,p)=1$.
Then we have
$$\chi \left( \sum_{\sigma \in \mathrm{Gal}(\mathbb{Q}(\zeta_{m})/\mathbb{Q})} \left( c^{\varepsilon_m}_{\mathbb{Q}(\zeta_{mp^n}),r, \eta_{\overline{\alpha}}} \right)^{\sigma} \cdot \sigma \right) 
= (-1)^r \cdot (r-1)! \cdot (mp^n)^{r-1}  \cdot \tau(\chi) \cdot \left( 1- \dfrac{ \chi(p) \cdot p^{r-1} }{ \overline{\alpha} }\right) \cdot \eta_{\overline{\alpha}} \otimes e_{-r} .$$
\end{enumerate}
\end{prop}
\begin{proof}
\begin{enumerate}
\item
It immediately follows from Theorem \ref{thm:integral-exponential}.(1).
\item
See \cite[2.4.2.Prop.]{perrin-riou-local-iwasawa}.
\item It follows from the comparison between 
$c^{\varepsilon_m}_{\mathbb{Q}(\zeta_{mp^n}),r, \eta_{\overline{\alpha}}}$
and
$c^{\varepsilon_{m\ell}}_{\mathbb{Q}(\zeta_{m \ell p^n}),r, \eta_{\overline{\alpha}}}$
via (\ref{eqn:explicit-formula-local-points})
and $\mathrm{Tr}_{\mathbb{Q}(\zeta_\ell)/\mathbb{Q}}(\zeta_\ell) = -1$.
\item It follows from (\ref{eqn:local-points-higher-class}).
\item It follows from (\ref{eqn:local-points-bottom-class}). 
\end{enumerate}
\end{proof}

%

\subsection{The equivariant big exponential map and the equivariant Perrin-Riou pairing}
By using $\Sigma^{\varepsilon_m}_{T_{\overline{f}}, r , \mathbb{Q}(\zeta_{mp^n}) }$ instead of $\Sigma^{\varepsilon}_{T_{\overline{f}}, r , \mathbb{Q}(\zeta_{mp^n}) }$, we also refine the big exponential map equivariantly as follows.
\begin{cor}
There exists a unique
 $F_\pi \otimes \mathscr{H}_{\mathbb{Q}_p}(\Gamma)[\mathrm{Gal}(\mathbb{Q}(\zeta_m)/\mathbb{Q})]$-equivariant homomorphism
$$\Omega^{\varepsilon_m}_{V_{\overline{f}}, 0, \mathbb{Q}(\zeta_m)} : \mathscr{D}_{\mathbb{Q}(\zeta_m)}(V_{\overline{f}})^{\widetilde{\Delta}_0} \to \mathscr{H}(\Gamma)[\mathrm{Gal}(\mathbb{Q}(\zeta_m)/\mathbb{Q})] \otimes_{\Lambda_{\mathcal{I}w}[\mathrm{Gal}(\mathbb{Q}(\zeta_m)/\mathbb{Q})]} \left( \bigoplus_v \dfrac{\mathrm{H}^1_{\mathcal{I}w}(\mathbb{Q}(\zeta_{m})_v, T_{\overline{f}})}{T^{G_{\mathbb{Q}(\zeta_{mp^\infty})_v}}_{\overline{f}}} \right)$$
satisfying all the properties of $\Omega^{\varepsilon}_{V_{\overline{f}}, 0}$ given in $\S$\ref{subsec:big-exponential-map} and $\S$\ref{subsec:explicit-reciprocity-law} 
where $v$ runs over all primes of $\mathbb{Q}(\zeta_{m})$ lying above $p$.
\end{cor}
\begin{proof}
It immediately follows from Theorem \ref{thm:big-exponential-map} by considering the $\mathrm{Gal}(\mathbb{Q}(\zeta_m)/\mathbb{Q})$-equivariant property of $\mathrm{exp}_{V_{\overline{f}}, \mathbb{Q}(\zeta_m)}$ with Choice \ref{choice:p-adic-orientation-enhanced}.
\end{proof}
\begin{rem}
We will use $\varepsilon^{-1}_m$ instead of $\varepsilon_m$ in the construction of $p$-adic $L$-functions in $\S$\ref{subsec:construction-p-adic-L}.
\end{rem}
For $ x = (x_{mp^{n}})_n \in \mathrm{H}^1_{\mathcal{I}w}(\mathbb{Q}(\zeta_{m}) \otimes \mathbb{Q}_p, T_f(k-r) )$ and $ y = (y_{mp^{n}})_n \in \mathrm{H}^1_{\mathcal{I}w}(\mathbb{Q}(\zeta_{m}) \otimes \mathbb{Q}_p, T_{\overline{f}}(r) )$,
we define the \textbf{equivariant Perrin-Riou pairing} $\langle -, - \rangle_{T_f(k-r),\mathbb{Q}(\zeta_{mp^\infty})}$ on 
$\mathrm{H}^1_{\mathcal{I}w}(\mathbb{Q}(\zeta_{m}) \otimes \mathbb{Q}_p, T_f(k-r) ) \times \mathrm{H}^1_{\mathcal{I}w}(\mathbb{Q}(\zeta_{m}) \otimes \mathbb{Q}_p, T_{\overline{f}}(r) )$
taking values in  $\Lambda_{\mathcal{I}w}[\mathrm{Gal}(\mathbb{Q}(\zeta_m)/\mathbb{Q})]$
by
$$
\langle x, y \rangle_{T_f(k-r),\mathbb{Q}(\zeta_{mp^\infty})}  := \varprojlim_n \left( \sum_{\sigma \in \mathrm{Gal}(\mathbb{Q}(\zeta_{mp^n})/\mathbb{Q})} \left[ ( x_{mp^{n}} )^{\sigma^{-1}}, y_{mp^{n}} \right]_{T_f(k-r),\mathbb{Q}(\zeta_{mp^n})} \cdot \sigma  \right) 
$$
in $\Lambda_{\mathcal{I}w}[\mathrm{Gal}(\mathbb{Q}(\zeta_m)/\mathbb{Q})]$.
As in $\S$\ref{subsec:explicit-reciprocity-law}, it also satisfies
$$\langle \lambda \cdot x, y \rangle_{T,K} = \lambda \cdot \langle x, y \rangle_{T_f(k-r),\mathbb{Q}(\zeta_{mp^\infty})}
 = \langle x, \iota(\lambda) \cdot y \rangle_{T,K}$$
where $\lambda \in \Lambda_{\mathcal{I}w}[\mathrm{Gal}(\mathbb{Q}(\zeta_m)/\mathbb{Q})]$ and $\iota : \Lambda_{\mathcal{I}w}[\mathrm{Gal}(\mathbb{Q}(\zeta_m)/\mathbb{Q})] \to \Lambda_{\mathcal{I}w}[\mathrm{Gal}(\mathbb{Q}(\zeta_m)/\mathbb{Q})]$ sends $\gamma \mapsto \gamma^{-1}$ for $\gamma \in \mathrm{Gal}(\mathbb{Q}(\zeta_{mp^\infty})/\mathbb{Q})$.
The pairing also extends linearly by taking the tensor product with $\mathscr{H}_{\mathbb{Q}_p}(\Gamma)$.
Note that $p$ is inverted in $\mathscr{H}_{\mathbb{Q}_p}(\Gamma)$ here again.

\subsection{Kato--Perrin-Riou $p$-adic $L$-functions for modular forms} \label{subsec:construction-p-adic-L}
\begin{defn}
Consider both $p$-stabilizations $\overline{f}_{\overline{\alpha}}$ and $\overline{f}_{\overline{\beta}}$ for $\overline{f}$.
Then the \textbf{equivariant $p$-adic $L$-function of $\overline{f}_{\overline{\alpha}}$ over $\mathbb{Q}(\zeta_{mp^\infty})$} is defined by
\begin{align*}
& L_p(\mathbb{Q}(\zeta_{mp^\infty}), \overline{f}_{\overline{\alpha}})  \\
 := \ & \mathrm{Tw}_1 \left( \left\langle \mathrm{loc}_p \mathfrak{Z}_{\mathbb{Q}(\zeta_{mp^\infty}), \delta_{f, \mathrm{cris}}}(f,k), \Omega^{\varepsilon^{-1}_m}_{V_{\overline{f}}, 0, \mathbb{Q}(\zeta_m)} \left(  \dfrac{1}{m} \cdot \zeta_m \cdot   (1+X)^{1/m} \otimes \eta_{\overline{\alpha}} \right) \right\rangle_{T_f(k), \mathbb{Q}(\zeta_{mp^\infty})} \right) \\
 \in \ & \mathscr{H}_{\mathrm{ord}_p(\overline{\alpha})}(\Gamma) [\mathrm{Gal}(\mathbb{Q}(\zeta_m)/\mathbb{Q})]
\end{align*}
and 
$L_p(\mathbb{Q}(\zeta_{mp^\infty}), \overline{f}_{\overline{\beta}}) \in \mathscr{H}_{\mathrm{ord}_p(\overline{\beta})}(\Gamma) [\mathrm{Gal}(\mathbb{Q}(\zeta_m)/\mathbb{Q})]$ is defined in the exactly same way.
\end{defn}
\begin{rem} 
\begin{enumerate}
\item 
In $\S$\ref{subsec:interpolation}, we show that our construction exactly gives the the same $p$-adic $L$-function constructed via modular symbols by Mazur--Tate--Teitelbaum \cite{mtt}.
\item If $\overline{f}$ is $p$-ordinary, then $$L_p(\mathbb{Q}(\zeta_{mp^\infty}), \overline{f}_{\overline{\alpha}}) \in  \Lambda_{\mathcal{I}w}[\mathrm{Gal}(\mathbb{Q}(\zeta_m)/\mathbb{Q})] $$
is the equivariant ordinary $p$-adic $L$-function
and 
$$ L_p(\mathbb{Q}(\zeta_{mp^\infty}), \overline{f}_{\overline{\beta}}) \in \mathscr{H}_{k-1}(\Gamma) [\mathrm{Gal}(\mathbb{Q}(\zeta_m)/\mathbb{Q})]$$ is the equivariant critical slope one.
\item 
In the ordinary case,  \cite[Wishful Thinking 8.2.2]{rubin-book}, which is much stronger than \cite[Conj. 8.2.6]{rubin-book}, holds due to the large image assumption.
\item The critical slope case (with $m=1$) is studied in \cite{loeffler-zerbes-wach-critical}.
See \cite{pollack-stevens-critical} for the overconvergent modular symbol counterpart.
\item
The critical slope $p$-adic $L$-function also satisfies the exactly same interpolation formula as the non-critical one if the restriction of $\rho_{\overline{f}}$ to $G_{\mathbb{Q}_p}$ is not a direct sum of characters (c.f. Assumption \ref{assu:standard}.(``non-CM")). However, the interpolation formula does not determine the $p$-adic $L$-function in this case.
On the other hand, Mazur--Tate elements are determined by the interpolation formula for Dirichlet characters and the norm relation in $F_\pi[\mathrm{Gal}(\mathbb{Q}(\zeta_{mp^n})/\mathbb{Q})]$. 
\end{enumerate}
\end{rem}
In \cite[$\S$3.2]{lei-compositio} and \cite[$\S$3.5.1]{lei-loeffler-zerbes_wach}, another construction via the Perrin-Riou pairing is given by
$$L^{\mathrm{LLZ}}_p(\mathbb{Q}(\zeta_{p^\infty}), \overline{f}_{\overline{\alpha}}) =
\mathrm{Tw}_1 \left( \left\langle  \Omega^{\varepsilon}_{V_{\overline{f}}, 0} \left(    (1+X) \otimes \eta_{\overline{\alpha}} \right) , \mathrm{loc}_p \mathfrak{Z}_{\mathbb{Q}(\zeta_{p^\infty}), \delta_{f, \mathrm{cris}}}(f,k) \right\rangle_{T_f(k), \mathbb{Q}(\zeta_{p^\infty})} \right) .$$
This construction gives us a \emph{very slightly different} interpolation formula, which is compatible with the (corrected) interpolation formula for Kitagawa's $p$-adic $L$-functions \cite{kitagawa}, \cite[Appendix B]{futterer-thesis}. Since they differ by a unit in the Iwasawa algebra, it does not affect the formulation of the Iwasawa main conjecture. See $\S$\ref{subsec:interpolation} for the precise computation.
\subsection{The comparison with Mazur--Tate--Teitelbaum $p$-adic $L$-functions} \label{subsec:interpolation}
We explicitly compute the twists and projections of $p$-adic $L$-functions
$\mathrm{pr}_{\mathbb{Q}(\zeta_{mp^n})} \circ \mathrm{Tw}_{r} \left( L_p(\mathbb{Q}(\zeta_{mp^\infty}), \overline{f}_{\overline{\alpha}})  \right)$
at critical points. Straightforward computation shows
{ \scriptsize
\begin{align} \label{eqn:interpolation-computation-part1}
\begin{split}
& \mathrm{pr}_{\mathbb{Q}(\zeta_{mp^n})} \circ \mathrm{Tw}_{r} \left( \left\langle \mathrm{loc}_p \mathfrak{Z}_{\mathbb{Q}(\zeta_{mp^\infty}), \delta_{f, \mathrm{cris}} }(f,k), \Omega^{\varepsilon^{-1}_m}_{V_{\overline{f}}, 0, \mathbb{Q}(\zeta_m)} \left(  \dfrac{1}{m} \cdot \zeta_m \cdot   (1+X)^{1/m} \otimes \eta_{\overline{\alpha}} \right) \right\rangle_{T_f(k), \mathbb{Q}(\zeta_{mp^\infty})} \right) \\
& = \sum_{\sigma'}   \left[ \mathrm{loc}_p \mathfrak{z}_{\mathbb{Q}(\zeta_{mp^n}), \delta_{f, \mathrm{cris}} }(f,k-r)^{\sigma'^{-1}}, \mathrm{pr}_{\mathbb{Q}(\zeta_{mp^n})}  \circ \mathrm{Tw}^{\varepsilon}_{r} \circ \Omega^{\varepsilon^{-1}_m}_{V_{\overline{f}}, 0, \mathbb{Q}(\zeta_m)} \left(  \dfrac{1}{m} \cdot \zeta_m \cdot   (1+X)^{1/m} \otimes \eta_{\overline{\alpha}} \right) \right]_{T_f(k), \mathbb{Q}(\zeta_{mp^n}) } \cdot \sigma'
\end{split}
\end{align}
}
where $\sigma'$ runs over $\mathrm{Gal}(\mathbb{Q}(\zeta_{mp^n})/\mathbb{Q})$.
Thus, it suffices to compute
\begin{align*}
& \mathrm{pr}_{\mathbb{Q}(\zeta_{mp^n})}  \circ \mathrm{Tw}_{r} \circ \Omega^{\varepsilon^{-1}_m}_{V_{\overline{f}}, 0, \mathbb{Q}(\zeta_m)} \left( \gamma^{-1}_m \cdot \left( 1/m \cdot \zeta_m \cdot   (1+X) \right) \otimes \eta_{\overline{\alpha}} \right) \\
&  =
\Omega^{\varepsilon^{-1}_m}_{V_{\overline{f}}, r, \mathbb{Q}(\zeta_{mp^n})}  \left( \gamma^{-1}_m \cdot \left( 1/m \cdot \zeta_m \cdot   (1+X) \right) \otimes \eta_{\overline{\alpha}} \right)
\end{align*}
due to (\ref{eqn:projection-Omega}) and Proposition \ref{prop:choice-varepsilon}.(1).
By (\ref{eqn:Omega-Sigma}), we have
\begin{align} \label{eqn:local-points-norm-relations}
\begin{split}
\Omega^{\varepsilon^{-1}_m}_{V_{\overline{f}}, r, \mathbb{Q}(\zeta_{mp^n})} \left( 
1/m \cdot \zeta_m \cdot   (1+X)^{1/m}  \otimes \eta_{\overline{\alpha}} \right)
& =
 \Sigma^{\varepsilon^{-1}_m}_{V_{\overline{f}}, r, \mathbb{Q}(\zeta_{mp^n})} \left( 
G^{\varphi^{-n}}_{\mathbb{Q}(\zeta_{m}), \eta_{\overline{\alpha}}} 
 \right) \\
& = \dfrac{1}{\overline{\alpha}^n} \cdot
 \Sigma^{\varepsilon^{-1}_m}_{V_{\overline{f}}, r, \mathbb{Q}(\zeta_{mp^n})} \left( 
G^{\sigma^{-n}}_{\mathbb{Q}(\zeta_{m}), \eta_{\overline{\alpha}}} 
 \right) .
\end{split}
\end{align}
By (\ref{eqn:Omega-exp}), we also have
\begin{align} \label{eqn:local-points-interpolation}
\begin{split}
& \Omega^{\varepsilon^{-1}_m}_{V_{\overline{f}}, r, \mathbb{Q}(\zeta_{mp^n})} \left( 
1/m \cdot \zeta_m \cdot   (1+X)^{1/m}  \otimes \eta_{\overline{\alpha}} \right) \\
& = (-1)^r \cdot (r-1)! \cdot (mp^n)^{r-1} \cdot \dfrac{1}{\overline{\alpha}^n} \cdot \mathrm{exp}_{V_{\overline{f}}(r), \mathbb{Q}(\zeta_{mp^n})} \left( G^{\sigma^{-n}}_{\mathbb{Q}(\zeta^{-1}_{m}), \eta_{\overline{\alpha}},r} (\zeta^{-1}_{p^n}-1) \right)
\end{split}
\end{align}
where
$G^{\sigma^{-n}}_{\mathbb{Q}(\zeta^{-1}_{m}), \eta_{\overline{\alpha}},r}$
is the solution of
$(1-\varphi ) G^{\sigma^{-n}}_{\mathbb{Q}(\zeta^{-1}_{m}), \eta_{\overline{\alpha}},r}
=m^{r-1} \cdot  \gamma^{-1}_m \cdot \left(  \zeta^{-1}_m \cdot   (1+X) \right) \otimes \eta_{\overline{\alpha}} \otimes e_{-r} $.
Indeed, $G^{\sigma^{-n}}_{\mathbb{Q}(\zeta^{-1}_{m}), \eta_{\overline{\alpha}},r}$ is 
exactly the same as
$G^{\sigma^{-n}}_{\mathbb{Q}(\zeta_{m}), \eta_{\overline{\alpha}},r}$ in $\S$\ref{subsec:computation-local-points} except replacing $\zeta_m$ by $\zeta^{-1}_m$.
Note that the norm relation follows from  (\ref{eqn:local-points-norm-relations}) and the interpolation formula follows from
 (\ref{eqn:local-points-interpolation}). 
We continue the computation from (\ref{eqn:interpolation-computation-part1}):
{ \scriptsize
\begin{align*}
& \sum_{\sigma'}   \left[ \mathrm{loc}_p \mathfrak{z}_{\mathbb{Q}(\zeta_{mp^n}), \delta_{f, \mathrm{cris}} }(f,k-r)^{\sigma'^{-1}}, \mathrm{pr}_{\mathbb{Q}(\zeta_{mp^n})}  \circ \mathrm{Tw}^{\varepsilon}_{r} \circ \Omega^{\varepsilon^{-1}_m}_{V_{\overline{f}}, 0, \mathbb{Q}(\zeta_m)} \left(  \dfrac{1}{m} \cdot \zeta_m \cdot   (1+X)^{1/m} \otimes \eta_{\overline{\alpha}} \right) \right]_{T_f(k), \mathbb{Q}(\zeta_{mp^n}) } \cdot \sigma' \\
& = \dfrac{ (-1)^r \cdot (r-1)! \cdot (mp^n)^{r-1} }{\overline{\alpha}^n} \cdot   \left[ \sum_{\sigma'} \mathrm{exp}^* \circ \mathrm{loc}_p \mathfrak{z}_{\mathbb{Q}(\zeta_{mp^n}), \delta_{f, \mathrm{cris}} }(f,k-r)^{\sigma'^{-1}} \cdot \sigma' , \sum_{\tau}    G^{\sigma^{-n}}_{\mathbb{Q}(\zeta^{-1}_{m}), \eta_{\overline{\alpha}},r} (\zeta^{-1}_{p^n}-1)^{\tau} \cdot \tau \right]_{\mathbf{D}( V_f(k) ), \mathbb{Q}(\zeta_{mp^n}) }
\end{align*}
}
where the sums run over $\mathrm{Gal}(\mathbb{Q}(\zeta_{mp^n})/\mathbb{Q})$, 
$\mathrm{exp} = \mathrm{exp}_{V_{\overline{f}}(r), \mathbb{Q}(\zeta_{mp^n})}$, and
$\mathrm{exp}^* = \mathrm{exp}^*_{V_{f}(k-r), \mathbb{Q}(\zeta_{mp^n})} $. A similar computation can be found in \cite[Lem. 3.2]{kurihara-invent}.
By evaluating the last term at Dirichlet characters of conductor $mp^n$, we obtain the following statement.
\begin{thm} \label{thm:comparison-kato-perrin-riou-mtt}
In the non-critical slope case, the $p$-adic $L$-function precisely coincides with the $p$-adic $L$-function of Mazur--Tate--Teitelbaum defined in \cite[$\S$13]{mtt}.
\end{thm}
\begin{proof}
See \cite[Prop. in $\S$14, Chap. I with the convention given in $\S$9, Chap. I]{mtt} for the interpolation formula for  the $p$-adic $L$-function of Mazur--Tate--Teitelbaum.
The conclusion immediately follows from Theorem \ref{thm:kato_interpolation}, Proposition \ref{prop:local-points-properties}.(4) and (5), the difference between $\overline{f}$ and $\overline{f}_{\overline{\alpha}}$, and
\begin{align} \label{eqn:gauss-sum-computation}
\begin{split}
{\displaystyle \sum_{\sigma_a \in \mathrm{Gal}(\mathbb{Q}(\zeta_{mp^n})/\mathbb{Q})} } \chi(a) \cdot \zeta^{-a}_{mp^n} & =
{\displaystyle \sum_{\sigma_a \in \mathrm{Gal}(\mathbb{Q}(\zeta_{mp^n})/\mathbb{Q})} } \chi(-1) \cdot \chi(-a) \cdot \zeta^{-a}_{mp^n} \\
& = \chi(-1) \cdot  {\displaystyle \sum_{\sigma_a \in \mathrm{Gal}(\mathbb{Q}(\zeta_{mp^n})/\mathbb{Q})} } \chi(-a) \cdot \zeta^{-a}_{mp^n} \\
& = \chi(-1) \cdot \tau(\chi) \\
& = \dfrac{mp^n}{\tau(\chi^{-1})} 
\end{split}
\end{align}
where $\chi$ is a Dirichlet character of conductor $mp^n$.
Since the interpolation formula determines the $p$-adic $L$-function in the non-critical slope case, the conclusion follows.
\end{proof}
\begin{rem}
Choice \ref{choice:p-adic-orientation-enhanced} genuinely affects the interpolation formula due to (\ref{eqn:gauss-sum-computation}).
This is observed in \cite[$\S$3.3 in Chap. IV]{futterer-thesis} and \cite[Ex. 9.2.14]{zahringer-thesis} and our construction yields a systematic explanation of their observation.
\end{rem}
Let $\chi$ be a finite order character on $\mathrm{Gal}(\mathbb{Q}(\zeta_{p^\infty})/\mathbb{Q})$.
Then we have
$$\dfrac{
\chi \cdot \chi^{r-1}_{\mathrm{cyc}} \left( L_p(\mathbb{Q}(\zeta_{p^\infty}), \overline{f}_{\overline{\alpha}}) \right)
}
{
\chi \cdot \chi^{r-1}_{\mathrm{cyc}} \left( L^{\mathrm{LLZ}}_p(\mathbb{Q}(\zeta_{p^\infty}), \overline{f}_{\overline{\alpha}}) \right)
} = \chi(-1) \cdot (-1)^{r-1}$$
due to (\ref{eqn:gauss-sum-computation}) and Proposition \ref{prop:choice-varepsilon}.(2).
Following the computation in \cite{mtt} and \cite[Thm. 1.11 in Appendix B]{futterer-thesis}, we also have
$$\dfrac{
\chi \cdot \chi^{r-1}_{\mathrm{cyc}} \left( L^{\mathrm{MTT}}_p(\mathbb{Q}(\zeta_{p^\infty}), \overline{f}_{\overline{\alpha}}) \right)
}{
\chi \cdot \chi^{r-1}_{\mathrm{cyc}} \left( L^{\mathrm{Kit}}_p(\mathbb{Q}(\zeta_{p^\infty}), \overline{f}_{\overline{\alpha}}) \right)
} = \chi(-1) \cdot (-1)^{r-1} $$
where 
$L^{\mathrm{MTT}}_p(\mathbb{Q}(\zeta_{p^\infty}), \overline{f}_{\overline{\alpha}})$ is Mazur--Tate--Teitelbaum's $p$-adic $L$-function \cite{mtt} and
$L^{\mathrm{Kit}}_p(\mathbb{Q}(\zeta_{p^\infty}), \overline{f}_{\overline{\alpha}})$ is Kitagawa's $p$-adic $L$-function \cite{kitagawa}, \cite[Appendix B]{futterer-thesis}. The complex conjugation in $\Gamma$ satisfies the interpolation formula $\chi(-1) \cdot (-1)^{r-1}$.

\subsection{Finite layer Coleman maps and Mazur--Tate elements} \label{subsec:finite-layer-coleman}
The following pairing generalizes the notion of the $P_n$-pairing of Kurihara \cite[$\S$3]{kurihara-invent}.
\begin{defn} 
We define the pairing
$$P_{\mathbb{Q}(\zeta_{mp^n}), r} : \mathrm{H}^1_{/f}( \mathbb{Q}(\zeta_{mp^n}) \otimes \mathbb{Q}_p, V_f(k-r)) 
\times 
\mathrm{H}^1_{f}( \mathbb{Q}(\zeta_{mp^n}) \otimes \mathbb{Q}_p, V_{\overline{f}}(r))
\to F_\pi[\mathrm{Gal}( \mathbb{Q}(\zeta_{mp^n}) / \mathbb{Q} )]$$
by
\begin{align*}
P_{\mathbb{Q}(\zeta_{mp^n}),r}(z_{mp^n},c_{mp^n}) & = \sum_{\sigma} \left[ z^{\sigma^{-1}}_{mp^n},c_{mp^n} \right]_{V_{f}(k-r),\mathbb{Q}(\zeta_{mp^n})} \cdot \sigma   \\
& = \sum_{\sigma} \left[ z_{mp^n},c^{\sigma}_{mp^n} \right]_{V_{f}(k-r),\mathbb{Q}(\zeta_{mp^n})} \cdot \sigma
\end{align*}
where
$ z_{mp^n}  \in \mathrm{H}^1_{/f}( \mathbb{Q}(\zeta_{mp^n}) \otimes \mathbb{Q}_p, V_f(k-r)) $,
$c_{mp^n} \in \mathrm{H}^1_{f}( \mathbb{Q}(\zeta_{mp^n}) \otimes \mathbb{Q}_p, V_{\overline{f}}(r))  $, and
$\sigma$ runs over $\mathrm{Gal}(\mathbb{Q}(\zeta_{mp^n})/\mathbb{Q})$.
\end{defn}
Following Definition \ref{defn:equivariant-local-points}, we define the \textbf{($\overline{\alpha}$-stabilized) equivariant local point associated to $(\mathbb{Q}(\zeta_{mp^n}),r, \eta_{\overline{\alpha}})$} by
$$c^{\varepsilon^{-1}_m}_{\mathbb{Q}(\zeta_{mp^n}),r, \eta_{\overline{\alpha}}}
=  \Sigma^{\varepsilon^{-1}_m}_{T_{\overline{f}}, r, \mathbb{Q}(\zeta_{mp^n})} (G^{\sigma^{-n}}_{\mathbb{Q}(\zeta_m), \eta_{\overline{\alpha}}}) \in \mathrm{H}^1_{f}( \mathbb{Q}(\zeta_{mp^n}) \otimes \mathbb{Q}_p, V_{\overline{f}}(r) ) \otimes \mathcal{O}_\pi[\alpha] .$$
Following (\ref{eqn:correct-periods}), we define the \textbf{equivariant local point associated to $(\mathbb{Q}(\zeta_{mp^n}),r, \eta_{\overline{f}})$} by
$$c^{\varepsilon^{-1}_m}_{\mathbb{Q}(\zeta_{mp^n}),r, \eta_{\overline{f}}} = \dfrac{\overline{\alpha}}{\overline{\alpha} - \overline{\beta}} \cdot c^{\varepsilon^{-1}_m}_{\mathbb{Q}(\zeta_{mp^n}),r, \eta_{\overline{\alpha}}}
+ \dfrac{\overline{\beta}}{\overline{\beta} - \overline{\alpha}} \cdot c^{\varepsilon^{-1}_m}_{\mathbb{Q}(\zeta_{mp^n}),r, \eta_{\overline{\beta}}} $$
and define the \textbf{finite layer Coleman map} 
$$\mathrm{Col}_{\mathbb{Q}(\zeta_{mp^n}), r}  :
 \mathrm{H}^1_{/f}( \mathbb{Q}(\zeta_{mp^n}) \otimes \mathbb{Q}_p, V_f(k-r)) 
\to F_\pi[\mathrm{Gal}( \mathbb{Q}(\zeta_{mp^n}) / \mathbb{Q} )] .$$
by
$\mathrm{Col}_{\mathbb{Q}(\zeta_{mp^n}), r}  \left( -  \right)
= P_{\mathbb{Q}(\zeta_{mp^n}), r} \left( - ,  c^{\varepsilon^{-1}_m}_{\mathbb{Q}(\zeta_{mp^n}), r, \eta_{\overline{f}}}  \right) $.
Furthermore, if we restrict the finite layer Coleman map to $T_f$, then it becomes integral, i.e. we have
$$\mathrm{Col}_{\mathbb{Q}(\zeta_{mp^n}), r}  :
 \mathrm{H}^1_{/f}( \mathbb{Q}(\zeta_{mp^n}) \otimes \mathbb{Q}_p, T_f(k-r)) 
\to \mathcal{O}_\pi[\mathrm{Gal}( \mathbb{Q}(\zeta_{mp^n}) / \mathbb{Q} )] $$
due to the choice of $\eta_{\overline{f}}$ in (\ref{eqn:correct-periods}) and Proposition \ref{prop:normalization_local_points_first}.
\begin{rem} 
For the case of elliptic curves, the integrality of finite layer Coleman maps is proved by the explicit computation of the points in formal groups after making the choice of N\'{e}ron periods \cite[Prop. 3.6]{kurihara-invent}.
In our case, we first prove the integrality of finite layer Coleman maps by making the careful choice of $\eta_{\overline{f}}$, but we could not figure out whether the corresponding Kato's zeta element associated to the optimal periods is integral or not in general.
This is where Conjecture \ref{conj:comparison-zeta-elements} comes from.
\end{rem}
Let $\mathfrak{z}_{\mathbb{Q}(\zeta_{mp^n}), r, \delta_{f, \mathrm{cris}}}$ be the modified Kato's Euler system in Corollary \ref{cor:kato-euler-systems-modification}.
Since 
$$\mathrm{Col}_{\mathbb{Q}(\zeta_{mp^n}), r}  \left( \mathrm{loc}^s_p  \mathfrak{z}_{\mathbb{Q}(\zeta_{mp^n}), r, \delta_{f, \mathrm{cris}}}  \right)  = P_{\mathbb{Q}(\zeta_{mp^n}), r} \left( \mathrm{loc}^s_p \mathfrak{z}_{\mathbb{Q}(\zeta_{mp^n}), r, \delta_{f, \mathrm{cris}}} ,  c^{\varepsilon^{-1}_m}_{\mathbb{Q}(\zeta_{mp^n}), r, \eta_{\overline{f}}}  \right) ,$$
the interpolation formula of $\mathrm{Col}_{\mathbb{Q}(\zeta_{mp^n}), r}  \left( \mathrm{loc}^s_p  \mathfrak{z}_{\mathbb{Q}(\zeta_{mp^n}), r, \delta_{f, \mathrm{cris}}}  \right)$ follows from
the interpolation formula of modified Kato's Euler systems (Corollary \ref{cor:kato-euler-systems-modification}) and the interpolation formulas of equivariant local points (Proposition \ref{prop:local-points-properties}).
Straightforward computation shows that it coincides with the the interpolation formula of Mazur--Tate elements (Proposition \ref{prop:mazur-tate-elements-properties}).
Thus, we have
$$\mathrm{Col}_{\mathbb{Q}(\zeta_{mp^n}), r}  \left( \mathrm{loc}^s_p  \mathfrak{z}_{\mathbb{Q}(\zeta_{mp^n}), r, \delta_{f, \mathrm{cris}}}  \right) = \theta_{\mathbb{Q}(\zeta_{mp^n}), r}(\overline{f}) \in   \mathcal{O}_\pi[\mathrm{Gal}( \mathbb{Q}(\zeta_{mp^n}) / \mathbb{Q} )] .$$
It completes the proof of Theorem \ref{thm:main_thm_finite_layer_coleman}.

\subsection{Canonical periods as optimal periods} \label{subsec:canonical_periods}
We prove Theorem \ref{thm:comparison-zeta-elements}.

We first compare two different normalizations (\ref{eqn:delta_f}) and (\ref{eqn:delta_f_cris})
\begin{align*}
\mathrm{per}_f(\omega_f) & = \delta^+_f \otimes C^\pm_{\mathrm{per}} \cdot \Omega^+_{\gamma, \omega_f} + \delta^-_f \otimes C^\mp_{\mathrm{per}} \cdot \Omega^-_{\gamma, \omega_f}  ,\\
\mathrm{per}_f(\omega_{f, \mathrm{cris}}) & = \delta^+_{f, \mathrm{cris}} \otimes C^\pm_{\mathrm{per}} \cdot \Omega^+_{\gamma, \omega_f} + \delta^-_{f, \mathrm{cris}} \otimes C^\mp_{\mathrm{per}} \cdot \Omega^-_{\gamma, \omega_{f} }
\end{align*}
when $(N,p)=1$ and $2 \leq k \leq p-1$.
Following the argument in \cite[$\S$14.22]{kato-euler-systems} based on the theory of Fontaine--Messing \cite{fontaine-messing}, the image of the ($k-1$)-th de Rham cohomology of an integral model of the Kuga--Sato variety in $S(f)$ coincides with $\mathcal{O}_\pi \omega_{f, \mathrm{cris}}$.
Indeed, all the Wach module ingredient, especially the lattice $M_f$, is replaced by the Fontaine--Laffaille module one here. 
The Fontaine--Laffaille functor is compatible with tensor product and dual as mentioned in Remark \ref{rem:fontaine-laffaille-functor}.
Following \cite[0.3. Main Thm.]{shokurov-holomorphic-kuga} and \cite[$\S$0.4]{shokurov-homology-kuga}, the de Rham cohomology of the Kuga--Sato variety is (integrally) isomorphic to the space of weight $k$ cuspforms at least when $k-2 < p$. 
In particular, the homologies of Kuga--Sato varieties and modular curves are directly compared in \cite{shokurov-homology-kuga}.
Therfore, we have $\mathcal{O}_\pi \omega_{f, \mathrm{cris}} = \mathcal{O}_\pi  \omega_f$ and we can choose $\omega_{f, \mathrm{cris}}$ and $\delta^{\pm}_{f, \mathrm{cris}}$ by $\omega_{f}$ and $\delta^{\pm}_{f}$, respectively.

In the ordinary case, the identification between $\omega_{f, \mathrm{cris}}$ and $\omega_{f}$ follows by Hida's theory of $\Lambda$-adic ordinary forms \cite[Chap. 7]{hida-blue} and the choice of the integral lattice in $\mathbf{D}_{\mathrm{dR}}(V_{F_\pi}(f))$ is given in \cite[$\S$17.5]{kato-euler-systems} by using the ordinary filtration. Thus, $\delta^{\pm}_{f, \mathrm{cris}}$ can also be identified with $\delta^{\pm}_{f}$.

We briefly review the canonical periods and show that they are optimal. 
\begin{thm}[Vatsal] \label{thm:mod-p-multi-one}
Let  $\mathfrak{m}$ be  the non-Eisenstein maximal ideal of the Hecke algebra.
If 
\begin{enumerate}
\item $M = N$ is prime to p and $2 \leq k \leq p-1$, or
\item $M = Np$ with $(N,p)=1$ and  $\overline{\rho}$ is $p$-ordinary and $p$-distinguished, 
\end{enumerate}
then
$$\mathrm{H}^1(\Gamma_1(M), \mathrm{Sym}^{k-2}(\mathcal{O}^2_\pi))^{\pm}_{\mathfrak{m}}
\simeq \mathrm{H}^1_c(\Gamma_1(M), \mathrm{Sym}^{k-2}(\mathcal{O}^2_\pi))^{\pm}_{\mathfrak{m}}
\simeq \mathrm{Hom}_{\mathcal{O}_\pi}(\mathbb{T}_{\mathfrak{m}}, \mathcal{O}_\pi)$$
where $\mathbb{T}_{\mathfrak{m}}$ is the Hecke algebra faithfully acting on $\mathrm{H}^1(\Gamma_1(M), \mathrm{Sym}^{k-2}(\mathcal{O}^2_\pi))_{\mathfrak{m}}$.
\end{thm}
\begin{proof}
See \cite[$\S$1.3 and Thm. 1.13]{vatsal-cong}.
\end{proof}
Since $\mathrm{Hom}_{\mathcal{O}_\pi}(\mathbb{T}_{\mathfrak{m}}, \mathcal{O}_\pi) \simeq S_k(\Gamma_1(M), \mathcal{O}_\pi)_{\mathfrak{m}}$ via the duality between cuspforms and the Hecke rings, Theorem \ref{thm:mod-p-multi-one} is equivalent to
$$\theta^{\pm} : \mathrm{H}^1(\Gamma_1(M), \mathrm{Sym}^{k-2}(\mathcal{O}^2_\pi))^{\pm}_{\mathfrak{m}} \simeq S_k(\Gamma_1(M), \mathcal{O}_\pi)_{\mathfrak{m}}$$
where $\theta^{\pm}( \delta^{\pm}_{\overline{f}, \mathrm{can}} ) = \overline{f}$.
Let $\omega_{\overline{f}, \mathrm{can}} \in \mathrm{H}^1(\Gamma_1(M), \mathrm{Sym}^{k-2}(\mathbb{C}^2))_{\mathfrak{m}}$ be the cohomology class  represented by the cocycle $$\left[ \gamma \mapsto 2 \pi \sqrt{-1} \cdot \int^{\gamma z_0}_{z_0} \overline{f}(z)(zX+Y)^{k-2}dz \right]$$
where $z_0$ is any base point on the upper half plane. This comes from the Eichler--Shimura isomorphism.
Then the relation
\begin{equation*}
\omega^{\pm}_{\overline{f}, \mathrm{can}} = \Omega^{\pm}_{\overline{f}, \mathrm{can}} \cdot \delta^{\pm}_{\overline{f}, \mathrm{can}}
\end{equation*}
determines the \textbf{canonical periods} $\Omega^{\pm}_{\overline{f}, \mathrm{can}}$ up to $p$-adic units.
See\cite[(7.13.2)]{kato-euler-systems} for the Eichler--Shimura isomorphism  with the convention of Kato's period integrals.
\begin{prop} \label{prop:canonical-periods-optimal}
Under the same assumptions on Theorem \ref{thm:mod-p-multi-one}, the canonical periods are optimal periods.
\end{prop}
\begin{proof}
See \cite[Rem. 1.12]{vatsal-cong} and \cite[Lem. 4.1]{vatsal-integralperiods-2013} if $k=2$ or $\overline{\rho}$ is  $p$-ordinary and $p$-distinguished. 
Now we assume that $f$ is non-ordinary at $p$ and $2 < k \leq p-1$.
Consider the map
\begin{align*}
\mathrm{H}^1_c(\Gamma_1(N), \mathrm{Sym}^{k-2}(\mathcal{O}^2_\pi)) & \to \mathrm{H}^1_c(\Gamma_1(N), \mathrm{Sym}^{k-2}(\mathcal{O}^2_\pi)) \otimes \mathbb{F}_\pi \\
& \simeq  \mathrm{H}^1_c(\Gamma_1(N), \mathrm{Sym}^{k-2}(\mathbb{F}^2_\pi)) \\
& \hookrightarrow \mathrm{H}^1_c(\Gamma_1(Np), \mathbb{F}_\pi)^{\omega^{k-2}}
\end{align*}
where the second isomorphism can be found in \cite[(5.4)]{edixhoven-integral-weight-two}, 
the last injective map follows from \cite[Theorem 3.4.(a)]{ash-stevens}, 
and $\mathrm{H}^1_c(\Gamma_1(Np), \mathbb{F}_\pi)^{\omega^{k-2}}$ is the submodule of 
$\mathrm{H}^1_c(\Gamma_1(Np), \mathbb{F}_\pi)$ on which the diamond operator at $p$ acts by $\omega^{k-2}$.
Indeed, the last map is induced from the map $\mathrm{Sym}^{k-2}(\mathbb{F}^2_\pi) \to \mathbb{F}_\pi$ defined by $P(X, Y ) \mapsto P(0,1)$ as in \cite[Lem. 4.4]{pw-mt}. Thus, the mod $p$ non-vanishing problem of modular symbols normalized by the canonical periods reduces to the weight two case, and it follows from \cite[Thm. 2.1]{glenn-cuspidal}. See also \cite[Thm. 4.5]{ash-stevens}.
\end{proof}
\begin{proof}[Proof of Theorem \ref{thm:comparison-zeta-elements}]
It immediately follows from Proposition \ref{prop:canonical-periods-optimal} and the Betti--\'{e}tale comparison and the integral Eichler--Shimura isomorphism under the assumptions of Theorem \ref{thm:mod-p-multi-one}.
\end{proof}
Theorem \ref{thm:comparison-zeta-elements} \emph{completes} (\ref{eqn:betti-etale-comparison-extended-integral}) in the following sense
\begin{equation} \label{eqn:betti-etale-comparison-extended-integral-renormalized}
\begin{gathered}
{ \scriptsize
\xymatrix{
 M_f \otimes_S \widehat{S}^{\mathrm{PD}} \otimes_{\mathbb{Z}_p} \widetilde{S} \ar@{^{(}->}[d] &  M_f \ar@{_{(}->}[l]  \\
\left( T_f \otimes \mathbf{A}_{\mathrm{cris}} \right)^{G_{\mathbb{Q}_p}} \otimes \mathbf{A}_{\mathrm{cris}} \ar[d] &  M^{k-r}_f \ar@{^{(}->}[u]  \\
T_f \otimes \mathbf{A}_{\mathrm{cris}} &  \mathcal{O}_\pi \omega_{f, \mathrm{cris}} = \left( M_{\overline{f}} / M^r_{\overline{f}} \right)^* \ar[d]_-{\mathrm{per}^+_f + \mathrm{per}^-_f} \ar@{=}[u] \\
T_f  \ar[u] &   \mathcal{O}_\pi \delta^+_{f, \mathrm{cris}} \otimes \Omega^+_{\delta_f, \omega_f} + \mathcal{O}_\pi \delta^-_{f, \mathrm{cris}} \otimes \Omega^-_{\delta_f, \omega_f}  \\
 \mathcal{O}_\pi \delta^+_{f, \mathrm{cris}}  + \mathcal{O}_\pi \delta^-_{f, \mathrm{cris}} \ar@{=}[u]  & \mathcal{O}_\pi \delta^+_{f, \mathrm{cris}}  + \mathcal{O}_\pi \delta^-_{f, \mathrm{cris}}  \ar[l] \ar[u]  
}
}
\end{gathered}
\end{equation}
with  $\omega_{f, \mathrm{cris}} = \omega_{f}$ and $\delta^\pm_{f, \mathrm{cris}}  = \delta^\pm_{f} $.

\section{Towards the ``weak" main conjecture of Mazur--Tate} \label{sec:mazur-tate}
We prove the following theorem, which covers many cases of the conjecture of Mazur--Tate on Fitting ideals of Selmer groups of modular forms. We keep Assumption \ref{assu:even-weight-central-critical}, i.e. $k$ is even and $r = k/2$, and write $\Lambda_n = \mathcal{O}_\pi[\mathrm{Gal}(\mathbb{Q}(\zeta_{p^n})/\mathbb{Q}) ]$.

\begin{thm}[Mazur--Tate conjecture] \label{thm:main_thm_mazur_tate_conj}
If all the conditions in Assumption \ref{assu:standard} hold and $\overline{\rho} \vert_{G_{\mathbb{Q}_p}}$ is irreducible, then we have
\begin{align} \label{eqn:mazur_tate_conj_with_error}
\begin{split}
& C_f \cdot \left(   \theta_{\mathbb{Q}(\zeta_{p^{n}}), k/2}(\overline{f}) \right)  \cdot \pi_n \left( \mathrm{Fitt}_{\Lambda_{\mathcal{I}w}}  \left( \left( \mathrm{Sel}_0(\mathbb{Q}(\zeta_{p^\infty}), W_{\overline{f}}(k/2))^\vee \right)_{\mathrm{mft}} \right) \right) \\
&  \subseteq \mathrm{Fitt}_{\Lambda_n} \left( \mathrm{Sel}(\mathbb{Q}(\zeta_{p^n}), W_{\overline{f}}(k/2))^\vee \right) 
\end{split}
\end{align}
for $n \geq 1$ where
\begin{itemize}
\item $C_f \in \mathcal{O}_\pi$ is a non-zero constant depending only on $f$ which is unique up to $\mathcal{O}^\times_\pi$, 
\item $\nu_{n-1, n} : \Lambda_{n-1} \to \Lambda_n$ is the trace map, 
\item $\pi_n : \Lambda_{\mathcal{I}w} \to \Lambda_n$ is the natural map, and
\item $M_{\mathrm{mft}}$ is the maximal finite torsion $\Lambda_{\mathcal{I}w}$-submodule of $M$.
\end{itemize}
If we further assume that $\mathrm{Sel}_0( \mathbb{Q}(\zeta_{p^\infty}), W_{\overline{f}}(k/2) )^\vee$ has no non-trivial finite $\Lambda_{\mathcal{I}w}$-torsion submodule,
then 
$$ C_f \cdot  \left( \theta_{\mathbb{Q}(\zeta_{p^{n}}), k/2}(\overline{f}) \right)  \subseteq \mathrm{Fitt}_{\Lambda_n} \left( \mathrm{Sel}(\mathbb{Q}(\zeta_{p^n}), W_{\overline{f}}(k/2))^\vee \right) .$$
If we further assume Conjecture \ref{conj:comparison-zeta-elements} (e.g. $2 \leq k \leq p-1$),
 then 
$$ \left( \theta_{\mathbb{Q}(\zeta_{p^{n}}), k/2}(\overline{f}) \right)  \subseteq \mathrm{Fitt}_{\Lambda_n} \left( \mathrm{Sel}(\mathbb{Q}(\zeta_{p^n}), W_{\overline{f}}(k/2))^\vee \right) .$$
In particular, Conjecture \ref{conj:mazur-tate} holds.
\end{thm}
\begin{rem} \label{rem:main_thm_mazur_tate_conj}
We discuss some aspects of Theorem \ref{thm:main_thm_mazur_tate_conj}.
\begin{enumerate}
\item 
The irreducibility of $\overline{\rho} \vert_{G_{\mathbb{Q}_p}}$ excludes the ordinary case (cf. \cite[$\S$2]{kim-kurihara}).
\item The constant error term $C_f$ in (\ref{eqn:mazur_tate_conj_with_error}) can be taken by any element in $\mathcal{O}_{(\pi)}$ satisfying  $\mathrm{ord}_\pi \left( C_f \right) = \mathrm{max} \left( \mathrm{ord}_\pi \left( C^+_{\mathrm{per}} \right) ,  \mathrm{ord}_\pi \left( C^-_{\mathrm{per}} \right) \right)$. It also appears in \cite{epw2} with no specification. 
\end{enumerate}
\end{rem}
Let $\chi : \mathrm{Gal}(\mathbb{Q}(\zeta_{p^n})/\mathbb{Q}) \to \overline{\mathbb{Q}}^\times_p$ be a character and $\mathcal{O}_\pi[\chi]$ the ring generated by the image of $\chi$ over $\mathcal{O}_\pi$.
The map $\chi$ naturally extends to an algebra homomorphism 
$\Lambda_n \to \mathcal{O}_\pi[\chi]$ defined by
$\sigma \mapsto \chi(\sigma)$
where $\sigma \in \mathrm{Gal}(\mathbb{Q}(\zeta_{p^n})/\mathbb{Q})$ and also denote it by $\chi$.
Then the \textbf{augmentation ideal at $\chi$} is defined by
$I_\chi := \mathrm{ker} \left( \chi: \Lambda_n \to \mathcal{O}_\pi[\chi] \right)$.
For $L \in \Lambda_n$, we say $L$ \textbf{vanishes to infinite order at $\chi$} if $L$ is contained in all powers of $I_\chi$. We say $L$ \textbf{vanishes to order $r$ at $\chi$} if $L \in I^r_\chi \setminus I^{r+1}_\chi$, and write $\mathrm{ord}_\chi \ L =r$. See \cite[(1.5)]{mazur-tate} for detail and  also \cite{ota-thesis, ota-rank-part, burns-kurihara-sano} for the recent development of the ``weak" vanishing conjecture.
\begin{cor}[``Weak" vanishing conjecture] \label{cor:mazur-tate-weak-vanishing}
If all the conditions in Assumption \ref{assu:standard} hold, $\overline{\rho} \vert_{G_{\mathbb{Q}_p}}$ is irreducible, and
 $\mathrm{Sel}_0( \mathbb{Q}(\zeta_{p^\infty}), W_{\overline{f}}(k/2) )^\vee$ has no non-trivial finite $\Lambda_{\mathcal{I}w}$-torsion submodule,
then we have
$$\mathrm{rk}_{\mathcal{O}_\pi[\chi]} \left( \mathrm{Sel}(\mathbb{Q}(\zeta_{p^n}), W_{\overline{f}}(k/2))^\vee \otimes_\chi \mathcal{O}_\pi[\chi] \right) \leq \mathrm{ord}_\chi \ \theta_{\mathbb{Q}(\zeta_{p^{n}}), k/2}(\overline{f})$$
where $\mathrm{Sel}(\mathbb{Q}(\zeta_{p^n}), W_{\overline{f}}(k/2))^\vee \otimes_\chi \mathcal{O}_\pi[\chi]$ is the $\chi$-isotypic quotient of $\mathrm{Sel}(\mathbb{Q}(\zeta_{p^n}), W_{\overline{f}}(k/2))^\vee$.
\end{cor}
\begin{proof}
It follows from Theorem \ref{thm:main_thm_mazur_tate_conj} and \cite[Prop. 3]{mazur-tate}.
\end{proof}
Our goal is to prove
$$  I^{\mathrm{an}}_{\mathbb{Q}(\zeta_{p^n}), k/2}(\overline{f})  = \left( \theta_{\mathbb{Q}(\zeta_{p^n}), k/2}(\overline{f}) \right) \subseteq \mathrm{Fitt}_{ \Lambda_n } \left( \mathrm{Sel}(\mathbb{Q}(\zeta_{p^n}), W_{\overline{f}}(k/2))^\vee \right)$$
under certain assumptions. 
The finite layer analogue of algebraic $p$-adic $L$-functions \`{a} la Perrin-Riou
is defined by
$$  I^{\mathrm{alg}}_{\mathbb{Q}(\zeta_{p^n}), k/2}(\overline{f}) = \pi_n \left( \mathrm{char}_{\Lambda_{\mathcal{I}w}}  \left( \mathrm{Sel}_0(\mathbb{Q}(\zeta_{p^\infty}), W_{\overline{f}}(k/2))^\vee \right)  \right)
\cdot  \left( \mathrm{Col}_{ \mathbb{Q}(\zeta_{p^n}) , k/2}(\mathrm{loc}_p  w_{n}) \right)$$
where
 $\pi_n : \Lambda_{\mathcal{I}w} \to \Lambda_n$ is the natural map,
$w_{n}$ is a $\Lambda_n$-generator of $\mathrm{H}^1_{\mathcal{I}w}(T_f(k/2))_{\Gamma_n}$,
and $\Gamma_n = \mathrm{Gal}(\mathbb{Q}(\zeta_{p^\infty})/\mathbb{Q}(\zeta_{p^n}))$.
\subsection{The failure of control theorem}
Since $\overline{\rho} \vert_{G_{\mathbb{Q}_p}}$  is irreducible, we have
$\mathrm{Sel}(\mathbb{Q}(\zeta_{p^\infty}), W_{\overline{f}}(k/2)) = \mathrm{Sel}_{\mathrm{rel}}(\mathbb{Q}(\zeta_{p^\infty}), W_{\overline{f}}(k/2))$
where the latter is the $p$-relaxed Selmer group \cite{perrin-riou-universal-norms, berger-universal-norms}.
Then we have an exact sequence
\[
{ \small
\xymatrix{
\mathrm{Sel}(\mathbb{Q}(\zeta_{p^n}), W_{\overline{f}}(k/2))
\ar@{^{(}->}[r] &
\mathrm{Sel}(\mathbb{Q}(\zeta_{p^\infty}), W_{\overline{f}}(k/2))^{\Gamma_n}
\ar[r] &
 \mathrm{H}^1_{/f}(\mathbb{Q}(\zeta_{p^n}) \otimes \mathbb{Q}_p, W_{\overline{f}}(k/2)) 
\oplus B_n
}
}
\]
where $B_n$ is a finite group with bounded size independent of $n$ and comes from the prime-to-$p$ local conditions.
Here, the last term measures the failure of the control theorem. See \cite[(2)]{pollack-algebraic} or \cite[Thm. 3.1]{iovita-pollack} for example.
Taking the Pontryagin dual and the local Tate duality, we have
\begin{equation*} 
{ \small
\xymatrix{
 \mathrm{H}^1_f(\mathbb{Q}(\zeta_{p^n})_p, T_{f}(k/2)) 
\oplus B^\vee_n
\ar[r] &
\left( \mathrm{Sel}(\mathbb{Q}(\zeta_{p^\infty}), W_{\overline{f}}(k/2))^\vee \right)_{\Gamma_n}
\ar@{->>}[r] &
\mathrm{Sel}(\mathbb{Q}(\zeta_{p^n}), W_{\overline{f}}(k/2))^\vee .
}
}
\end{equation*}
On the other hand, we have the control theorem for fine Selmer groups as follows
\begin{equation} \label{eqn:sel0-control}
{ \small
\xymatrix{
 B^\vee_n
\ar[r] &
\left( \mathrm{Sel}_0(\mathbb{Q}(\zeta_{p^\infty}), W_{\overline{f}}(k/2))^\vee \right)_{\Gamma_n}
\ar@{->>}[r] &
\mathrm{Sel}_0(\mathbb{Q}(\zeta_{p^n}), W_{\overline{f}}(k/2))^\vee .
}
}
\end{equation}
Using the Poitou--Tate sequence, 
we also have an exact sequence
\begin{equation} \label{eqn:comparison-sel-sel0}
{ \small
\xymatrix{
\dfrac{ \mathrm{H}^1_{/f}(\mathbb{Q}(\zeta_{p^n})  \otimes \mathbb{Q}_p, T_f(k/2))  }{\mathrm{loc}^s_p \mathrm{H}^1(\mathbb{Q}_{\Sigma}/\mathbb{Q}(\zeta_{p^n}), T_f(k/2)) }  
 \ar[r] & 
\mathrm{Sel}(\mathbb{Q}(\zeta_{p^n}), W_{\overline{f}}(k/2))^\vee
\ar@{->>}[r] &
\mathrm{Sel}_0(\mathbb{Q}(\zeta_{p^n}), W_{\overline{f}}(k/2))^\vee .
}
}
\end{equation}
\begin{rem} \label{rem:dual-to-epw}
In \cite{epw2}, the comparison between Selmer groups and $p$-relaxed Selmer groups is used as in \cite{pollack-algebraic, iovita-pollack}. Since our argument is based on the comparison between Selmer groups and fine Selmer groups, 
our argument is a ``dual" variant of that of \cite{epw2} in some sense.
Such a ``dual" relation already appeared in \cite{kurihara-invent} and \cite{pollack-algebraic}.
\end{rem}
\subsection{Fitting ideal toolbox}
We recall well-known properties of Fitting ideals. 
All the modules in this subsection are finitely presented over their base rings.
\begin{lem} \label{lem:fitting_ideals_product}
Let $M_i$ be an $R_i$-module for $i = 1, 2$.
Then
$\mathrm{Fitt}_{R_1}(M_1) \times \mathrm{Fitt}_{R_2}(M_2) = \mathrm{Fitt}_{R_1 \times R_2}(M_1 \times M_2) $.
\end{lem}
\begin{lem} \label{lem:fitting_ideals}
Let $A$, $B$, and $C$ be $R$-modules.
\begin{enumerate}
\item If $A \to B$ is surjective, then $\mathrm{Fitt}_R(A)  \subseteq  \mathrm{Fitt}_R(B)$.
\item If $0 \to A \to B \to C \to 0$ is exact, then
$\mathrm{Fitt}_R(A) \cdot    \mathrm{Fitt}_R(C) \subseteq  \mathrm{Fitt}_R(B) $.
\item If $ \pi : R \to S$ be a surjective ring homomorphism, then $ \mathrm{Fitt}_S(A \otimes_R S) = \pi (  \mathrm{Fitt}_R(A) ) $.
\end{enumerate}
\end{lem}

\begin{lem} \label{lem:fitting_sub}
If $A \subseteq B$ as finitely generated $\Lambda_n$-modules, then 
$\mathrm{Fitt}_{\Lambda_n}(B) \subseteq  \mathrm{Fitt}_{\Lambda_n}(A) $. 
\end{lem}
\begin{proof}
By \cite[Lem. A.10]{kim-kurihara}, the statement over $\Lambda^{(i)}_n$ holds for each $i$.
By Lemma \ref{lem:fitting_ideals_product}, the conclusion holds.
\end{proof}
\subsection{A local computation}
Let $\mathrm{loc}^s_p w_n$ be the image of a $\Lambda_{\mathcal{I}w}$-generator $\mathbf{w} = (w_n)_n$ of $\mathrm{H}^1_{\mathcal{I}w}(T_f(k/2))$ in $\mathrm{H}^1_{/f}(\mathbb{Q} (\zeta_{p^n}) \otimes \mathbb{Q}_p, T_f(k/2) )$.
\begin{prop} \label{prop:local-computation-algebraic-theta}
$$  \left( P_{\mathbb{Q}(\zeta_{p^{n}}), k/2}(\mathrm{loc}^s_p w_{n}, c^{\varepsilon^{-1}}_{\mathbb{Q}(\zeta_{p^{n}}), k/2, \eta_{\overline{f}}}  ) \right)  \subseteq \mathrm{Fitt}_{\Lambda_n} \left( \dfrac{  \mathrm{H}^1_{/f}(\mathbb{Q} (\zeta_{p^n}) \otimes \mathbb{Q}_p, T_f(k/2) ) }{ \left( \mathrm{loc}^s_p w_n \right) } \right) .$$
\end{prop}
\begin{proof}
Write 
$c_n = c^{\varepsilon^{-1}}_{\mathbb{Q}(\zeta_{p^{n}}), k/2, \eta_{\overline{f}}}$ and $\psi_{c_n} (-) = P_{\mathbb{Q}(\zeta_{p^n}), k/2} ( -,  c_n )$ for convenience.
Suppose that $z \in \mathrm{ker} (\psi_{c_n} ) \subseteq \mathrm{H}^1_{/f}(\mathbb{Q}(\zeta_{p^n}) \otimes \mathbb{Q}_p, T_f(k/2))$.
Then 
$\chi \left( P_{\mathbb{Q}(\zeta_{p^n}), k/2} ( z, c_n ) \right)  = 0$
for all characters $\chi$ on $\mathrm{Gal}(\mathbb{Q}(\zeta_{p^n})/\mathbb{Q})$.
By the linear independence of the characters, we have
$$\left[ z^\sigma, c_n \right]_{V_f(k/2), \mathbb{Q}(\zeta_{p^n})} 
=
\left[ \mathrm{exp}^* \ z^\sigma, \mathrm{log} \ c_0 \right]_{\mathbf{D}(V_f(k/2)), \mathbb{Q}(\zeta_{p^n})}
 = 0$$
for every $\sigma \in \mathrm{Gal}(\mathbb{Q}(\zeta_{p^n})/\mathbb{Q})$.
We choose any $\mathcal{O}_\pi$-basis $\eta$ of $M_{\overline{f}}/M^{k/2}_{\overline{f}}$.
Following the argument in \cite[Prop. 3.11]{lei-loeffler-zerbes-canadian} and the formulas in \cite[Lem. 7.2]{kurihara-invent}, we have 
$$a \cdot \eta \otimes e_{- k/2} = \mathrm{log} \ c_0
\pmod{ M^{k/2}_{\overline{f}} }
$$
for some non-zero $a \in \mathcal{O}_\pi$.
Thus, we also have
$$\left[ \mathrm{exp}^* \ z^{\sigma}, \eta \otimes e_{- k/2} \right]_{\mathbf{D}(V_f(k/2)), \mathbb{Q}(\zeta_{p^n})} = 0$$
for every $\sigma \in \mathrm{Gal}(\mathbb{Q}(\zeta_{p^n})/\mathbb{Q})$. In particular, $z = 0$ in $\mathrm{H}^1_{/f}(\mathbb{Q} (\zeta_{p^n}) \otimes \mathbb{Q}_p, T_f(k/2) )$.
Thus, $\psi_{c_n}$ is injective and the induced map
$$\overline{\psi}_{c_n} : \dfrac{ \mathrm{H}^1_{/f}(\mathbb{Q} (\zeta_{p^n}) \otimes \mathbb{Q}_p, T_f(k/2) ) }{ \left( \mathrm{loc}^s_p w_n \right) }  \to \Lambda_n / \left(  P_{\mathbb{Q}(\zeta_{p^n}), k/2} ( \mathrm{loc}^s_p w_n, c_n ) \right) $$
 is also injective.
Applying Lemma \ref{lem:fitting_sub} to $ c^{\varepsilon^{-1}}_{\mathbb{Q}(\zeta_{p^{n}}), k/2, \eta_{\overline{f}}}$, we have
$$\left(
P_{\mathbb{Q}(\zeta_{p^n}), k/2} \left( \mathrm{loc}^s_p w_n,  c^{\varepsilon^{-1}}_{\mathbb{Q}(\zeta_{p^{n}}), k/2, \eta_{\overline{f}}}\right)  
 \right) \subseteq \mathrm{Fitt}_{\Lambda_n} \left( \dfrac{  \mathrm{H}^1_{/f}(\mathbb{Q} (\zeta_{p^n}) \otimes \mathbb{Q}_p, T_f(k/2) ) }{ \left( \mathrm{loc}^s_p w_n \right) } \right) .$$
\end{proof}
\begin{rem}
In \cite[Page 221]{kurihara-invent}, the injectivity of $\psi_{c_n}$ is proved under a rather strong assumption.
\end{rem}
\subsection{Fitting ideal argument}
Now we prove Theorem \ref{thm:main_thm_mazur_tate_conj} step by step.
\subsubsection{Step 1: simple computation in finite layers}
We have
{ \small
\begin{align} \label{eqn:step-one}
\begin{split}
 &  \mathrm{Fitt}_{\Lambda_n} \left( \dfrac{ \mathrm{H}^1_{/f}(\mathbb{Q}(\zeta_{p^n})  \otimes \mathbb{Q}_p, T_{f}(k/2)) }{  \mathrm{loc}^s_p  \mathrm{H}^1(\mathbb{Q}_{\Sigma}/\mathbb{Q}(\zeta_{p^n}), T_{f}(k/2))}  \right)  \cdot  \pi_n \left( \mathrm{Fitt}_{\Lambda}  \left( \mathrm{Sel}_0(\mathbb{Q}(\zeta_{p^\infty}), W_{\overline{f}}(k/2))^\vee \right)  \right) \\
= \ &  \mathrm{Fitt}_{\Lambda_n} \left( \dfrac{ \mathrm{H}^1_{/f}(\mathbb{Q}(\zeta_{p^n})  \otimes \mathbb{Q}_p, T_{f}(k/2)) }{  \mathrm{loc}^s_p  \mathrm{H}^1(\mathbb{Q}_{\Sigma}/\mathbb{Q}(\zeta_{p^n}), T_{f}(k/2))}  \right)  \cdot \mathrm{Fitt}_{\Lambda_n} \left( \left( \mathrm{Sel}_0(\mathbb{Q}(\zeta_{p^\infty}), W_{\overline{f}}(k/2))^\vee \right)_{\Gamma_n} \right)  \\
\subseteq \ & \mathrm{Fitt}_{\Lambda_n} \left( \dfrac{ \mathrm{H}^1_{/f}(\mathbb{Q}(\zeta_{p^n})  \otimes \mathbb{Q}_p, T_{f}(k/2)) }{ \mathrm{loc}^s_p \mathrm{H}^1(\mathbb{Q}_{\Sigma}/\mathbb{Q}(\zeta_{p^n}), T_{f}(k/2))}  \right)  \cdot \mathrm{Fitt}_{\Lambda_n}  \left( \mathrm{Sel}_0(\mathbb{Q}(\zeta_{p^n}), W_{\overline{f}}(k/2))^\vee \right) \\
\subseteq \ & \mathrm{Fitt}_{\Lambda_n} \left( \mathrm{Sel}(\mathbb{Q}(\zeta_{p^n}), W_{\overline{f}}(k/2))^\vee \right) .  
\end{split}
\end{align}
}
where $\pi_n : \Lambda_{\mathcal{I}w} \to \Lambda_n$ is the natural map, and
the first equality follows from Lemma \ref{lem:fitting_ideals}.(3),
the second inclusion follows from  (\ref{eqn:sel0-control}) and Lemma \ref{lem:fitting_ideals}.(1), and
 the last inclusion follows from (\ref{eqn:comparison-sel-sel0}) and Lemma \ref{lem:fitting_ideals}.(2).

\subsubsection{Step 2: from the infinite layer to finite layers}
Write
$\mathrm{Sel}^{\vee}_{0, \mathrm{mft}} =  \left( \mathrm{Sel}_0(\mathbb{Q}(\zeta_{p^\infty}), W_{\overline{f}}(k/2))^\vee \right)_{\mathrm{mft}}$
for convenience.
Consider the exact sequence
\[
\xymatrix{
0 \ar[r] & \mathrm{Sel}^{\vee}_{0, \mathrm{mft}} \ar[r] &
\mathrm{Sel}_0(\mathbb{Q}(\zeta_{p^\infty}), W_{\overline{f}}(k/2))^\vee  \ar[r] &
\mathcal{S} \ar[r] & 0
}
\]
where $\mathcal{S}$ is defined by this sequence. Since $\mathcal{S}$ does not have any non-trivial finite $\Lambda_{\mathcal{I}w}$-submodule, the projective dimension of $\mathcal{S}$ is $\leq 1$.
Then it is known that
$$\mathrm{Fitt}_{\Lambda_{\mathcal{I}w}}  \mathrm{Sel}_0(\mathbb{Q}(\zeta_{p^\infty}), W_{\overline{f}}(k/2))^\vee = \mathrm{Fitt}_{\Lambda_{\mathcal{I}w}} \left( \mathrm{Sel}^{\vee}_{0, \mathrm{mft}} \right)  \cdot \mathrm{Fitt}_{\Lambda_{\mathcal{I}w}}  \mathcal{S} .$$
Since
$\mathrm{char}_{\Lambda_{\mathcal{I}w}} \left( \mathrm{Sel}^{\vee}_{0, \mathrm{mft}} \right)$
is trivial and the projective dimension of $\mathcal{S}$ is $\leq 1$, we also have
$\mathrm{Fitt}_{\Lambda_{\mathcal{I}w}} \mathcal{S}  = \mathrm{char}_{\Lambda_{\mathcal{I}w}} \mathcal{S}  = \mathrm{char}_{\Lambda_{\mathcal{I}w}}  \mathrm{Sel}_0(\mathbb{Q}(\zeta_{p^\infty}), W_{\overline{f}}(k/2))^\vee$.
Thus, we have
{ \small
\begin{align} \label{eqn:step-two}
\begin{split}
 &  \mathrm{Fitt}_{\Lambda_n} \left( \dfrac{ \mathrm{H}^1_{/f}(\mathbb{Q}(\zeta_{p^n})  \otimes \mathbb{Q}_p, T_{f}(k/2)) }{  \mathrm{loc}^s_p  \mathrm{H}^1(\mathbb{Q}_{\Sigma}/\mathbb{Q}(\zeta_{p^n}), T_{f}(k/2))}  \right) \cdot  \pi_n \left( \mathrm{Fitt}_{\Lambda_{\mathcal{I}w}}  \left(  \mathrm{Sel}^{\vee}_{0, \mathrm{mft}}  \right) \right) \\
& \cdot  \pi_n \left( \mathrm{char}_{\Lambda_{\mathcal{I}w}}  \left( \mathrm{Sel}_0(\mathbb{Q}(\zeta_{p^\infty}), W_{\overline{f}}(k/2))^\vee \right)  \right) \\
= \ &  \mathrm{Fitt}_{\Lambda_n} \left( \dfrac{ \mathrm{H}^1_{/f}(\mathbb{Q}(\zeta_{p^n})  \otimes \mathbb{Q}_p, T_{f}(k/2)) }{  \mathrm{loc}^s_p  \mathrm{H}^1(\mathbb{Q}_{\Sigma}/\mathbb{Q}(\zeta_{p^n}), T_{f}(k/2))}  \right)  \cdot  \pi_n \left( \mathrm{Fitt}_{\Lambda_{\mathcal{I}w}}  \left( \mathrm{Sel}_0(\mathbb{Q}(\zeta_{p^\infty}), W_{\overline{f}}(k/2))^\vee \right)  \right) .
\end{split}
\end{align}
}

\subsubsection{Step 3: completion in the algebraic side} \label{subsubsec:err_n}
We have
\begin{align} \label{eqn:step-three}
\begin{split}
   \left( P_{\mathbb{Q}(\zeta_{p^{n}}), k/2}(\mathrm{loc}^s_p w_{n}, c^{\varepsilon^{-1}}_{\mathbb{Q}(\zeta_{p^{n}}), k/2, \eta_{\overline{f}}}  ) \right) 
& \subseteq   \mathrm{Fitt}_{\Lambda_n} \left( \dfrac{ \mathrm{H}^1_{/f}(\mathbb{Q}(\zeta_{p^n}) \otimes \mathbb{Q}_p, T_{f}(k/2)) }{ \mathrm{loc}^s_p \mathrm{H}^1_{\mathcal{I}w}(T_{f}(k/2))_{\Gamma_n} }  \right) \\
& \subseteq   \mathrm{Fitt}_{\Lambda_n} \left( \dfrac{ \mathrm{H}^1_{/f}(\mathbb{Q}(\zeta_{p^n})  \otimes \mathbb{Q}_p, T_{f}(k/2)) }{  \mathrm{loc}^s_p \mathrm{H}^1(\mathbb{Q}_{\Sigma}/\mathbb{Q}(\zeta_{p^n}), T_{f}(k/2))}  \right) 
\end{split}
\end{align}
where 
 the first inclusion follows from Proposition \ref{prop:local-computation-algebraic-theta} and
 the second inclusion follows from Lemma \ref{lem:fitting_ideals}.(1).
The combination of (\ref{eqn:step-one}), (\ref{eqn:step-two}), and (\ref{eqn:step-three}) shows that
\begin{equation} \label{eqn:mazur-tate-algebraic-part}
  I^{\mathrm{alg}}_{\mathbb{Q}(\zeta_{p^n}), k/2}(\overline{f}) \cdot \pi_n \left( \mathrm{Fitt}_{\Lambda_{\mathcal{I}w}}  \left( \mathrm{Sel}^{\vee}_{0, \mathrm{mft}} \right) \right) \subseteq \mathrm{Fitt}_{\Lambda_n} \left( \mathrm{Sel}(\mathbb{Q}(\zeta_{p^n}), W_{\overline{f}}(k/2))^\vee \right) 
\end{equation}
and it means that we have only worked with the algebraic side of Iwasawa theory.
If we further assume that $\mathrm{Sel}_0( \mathbb{Q}(\zeta_{p^\infty}), W_{\overline{f}}(k/2) )^\vee$ has no non-trivial finite $\Lambda_{\mathcal{I}w}$-torsion submodule, then we have
$$  I^{\mathrm{alg}}_{\mathbb{Q}(\zeta_{p^n}), k/2}(\overline{f})  \subseteq \mathrm{Fitt}_{\Lambda_n} \left( \mathrm{Sel}(\mathbb{Q}(\zeta_{p^n}), W_{\overline{f}}(k/2))^\vee \right) .$$
\subsubsection{Step 4: application of Theorem \ref{thm:main_thm_finite_layer_coleman}}
Let $\mathbf{z} = \mathbf{z}_{\mathbb{Q}, \gamma_0}(f, k/2) \in \mathrm{H}^1_{\mathcal{I}w}(T_{f}(k/2))$ be the canonical Kato's Euler system over the cyclotomic tower
and $\mathbf{w}$ a $\Lambda_{\mathcal{I}w}$-generator of $\mathrm{H}^1_{\mathcal{I}w}(T_{f}(k/2))$.
Then we can write
$\mathbf{z} = \alpha \cdot \mathbf{w}$
where $\alpha \in \Lambda_{\mathcal{I}w}$, and indeed, we have 
$\alpha = \mathrm{char}_{\Lambda_{\mathcal{I}w}}( \mathrm{H}^1_{\mathcal{I}w}(T_{f}(k/2)) / \mathbf{z}  ) $.
Denote by $z_n $ the image of $\mathbf{z}$ in $\mathrm{H}^1_{\mathcal{I}w}(T_f(k/2))_{\Gamma_n}$ and by $w_{n}$ the image of $\mathbf{w}$ in $\mathrm{H}^1_{\mathcal{I}w}(T_f(k/2))_{\Gamma_n}$.
Theorem \ref{thm:main_thm_mazur_tate_conj} follows from (\ref{eqn:mazur-tate-algebraic-part}) and the following computation
\begin{align*}
  &  C_f \cdot  \left( \theta_{\mathbb{Q}(\zeta_{p^{n}}), k/2}( \overline{f}  ) \right)  \\
\subseteq \ &   \left( \mathrm{Col}_{\mathbb{Q}(\zeta_{p^{n}}), k/2}(\mathrm{loc}^s_p z_{n}  ) \right)  \\
= \ &   \left( P_{\mathbb{Q}(\zeta_{p^{n}}), k/2}(\mathrm{loc}^s_p z_{n}, c^{\varepsilon^{-1}}_{\mathbb{Q}(\zeta_{p^{n}}), k/2, \eta_{\overline{f}}}   ) \right) \\
= \ &   \left( P_{\mathbb{Q}(\zeta_{p^n}), k/2}( \mathrm{loc}^s_p \left( \pi_n \left( \alpha \right) \cdot w_n \right), c^{\varepsilon^{-1}}_{\mathbb{Q}(\zeta_{p^{n}}), k/2, \eta_{\overline{f}}} ) \right)   \\
= \ &   \left( P_{\mathbb{Q}(\zeta_{p^n}), k/2}( \mathrm{loc}^s_p w_n,  c^{\varepsilon^{-1}}_{\mathbb{Q}(\zeta_{p^{n}}), k/2, \eta_{\overline{f}}} ) \right)  \cdot \pi_n \left( \alpha \right) \\
= \ &  \left( P_{\mathbb{Q}(\zeta_{p^{n}}), k/2}(\mathrm{loc}^s_p w_{n}, c^{\varepsilon^{-1}}_{\mathbb{Q}(\zeta_{p^{n}}), k/2, \eta_{\overline{f}}}   ) \right)  \cdot \pi_n \left( \mathrm{char}_{\Lambda_{\mathcal{I}w}}( \mathrm{H}^1_{\mathcal{I}w}(T_{f}(k/2)) / \mathbf{z}  ) \right) \\
\subseteq \ &   \left( P_{\mathbb{Q}(\zeta_{p^{n}}), k/2}(\mathrm{loc}^s_p w_{n}, c^{\varepsilon^{-1}}_{\mathbb{Q}(\zeta_{p^{n}}), k/2, \eta_{\overline{f}}}   ) \right)  \cdot  \pi_n \left( \mathrm{char}_{\Lambda_{\mathcal{I}w}}  \left( \mathrm{Sel}_0(\mathbb{Q}(\zeta_{p^\infty}), W_{\overline{f}}(k/2))^\vee \right)  \right) \\
= \ & I^{\mathrm{alg}}_{\mathbb{Q}(\zeta_{p^n}), k/2}(\overline{f})
\end{align*}
where $C_f$ is the constant error term in Theorem \ref{thm:main_thm_mazur_tate_conj},
the first inclusion follows from Theorem \ref{thm:main_thm_finite_layer_coleman},
and the last inclusion follows from Theorem \ref{thm:kato-divisibility} and the pseudo-isomorphism $\mathrm{Sel}_0(\mathbb{Q}(\zeta_{p^\infty}), W_{\overline{f}}(k/2))^\vee \sim \mathrm{H}^2_{\mathcal{I}w}(T_{f}(k/2))$ \cite[Prop. 7.1]{kobayashi-thesis}.

\section{Towards the Iwasawa main conjecture for modular forms} \label{sec:IMC}
We prove the following statement, which improves the main result of \cite{kks}.
\begin{thm}[Iwasawa main conjecture]  \label{thm:main_thm_main_conj}
Let $p \geq 5$ be a prime. Keep all the statements in Assumption \ref{assu:standard}.
If Conjecture \ref{conj:comparison-zeta-elements} holds (e.g. $2 \leq k \leq p-1$ or $\overline{\rho}$ is  $p$-ordinary and $p$-distinguished), then the following statements are equivalent.
\begin{enumerate}
\item $ \left \lbrace
    \begin{array}{ll}
     \widetilde{\delta}_{\mathbb{Q}(\zeta_m)}(\overline{f},r) \neq 0  & \textrm{ when } i=0 , \\ 
     \widetilde{\delta}_{\mathbb{Q}(\zeta_m)}(\overline{f},i,r) \neq 0  & \textrm{ when } i \neq 0
    \end{array}
    \right.
$
for some square-free product $m$ of Kolyvagin primes,
\item
the Kolyvagin system associated to $\mathfrak{z}_{\mathbb{Q}(\zeta_{m}), \delta_{f, \mathrm{cris}}}(f,i, k-r)$ for $T_{f, i}(k-r)$ does not vanish modulo $\pi$ and
\item the main conjecture for $T_{f,i}(k-r)$ holds and the local Tamagawa numbers are not divisible by $\pi$.
\end{enumerate}
In particular, those Kurihara numbers do not vanish for every $i = 0, \cdots, p-2$ if and only if the main conjecture for $T_{f}(k-r)$ over the full cyclotomic extension $\mathbb{Q}(\zeta_{p^\infty})$ of $\mathbb{Q}$ holds  and the local Tamagawa numbers are not divisible by $\pi$.
\end{thm}
If we restrict ourselves to the case of elliptic curves, we obtain the following statement.
\begin{cor}
Let $E$ be an elliptic curve over $\mathbb{Q}$ and $p \geq 5$ a non-anomalous good reduction prime such that 
$E[p]$ is a surjective Galois representation. Then the following statements are equivalent.
\begin{enumerate}
\item $\widetilde{\delta}_{m}(E) = \widetilde{\delta}_{\mathbb{Q}(\zeta_m)}(E, 1) \neq 0$ for some square-free product $m$ of Kolyvagin primes.
\item Kato's Kolyvagin system for the $p$-adic Tate module of $E$ does not vanish modulo $p$.
\item The Iwasawa main conjecture for $E$ over $\mathbb{Q}_{\infty}$ holds and all the local Tamagawa factors are not divisible by $p$.
\end{enumerate}
\end{cor}
As a consequence, if the Iwasawa main conjecture holds, then we are able to describe the structures of $p$-strict Selmer groups and usual Selmer groups following \cite[Thm. 5.2.12]{mazur-rubin-book} and \cite{kim-structure-selmer}.

Theorem \ref{thm:main_thm_main_conj} also has the following immediate consequence on the analytic rank zero case of the Bloch--Kato conjecture for modular forms.
\begin{cor}[Bloch--Kato conjecture] \label{cor:bloch-kato}
Keep all the statements in Assumption \ref{assu:standard}.
If  
Conjecture \ref{conj:comparison-zeta-elements} holds (e.g. $2 \leq k \leq p-1$ or $\overline{\rho}$ is  $p$-ordinary and $p$-distinguished), 
and
$$ \widetilde{\delta}_{\mathbb{Q}(\zeta_m)}(\overline{f},r) \neq 0$$
for some square-free product $m$ of Kolyvagin primes,
then the following statements hold.
\begin{enumerate}
\item Suppose that $L(\overline{f}, k/2) \neq 0$ if $k$ is even.
\begin{enumerate}
\item 
A generalized version of Kato's Birch and Swinnerton-Dyer conjecture \cite[Conj. 2.3.3]{kato-lecture-1} holds, i.e.
$\mathrm{H}^2(\mathbb{Z}[1/p],j_*T_f(k-r))$ is finite and
$$\mathrm{ord}_\pi \left( [\mathrm{H}^1(\mathbb{Z}[1/p],j_*T_f(k-r)):  z_{\mathbb{Q} , \delta_{f, \mathrm{cris}}}(f, k-r) ] \right) = \mathrm{length}_{\mathcal{O}_\pi}\left( \mathrm{H}^2(\mathbb{Z}[1/p],j_*T_f(k-r)) \right)$$
where $\mathrm{H}^i(\mathbb{Z}[1/p],j_*T_f(k-r)) = \mathrm{H}^i_{\mathrm{\acute{e}t}}(\mathrm{Spec} (\mathbb{Z}[1/p]), j_*T_f(k-r))$.
\item 
If we further assume $2 \leq k \leq p-1$, then the $\pi$-part of 
the corresponding ``rank zero" Bloch--Kato's Tamagawa number conjecture \cite[Conj. 5.15]{bloch-kato} holds, i.e. 
$ \mathrm{Sel}(\mathbb{Q}, W_{\overline{f}}(r))$ is finite and
\begin{equation} \label{eqn:bloch-kato-conjecture-higher-weight}
\mathrm{ord}_\pi \left(  \dfrac{L( \overline{f} ,  r)}{ (-2\pi \sqrt{-1})^{ r - 1 } \cdot \Omega^{\pm}_{\overline{f}} } \right) =  \mathrm{length}_{\mathcal{O}_\pi} \left( \mathrm{Sel}(\mathbb{Q}, W_{\overline{f}}(r)) \right)
\end{equation}
where $\pm $ is the sign of $(-1)^{r -1}$.
\end{enumerate}
\item Suppose that $k$ is even.
If $L(\overline{f}, k/2) = 0$, then $\mathrm{cork}_{\mathcal{O}_\pi} \mathrm{Sel}(\mathbb{Q}, W_{\overline{f}}(k/2)) \geq 1$.
\end{enumerate}
\end{cor}
\begin{proof}
\begin{enumerate}
\item
The first statement immediately follows from Theorem \ref{thm:main_thm_main_conj}.
The second statement follows from the first one and \cite[Prop. 14.21]{kato-euler-systems}, which uses Fontaine--Laffaille modules; thus, the weight should lie in $2 \leq k \leq p-1$ in the second statement. See \cite[Thm. 4.3.(iii)]{bloch-kato} for details. Note that $L(\overline{f}, r)$ does not vanish unless $k$ is even and $r = k/2$ as in \cite[Thm. 13.5.(1)]{kato-euler-systems}.
\item It follows from Theorem \ref{thm:main_thm_main_conj} and the control theorem for fine Selmer groups. Although the fine Selmer group over $\mathbb{Q}_\infty$ may have a non-trivial finite Iwasawa submodule, it does not affect the corank of the Selmer group over $\mathbb{Q}$.
\end{enumerate}
\end{proof}
\begin{rem}
Corollary \ref{cor:bloch-kato}.(1).(a) refines \cite[Thm. 14.5]{kato-euler-systems} and
Corollary \ref{cor:bloch-kato}.(1).(b) refines \cite[Thm. 14.2]{kato-euler-systems} by giving the exact bounds.
Regarding an application of Corollary \ref{cor:bloch-kato}.(2), see \cite[Rem. (b) after Thm. 3.35]{skinner-urban} and \cite{skinner-urban-icm}.
See also Proposition \ref{prop:kato-main-conj-to-bsd-formula-rank-zero}.
\end{rem}
We assume Conjecture \ref{conj:comparison-zeta-elements} in this section.
\subsection{Kolyvagin derivatives}
Let $m$ be a square-free product of Kolyvagin primes and consider
the following Mazur--Tate elements with Teichm\"{u}ller twists
$$
\omega^i \left( \theta_{\mathbb{Q}(\zeta_{mp}), r}(\overline{f}) \right)  =  \sum_{a \in (\mathbb{Z}/mp\mathbb{Z})^\times }   \lambda^{\pm, \mathrm{opt}}(\overline{f}, z^{r-1}, a, mp)  \cdot \omega^{i}( a_{(p)}) \cdot \sigma_{a_{(m)}} 
$$
in $\mathcal{O}_\pi[\mathrm{Gal}(\mathbb{Q}(\zeta_m)/\mathbb{Q})]$
 for $i = 0, \cdots , p-2$ where $a_{(p)} = a \pmod{p}$ and $a_{(m)} = a \pmod{m}$, and the sign of modular symbols coincides with that of $(-1)^{r-1} \cdot \omega^i(-1)$.

Recall that we fix a primitive root $\eta_\ell$  for each Kolyvagin prime $\ell$, and denote by
$\sigma_{\eta_\ell} \in \mathrm{Gal}(\mathbb{Q}(\zeta_\ell)/\mathbb{Q})$ the corresponding generator.
For a square-free product $m$ of Kolyvagin primes, the \textbf{Kolyvagin derivative operator at $\mathbb{Q}(\zeta_m)$} is defined by 
$$D_{\mathbb{Q}(\zeta_m)} := \prod_{\ell \vert m} D_{\mathbb{Q}(\zeta_\ell)} = \prod_{\ell \vert m} \left( \sum_{i=1}^{\ell-2} i \sigma^i_{\eta_\ell} \right) \in \mathbb{Z}[\mathrm{Gal}(\mathbb{Q}(\zeta_m)/\mathbb{Q})]$$
 where $\mathrm{Gal}(\mathbb{Q}(\zeta_\ell)/\mathbb{Q})$ is regarded as a subgroup of $\mathrm{Gal}(\mathbb{Q}(\zeta_m)/\mathbb{Q})$. 

\subsection{An application of the theory of Kolyvagin systems}
We first decompose the (modified) Kato's Euler system
\begin{align*}
 \mathfrak{z}_{\mathbb{Q}(\zeta_{mp}), \delta_{f, \mathrm{cris}}}(f, k-r) & = \left( \mathfrak{z}_{\mathbb{Q}(\zeta_{m}), \delta_{f, \mathrm{cris}}}(f, i, k-r) \right)_{i=0, \cdots, p-2} \\
 & \in \bigoplus_{i=0}^{p-2} \mathrm{H}^1(\mathbb{Q}(\zeta_{m}), T_{f,i}(k-r))
\end{align*}
and each component $\mathfrak{z}_{\mathbb{Q}(\zeta_{m}), \delta_{f, \mathrm{cris}}}(f, i, k-r)$ is also an Euler system for $T_{f, i}(k-r)$ \cite[$\S$2.4]{rubin-book}.
Consider the map
\begin{equation*} \label{eqn:kolyvagin-derivative-diagram}
\begin{split}
{ \scriptsize
\xymatrix{
\mathrm{H}^1(\mathbb{Q}(\zeta_{m}), T_{f, i}(k-r)) \ar[d]_-{ D_{\mathbb{Q}(\zeta_{m})} \circ \mathrm{Col}^{(i)}_{\mathbb{Q}(\zeta_{m}), r} \circ \mathrm{loc}^s_p } & \mathfrak{z}_{\mathbb{Q}(\zeta_{m}), \delta_{f, \mathrm{cris}}}(f,i, k-r) \ar@{|->}[d] \\
\mathcal{O}_\pi[\mathrm{Gal}(\mathbb{Q}(\zeta_{m})/\mathbb{Q})] \ar[d]_-{\pmod{\pi}} &  D_{\mathbb{Q}(\zeta_{m})} \cdot \omega^i \left( \theta_{\mathbb{Q}(\zeta_{mp}), r} (\overline{f}) \right) \ar@{|->}[d] \\
  \mathbb{F}_\pi[\mathrm{Gal}(\mathbb{Q}(\zeta_{m})/\mathbb{Q})] & \overline{D_{\mathbb{Q}(\zeta_{m})} \cdot \omega^i \left( \theta_{\mathbb{Q}(\zeta_{mp}), r} (\overline{f}) \right)}
}
}
\end{split}
\end{equation*}
where 
$\mathrm{Col}_{\mathbb{Q}(\zeta_{mp}), r} = \bigoplus_i \mathrm{Col}^{(i)}_{\mathbb{Q}(\zeta_{m}), r}$
is the decomposition of the finite layer Coleman map, which is compatible with the decomposition of Euler systems.

Let $\overline{D_{\mathbb{Q}(\zeta_{m})} \mathfrak{z}_{\mathbb{Q}(\zeta_{m}), \delta_{f, \mathrm{cris}}}(f, i, k-r)} \in \mathrm{H}^1(\mathbb{Q}(\zeta_{m}), T_{f, i}(k-r)/\pi T_{f, i}(k-r))$ be the mod $\pi$ reduction of the Kolyvagin derivative of $\mathfrak{z}_{\mathbb{Q}(\zeta_{m}), \delta_{f, \mathrm{cris}}}(f,i, k-r)$ and denote its image in $\mathrm{H}^1(\mathbb{Q}, T_{f, i}(k-r)/\pi T_{f, i}(k-r))$ by $\kappa^{(1)}_m$.
Note that  $\overline{D_{\mathbb{Q}(\zeta_{m})} \mathfrak{z}_{\mathbb{Q}(\zeta_{m}), \delta_{f, \mathrm{cris}}}(f, i, k-r)} \neq 0$ if and only if $\kappa^{(1)}_m \neq 0$.

Since the finite layer Coleman map $\mathrm{Col}^{(i)}_{\mathbb{Q}(\zeta_{m}), r}$ is $\mathrm{Gal}(\mathbb{Q}(\zeta_m)/\mathbb{Q})$-equivariant, 
the non-vanishing of $\overline{D_{\mathbb{Q}(\zeta_{m})} \cdot \omega^i \left( \theta_{\mathbb{Q}(\zeta_{mp}), r} (\overline{f}) \right)}$ implies the non-triviality of $\overline{D_{\mathbb{Q}(\zeta_{m})} \mathfrak{z}_{\mathbb{Q}(\zeta_{m}), \delta_{f, \mathrm{cris}}}(f, i, k-r)}$.

Let 
 $\ks^{(1)} = \left\lbrace \kappa^{(1)}_m \in \mathrm{H}^1(\mathbb{Q}, T_{f, i}(k-r)/\pi T_{f, i}(k-r)) \right\rbrace_m$
be (mod $\pi$) Kato's Kolyvagin system for $T_{f,i}(k-r) / \pi T_{f,i}(k-r)$ following \cite[Thm. 3.2.4]{mazur-rubin-book}.

By applying the same argument of the proof of Proposition \ref{prop:local-computation-algebraic-theta} to this setting, we obtain the injectivity of  $\mathrm{Col}^{(i)}_{\mathbb{Q}(\zeta_{m}), r}$ on $\mathrm{H}^1_{/f}(\mathbb{Q}(\zeta_{m}) \otimes \mathbb{Q}_p, T_{f, i}(k-r))$.
We suppose that $\overline{D_{\mathbb{Q}(\zeta_{m})} \cdot \omega^i \left( \theta_{\mathbb{Q}(\zeta_{mp}), r} (\overline{f}) \right)}$ vanishes identically.
Then the injectivity of  $\mathrm{Col}^{(i)}_{\mathbb{Q}(\zeta_{m}), r}$ implies that $\kappa^{(1)}_m$ lies in the mod $\pi$ Bloch--Kato Selmer groups \cite[Prop. 6.2.2]{mazur-rubin-book} for every $m \in \mathcal{N}$.
Then the core rank of this Kolyvagin system is zero, so $\ks^{(1)}$ becomes trivial by \cite[Thm. 4.2.2]{mazur-rubin-book}.
To sum up, the following statements are equivalent.
\begin{itemize}
\item the non-vanishing of $\overline{D_{\mathbb{Q}(\zeta_{m})} \cdot \omega^i \left( \theta_{\mathbb{Q}(\zeta_{mp}), r} (\overline{f}) \right)}$ for some  $m \in \mathcal{N}$.
\item the non-triviality of $\ks^{(1)}$.
\end{itemize}
We recall the main result of \cite{kazim-Lambda-adic}, which plays an essential role to obtain the equivalence.
\begin{thm}[B\"{u}y\"{u}kboduk] \label{thm:rigidity}
Let $\KS(T_{f}(k-r) \otimes \Lambda^{(i)}_{\mathrm{Iw}}  )$ be the module of the $\Lambda^{(i)}_{\mathrm{Iw}}$-adic Kolyvagin system for $T_{f}(k-r)$ and
$\KS(T_{f,i}(k-r))$ the module of the Kolyvagin system for $T_{f,i}(k-r)$.
If we assume that all the statements in Assumption \ref{assu:standard} and the local Tamagawa numbers are not divisible by $\pi$,
then 
\begin{enumerate}
\item $\KS(T_{f}(k-r) \otimes \Lambda^{(i)}_{\mathrm{Iw}}  )$ is free of rank one over $\Lambda^{(i)}_{\mathrm{Iw}}$, and
\item the natural map $\KS(T_{f}(k-r) \otimes \Lambda^{(i)}_{\mathrm{Iw}}  ) \to \KS(T_{f,i}(k-r)  )$ is surjective.
\end{enumerate}
\end{thm}
\begin{proof}
See \cite[Thm. 3.23]{kazim-Lambda-adic}. 
\end{proof}
Denote by $\ks^{\infty}$ the $\Lambda^{(i)}_{\mathrm{Iw}}$-adic Kato's Kolyvagin system associated to 
the canonical Kato's Euler system for $T_{f, i}(k-r)$ with $\delta_{f, \mathrm{cris}}$ obtained by applying the $\Lambda^{(i)}_{\mathrm{Iw}}$-adic version of the Euler-to-Kolyvagin system map \cite[Thm. 5.3.3]{mazur-rubin-book}.
Since $\KS(T_{f,i}(k-r)  )$ is free of rank one over $\mathcal{O}_\pi$ thanks to \cite[Thm. 5.2.10]{mazur-rubin-book}, the following corollary is immediate.
\begin{cor}
Under the same setting of Theorem \ref{thm:rigidity},
$\ks^{\infty}$ is a $\Lambda^{(i)}_{\mathrm{Iw}}$-generator of $\KS(T_{f}(k-r) \otimes \Lambda^{(i)}_{\mathrm{Iw}}  )$ if and only if $\ks^{(1)}$ does not vanish.
\end{cor}
Any $\Lambda^{(i)}_{\mathrm{Iw}}$-generator $\ks^{\infty}_{\mathrm{prim}}$ of $\KS(T_{f}(k-r) \otimes \Lambda^{(i)}_{\mathrm{Iw}}  )$ satisfies the Iwasawa main conjecture (with $\ks^{\infty}_{\mathrm{prim}}$) \cite[Def. 5.3.9 and Thm. 5.3.10]{mazur-rubin-book}.
Since $\KS(T_{f}(k-r) \otimes \Lambda^{(i)}_{\mathrm{Iw}}  )$ is free of rank one over $\Lambda^{(i)}_{\mathrm{Iw}}$,
$\ks^{\infty}_{\mathrm{prim}}  = f \cdot \ks^{\infty}$ with $ f \in \Lambda^{(i)}_{\mathrm{Iw}}$.
Under our working assumptions, the following statements are equivalent.
\begin{itemize}
\item The Iwasawa main conjecture for $T_{f}(k-r)$ holds and the local Tamagawa numbers are not divisible by $\pi$.
\item $f$ is a unit in $\Lambda^{(i)}_{\mathrm{Iw}}$.
\item $\ks^{\infty}$ is a $\Lambda^{(i)}_{\mathrm{Iw}}$-generator of $\KS(T_{f}(k-r) \otimes \Lambda^{(i)}_{\mathrm{Iw}} )$.
\item $\ks^{(1)}$ does not vanish.
\end{itemize}
Theorem \ref{thm:main_thm_main_conj} now follows from the equivalence between
\begin{itemize}
\item the non-vanishing of $\overline{D_{\mathbb{Q}(\zeta_{m})} \cdot \omega^i \left( \theta_{\mathbb{Q}(\zeta_{mp}), r} (\overline{f}) \right)}  $ in $\mathbb{F}_\pi[\mathrm{Gal}(\mathbb{Q}(\zeta_{m})/\mathbb{Q})]$  for some $m \in \mathcal{N}$ and
\item the non-vanishing of $\widetilde{\delta}_{\mathbb{Q}(\zeta_m)}(\overline{f},i,r)$ in $\mathbb{F}_\pi$ for some $m \in \mathcal{N}$.
\end{itemize}
Its proof is given in the next subsection.
\subsection{Completion of the proof}
Write $m = \prod^{s}_{j=1} \ell_j$.
\begin{lem}[Kurihara] \label{lem:kurihara-expansion}
Let $\mathfrak{a}_m = \left( \pi , \left( \sigma_{\eta_{\ell_1}} - 1 \right)^2, \cdots,  \left( \sigma_{\eta_{\ell_s}} - 1 \right)^2 \right) \subseteq \mathcal{O}_\pi[\mathrm{Gal}(\mathbb{Q}(\zeta_m)/\mathbb{Q})]$. Then we have
\begin{align*}
\omega^i \left( \theta_{\mathbb{Q}(\zeta_{mp}), r}(\overline{f}) \right) & \equiv a \cdot \prod^s_{j=1} \left( \sigma_{\eta_{\ell_j}} - 1 \right) \pmod{\mathfrak{a}_m}
\end{align*}
where $a \in \mathbb{F}_\pi$ in  $\mathcal{O}_\pi[\mathrm{Gal}(\mathbb{Q}(\zeta_m)/\mathbb{Q})] / \mathfrak{a}_m$.
\end{lem}
\begin{proof}
See \cite[Lem. 4.4]{kurihara-documenta}.
\end{proof}
We have
\begin{align*}
\sigma_{a_{(m)}} & = \prod_{j=1}^s \left( \sigma_{\eta_i} \right)^{\mathrm{log}_{\mathbb{F}_{\ell_j}}(a_{(m)})}   = \prod_{j=1}^s \left( 1+ \left(  \sigma_{\eta_j} -1 \right) \right)^{\mathrm{log}_{\mathbb{F}_{\ell_j}}(a_{(m)})} \\
& \equiv \prod_{j=1}^s \left( 1 + \overline{\mathrm{log}_{\mathbb{F}_{\ell_j}}(a_{(m)})} \cdot \left( \sigma_{\eta_j} -1 \right)  \right)  \pmod{\mathfrak{a}_m} 
\end{align*}
where $\mathrm{Gal}(\mathbb{Q}(\zeta_m)/\mathbb{Q}) \simeq \prod^{s}_{j =1} \mathrm{Gal}(\mathbb{Q}(\zeta_{\ell_j})/\mathbb{Q})$ is used in the first equality and  $\overline{\mathrm{log}_{\mathbb{F}_{\ell_j}}(a_{(m)})}$ is the image of $\mathrm{log}_{\mathbb{F}_{\ell_j}}(a_{(m)})$ in $\mathbb{F}_p$.
We expand Mazur--Tate elements (``Taylor expansion at  $\prod_{j=1}^s \left( \sigma_{\eta_j} -1 \right)$ modulo $\mathfrak{a}_m$")
\begin{align*}
\begin{split}
& \omega^i \left( \theta_{\mathbb{Q}(\zeta_{mp}), r} (\overline{f}) \right) \\
 = \ & \sum_{a \in (\mathbb{Z}/mp\mathbb{Z})^\times }   \lambda^{\pm, \mathrm{opt}}(\overline{f}, z^{r-1}, a, mp) \cdot \omega^i (a_{(p)})\cdot \sigma_{a_{(m)}}  \\
 \equiv \ & \sum_{a \in (\mathbb{Z}/mp\mathbb{Z})^\times } \overline{  \lambda^{\pm, \mathrm{opt}}(\overline{f}, z^{r-1}, a, mp) } \cdot \overline{ \omega^i (a_{(p)}) } \cdot \prod_{j=1}^s \left( 1 + \overline{\mathrm{log}_{\mathbb{F}_{\ell_j}}(a_{(m)})} \cdot \left( \sigma_{\eta_j} -1 \right) \right) \pmod{\mathfrak{a}_m} \\
  \equiv \ & \sum_{a \in (\mathbb{Z}/mp\mathbb{Z})^\times } \overline{  \lambda^{\pm, \mathrm{opt}}(\overline{f}, z^{r-1}, a, mp) } \cdot \overline{ \omega^i (a_{(p)}) } \cdot \prod_{j=1}^s  \overline{\mathrm{log}_{\mathbb{F}_{\ell_j}}(a_{(m)})} \cdot \left( \sigma_{\eta_j} -1 \right) \pmod{\mathfrak{a}_m} 
\end{split}
\end{align*}
in $\mathcal{O}_\pi[\mathrm{Gal}(\mathbb{Q}(\zeta_m)/\mathbb{Q})]/\mathfrak{a}_m$ where the last equality follows from Lemma \ref{lem:kurihara-expansion}.
We compute the Kolyvagin derivatives of Mazur--Tate elements modulo $\mathfrak{a}_m$.
\begin{align*}
\begin{split}
& D_{\mathbb{Q}(\zeta_m)} \omega^i \left( \theta_{\mathbb{Q}(\zeta_{mp}), r} (\overline{f}) \right) \\
\equiv \ & D_{\mathbb{Q}(\zeta_m)} \left( \sum_{a \in (\mathbb{Z}/mp\mathbb{Z})^\times } \overline{ \lambda^{\pm, \mathrm{opt}}(\overline{f}, z^{r-1}, a, mp) } \cdot \overline{ \omega^i(a_{(p)}) } \cdot \prod_{j=1}^s  \overline{\mathrm{log}_{\mathbb{F}_{\ell_j}}(a_{(m)})} \cdot \left( \sigma_{\eta_j} -1 \right) \right)  \pmod{\mathfrak{a}_m}\\
= \ & \sum_{a \in (\mathbb{Z}/mp\mathbb{Z})^\times } \overline{ \lambda^{\pm, \mathrm{opt}}(\overline{f}, z^{r-1}, a, mp) } \cdot \overline{ \omega^i(a_{(p)}) } \cdot \prod_{j=1}^s D_{\mathbb{Q}(\zeta_{\ell_j})} \overline{\mathrm{log}_{\mathbb{F}_{\ell_j}}(a_{(m)})} \cdot \left( \sigma_{\eta_j} -1 \right) \\
= \ & \sum_{a \in (\mathbb{Z}/mp\mathbb{Z})^\times } \overline{ \lambda^{\pm, \mathrm{opt}}(\overline{f}, z^{r-1}, a, mp) } \cdot \overline{ \omega^i(a_{(p)}) } \cdot \prod_{j=1}^s \overline{\mathrm{log}_{\mathbb{F}_{\ell_j}}(a_{(m)})} \cdot \left( \ell_j - 1 - \mathrm{Tr}_{\mathbb{Q}(\zeta_{\ell_j})} \right)   \\
= \ & \sum_{a \in (\mathbb{Z}/mp\mathbb{Z})^\times } \overline{ \lambda^{\pm, \mathrm{opt}}(\overline{f}, z^{r-1}, a, mp) } \cdot \overline{ \omega^i(a_{(p)}) } \cdot \prod_{j=1}^s \overline{\mathrm{log}_{\mathbb{F}_{\ell_j}}(a_{(m)})} \cdot \left(  - \mathrm{Tr}_{\mathbb{Q}(\zeta_{\ell_j})} \right) 
\end{split}
\end{align*}
where $\mathrm{Tr}_{\mathbb{Q}(\zeta_\ell)/\mathbb{Q}} = {\displaystyle \sum_{i=1}^{\ell-1} } \sigma^i_{\eta_\ell}$, 
the third equality follows from the well-known identity
$(\sigma_{\eta_\ell} - 1)D_{\mathbb{Q}(\zeta_\ell)} = \ell - 1 - \mathrm{Tr}_{\mathbb{Q}(\zeta_\ell)/\mathbb{Q}}$, and the fourth equality follows from that $m \in \mathcal{N}$.
By considering the action of $D_{\mathbb{Q}(\zeta_m)} \omega^i \left( \theta_{\mathbb{Q}(\zeta_{mp}), r} (\overline{f}) \right)$ on $\zeta^{-1}_m$,
we have equality
\begin{align*}
& D_{\mathbb{Q}(\zeta_m)} \omega^i \left( \theta_{\mathbb{Q}(\zeta_{mp}), r} (\overline{f}) \right) \cdot \zeta^{-1}_m \\
 & \equiv \sum_{a \in (\mathbb{Z}/mp\mathbb{Z})^\times } \overline{ \lambda^{\pm, \mathrm{opt}}(\overline{f}, z^{r-1}, a, mp) } \cdot \overline{ \omega^i(a_{(p)}) } \cdot \prod_{j=1}^s \overline{\mathrm{log}_{\mathbb{F}_{\ell_j}}(a_{(m)})} \pmod{\mathfrak{a}_m} \\
& = \widetilde{\delta}_{\mathbb{Q}(\zeta_m)}(\overline{f},i,r) .
\end{align*}
Thus, the following statements are equivalent.
\begin{itemize}
\item $D_{\mathbb{Q}(\zeta_m)} \omega^i \left( \theta_{\mathbb{Q}(\zeta_{mp}), r} (\overline{f}) \right)$ does not vanish modulo $\mathfrak{a}_m$.
\item $\widetilde{\delta}_{\mathbb{Q}(\zeta_m)}(\overline{f},i,r)$ does not vanish (modulo $\pi$).
\end{itemize}
On the other hand, since 
$$ \overline{D_{\mathbb{Q}(\zeta_{m})} \mathfrak{z}_{\mathbb{Q}(\zeta_{m}), \delta_{f, \mathrm{cris}}}(f, i, k-r)} \in \mathrm{H}^1(\mathbb{Q}(\zeta_m), T_{f, i}(k-r)/\pi T_{f, i}(k-r) )^{\mathrm{Gal}(\mathbb{Q}(\zeta_m)/\mathbb{Q})}$$
(thanks to \cite[Lem. 4.4.2]{rubin-book})
and
$$\left( \left( \sigma_{\eta_{\ell_1}} - 1 \right)^2, \cdots,  \left( \sigma_{\eta_{\ell_s}} - 1 \right)^2 \right) \cdot \mathrm{H}^1(\mathbb{Q}(\zeta_m), T_{f, i}(k-r)/\pi T_{f, i}(k-r) )^{\mathrm{Gal}(\mathbb{Q}(\zeta_m)/\mathbb{Q})} = 0,$$
the following statements are equivalent.
\begin{itemize}
\item $\overline{D_{\mathbb{Q}(\zeta_{m})} \mathfrak{z}_{\mathbb{Q}(\zeta_{m}), \delta_{f, \mathrm{cris}}}(f, i, k-r)}$ does not vanish.
\item $D_{\mathbb{Q}(\zeta_{m})} \mathfrak{z}_{\mathbb{Q}(\zeta_{m}), \delta_{f, \mathrm{cris}}}(f, i, k-r)$ does not vanish modulo $\mathfrak{a}_m$.
\end{itemize}
If  $D_{\mathbb{Q}(\zeta_m)} \omega^i \left( \theta_{\mathbb{Q}(\zeta_{mp}), r} (\overline{f}) \right)$ does not vanish modulo $\mathfrak{a}_m$, 
then $\overline{D_{\mathbb{Q}(\zeta_m)} \omega^i \left( \theta_{\mathbb{Q}(\zeta_{mp}), r} (\overline{f}) \right)}$ does not vanish.

Thus, it suffices to prove the converse.
Suppose that
$\overline{D_{\mathbb{Q}(\zeta_m)} \omega^i \left( \theta_{\mathbb{Q}(\zeta_{mp}), r} (\overline{f}) \right)}$ does not vanish for some $m  \in \mathcal{N}$.
Then
$\overline{D_{\mathbb{Q}(\zeta_{m})} \mathfrak{z}_{\mathbb{Q}(\zeta_{m}), \delta_{f, \mathrm{cris}}}(f, i, k-r)}$ does not vanish  since 
$$D_{\mathbb{Q}(\zeta_m)} \omega^i \left( \theta_{\mathbb{Q}(\zeta_{mp}), r} (\overline{f}) \right) = \mathrm{Col}^{(i)}_{\mathbb{Q}(\zeta_{m}), r} \circ \mathrm{loc}^s_p \left(  D_{\mathbb{Q}(\zeta_{m})} \mathfrak{z}_{\mathbb{Q}(\zeta_{m}), \delta_{f, \mathrm{cris}}}(f, i, k-r) \right).$$
By the above equivalence,
$D_{\mathbb{Q}(\zeta_{m})} \mathfrak{z}_{\mathbb{Q}(\zeta_{m}), \delta_{f, \mathrm{cris}}}(f, i, k-r)$ does not vanish modulo $\mathfrak{a}_m$.
We now suppose that 
$D_{\mathbb{Q}(\zeta_m)} \omega^i \left( \theta_{\mathbb{Q}(\zeta_{mp}), r} (\overline{f}) \right)$ vanishes modulo $\mathfrak{a}_m$ for every $m \in \mathcal{N}$.
As mentioned above, the non-triviality properties of 
\[
\xymatrix{
\left\lbrace \overline{D_{\mathbb{Q}(\zeta_{m})} \mathfrak{z}_{\mathbb{Q}(\zeta_{m}), \delta_{f, \mathrm{cris}}}(f, i, k-r)} \right\rbrace_m \textrm{ and }
\left\lbrace D_{\mathbb{Q}(\zeta_{m})} \mathfrak{z}_{\mathbb{Q}(\zeta_{m}), \delta_{f, \mathrm{cris}}}(f, i, k-r) \pmod{\mathfrak{a}_m} \right\rbrace_m
}
\]
are equivalent.
By the equivalence stated before Theorem \ref{thm:rigidity}, $\ks^{(1)}$ becomes trivial under the assumption.
Then $\overline{D_{\mathbb{Q}(\zeta_m)} \omega^i \left( \theta_{\mathbb{Q}(\zeta_{mp}), r} (\overline{f}) \right)}$ should vanish for every $m \in \mathcal{N}$, but it is a contradiction.
Thus, the following statements are also equivalent.
\begin{itemize}
\item $\overline{D_{\mathbb{Q}(\zeta_m)} \omega^i \left( \theta_{\mathbb{Q}(\zeta_{mp}), r} (\overline{f}) \right)}$ does not vanish  for some $m \in \mathcal{N}$.
\item $D_{\mathbb{Q}(\zeta_m)} \omega^i \left( \theta_{\mathbb{Q}(\zeta_{mp}), r} (\overline{f}) \right)$ does not vanish modulo $\mathfrak{a}_m$  for some $m \in \mathcal{N}$.
\end{itemize}
Thus, we obtain the equivalence we need.

When $i = 0$, it is not difficult to observe that
\begin{align*}
\omega^0 \left( \theta_{\mathbb{Q}(\zeta_{mp}), r}(\overline{f}) \right) & = \left( 1 - a_p(\overline{f}) \cdot p^{-r}  \cdot \mathrm{Frob}^{-1}_p + \overline{\psi}(p) \cdot p^{k-1-2r} \cdot \mathrm{Frob}^{-2}_p \right) \cdot   \theta_{\mathbb{Q}(\zeta_{m}), r}(\overline{f}) , \\
\omega^0 \left(  \mathrm{Col}_{\mathbb{Q}(\zeta_{mp}), r} \left( - \right)  \right) & =
\left( 1 - a_p(\overline{f}) \cdot p^{-r}  \cdot \mathrm{Frob}^{-1}_p +  \overline{\psi}(p) \cdot p^{k-1-2r} \cdot \mathrm{Frob}^{-2}_p \right) \cdot  \mathrm{Col}_{\mathbb{Q}(\zeta_{m}), r} \left( - \right)  
\end{align*}
since Kato's Euler systems have no contribution to the $p$-power direction in the norm compatibility.
Write $E_p(\overline{f},r) = 1 - a_p(\overline{f}) \cdot p^{-r}  \cdot \mathrm{Frob}^{-1}_p + \overline{\psi}(p) \cdot p^{k-1-2r} \cdot \mathrm{Frob}^{-2}_p $ and note that  $\mathfrak{z}_{\mathbb{Q}(\zeta_{m}), \delta_{f, \mathrm{cris}}}(f, 0, k-r) =  \mathfrak{z}_{\mathbb{Q}(\zeta_{m}), \delta_{f, \mathrm{cris}}}(f, k-r)$.
Then we have
\[
{ \scriptsize
\xymatrix{
\mathrm{H}^1(\mathbb{Q}(\zeta_{m}), T_{f}(k-r)) \ar[d]^-{\mathrm{loc}^s_p} & \mathfrak{z}_{\mathbb{Q}(\zeta_{m}), \delta_{f, \mathrm{cris}}}(f, k-r) \ar@{|->}[d] \\
 \mathrm{H}^1_{/f}(\mathbb{Q}(\zeta_{m}) \otimes \mathbb{Q}_p, T_{f, i}(k-r)) \ar[d]^-{\mathrm{Col}^{(0)}_{\mathbb{Q}(\zeta_{m}), r}} &  \mathrm{loc}^s_p \mathfrak{z}_{\mathbb{Q}(\zeta_{m}), \delta_{f, \mathrm{cris}}}(f,  k-r) \ar@{|->}[d] \\
E_p(\overline{f},r) \cdot  \mathcal{O}_\pi[\mathrm{Gal}(\mathbb{Q}(\zeta_{m})/\mathbb{Q})] &  \mathrm{Col}^{(0)}_{\mathbb{Q}(\zeta_{mp}), r} \mathrm{loc}^s_p \mathfrak{z}_{\mathbb{Q}(\zeta_{m}), \delta_{f, \mathrm{cris}}}(f, k-r)
}
}
\]
and
$ \mathrm{Col}^{(0)}_{\mathbb{Q}(\zeta_{m}), r} \mathrm{loc}_p \mathfrak{z}_{\mathbb{Q}(\zeta_{m}), \delta_{f, \mathrm{cris}}}(f, k-r)
= E_p(\overline{f},r) \cdot  \mathrm{Col}_{\mathbb{Q}(\zeta_{m}), r} \mathrm{loc}_p \mathfrak{z}_{\mathbb{Q}(\zeta_{m}), \delta_{f, \mathrm{cris}}}(f, k-r) $.
Since the Frobenius acts trivially on Kurihara numbers, we observe that
the $p$-th Euler factor $ 1 - a_p(\overline{f}) \cdot p^{-r}  + \overline{\psi}(p) \cdot p^{k-1-2r} $ appears, but it cancels out if it is non-zero. 
Thanks to Assumption \ref{assu:standard}.(ENV), we are safe.
Thus, the Kurihara number for the $i=0$ case naturally appears from this computation.

\section{Towards Conjecture A of Coates--Sujatha} \label{sec:mu-conjecture}
We prove Theorem \ref{thm:main_thm_mu_zero_conj} in $\S$\ref{subsec:mazur-tate-coates-sujatha} and Corollary \ref{cor:mu-elliptic-curves} in $\S$\ref{subsec:mu-ordinary} and $\S$\ref{subsec:mu-supersingular}.
\begin{thm}[Conjecture A] \label{thm:main_thm_mu_zero_conj}
Let $p$ be an odd prime.
Keep Assumption \ref{assu:standard} as well as Conjecture \ref{conj:comparison-zeta-elements} (e.g. $2 \leq k \leq p-1$ or $\overline{\rho}$ is  $p$-ordinary and $p$-distinguished).
If
\begin{equation*}
\sum_{a \in (\mathbb{Z}/p\mathbb{Z})^\times} \overline{ \lambda^{\pm, \mathrm{opt}}(\overline{f}, z^{r-1}, a(1+bp), p^n) } \cdot \overline{ \omega^{i}(a)  }  \neq 0 
\end{equation*}
in $\mathbb{F}_\pi$ for some $b \in \mathbb{Z}/p^{n-1}\mathbb{Z}$ with $n \geq 1$, then 
then we obtain the $\mu$-part of the main conjecture and Conjecture A of Coates--Sujatha
$$\mu \left( \dfrac{\mathrm{H}^1_{\mathrm{Iw}}(T_{f,i}(k-r)) }{ \Lambda^{(i)}_{\mathrm{Iw}}\mathbf{z}_{\mathbb{Q}, \delta_{f, \mathrm{cris}} }(f, i, k-r)} \right) =\mu \left( \mathrm{H}^2_{\mathrm{Iw}}(T_{f,i}(k-r)) \right)  = 0.$$
\end{thm}
The verification of Conjecture A is essential to complete the numerical verification of the Iwasawa main conjecture for families of modular forms of arbitrary level via congruences (Corollary \ref{cor:main_thm_main_conj_simple}).

The same numerical criterion also implies the $\mu=0$ conjecture for $p$-adic $L$-functions of elliptic curves with good ordinary reduction (Greenberg's conjecture \cite[Conj. 1.11]{greenberg-lnm}) and for $\pm$-$p$-adic $L$-functions of elliptic curves with good supersingular reduction (Pollack's conjecture \cite[Conj. 6.3]{pollack-thesis}).
\begin{cor} \label{cor:mu-elliptic-curves}
Assume that $\rho_f$ has large image.
Let $E$ be an (optimal) elliptic curve over $\mathbb{Q}$ corresponding to $f$, i.e. $F = \mathbb{Q}$, $\psi = \mathbf{1}$, $f = \overline{f}$, $k=2$, and $r=1$.
If
$$ \sum_{a \in (\mathbb{Z}/p\mathbb{Z})^\times} \overline{ \lambda^{\pm, \mathrm{opt}}(f, 1, a(1+bp), p^n) }  \neq 0  $$
in $\mathbb{F}_p$ for some $b \in \mathbb{Z}/p^{n-1}\mathbb{Z}$ with $n \geq 1$,
then the following statements hold.
\begin{enumerate}
\item When $E$ is good ordinary at $p$ and $a_p(E) \not\equiv 1 \pmod{p}$, we obtain
$$\mu \left(  \Lambda_{\mathrm{Iw}} / L_p(E) \right) = \mu \left(  \mathrm{Sel}(\mathbb{Q}_\infty, E[p^\infty])^\vee \right) = 0.$$
\item When $E$ is good supersingular at $p$ with $a_p(E) = 0$, we obtain 
$$
\begin{array}{ll}
\mu \left(  \Lambda_{\mathrm{Iw}} / L^-_p(E) \right) = \mu \left(  \mathrm{Sel}^{+}(\mathbb{Q}_\infty, E[p^\infty])^\vee \right) = 0 & \textrm{ when } n \textrm{ is odd}, \\
\mu \left(  \Lambda_{\mathrm{Iw}} / L^+_p(E) \right) = \mu \left(  \mathrm{Sel}^{-}(\mathbb{Q}_\infty, E[p^\infty])^\vee \right) = 0 &  \textrm{ when } n \textrm{ is even}.
\end{array}
$$
\end{enumerate}
\end{cor}
\begin{rem}
 Theorem \ref{cor:mu-elliptic-curves}.(1) generalizes to Hida families \cite{epw} 
The N\'{e}ron periods of optimal elliptic curves are optimal periods \cite[Rem. 5.4 and 5.5]{pollack-thesis} and \cite[Thm. 2.1]{glenn-cuspidal}.
\end{rem}

We assume Conjecture \ref{conj:comparison-zeta-elements} in this section again.
\subsection{Mazur--Tate elements and the conjecture of Coates--Sujatha} \label{subsec:mazur-tate-coates-sujatha}
Recall the decomposition of finite layer Coleman maps
\[
\xymatrix{
 \mathrm{H}^1(\mathbb{Q}(\zeta_{p^n}) \otimes \mathbb{Q}_p, T_f(k-r)) \ar[rr]^-{ \mathrm{Col}_{\mathbb{Q}(\zeta_{p^n}), r} } \ar[d]^-{\simeq} & & \Lambda_n \ar[d]^-{\simeq} \\
 \bigoplus_i \mathrm{H}^1(\mathbb{Q}_{n-1, p}, T_{f,i}(k-r)) \ar[rr]^-{  \bigoplus_i \mathrm{Col}^{(i)}_{\mathbb{Q}_{n-1}, r} } & & \bigoplus_i \Lambda^{(i)}_n .
}
\]
Then Theorem \ref{thm:main_thm_finite_layer_coleman} shows that
\begin{align} \label{eqn:coleman-map-p-power}
\begin{split}
& \mathrm{Col}^{(i)}_{\mathbb{Q}_{n-1}, r} \left(  \mathrm{loc}_p \left( \mathfrak{z}_{\mathbb{Q}_{n-1}, \delta_{f, \mathrm{cris}} }(f, i, k-r) \right)  \right) \\
& =
\omega^i \left( \mathrm{Col}_{\mathbb{Q}(\zeta_{p^n}), r} \left( \mathrm{loc}_p \left( \mathfrak{z}_{\mathbb{Q}(\zeta_{p^n}), \delta_{f, \mathrm{cris}} }(f, k-r) \right) \right) \right) \\
& = \omega^i \left( \theta_{\mathbb{Q}(\zeta_{p^n}), r}(\overline{f}) \right) \in \Lambda^{(i)}_n .
\end{split}
\end{align}
Since
$$
  \omega^i \left( \theta_{\mathbb{Q}(\zeta_{p^n}), r}(\overline{f}) \right)  =   \sum_{b \in \mathbb{Z}/p^{n-1}\mathbb{Z}} \left(  \sum_{a \in (\mathbb{Z}/p\mathbb{Z})^\times} \lambda^{\pm, \mathrm{opt}}(\overline{f}, z^{r-1}, a(1+bp), p^n) \cdot \omega^i(a) \right) \cdot \sigma_{1+bp} 
,$$
if $\sum_{a \in (\mathbb{Z}/p\mathbb{Z})^\times} \overline{ \lambda^{\pm, \mathrm{opt}}(\overline{f}, z^{r-1}, a(1+bp), p^n) } \cdot \overline{ \omega^{i}(a)  } \neq 0$
for some $n \geq 1$,
then $\omega^i \left( \theta_{\mathbb{Q}(\zeta_{p^n}), r}(\overline{f}) \right) \in \Lambda^{(i)}_n$ is non-zero modulo $\pi$.
Furthermore, (\ref{eqn:coleman-map-p-power}) implies that
$\mathfrak{z}_{\mathbb{Q}_{n-1}, \delta_{f, \mathrm{cris}} }(f, i, k-r) \in \mathrm{H}^1(\mathbb{Q}_{n-1, p}, T_{f,i}(k-r))$
 is non-zero mod $\pi$ and 
$\mathfrak{Z}_{\mathbb{Q},  \delta_{f, \mathrm{cris}} }(f, i, k-r) \in \mathrm{H}^1_{\mathrm{Iw}}(T_{f, i}(k-r))$
 is also non-zero mod $\pi$.
By the Euler system divisibility (Theorem \ref{thm:kato-divisibility}), we also have
$\mu \left( \mathrm{H}^2_{\mathrm{Iw}}(T_{f, i}(k-r)) \right) = 0$.

\subsection{Elliptic curves with good ordinary reduction} \label{subsec:mu-ordinary}
Let $E$ be an (optimal) elliptic curve over $\mathbb{Q}$ with non-anomalous good ordinary reduction at $p$ corresponding to $f$, i.e. $p \nmid a_p(f) \cdot (a_p(f)-1)$.
Let $n \geq 0$ be an integer.
Write 
$\theta_{\mathbb{Q}_{n}}(E) = \mathrm{Tr}_{\mathbb{Q}(\zeta_{p^{n+1}})/\mathbb{Q}_n} \left( \theta_{\mathbb{Q}(\zeta_{p^{n+1}}), 1}(f) \right) $,
$$\vartheta_{\mathbb{Q}(\zeta_{p^{n+1}})}(E)
= \dfrac{1}{\alpha^{n+1}} \cdot \left( \theta_{\mathbb{Q}(\zeta_{p^{n+1}}),1}(f) - \dfrac{1}{\alpha} \cdot \nu_{\mathbb{Q}(\zeta_{p^{n}})/\mathbb{Q}(\zeta_{p^{n+1}})} \left(  \theta_{\mathbb{Q}(\zeta_{p^{n}}),1}(f) \right) \right)   ,$$
and
$\vartheta_{\mathbb{Q}_n}(E) = \mathrm{Tr}_{\mathbb{Q}(\zeta_{p^{n+1}})/\mathbb{Q}_n} \left( \vartheta_{\mathbb{Q}(\zeta_{p^{n+1}})}(E) \right)   $.
Then $\vartheta_{\mathbb{Q}_n}(E)$ forms a compatible sequence in $\mathbb{Z}_p[\mathrm{Gal}(\mathbb{Q}_n/\mathbb{Q})]$ and the inverse limit is the $p$-adic $L$-function $L_p(E)$ of $E$ in $\Lambda^{(0)}_{\mathrm{Iw}}$.
\begin{prop}
If $p \nmid a_p(f) \cdot (a_p(f)-1)$, then
 $\theta_{\mathbb{Q}_n}(E)$ is a multiple of $\vartheta_{\mathbb{Q}_n}(E)$ by an element of $\mathbb{Z}_p[\mathrm{Gal}(\mathbb{Q}_n/\mathbb{Q})]$, i.e.
$\theta_{\mathbb{Q}_n}(E) = c \cdot \vartheta_{\mathbb{Q}_n}(E)$
where $c \in \mathbb{Z}_p[\mathrm{Gal}(\mathbb{Q}_n/\mathbb{Q})]$.
\end{prop}
\begin{proof}
We use induction on $n$.
We write $\vartheta_{\mathbb{Q}}(E) = \dfrac{L(f_\alpha, 1)}{\Omega^+_E}$.
By the interpolation formula, we have
$\vartheta_{\mathbb{Q}}(E) = \left( 1 - \dfrac{1}{\alpha} \right) \cdot \theta_{\mathbb{Q}}(E)  $
and $\left( 1 - \dfrac{1}{\alpha} \right)$ is invertible in $\mathbb{Z}_p$ since $p \nmid  (a_p(f)-1)$.
Looking at the norm relation in \cite[$\S$1.3]{mazur-tate} and the definition of $\vartheta_{\mathbb{Q}(\zeta_p)}(E)$, we also have
$\pi_{1,0} (\vartheta_{\mathbb{Q}(\zeta_p)}(E)  ) = \left( 1 - \dfrac{1}{\alpha} \right) \cdot \vartheta_{\mathbb{Q}}(E) $.
We compute
\begin{align*}
\vartheta_{\mathbb{Q}(\zeta_p)}(E) & =\frac{1}{\alpha} \left(  \theta_{\mathbb{Q}(\zeta_p)}(E) -  \frac{1}{\alpha} \cdot \nu_{0,1} (\theta_{\mathbb{Q}}(E))  \right) \\
& = \frac{1}{\alpha} \left( \theta_{\mathbb{Q}(\zeta_p)}(E) -  \frac{1}{\alpha} \cdot \nu_{0,1} \left( \left( 1 - \dfrac{1}{\alpha} \right)^{-1} \cdot \vartheta_{\mathbb{Q}}(E)  \right) \right)  \\
& = \frac{1}{\alpha} \left( \theta_{\mathbb{Q}(\zeta_p)}(E) -  \frac{1}{\alpha} \cdot \nu_{0,1} \left( \left( 1 - \dfrac{1}{\alpha} \right)^{-2} \cdot   \pi_{1,0} (\vartheta_{\mathbb{Q}(\zeta_p)}(E)  ) \right) \right)  \\
& = \frac{1}{\alpha} \left( \theta_{\mathbb{Q}(\zeta_p)}(E) -  \frac{1}{\alpha} \cdot a \cdot \vartheta_{\mathbb{Q}(\zeta_p)}(E) \right) 
\end{align*}
where $a$ in $\mathbb{Z}_p[\mathrm{Gal}(\mathbb{Q}(\zeta_p)/\mathbb{Q})]$.
Thus, we have
$\theta_{\mathbb{Q}(\zeta_p)}(E) = \left( \alpha +  \dfrac{1}{\alpha} \cdot a \right) \vartheta_{\mathbb{Q}(\zeta_p)}(E) $.

Now suppose that 
$\theta_{\mathbb{Q}(\zeta_{p^{n-1}})}(E) = b \cdot \vartheta_{\mathbb{Q}(\zeta_{p^{n-1}})}(E)$
with $b \in \mathbb{Z}_p[\mathrm{Gal}(\mathbb{Q}(\zeta_{p^{n-1}})/\mathbb{Q})]$.
Then we have
\begin{align*}
\vartheta_{\mathbb{Q}( \zeta_{p^n} )}(E) & = \frac{1}{\alpha^n} \left(  \theta_{\mathbb{Q}( \zeta_{p^n} )}(E) -  \frac{1}{\alpha} \cdot \nu_{n,n-1} ( \theta_{\mathbb{Q}( \zeta_{p^{n-1}} )}(E) )  \right)\\
& = \frac{1}{\alpha^n} \left(  \theta_{\mathbb{Q}( \zeta_{p^n} )}(f) -  \frac{1}{\alpha} \cdot \nu_{n,n-1} (  b \cdot \vartheta_{\mathbb{Q}(\zeta_{p^{n-1}})}(E)  )  \right)\\
& = \frac{1}{\alpha^n} \left(  \theta_{\mathbb{Q}( \zeta_{p^n} )}(f) -  \frac{1}{\alpha} \cdot \nu_{n,n-1} (  b \cdot \pi_{n, n-1} ( \vartheta_{\mathbb{Q}(\zeta_{p^{n}})}(E) )  ) \right) \\
& = \frac{1}{\alpha^n} \left(  \theta_{\mathbb{Q}( \zeta_{p^n} )}(f) -  \frac{1}{\alpha} \cdot c \cdot \vartheta_{\mathbb{Q}(\zeta_{p^{n}})}(E) )  \right)
\end{align*}
where $c \in \mathbb{Z}_p[\mathrm{Gal}(\mathbb{Q}(\zeta_{p^{n}})/\mathbb{Q})]$.
In the same manner, we have
$$\theta_{\mathbb{Q}( \zeta_{p^n} )}(f) = \left( \alpha^n +  \frac{1}{\alpha} \cdot c \right) \cdot \vartheta_{\mathbb{Q}( \zeta_{p^n} )}(E).$$
By using the decomposition via the Teichm\"{u}ller character, the conclusion follows.
\end{proof}
\begin{cor}
Let $\mathfrak{a}$ be an ideal of $\mathbb{Z}_p[\mathrm{Gal}(\mathbb{Q}_n/\mathbb{Q})]$.
If $\vartheta_{\mathbb{Q}_n}(E) \in \mathfrak{a}$, then $\theta_{\mathbb{Q}_n}(E) \in \mathfrak{a}$.
\end{cor}
If we take $\mathfrak{a} = p\mathbb{Z}_p[\mathrm{Gal}(\mathbb{Q}_n/\mathbb{Q})]$, then we immediately obtain 
$\mu \left(  \Lambda_{\mathrm{Iw}} / L_p(E) \right) = 0$ from the above corollary.
Also, Corollary \ref{cor:mu-elliptic-curves}.(1) follows from the Euler system divisibility of the main conjecture of Mazur--Greenberg, which is equivalent to Theorem \ref{thm:kato-divisibility}.
\subsection{Elliptic curves with supersingular reduction} \label{subsec:mu-supersingular}
We assume that $f$ corresponds to an (optimal) elliptic curve over $\mathbb{Q}$ with good supersingular reduction at $p \geq 3$.
We further assume that
$a_p(f) = 0 $,
and indeed it always holds if $p \geq 5$ due to the Hasse bound.

We recall the identification
 $\Lambda^{(0)}_{\mathrm{Iw}} = \mathbb{Z}_p\llbracket \mathrm{Gal}(\mathbb{Q}_\infty/\mathbb{Q}) \rrbracket$ with $\mathbb{Z}_p\llbracket X \rrbracket$ defined by sending a topological generator of $\mathrm{Gal}(\mathbb{Q}_\infty/\mathbb{Q})$ to $1+X$.
Let $\omega^{\pm}_0(X)  := X$, $\widetilde{\omega}^{\pm}_0(X)  := 1$, and
\[
\xymatrix@R=0em{
{\displaystyle \omega^{+}_n = \omega^{+}_n(X) := X\cdot \prod_{2 \leq m \leq n, m: \textrm{ even}}\Phi_m(1+X) } , &
{\displaystyle \omega^{-}_n = \omega^{-}_n(X)  := X\cdot \prod_{1 \leq m \leq n, m: \textrm{ odd}}\Phi_m(1+X) } , \\
{\displaystyle \widetilde{\omega}^{+}_n = \widetilde{\omega}^{+}_n(X)  := \prod_{2 \leq m \leq n, m: \textrm{ even}}\Phi_m(1+X) } , &
{\displaystyle \widetilde{\omega}^{-}_n = \widetilde{\omega}^{-}_n(X)  := \prod_{1 \leq m \leq n, m: \textrm{ odd}}\Phi_m(1+X) } .
}
\]
Then we have $\omega_n(X) = \omega^{\pm}_n(X) \cdot \widetilde{\omega}^{\mp}_n(X)$, respectively. We also regard
$\omega^{\pm}_n$, $\widetilde{\omega}^{\pm}_n$ as elements in $\Lambda^{(0)}_{n+1}$ or $\Lambda^{(0)}_{\mathrm{Iw}}$ via the identification.
\begin{prop}[Pollack] \label{prop:pollack_theta_pm}
$\theta_n(E) \equiv \widetilde{\omega}^{\mp}_n \cdot L^{\mp}_p(E) \Mod{\omega_n}$
in $\Lambda^{(0)}_{n+1}$ if $n$ is even/odd, respectively.
\end{prop}
\begin{proof}
See \cite[Prop. 6.18]{pollack-thesis}.
\end{proof}
If the numerical criterion in Theorem \ref{thm:main_thm_mu_zero_conj} holds,
then $\theta_{n-1}(E)$ does not vanish modulo $p$.
By Proposition \ref{prop:pollack_theta_pm}, if $n-1$ is even (resp. odd), then $L^-_p(E)$ (resp. $L^+_p(E)$) does not vanish modulo $p$; thus, $\mu \left(  \mathrm{Sel}^{+}(\mathbb{Q}_\infty, E[p^\infty])^\vee \right) = 0$ (resp. $\mu \left(  \mathrm{Sel}^{-}(\mathbb{Q}_\infty, E[p^\infty])^\vee \right) = 0$) by the Euler system divisibility of the $\pm$-main conjectures, which is equivalent to Theorem \ref{thm:kato-divisibility}.
Thus, Corollary \ref{cor:mu-elliptic-curves}.(2) follows.

\section{Applications to Birch and Swinnerton-Dyer conjecture} \label{sec:app-bsd}
We discuss two applications to Birch and Swinnerton-Dyer conjecture. More general applications are studied in the subsequent works \cite{kim-structure-selmer, kim-refined-tamagawa}.
\subsection{On the $p$-part of the BSD formula for elliptic curves of analytic rank one} \label{subsec:appendix-applications}
We first discuss an application of Theorem \ref{thm:main_thm_main_conj_simple} and Corollary \ref{cor:main_thm_main_conj_simple} to the arithmetic of elliptic curves of analytic rank one via the work of Kobayashi on $p$-adic Gross--Zagier formula \cite{kobayashi-gross-zagier}. 

Let $E$ be an (optimal) elliptic curve of conductor $N$ with good reduction at $p >2$ and  $\mathcal{K}$ an imaginary quadratic field such that all the prime factors of $Np$ split in $\mathcal{K}$ and $L(E_{\mathcal{K}}, 1 ) \neq 0 $
where $E_{\mathcal{K}}$ is the quadratic twist of $E$ by $\mathcal{K}$. See \cite{bump-friedberg-hoffstein, murty-murty-mean} for the existence of such an imaginary quadratic field. Let $f$ be the newform corresponding to $E$ and $f_{\mathrm{min}}$ the newform  of minimal level congruent to $f$ modulo $p$ obtained by level lowering \cite{diamond-taylor-non-optimal}. The same notational rule applies to $E_{\mathcal{K}}$.

\begin{thm}[Kobayashi] \label{thm:rank-one-bsd-over-Q-kobayashi}
Suppose that
\begin{enumerate}
\item $\mathrm{ord}_{s=1}L(E, s) = 1$,
\item $E[p]$ is a surjective Galois representation, and
\item the $p$-adic height pairing is non-degenerate and $a_p(E) \not\equiv 1 \pmod{p}$ when $p$ is good ordinary for $E$.
\end{enumerate}
If
\begin{enumerate}
\item[(i)]
$\widetilde{\delta}_{m}(f_{\mathrm{min}}, 1) \neq 0, \widetilde{\delta}_{m_{\mathcal{K}}}(f_{\mathcal{K}, \mathrm{min}} , 1)  \neq 0$ in $\mathbb{F}_p$\\
for some square-free products $m$ and $m_{\mathcal{K}}$ of Kolyvagin primes for $f_{\mathrm{min}}$ and $f_{\mathcal{K}, \mathrm{min}}$, respectively, and
\item[(ii)] 
${ \displaystyle \sum_{a \in (\mathbb{Z}/p\mathbb{Z})^\times} \overline{ \lambda^{\pm, \mathrm{opt}}(f, 1, a(1+bp), p^n) }  \neq 0, \displaystyle \sum_{a \in (\mathbb{Z}/p\mathbb{Z})^\times} \overline{ \lambda^{\pm, \mathrm{opt}}(f_{\mathcal{K}}, 1, a(1+b'p), p^n) }  \neq 0}$ 
in $\mathbb{F}_p$ \\
for some $b, b' \in \mathbb{Z}/p^{n-1}\mathbb{Z}$ with $n \geq 1$,
\end{enumerate}
then the $p$-part of BSD formula for $E$ holds, i.e.
$$ \left\vert \dfrac{L'(E, 1)}{ \Omega_{E} \cdot \mathrm{Reg}(E/\mathbb{Q}) } \right\vert_{p} = 
 \left\vert \#\sha(E/\mathbb{Q}) \cdot \prod_{\ell \vert N} c_\ell(E/\mathbb{Q}) \right\vert_{p}
$$
where 
$\Omega_{E}$ is the real period of $E$,
$\mathrm{Reg}(E/\mathbb{Q}) := \dfrac{\langle y,y \rangle_{\mathrm{NT}}}{[E(\mathbb{Q}): \mathbb{Z} \cdot y]^2}$ for any non-torsion $y \in E(\mathbb{Q})$,
$\langle -,- \rangle_{\mathrm{NT}}$ is the N\'{e}ron--Tate height pairing,
 and $c_\ell(E/\mathbb{Q})$ is the local Tamagawa factor of $E$ at $\ell$.
If all the local Tamagawa factors of $E$ are prime to $p$, then we can bypass the numerical computation (ii).
\end{thm}
\begin{proof}
This is the combination of Theorem \ref{thm:main_thm_main_conj_simple}, Corollary \ref{cor:main_thm_main_conj_simple}, and\cite[Corollary 1.3.(3)]{kobayashi-gross-zagier}.
\end{proof}
See \cite{kazim-pollack-sasaki} for the recent developments on the non-degeneracy of the $p$-adic height pairing for the ordinary case.
One advantage of Theorem \ref{thm:rank-one-bsd-over-Q-kobayashi} is that no assumption on the conductor is imposed.

\subsection{Kato's Kolyvagin systems and Heegner point Kolyvagin systems} \label{subsec:appendix-heegner}
We study the applications of Theorem \ref{thm:main_thm_main_conj_simple} to Kolyvagin's conjecture on the indivisibility of Heegner point Kolyvagin systems following Wei Zhang's work \cite{wei-zhang-mazur-tate}. 
The immediate consequences include
the $p$-converse to a theorem of Gross--Zagier and Kolyvagin,
the structure theorem on the $p$-part of Tate--Shafarevich groups, and
the $p$-part of the BSD formula of modular abelian varieties of analytic rank one.
Since Wei Zhang's work already completes the $p$-ordinary case, our main focus lies in the $p$-non-ordinary case. 

\subsubsection{The anticyclotomic setting}
Let $p > 3$ be a prime and $f \in S_2(\Gamma_0(N))$ a newform with $(N,p) = 1$.
Fix a prime  $\pi$ of the Hecke field $F$ of $f$ lying above $p$. 
Let $\mathcal{K}$ be an imaginary quadratic field of discriminant $D_{\mathcal{K}} < 0$ with $(D_{\mathcal{K}},N p)=1$.
Write $N = N^+ \cdot N^-$ where a prime factor of $N^+$ splits in $\mathcal{K}$ and a prime factor of $N^-$ is inert in $\mathcal{K}$.
We always assume that $N^-$ is square-free and denote by  $\nu(N^-)$ the number of prime factors of $N^-$. 
Then the root number for $f$ over $\mathcal{K}$ is $(-1)^{1+\nu(N^-)}$. In this subsection, we assume that $\nu(N^-)$ is even, i.e. the \emph{generalized Heegner hypothesis} for $(f, \mathcal{K})$ is satisfied.

A prime $\ell$ is an \textbf{anticyclotomic Kolyvagin prime for $(f,\mathcal{K}, \pi)$} if $\ell$ is prime to $N \cdot D_{\mathcal{K}} \cdot p$, inert in $\mathcal{K}$, and the Kolyvagin index $M(\ell) := \mathrm{min}(v_p(\ell+1), v_\pi(a_\ell(f)))$ is strictly positive. The last condition is equivalent to that $\ell \equiv -1 \pmod{p}$ and $a_\ell(f) \equiv 0 \pmod{\pi}$.
Denote by $\mathcal{N}^{\mathrm{ac}}$ the set of square-free products of anticyclotomic Kolyvagin primes for $(f,\mathcal{K}, \pi)$.
Define $M(n) = \mathrm{min} \lbrace M(\ell) : \ell \vert n \rbrace$ if $n > 1$ and $M(1) = \infty$ where $n \in \mathcal{N}^{\mathrm{ac}}$.

Denote by $\mathcal{K}(n)$ the ring class field of conductor $n$ of $\mathcal{K}$ and by $A_f$ the modular abelian variety associated to $f$.
To each Heegner point $y(n) \in A_f(\mathcal{K}(n))$ of conductor $n$ and $M \leq M(n)$, there exists a cohomology class
$c^{\mathrm{Heeg}}_{M}(n) \in \mathrm{H}^1(\mathcal{K}, T_f/\pi^M T_f)$
where $n$ runs over the square-free products of anticyclotomic Kolyvagin primes. See \cite[$\S$3.7]{wei-zhang-mazur-tate} for the explicit construction of $c^{\mathrm{Heeg}}_{M}(n)$.

Following the terminology of \cite{mazur-rubin-book}, suitably modified for the Heegner point
setting \cite{howard-kolyvagin, howard-gl2-type}, we call the collection 
$$\ks^{\mathrm{Heeg}} = \left\lbrace c^{\mathrm{Heeg}}_{M}(n) \in \mathrm{H}^1(\mathcal{K}, T_f/\pi^M T_f) : n \in \mathcal{N}^{\mathrm{ac}} , M \leq M(n) \right\rbrace $$
 the Heegner point Kolyvagin system\footnote{We do not distinguish Kolyvagin systems and weak Kolyvagin systems seriously in this subsection.} for $(f, \mathcal{K}, \pi)$.

\subsubsection{Kolyvagin conjecture}
For a fixed $n \in \mathcal{N}^{\mathrm{ac}}$,
define  the divisibility index 
$\mathscr{M}(n)$
by the maximal non-negative integer (or the infinity) $\mathscr{M}$ such that 
$c^{\mathrm{Heeg}}_M(n) \in \pi^\mathscr{M} \mathrm{H}^1(\mathcal{K}, T_f/\pi^M T_f)$
for all $M \leq M(n)$. Define
$\mathscr{M}_r := \mathrm{min}_{\nu(n)=r}  \mathscr{M}(n)$
and Kolyvagin shows that $\mathscr{M}_r \geq \mathscr{M}_{r+1} \geq 0$.
We also define $\mathscr{M}_\infty(f) := \lim_{r \to \infty} \mathscr{M}_r$ as the minimum of $\mathscr{M}_r$ for varying $r \geq 0$. The following conjecture is due to Kolyvagin (as stated in \cite[Conj. A]{kolyvagin-selmer} and \cite[Conj. 3.2]{wei-zhang-mazur-tate}).
\begin{conj}[Kolyvagin] \label{conj:kolyvagin-conjecture}
If $\rho_f$ has large image, then $\ks^{\mathrm{Heeg}} \neq 0$. Equivalently, $\mathscr{M}_\infty(f) < \infty$.
\end{conj}
\subsubsection{Statement of the main results}
For a residual Galois representation $\overline{\rho}$, denote by $S_2(\overline{\rho})$ the set of newforms of weight two and trivial nebentypus whose residual representations are isomorphic to $\overline{\rho}$ and levels are not divisible by $p$.
For a given $f \in S_2(\overline{\rho})$, denote by $f_{\mathrm{min}} \in S_2(\overline{\rho})$ any congruent newform of minimal level via level lowering.
Denote by $\mathrm{Ram}(f, \overline{\rho})$ the set of primes $\ell$ dividing $N$ exactly and $\overline{\rho}$ is ramified at $\ell$.

The following statement is a slight refinement of Wei Zhang's result on Kolyvagin's conjecture \cite[Thm. 1.1 and Thm. 9.3]{wei-zhang-mazur-tate}.
Let $f =  \sum_{n \geq 1} a_n(f) q^n \in S_2(\Gamma_0(N))$ be a newform and $\pi$ a fixed prime of the Hecke field $F$ of $f$ lying above a prime $p \geq 5$ with $(N,p)=1$.
Let $\mathcal{K}$ be an imaginary quadratic field of discriminant $D_{\mathcal{K}}$ with $(N p, D_{\mathcal{K}}) = 1$.
Denote by $f_{\mathcal{K}}$ the twist of $f$ by the quadratic character associated to $\mathcal{K}$.
\begin{thm}[Kolyvagin's conjecture] \label{thm:main-thm-kolyvagin-conjecture}
Suppose that the triple $(f, \mathcal{K}, \pi)$ satisfies the following properties.
\begin{enumerate}
\item $\rho_f$ has large image.
\item $N^-$ is square-free with even number of prime factors.
\item $\mathrm{Ram}(f, \overline{\rho})$ contains all primes $\ell \vert N^-$ such that $\ell \equiv \pm 1 \pmod{p}$.
\item $\mathrm{Ram}(f, \overline{\rho})$ contains all primes $\ell$ such that $\ell \Vert N^+$.
\item If $F\neq \mathbb{Q}$, then  we have $\mathrm{H}^1(\mathbb{Q}_\ell, \overline{\rho} ) = \mathrm{H}^0(\mathbb{Q}_\ell, \overline{\rho} ) = 0$ for all prime $\ell$ with $\ell^2 \vert N^+$.
\end{enumerate}
If 
\begin{enumerate}
\item[(i)]
$\widetilde{\delta}_{m}(f_{\mathrm{min}}, 1) \neq 0$ and $\widetilde{\delta}_{m_{\mathcal{K}}}(f_{\mathcal{K}, \mathrm{min}} , 1)  \neq 0$
in $\mathbb{F}_\pi$\\
for some square-free products $m$ and $m_{\mathcal{K}}$ of Kolyvagin primes for $f_{\mathrm{min}}$ and $f_{\mathcal{K}, \mathrm{min}}$, respectively, and
\item[(ii)] 
${ \displaystyle \sum_{a \in (\mathbb{Z}/p\mathbb{Z})^\times} \overline{ \lambda^{\pm, \mathrm{opt}}(f, 1, a(1+bp), p^n) }  \neq 0, \displaystyle \sum_{a \in (\mathbb{Z}/p\mathbb{Z})^\times} \overline{ \lambda^{\pm, \mathrm{opt}}(f_{\mathcal{K}}, 1, a(1+b'p), p^n) }  \neq 0}$ 
 in $\mathbb{F}_\pi$\\
 for some $b, b' \in \mathbb{Z}/p^{n-1}\mathbb{Z}$ with $n \geq 1$, respectively,
\end{enumerate}
then $\ks^{\mathrm{Heeg}}$ is $\pi$-indivisible, i.e. $c^{\mathrm{Heeg}}_{1}(n)$ does not vanish for some $n \in \mathcal{N}^{\mathrm{ac}}$. Equivalently, $\mathscr{M}_\infty(f) = 0$.
Thus, Kolyvagin conjecture (Conjecture \ref{conj:kolyvagin-conjecture}) holds for $(f, \mathcal{K} , \pi)$, i.e. $\ks^{\mathrm{Heeg}}$ is non-zero. Equivalently, $\mathscr{M}_\infty(f) < \infty$.
\end{thm}
\begin{proof}
This is an application of Theorem \ref{thm:main_thm_main_conj} and Theorem \ref{thm:main_thm_mu_zero_conj} to the main result of \cite{wei-zhang-mazur-tate}.
The most argument in \cite{wei-zhang-mazur-tate} immediately applies to this setting.
In \cite{wei-zhang-mazur-tate}, Hypothesis $\heartsuit$ (Assumption \ref{assu:heart}) is imposed, which includes Condition (3), (4) and (5). The relaxation of  Hypothesis $\heartsuit$ is explained in \S\ref{subsubsec:relaxing-heartsuit}.
The $p$-ordinary condition is also imposed therein in order to invoke the $p$-part of the BSD formula for rank zero elliptic curves due to Kato and Skinner--Urban \cite{kato-euler-systems, skinner-urban}. See \cite[Thm. 7.1 and Thm. 7.2]{wei-zhang-mazur-tate}. The $p$-ordinary condition is replaced by (i) and (ii), and this replacement is explained in \S\ref{subsubsec:relaxing-p-ordinary}, especially, Proposition \ref{prop:kato-main-conj-to-bsd-formula-rank-zero}.
\end{proof}
One interesting feature of Theorem \ref{thm:main-thm-kolyvagin-conjecture} is that Kolyvagin conjecture on the $p$-indivisibility of Heegner point Kolyvagin systems (for non-ordinary primes) can also be verified numerically.
Another interesting aspect is that the $p$-indivisibility of Kato's Kolyvagin systems and the $p$-indivisibility of Heegner point Kolyvagin systems are tightly related although the natures of these Kolyvagin systems are \emph{very different}.
The following corollary is a conceptual version of Theorem \ref{thm:main-thm-kolyvagin-conjecture} explicitly revealing this aspect, which is more than the analogy.

Denote by $\ks^{\mathrm{Kato}}$ the Kato's Kolyvagin system associated to the canonical Kato's Euler system.
\begin{cor}[``derived" Perrin-Riou conjecture] \label{cor:derived-perrin-riou}
Let $(f, \mathcal{K} , \pi)$ be as in Theorem \ref{thm:main-thm-kolyvagin-conjecture}.
If
\begin{enumerate}
\item[(i)] $\ks^{\mathrm{Kato}}$ is $\pi$-indivisible for $f_{\mathrm{min}}$ and $f_{\mathcal{K}, \mathrm{min}}$ (``primitive" in the sense of \cite{mazur-rubin-book}) and
\item[(ii)] the cyclotomic $\mu$-invariants for fine Selmer groups vanish for $f$ and $f_{\mathcal{K}}$,
\end{enumerate}
then $\ks^{\mathrm{Heeg}}$ is $\pi$-indivisible for $(f, \mathcal{K} , p)$.
\end{cor}
\begin{rem}
Corollary \ref{cor:derived-perrin-riou} asserts that 
\emph{the mod $p$ non-vanishing of Kato's Kolyvagin systems implies the mod $p$ non-vanishing of Heegner point Kolyvagin systems 
(modulo some congruence argument via vanishing of the $\mu$-invariants).}
Thus, it can be regarded as a ``derived" variant of Perrin-Riou conjecture on the relation between Kato's zeta elements and Heegner points at the bottom layer \cite{perrin-riou-rational-pts}.
\end{rem}
\begin{rem}
In order to prove Kolyvagin conjecture at non-ordinary primes \emph{unconditionally}, we need the equality of the cyclotomic Iwasawa main conjecture for non-ordinary modular forms of weight two and \emph{general} level.
For example, the square-free level assumption excludes $f_{\mathcal{K}}$ since the level of $f_{\mathcal{K}}$ is $N \cdot \vert D_{\mathcal{K}} \vert^2$. (cf. \cite{wan-main-conj-ss-ec, sprung-main-conj-ss, castella-ciperiani-skinner-sprung}.) 
\end{rem}

The vanishing order $\mathrm{ord} ( \ks^{\mathrm{Heeg}} )$ of the Kolyvagin system $\ks^{\mathrm{Heeg}}$ is defined by the minimal number of prime factors of $n \in \mathcal{N}^{\mathrm{ac}}$ such that 
$c^{\mathrm{Heeg}}_M(n) \neq 0$ for some $M \leq M(n)$.
Let $\mathrm{Sel}(\mathcal{K}, W_f(1))^{\pm}$ denote the eigenspace with eigenvalue $\pm 1$ of $\mathrm{Sel}(\mathcal{K}, W_f(1))$ under
the complex conjugation\footnote{The reader should not confuse this with Kobayashi's signed Selmer groups.}.
Let $r^{\pm}_\pi(f/\mathcal{K})$ be the $\mathcal{O}_\pi$-corank of $ \mathrm{Sel}(\mathcal{K}, W_f(1))^{\pm}$.
Combining Theorem \ref{thm:main-thm-kolyvagin-conjecture} with Kolyvagin's theorem \cite{kolyvagin-selmer, masoero-structure}, the following result on the relation between Selmer coranks and the vanishing order of Kolyvagin systems easily follows.
\begin{thm}[Selmer coranks and $\mathrm{ord} ( \ks^{\mathrm{Heeg}} )$] \label{thm:selmer-coranks-vanishing-order}
Let $(f, \mathcal{K} , \pi)$ be as in Theorem \ref{thm:main-thm-kolyvagin-conjecture}.
If the mod $p$ numerical criteria (i) and (ii) in Theorem \ref{thm:main-thm-kolyvagin-conjecture} hold for $f$ and $f_{\mathcal{K}}$, then we have
\begin{enumerate}
\item $\mathrm{ord} ( \ks^{\mathrm{Heeg}} ) = \mathrm{max} \lbrace r^+_\pi(f/\mathcal{K}), r^-_\pi(f/\mathcal{K}) \rbrace - 1$,
\item $r^{\bullet}_\pi(f/\mathcal{K}) = \mathrm{ord} ( \ks^{\mathrm{Heeg}} ) +1$, and
\item $0 \leq \mathrm{ord} ( \ks^{\mathrm{Heeg}} ) - r^{ - \bullet}_\pi(f/\mathcal{K}) \equiv 0 \pmod{2}$,
\end{enumerate}
where $\bullet$ is the sign of $\epsilon(f/\mathbb{Q}) \cdot (-1)^{ \mathrm{ord} ( \ks^{\mathrm{Heeg}} ) +1 }$  and $\epsilon(f/\mathbb{Q})$ is the root number of $f$ over $\mathbb{Q}$.
\end{thm}
\begin{proof}
It essentially follows from Theorem \ref{thm:main-thm-kolyvagin-conjecture} and \cite[Thm. 4]{kolyvagin-selmer}.
Indeed, this is \cite[Thm. 1.2 and Thm. 11.2]{wei-zhang-mazur-tate} where the statements are given only for elliptic curves over $\mathbb{Q}$ since so is the setting of \cite{kolyvagin-selmer}. The generalization of the setting to (even higher weight) modular forms is now developed in \cite[Thm. 8.4]{masoero-structure}.
\end{proof}

\begin{thm}[$p$-converse to Gross--Zagier and Kolyvagin over $\mathcal{K}$] \label{thm:main-thm-converse-1}
Let $(f, \mathcal{K}, \pi)$ be as in Theorem \ref{thm:main-thm-kolyvagin-conjecture}, and
assume that the mod $p$ numerical criteria (i) and (ii) in Theorem \ref{thm:main-thm-kolyvagin-conjecture} hold for $f$ and $f_{\mathcal{K}}$ 
 If $\mathrm{cork}_{\mathcal{O}_\pi} \mathrm{Sel}(\mathcal{K}, W_f(1))=1$, then the Heegner point $y_{\mathcal{K}} \in A_f(\mathcal{K})$ is non-torsion. In particular, we have
$\mathrm{ord}_{s=1}(L(f/\mathcal{K}, s))  =1$, $\mathrm{dim}_{\mathbb{Q}}(A_f(\mathcal{K}) \otimes_{\mathbb{Z}} \mathbb{Q} ) = [F:\mathbb{Q}]$, and $\sha(A_f/\mathcal{K})$ is finite.
\end{thm}
\begin{proof}
See \cite[Proof of Thm. 1.3]{wei-zhang-mazur-tate}. The ``In particular" part follows from Gross--Zagier formula \cite{yuan-zhang-zhang} and Euler systems of Heegner points \cite{nekovar-euler-systems}.
\end{proof}
\begin{rem}
In Theorem \ref{thm:main-thm-converse-1}, $p$ does not need to split in $\mathcal{K}/\mathbb{Q}$ even in the non-ordinary case.
\end{rem}
The following structure theorem of $\sha(A_f/\mathcal{K})$ also follows from Theorem \ref{thm:main-thm-kolyvagin-conjecture}.
\begin{thm}[Structure of $\sha$] \label{thm:structure-sha-kolyvagin}
Let $(f, \mathcal{K}, \pi)$ be as in Theorem \ref{thm:main-thm-kolyvagin-conjecture}, and assume that the mod $p$ numerical criteria (i) and (ii) in Theorem \ref{thm:main-thm-kolyvagin-conjecture} hold for $f$ and $f_{\mathcal{K}}$.
If  $\mathrm{cork}_{\mathcal{O}_\pi} \mathrm{Sel}(\mathcal{K}, W_f(1))=1$, then
 we have
$$\sha(A_f/\mathcal{K})^{\pm}[\pi^\infty] \simeq \bigoplus_{i \geq 1} \mathrm{End}_{\mathbb{Z}}(A_f)/\pi^{a^{\pm}_i}\mathrm{End}_{\mathbb{Z}}(A_f)$$
with $a^{\pm}_1 \geq a^{\pm}_2 \geq \cdots$ determined by the following formula
\begin{align*}
a^{ \mathrm{sgn} \left( \epsilon(f/\mathbb{Q}) \cdot (-1)^{ \mathrm{ord} ( \ks^{\mathrm{Heeg}} ) +1 } \right) }_i
& = \mathscr{M}_{\mathrm{ord} ( \ks^{\mathrm{Heeg}} ) + 2i -1} - \mathscr{M}_{\mathrm{ord} ( \ks^{\mathrm{Heeg}} ) + 2i} , \\
a^{ \mathrm{sgn} \left( - \epsilon(f/\mathbb{Q}) \cdot (-1)^{ \mathrm{ord} ( \ks^{\mathrm{Heeg}} ) +1 } \right) }_i
& = \mathscr{M}_{\mathrm{ord} ( \ks^{\mathrm{Heeg}} ) + 2i -2} - \mathscr{M}_{\mathrm{ord} ( \ks^{\mathrm{Heeg}} ) + 2i -1} .
\end{align*}
\end{thm}
\begin{proof}
See \cite[Rem. 18]{wei-zhang-mazur-tate}, \cite[Thm. 1]{kolyvagin-selmer}, and \cite[Thm. 8.4]{masoero-structure}.
\end{proof}
In order to deduce the corresponding result for $f$ over $\mathbb{Q}$ from Theorem \ref{thm:main-thm-converse-1}, we need to make the following choice of auxiliary imaginary quadratic fields. 
\begin{choice} \label{choice:imag-quad-fields}
For a pair $(f, \pi)$, we choose two imaginary quadratic fields $\mathcal{K}_0$ and $\mathcal{K}_1$ satisfying the following conditions.
\begin{enumerate}
\item Both $(f, \mathcal{K}_0, \pi)$ and $(f, \mathcal{K}_1, \pi)$ satisfy all the conditions of Theorem \ref{thm:main-thm-kolyvagin-conjecture}.
\item $L(f_{\mathcal{K}_0}, 1) \neq 0$.
\item $L'(f_{\mathcal{K}_1}, 1) \neq 0$ when $\epsilon(f/\mathbb{Q}) = 1$.
\end{enumerate}
See \cite[Proof of Thm. 1.4]{wei-zhang-mazur-tate} for making such choices of $\mathcal{K}_0$ and $\mathcal{K}_1$.
\end{choice}
\begin{thm}[$p$-converse to Gross--Zagier and Kolyvagin over $\mathbb{Q}$] \label{thm:main-thm-converse-2}
Let $p$ be a prime $\geq 5$ and $f \in S_2(\Gamma_0(N))$ a newform.
Let $\mathcal{K}_0$ and $\mathcal{K}_1$ be the imaginary quadratic fields chosen in Choice \ref{choice:imag-quad-fields}.
Assume that $(f, \pi )$ satisfies the following statements.
\begin{enumerate}
\item $\rho_f$ has large image.
\item If $\ell \equiv \pm 1 \pmod{p}$ and $\ell$ divides $N$ exactly, then $\overline{\rho}_f$ is ramified at $\ell$.
\item  If $F\neq \mathbb{Q}$, then  we have $\mathrm{H}^1(\mathbb{Q}_\ell, \overline{\rho}_f ) = \mathrm{H}^0(\mathbb{Q}_\ell, \overline{\rho}_f ) = 0$ for all prime $\ell$ with $\ell^2 \vert N^+$.
\end{enumerate}
If the mod $p$ numerical criteria (i) and (ii) in Theorem \ref{thm:main-thm-kolyvagin-conjecture} hold for $f$, $f_{\mathcal{K}_0}$, and $f_{\mathcal{K}_1}$, then  we have the following statements.
\begin{enumerate}
\item[(a)]
If $\mathrm{cork}_{\mathcal{O}_\pi}\mathrm{Sel}(\mathbb{Q}, W_f(1))=1$, then 
$\mathrm{ord}_{s=1}(L(f, s)) =1$, $\mathrm{dim}_{\mathbb{Q}}(A_f(\mathbb{Q}) \otimes_{\mathbb{Z}} \mathbb{Q} ) = [F:\mathbb{Q}]$, and $\sha(A_f/\mathbb{Q})$ is finite.
\item[(b)]
If $\mathrm{ord}_{s=1} L(f,s) >1 $,
then $\mathrm{cork}_{\mathcal{O}_\pi} \mathrm{Sel}(\mathbb{Q}, W_f(1))$ is at least two (three, resp.) if $\epsilon(E/\mathbb{Q})$ is 1 ($-1$, resp.).
\end{enumerate}
\end{thm}
\begin{proof}
See \cite[Proof of Thm. 1.4]{wei-zhang-mazur-tate}. Note that (a) uses $\mathcal{K}_0$ and  (b) uses $\mathcal{K}_1$. 
\end{proof}

\begin{thm}[Rank one BSD formula for $A_f$ over $\mathcal{K}$] \label{thm:rank-one-bsd-over-K}
Let $(f, \mathcal{K}, \pi )$ be as in Theorem \ref{thm:main-thm-kolyvagin-conjecture} and
 assume that the mod $p$ numerical criteria (i) and (ii) in Theorem \ref{thm:main-thm-kolyvagin-conjecture} hold for $f$ and $f_{\mathcal{K}}$.
If $\mathrm{ord}_{s=1} L(f/\mathcal{K}, s) =1$, then the $\pi$-part of the BSD formula for $A_f$ over $\mathcal{K}$ holds, i.e.
$$ \mathrm{ord}_\pi  \left( \dfrac{L^*(A_f/\mathcal{K}, 1)}{ \Omega^{\mathrm{can}}_{A_f} \cdot \vert D_{\mathcal{K}} \vert^{-1/2} \cdot \mathrm{Reg}(A_f/\mathcal{K}) } \right) = 
\mathrm{length}_{\mathcal{O}_\pi} \sha(A_f/\mathcal{K})[\pi^\infty]
+ \mathrm{ord}_\pi \left(  \prod_{\ell \vert N^-} c_\ell(A_f/\mathcal{K}) \right)
$$
where 
$L^*(A_f/\mathcal{K}, 1)$  is the leading term of the Taylor expansion of $L(A_f/\mathcal{K}, s)$ at $s = 1$,
$\Omega^{\mathrm{can}}_{A_f}$ is the canonical period of $A_f$,
$\mathrm{Reg}(A_f/\mathcal{K}) := \dfrac{\langle y,y \rangle_{\mathrm{NT}}}{[A_f(\mathcal{K}): \mathrm{End}_{\mathbb{Z}}(A_f) \cdot y]^2}$ for any non-torsion $y \in A_f(\mathcal{K})$,
$\langle -,- \rangle_{\mathrm{NT}}$ is the N\'{e}ron--Tate height pairing,
 and 
$c_\ell(A_f/\mathbb{Q})$ and $c_\ell(A_f/\mathcal{K})$ are local Tamagawa factors.
\end{thm}
\begin{proof}
See \cite[Proof of Thm. 10.2]{wei-zhang-mazur-tate}. The Tamagawa factors at primes dividing $N^+$ are $\pi$-adic units due to Assumption (4) in Theorem \ref{thm:main-thm-kolyvagin-conjecture}.
\end{proof}
\begin{thm}[Rank one BSD formula for $A_f$ over $\mathbb{Q}$] \label{thm:rank-one-bsd-over-Q}
Let $(f, \pi)$ be as in Theorem \ref{thm:main-thm-converse-2} and  assume that the mod $p$ numerical criteria (i) and (ii) in Theorem \ref{thm:main-thm-kolyvagin-conjecture} hold for $f$ and $f_{\mathcal{K}_0}$ (as in Choice \ref{choice:imag-quad-fields}).
If $\mathrm{ord}_{s=1} L(f, s) =1$, then the $\pi$-part of the BSD formula for $A_f$ over $\mathbb{Q}$ holds, i.e.
$$\mathrm{ord}_{\pi}  \left( \dfrac{L^*(A_f, 1)}{ \Omega_{A_f} \cdot \mathrm{Reg}(A_f/\mathbb{Q}) } \right) = 
\mathrm{length}_{\mathcal{O}_\pi}\sha(A_f/\mathbb{Q})[\pi^\infty] +
 \mathrm{ord}_{\pi}  \left( \prod_{\ell \vert N} c_\ell(A_f/\mathbb{Q}) \right)
$$
where 
$\Omega_{A_f}$ is the real period of $A_f$ and
$\mathrm{Reg}(A_f/\mathbb{Q}) := \dfrac{\langle y,y \rangle_{\mathrm{NT}}}{[A_f(\mathbb{Q}): \mathrm{End}_{\mathbb{Z}}(A_f) \cdot y]^2}$ for any non-torsion $y \in A_f(\mathbb{Q})$.
\end{thm}
\begin{proof}
See \cite[Proof of Thm. 10.3]{wei-zhang-mazur-tate}.
\end{proof}

\subsubsection{Relaxing Hypothesis $\heartsuit$} \label{subsubsec:relaxing-heartsuit}
We recall ``Hypothesis $\heartsuit$" in \cite[(xv), Notations]{wei-zhang-mazur-tate}.
\begin{assu}[Hypothesis $\heartsuit$] \label{assu:heart}
Let $\mathrm{Ram}(f, \overline{\rho})$ be the set of primes $\ell$ dividing $N$ exactly and $\overline{\rho}$ is ramified at $\ell$.
\begin{enumerate}
\item $\mathrm{Ram}(f, \overline{\rho})$ contains all primes $\ell \vert N^-$ such that $\ell \equiv \pm 1 \pmod{p}$.
\item $\mathrm{Ram}(f, \overline{\rho})$ contains all primes $\ell$ such that $\ell \Vert N^+$
\item If $N$ is not square-free, then $\#\mathrm{Ram}(f, \overline{\rho}) \geq 1$.
\item If $N$ is not square-free, then $\mathrm{Ram}(f, \overline{\rho})$ contains a prime $\ell \mid N^-$ or there exist at least two prime factors $\ell \Vert N^+$.
\item  If $F\neq \mathbb{Q}$, then  we have $\mathrm{H}^1(\mathbb{Q}_\ell, \overline{\rho} ) = \mathrm{H}^0(\mathbb{Q}_\ell, \overline{\rho} ) = 0$ for all prime $\ell$ with $\ell^2 \vert N^+$.
\end{enumerate}
\end{assu}
The following remark explains how each condition in Hypothesis $\heartsuit$ is used in Wei Zhang's argument and shows (3) and (4) are unnecessary by improving a certain argument therein.
\begin{rem}
\begin{enumerate}
\item Assumption \ref{assu:heart}.(1) is important to have the mod $p$ multiplicity one of the character group of Shimura curves of discriminant divisible by $N^-$ in \cite[Lem. 3.3 and Thm. 4.3]{wei-zhang-mazur-tate}.
This is related to ``Mazur's principle". See also \cite{pw-mu} and \cite[(4.8)]{wei-zhang-mazur-tate}.
\item Assumption \ref{assu:heart}.(2) implies that all the Tamagawa factors at primes dividing $N^+$ are $\pi$-adic units.
The $\pi$-indivisibility of Heegner point Kolyvagin systems fails without this assumption. See \cite{jetchev-global-divisibility}.
\item Assumptions \ref{assu:heart}.(3) and (4) are used only in \cite[Thm. 6.4]{wei-zhang-mazur-tate} on the congruence number-Tamagawa exponent computation based on the approach of Ribet--Takahashi \cite{ribet-takahashi, takahashi-jnt}.
The approach of Ribet--Takahashi becomes subtle when the level is not square-free. See \cite{agashe-ribet-stein} for the comparison between modular degrees and congruence numbers.
Indeed, \cite[Thm. 6.4]{wei-zhang-mazur-tate} can be replaced by the work of the author with K.~Ota \cite{kim-ota} based on an $R=\mathbb{T}$ theorem without these assumptions. 
\item Assumptions \ref{assu:heart}.(5) is used only in \cite[Thm. 5.2]{wei-zhang-mazur-tate}.
It is automatic in the case of elliptic curves as proved in \cite[Lem. 5.1.(2)]{wei-zhang-mazur-tate}.
\end{enumerate}
\end{rem}
\subsubsection{Replacing the $p$-ordinary condition} \label{subsubsec:relaxing-p-ordinary}
As already pointed out in \cite[$\S$7.1]{wei-zhang-mazur-tate}, the $p$-ordinary assumption is only used in \cite{wei-zhang-mazur-tate} in order to invoke the $p$-part of the BSD formula for the level raising newforms of $f$ and $f_{\mathcal{K}}$ of level $Nq$ and $N \cdot \vert D_{\mathcal{K}} \vert^2 \cdot q$ at a 1-admissible prime $q$ (cf.~\cite[(xiv), Notations]{wei-zhang-mazur-tate}) whose Selmer coranks are zero. See \cite[Thm. 7.1 and Proof of Thm. 7.2]{wei-zhang-mazur-tate} how the work of Kato \cite{kato-euler-systems} and Skinner--Urban \cite{skinner-urban} are used.
This is crucial for the first step in the induction argument of the proof of Kolyvagin conjecture \cite[Thm. 9.1]{wei-zhang-mazur-tate}.

If we can verify the $p$-part of the rank zero BSD formula for the non-ordinary case, then all the results in \cite{wei-zhang-mazur-tate} immediately generalizes to the non-ordinary case.
\begin{prop} \label{prop:kato-main-conj-to-bsd-formula-rank-zero}
Suppose that $\overline{\rho}_f$ is irreducible and Kato's main conjecture holds for $T_f(1)$ at a good prime $p>2$ and $\mathrm{Sel}(\mathbb{Q}, W_f(1))$ is finite.
Then the $\pi$-part of the Bloch--Kato conjecture holds for $f$, i.e.
$$\mathrm{ord}_\pi \left( \dfrac{L(f,1)}{\Omega^+_f} \right) = \mathrm{length}_{\mathcal{O}_\pi} \left( \mathrm{Sel}(\mathbb{Q}, W_f(1)) \right) + \mathrm{ord}_\pi \left(  \prod_{\ell \vert N} c_\ell (A_f/\mathbb{Q})  \right) .$$
\end{prop}
\begin{proof}
We split the proof into two cases.
\begin{enumerate}
\item 
When $f$ is ordinary at $p$, it follows from the main conjecture of Mazur--Greenberg 
\cite[$\S$17.13]{kato-euler-systems} and the Euler characteristic of Selmer groups \cite[Thm. 4.1]{greenberg-lnm}. 
\item
When $f$ is non-ordinary at $p$, it also follows from the signed Iwasawa main conjecture \cite[Cor. 6.6]{lei-loeffler-zerbes_wach}, the interpolation formula of signed $p$-adic $L$-function at the trivial character  \cite[Prop. 3.28]{lei-loeffler-zerbes_wach}, and the Euler characteristic of signed Selmer groups \cite[Cor. 2.10]{ponsinet-signed}. See also \cite{castella-ciperiani-skinner-sprung}.
\end{enumerate}
\end{proof}

\bibliographystyle{amsalpha}
\bibliography{library}

\end{document}